\documentclass[11pt]{article}

\usepackage{etex}

\usepackage[utf8x]{inputenc}
\usepackage{microtype}
\usepackage[T1]{fontenc}
\usepackage{ae,aecompl}
\usepackage{times}
\usepackage{float}
\usepackage{mathrsfs}
\usepackage{verbatim,amsmath,amsthm,amsfonts,amssymb,latexsym,graphicx,mathtools,extpfeil,color,mathabx}
\usepackage{stmaryrd}
\usepackage{epstopdf,pinlabel}
\usepackage[all]{xy}
\usepackage{graphicx}
\usepackage{tabularx,caption}
\usepackage{subcaption}
\usepackage{tikz}
\usetikzlibrary{decorations.markings}
\usepackage{hyperref}
\hypersetup{
  colorlinks = true,
  linkcolor  = black
}
\usepackage[alphabetic,backrefs,msc-links]{amsrefs}
\DeclareMathAlphabet{\mathpzc}{OT1}{pzc}{m}{it}
\usepackage{accents}

\theoremstyle{plain}

\usepackage{aliascnt}
\numberwithin{equation}{section}

\newcommand\bbW{{\mathbb W}}
\newcommand\bbX{{\mathbb X}}
\newcommand\bbA{{\mathbb A}}
\newcommand\Ch{{\mathrm{Ch}}}

\newcommand{\CS}{{\rm CS}}

\newcommand{\rI}{{\rm I}}

\newcommand{\bi}{{\bf i}}

\newcommand{\bm}{{\bf m}}

\newcommand{\bv}{{\bf v}}

\newcommand{\bx}{{\bf x}}

\newcommand{\bA}{{\bf A}}

\newcommand{\bC}{{\bf C}}

\newcommand{\bH}{{\bf H}}

\newcommand{\bQ}{{\bf Q}}
\newcommand{\bR}{{\bf R}}

\newcommand{\bZ}{{\bf Z}}

\newcommand{\fc}{{\mathfrak c}}
\newcommand{\fd}{{\mathfrak d}}

\newcommand{\fl}{{\mathfrak l}}

\newcommand{\fr}{{\mathfrak r}}

\newcommand{\ft}{{\mathfrak t}}

\newcommand{\fB}{{\mathfrak B}}
\newcommand{\fC}{{\mathfrak C}}
\newcommand{\fD}{{\mathfrak D}}

\newcommand{\fF}{{\mathfrak F}}

\newcommand{\fW}{{\mathfrak W}}
\newcommand{\fX}{{\mathfrak X}}

\newcommand{\Z}{\bZ}

\newcommand{\Q}{\bQ}
\newcommand{\R}{\bR}
\newcommand{\C}{\bC}
\renewcommand{\H}{\bH}

\newcommand{\su}{\mathfrak{su}}

\newcommand{\gl}{\mathfrak{gl}}

\newcommand{\SU}{{\rm SU}}

\newcommand{\U}{{\rm U}}

\DeclareMathOperator{\Ad}{Ad}

\DeclareMathOperator{\Hol}{Hol}

\DeclareMathOperator{\PD}{PD}

\DeclareMathOperator{\ad}{ad}

\DeclareMathOperator{\rk}{rk}
\DeclareMathOperator{\coker}{coker}

\DeclareMathOperator{\ind}{index}

\DeclareMathOperator{\sign}{sign}

\DeclareMathOperator{\tr}{tr}

\renewcommand{\Im}{\operatorname{Im}}
\renewcommand{\Re}{\operatorname{Re}}
\renewcommand{\det}{\operatorname{det}}

\newcommand{\hol}{\mathrm{hol}}
\newcommand{\id}{{\rm id}}

\renewcommand{\epsilon}{\varepsilon}

\def\({\mathopen{}\left(}
\def\){\right)\mathclose{}}
\def\<{\mathopen{}\left<}
\def\>{\right>\mathclose{}}

\usepackage{multicol, color}

\definecolor{gold}{rgb}{0.85,.66,0}
\definecolor{cherry}{rgb}{0.9,.1,.2}
\definecolor{burgundy}{rgb}{0.8,.2,.2}
\definecolor{orangered}{rgb}{0.85,.3,0}
\definecolor{orange}{rgb}{0.85,.4,0}
\definecolor{olive}{rgb}{.45,.4,0}
\definecolor{lime}{rgb}{.6,.9,0}
\definecolor{green}{rgb}{.2,.7,0}
\definecolor{grey}{rgb}{.4,.4,.2}
\definecolor{brown}{rgb}{.4,.3,.1}

\def\makeautorefname#1#2{\AtBeginDocument{\expandafter\def\csname#1autorefname\endcsname{#2}}}

\newcommand{\mynewtheorem}[2]{
  \newaliascnt{#1}{equation}
  \newtheorem{#1}[#1]{#2}
  \aliascntresetthe{#1}
  \makeautorefname{#1}{#2}
}
\mynewtheorem{theorem}{Theorem}
\mynewtheorem{prop}{Proposition}
\mynewtheorem{cor}{Corollary}
\mynewtheorem{construction}{Construction}
\mynewtheorem{lemma}{Lemma}
\mynewtheorem{conjecture}{Conjecture}
\mynewtheorem{question}{Question}
\mynewtheorem{problem}{Problem}
\mynewtheorem{condition}{Condition}

\numberwithin{substep}{step}
\makeautorefname{step}{Step}
\makeautorefname{substep}{Step}

\numberwithin{subcase}{case}
\makeautorefname{case}{Case}
\makeautorefname{subcase}{case}

\theoremstyle{remark}
\mynewtheorem{remark}{Remark}

\theoremstyle{definition}
\mynewtheorem{definition}{Definition}
\mynewtheorem{example}{Example}
\mynewtheorem{exercise}{Exercise}
\mynewtheorem{convention}{Convention}
\newtheorem*{convention*}{Convention}
\newtheorem*{conventions*}{Conventions}

\makeautorefname{chapter}{Chapter}
\makeautorefname{section}{Section}
\makeautorefname{subsection}{Section}
\makeautorefname{subsubsection}{Section}

\DeclareMathOperator{\I}{I}

\tikzset{middlearrow/.style={
        decoration={markings,
            mark= at position 0.5 with {\arrow{#1}} ,
        },
        postaction={decorate}
    }
}

\title{\large \bf Surgery, Polygons and $\SU(N)$-Floer Homology}
\author{\bf \sc \large Lucas Culler, Aliakbar Daemi, Yi Xie}
\date{}
\begin{document}
\maketitle
\begin{abstract}
	Surgery exact triangles in various 3-manifold Floer homology theories provide
	 an important tool in studying and computing the relevant Floer homology groups. 
	 These exact triangles relate the invariants of 3-manifolds, obtained by 
	three different Dehn surgeries on a fixed knot. In this paper, the behavior of $\SU(N)$-instanton Floer homology with
	respect to Dehn surgery is studied. In particular, it is shown that 
	there are surgery exact tetragons and pentagons, respectively, for $\SU(3)$- and $\SU(4)$-instanton 
	Floer homologies. It is also conjectured that $\SU(N)$-instanton Floer homology in general admits a surgery 
	exact $(N+1)$-gon. An essential step in 
	the proof is the construction of a family of asymptotically cylindrical metrics on ALE spaces of type $A_n$. 
	This family is parametrized by the $(n-2)$-dimensional associahedron and consists of
	anti-self-dual metrics with positive scalar curvature. The metrics in the family also admit a torus symmetry.
\end{abstract}
\newpage
\tableofcontents
\newpage
\section{Introduction}

In 1985, Casson introduced an integer-valued invariant of integral homology spheres \cite{AM:Cass}. Casson's invariant for an integral homology sphere $Y$, often denoted by $\lambda(Y)$, was later reinterpreted by Taubes as an algebraic count of non-trivial flat $\SU(2)$-connections on $Y$ \cite{Cliff:Cas-invt}. Remarkably, Casson developed a series of formulas that describes how $\lambda$ behaves with respect to Dehn surgery on knots. These surgery formulas allowed him to identify the mod $2$ reduction of $\lambda$ in terms of {\it Rochlin invariant}. Consequently, he showed that  the Poincar\'e conjecture holds for an integral homology sphere whose Rochlin invariant is non-trivial.

Suppose $K$ is a knot in an integral homology sphere $Y$. For a pair of co-prime integer numbers $p$ and $q$, let $Y_{p/q}$ denote the result of performing $(p/q)$-surgery on $K$. The 3-manifolds $Y_{1/q}(K)$ are integral homology spheres and their Casson's invariants satisfy the following surgery formulas \cite{AM:Cass}:
\begin{equation} \label{SU(2)-surgery-formula}
	\lambda(Y_{1/(q+1)}(K))=\lambda(Y_{1/q}(K))+\frac{1}{2}\frac{d^2}{dt^2}\Delta_K(1)
\end{equation}
where $\Delta_K(t)$ is the symmetrized Alexander polynomial of $K$. Therefore, we have:
\begin{equation}\label{SU(2)-surgery-formula-2}
	\lambda(Y_{1/(q+2)}(K))-2\lambda(Y_{1/(q+1)}(K))+\lambda(Y_{1/q}(K))=0
\end{equation}
This identity implies that $\lambda(Y_{1/q}(K))=aq+b$ for constants $a$ and $b$ that depend only on $Y$ and $K$.

Casson invariant and surgery formulas \eqref{SU(2)-surgery-formula} and \eqref{SU(2)-surgery-formula-2} can be categorified in terms of Floer's instanton homology. For an integral homology sphere $Y$, Floer used {\it anti-self-dual} $\SU(2)$-connections on $\R \times Y$ to define a $\Z/8$-graded abelian group, which we denote by $\rI^2_*(Y)$ \cite{Fl:I}\footnote{Here the superscript $2$ stands for the choice of $\SU(2)$ as the structure group of the connections involved in the definition of instanton Floer homology.}. Floer also used a similar construction to define an instanton homology group $\rI^2_*(Y,\gamma)$ for a pair of a 3-manifold $Y$ and $\gamma\in H^2(Y,\Z)$ such that the pairing of $\gamma$ and an oriented embedded surface in $Y$ is an odd number \cite{F:sur-rel,DB:sur-rel}. A pair $(Y,\gamma)$ is called {\it admissible} if either it satisfies this condition or $Y$ is an integral homology sphere. For an integral homology sphere $Y$, the integer number $\lambda(Y)$ is half of the Euler characteristic of $\rI^2_*(Y)$ \cite{Cliff:Cas-invt,Fl:I}. Moreover, for a knot $K$ in $Y$, the instanton Floer homology groups $\rI^2_*(Y_{1/(q+1)}(K))$ and $\rI^2_*(Y_{1/q}(K))$ fit into an exact triangle as follows \cite{F:sur-rel,DB:sur-rel,Ch:I-Cube}:
\begin{equation}
	\label{surg-triangle}
	\xymatrix{
	\rI^2_*(Y_0(K),\gamma_0) \ar[rr]& &
	 \rI^2_*(Y_{1/(q+1)}(K)) \ar[dl]\\
	& \rI^2_*(Y_{1/q}(K)) \ar[ul] &
	}
\end{equation}
Here $Y_0(K)$, the $0$-surgery on $K$, has the same homology as $S^1 \times S^2$, and $\gamma_0$ is a generator of $H^2(Y_0(K),\Z)$. The Euler characteristic of the instanton Floer homology group of the admissible pair $(Y_0(K),\gamma_0)$ is equal to $-\frac{d^2}{dt^2}\Delta_K(1)$. In the above diagram, the horizontal map has degree $1$ and the remaining two maps have degree $0$. Therefore, \eqref{SU(2)-surgery-formula} is an immediate consequence of Floer's {\it surgery exact triangle}. There is also an exact triangle in correspondence to the identity in \eqref{SU(2)-surgery-formula-2} which has the form in \eqref{surg-triangle-2}. (See the discussion following Theorem \ref{surgery-polygon-thm}.)
\begin{equation} \label{surg-triangle-2}
	\xymatrix{
	\rI^2_*(Y_{1/q}(K)) \ar[rr]& &
	 \rI^2_*(Y_{1/(q+2)}(K)) \ar[dl]\\
	& \rI^2_*(Y_{1/(q+1)}(K))\oplus \rI^2_*(Y_{1/(q+1)}(K)) \ar[ul] &
	}
\end{equation}

Both Casson invariant and instanton Floer homology have various extensions obtained by replacing $\SU(2)$ with special unitary groups of higher rank. Flat $\SU(3)$ connections on an integral homology sphere were used in several ways to define analogues of Casson invariant \cite{BH:3-Casson, BHK:3-Casson,CLM:3-Casson}. In \cite{KM:YAFT}, Kronheimer and Mrowka used $\SU(N)$-Yang-Mills gauge theory to define instanton Floer homology $\rI^N_*(Y,\gamma)$ for an {\it $N$-admissible pair} of a 3-manifold $Y$ and a cohomology class $\gamma \in H^2(Y,\Z)$. For $N\geq 3$, $(Y,\gamma)$ is $N$-admissible, if there is an oriented embedded surface in $Y$ whose pairing with $\gamma$ is coprime to $N$. A pair is $2$-admissible, if it is admissible in the above sense. If $N\geq 3$ and $Y$ is an integral homology sphere, then $\gamma$ is necessarily trivial and $(Y,\gamma)$ is not $N$-admissible. Therefore, $\SU(N)$-instanton Floer homology is not presently defined for integral homology spheres and the relationship between $\SU(3)$ Casson invariants and Floer homology is not clear.

Surgery formulas and surgery exact triangles as above are ubiquitous in low dimensional topology. Motivated by Floer's exact triangle, analogous exact triangles were constructed for other 3-manifold Floer homologies \cite{OzSz:HF-app,KMOS:mon-lens}. A closely related exact triangle is also built into the definition of Khovanov homology \cite{Kh:Jones-cat}.
The Jones polynomial \cite{Jones:Jones} or more generally the HOMFLY polynomial \cite{HOMFLY:HOMFLY} are also defined by analogues of surgery formula \eqref{SU(2)-surgery-formula} in the context of knot invariants. Our main concern in this article is to investigate how $\SU(N)$-instanton Floer homologies behave with respect to Dehn surgery.

The values of $\SU(3)$-Casson invariant constructed in \cite{BHK:3-Casson} have been computed for Dehn surgeries on torus knots in $S^3$ \cite{BHK:3-Casson-comp}. In particular, it is shown that this invariant, denoted by $\tau$, has the following form:
\[
  \tau(S^3_{1/q}(K))=aq^2+bq
\]
which is in contrast to the linear behavior of $\lambda$. However, $\tau$ for these special choices of $K$ satisfy the following modified version of \eqref{SU(2)-surgery-formula-2}:
\[
  \tau(Y_{1/q}(K))-3\tau(Y_{1/(q+1)}(K))+3\tau(Y_{1/(q+2)}(K))-\tau(Y_{1/(q+3)}(K))=0
\]
Therefore, one might hope that there is a Casson type invariant $\lambda_N$ of integral homology spheres, defined by counting $\SU(N)$ flat connections, which satisfies the following surgery formula:
\begin{equation} \label{order-N-surgery}
	\sum_{i=0}^{N} \binom{N}{i}(-1)^i \lambda_N(Y_{1/(q+i)}(K))=0
\end{equation}
In this article, we construct a {\it surgery exact polygon} for $\SU(N)$-Floer homology in the case that $N=2$, $3$ or $4$. This surgery exact polygon can be regarded as a categorification of \eqref{order-N-surgery}.

\subsection{Statement of Results}

Fix an integer $N\geq 2$, a 3-manifold $\mathcal Y$ with torus boundary and an oriented 1-manifold $w$  embedded in the interior of $\mathcal Y$. Let $\lambda$ and $\mu$ be two curves on the boundary of $\mathcal Y$ such that $\lambda\cdot \mu$, the algebraic intersection of $\lambda$ and $\mu$, is equal to $1$. Let $Y_i$ be the closed 3-manifold obtained by gluing the solid torus $S^1\times D^2$ such that $\{{\rm point}\}\times \partial D^2$ is glued along a loop in $\partial Y$ homologous to $i\mu+\lambda$. We also denote $K_i$ for the core of the solid torus $S^1\times D^2$ in $Y_i$. For each subset $S$ of $\{0,1,\dots,N-1\}$ with $i$ elements, consider the following cohomology class in $Y_i$:
\[
  \gamma_S={\rm P.D.}(w+\sigma_S\cdot K_i) \in H^2(Y_i,\Z)
\]
where $\sigma_S$ denotes the sum of the elements of $S$. The main result of the present pair can be summarized as in Theorem \ref{surgery-polygon-thm}. A more detailed version of this theorem is stated as Theorem \ref{main} in Section \ref{Floer-main}.
\begin{theorem}\label{surgery-polygon-thm}
	Suppose $\mathcal Y$ and $w$ are chosen such that $(Y_i,\gamma_S)$ is $N$-admissible for any $i$ and any set
	$S\subseteq \{0,1,\dots,N-1\}$ with $i$ elements. If $N\leq 4$, then $\SU(N)$-instanton Floer homology
	associates an exact $(N+1)$-gon to $(\mathcal Y,w)$. The chain complex associated to the $i^{ th}$ vertex of this
	exact $(N+1)$-gon is equal to:
	\[
	  \bigoplus_{|S|=i} \fC^N_*(Y_i,\gamma_S)
	\]
	where $ \fC^N_*(Y_i,\gamma_S)$ denotes the instanton Floer complex
	associated to the admissible pair $(Y_i,\gamma_S)$.
\end{theorem}
\noindent
An exact $n$-gon consists of $n$ chain complexes and homomorphisms between any pair of chain complexes that are supposed to satisfy a series of identities. The definition of exact $n$-gons, as generalization of exact triangles, is discussed in Section \ref{exact-polygon}. In general, we also believe the following conjecture holds:
\begin{conjecture}\label{surgery-polygon-conj}
	Theorem \ref{surgery-polygon-thm} holds without the assumption $N\leq 4$.
\end{conjecture}

Suppose $\mathcal Y$ is given by the complement of a knot $K$ in an integral homology sphere $Y$. Let $m$ be the meridian of $K$ and $l$ be the longitude of $K$ fixed by a Seifert surface and oriented such that $l\cdot m=1$. Let $\lambda$ and $\mu$ be closed curves on the boundary of $\mathcal Y$ whose homology classes are given as below:
\[
  [\lambda]=[m+(q+2)l]  \hspace{1cm} [\mu]=[-l]
\]
Then $Y_i$ is diffeomorphic to the integral homology sphere $Y_{1/(q+2-i)}(K)$, the manifold obtained by $1/(q+2-i)$-surgery on $K$. If we assume $N=2$, then all pairs $(Y_i,\gamma_S)$ are admissible. Therefore, \ref{surgery-polygon-thm} gives rise to an exact triangle of the Floer chain complexes of the 3-manifolds $Y_{1/(q+2)}(K)$, $Y_{1/(q+1)}(K)$ and $Y_{1/q}(K)$. As it is expected, any exact triangle, in the sense defined in Section \ref{exact-polygon}, induces an exact sequence of homology groups. In this case, the resulting exact sequence of Floer homology groups is given in \eqref{surg-triangle}.

For $n\geq 4$, an exact $n$-gon does not induce an $n$-periodic exact sequence at the level of homology groups. Nevertheless, there are spectral sequences associated to an exact $n$-gon generalizing the exact sequence of homology groups for $n=3$. These spectral sequences are discussed in Proposition \ref{spectral-sequence}. An immediate consequence of Proposition \ref{spectral-sequence} is the following corollary of Theorem \ref{surgery-polygon-thm}:

\begin{cor}
	Suppose $N$, $\mathcal Y$ and $w$ are given as in Theorem \ref{surgery-polygon-thm}. Then:
	\vspace{-5pt}
	\begin{itemize}
		\item[(i)] there is a spectral sequence that converges to the trivial group and its second page is equal to:
				\[
	  			  \bigoplus_{0\leq i\leq N}\bigoplus_{|S|=i} \rI^N_*(Y_i,\gamma_S);
				\]
		\item[(ii)] there is a spectral sequence that converges to $\rI^N_*(Y_0,\gamma_{\emptyset})$ and
		its second page is equal to:
				\[
	  			  \bigoplus_{1\leq i\leq N}\bigoplus_{|S|=i} \rI^N_*(Y_i,\gamma_S).
				\]		
	\end{itemize}
\end{cor}

Theorem \ref{surgery-polygon-thm} also has a consequence at the level of Euler characteristics. The instanton Floer homology group $\rI_*^N(Y,\gamma)$ has a canonical $\Z/2\Z$-grading. (See Subsection \ref{inst-floer}.) In particular, we can use this grading to define Euler characteristic $\lambda_N(Y,\gamma)$ for an $N$-admissible pair $(Y,\gamma)$.
\begin{cor}\label{Euler-char}
	Suppose $N$, $\mathcal Y$ and $w$ are given as in Theorem \ref{surgery-polygon-thm}.
	Let $\nu$ denote a generator of the kernel of the map $H_1(\partial \mathcal Y) \to H_1(\mathcal Y)$. Then we have:
	\begin{equation} \label{Euler-char-equ}
	  \sum_{i=0}^N(-1)^{i+\epsilon_i} \sum_{|S|=i} \lambda_N(Y_i,\gamma_S)=0.
	\end{equation}
	If $N$ is odd $\epsilon_i=0$, and if $N$ is even, $\epsilon_i=\sum_{j=1}^i\epsilon_j'$ where:
	\begin{equation}\label{epsilon-i}
		\epsilon_j'=\left \{
		\begin{array}{ll}
			0&\text{if }\sign((\lambda+j\mu)\cdot \nu)=\sign((\lambda+(j-1)\mu)\cdot \nu);\\
			1&\text{if }\sign((\lambda+j\mu)\cdot \nu)=-\sign((\lambda+(j-1)\mu)\cdot \nu);\\
			0&(\lambda+(j-1)\mu)\cdot \nu=0;\\
			1&(\lambda+j\mu)\cdot \nu=0.\\
		\end{array}
		\right.
	\end{equation}
	In particular, if $\nu=\lambda$ or $\mu$, then all constants $\epsilon_i$ are zero.
\end{cor}
\noindent
Hypothetically assume that the definition of $\SU(N)$-instanton Floer homology can be extended to integral homology spheres as finitely generated abelian group such that Theorem \ref{surgery-polygon-thm} continues to hold. Then we would be able to apply Corollary \ref{Euler-char} to knot complements in integral homology spheres and derive \eqref{order-N-surgery}.

\subsection{Outline of Contents}
To prove Theorem \ref{surgery-polygon-thm}, we need to construct a series of maps between various instanton Floer chain complexes associated to the 3-manifolds $Y_j$. These maps are defined with the aid of a family of $4$-manifolds $W^j_k$, which are constructed in Section \ref{topology}. We shall define a family of metrics on each $W^j_k$ that is parametrized by an {\it associahedron}. Section \ref{metrics} of the paper is devoted to the definition of this family of metrics. The main result of this relatively long section is summarized in Theorem \ref{Wmetrics}. An impatient reader may want to focus only on understanding the statement of Theorem \ref{Wmetrics} after glancing at the rest of Section \ref{metrics}.

Counting solutions of the {\it Anti-Self-Duality} equation with respect to the families of metrics on 4-manifolds $W^j_k$  gives rise to the definition of maps of the exact $(N+1)$-gon in Theorem \ref{surgery-polygon-thm}. A priori, these maps satisfy the identities required for an exact $(N+1)$-gon with some additional terms. The additional terms, which we need to show that they sum up to zero, can be divided into two parts. In order to obtain vanishing of the first series of maps, we have to make a careful choice of our families of metrics in Section \ref{metrics}. We achieve this goal using {\it Gibbons-Hawking} metrics on the {\it ALE spaces} of type $A_{N-1}$. The main part of the proof of Theorem \ref{surgery-polygon-thm} is to show that the additional terms of the second type give us the zero map. The essential step to carry out this task is given in Subsection \ref{comp-red}. In the remaining subsections of Section \ref{ASD-mod-spaces}, we review some general facts about the moduli space of ASD equation.

The proof of the main theorem is given in Section \ref{Floer-main}. Section \ref{perturbations} is devoted to some technical results about regularity of the moduli spaces involved in our construction. Most steps in the proof of Theorem \ref{surgery-polygon-thm} can be replicated to address Conjecture \ref{surgery-polygon-conj}. Thus we work with arbitrary $N$ for the most of the paper. Only in Subsection \ref{XNl-reg}, we have to limit ourselves to the case that $N\leq 4$. Reproducing the result of this subsection for arbitrary $N$ seems to require new gluing techniques for the moduli spaces of ASD connections, and we hope to address that issue elsewhere. The statement of our main theorem is motivated by the role of physicists' Coulomb branch in $\SU(N)$-instanton Floer homology. We also hope to discuss this circle of ideas in a separate paper.

\subsection{Conventions}

We use the following orientation conventions throughout the paper. Let  $M$ be a manifold with boundary. Unless otherwise specified, the boundary of $M$ is oriented with the outward-normal-first convention. If $L$ is a $\U(1)$-bundle over an oriented manifold, we use the fiber-first convention to orient the total space $L$. In this paper, $L(k,1)$, as an oriented 3-manifold, is identified with the $\U(1)$-bundle of Euler number $k$ over the 2-dimensional sphere. If $M$ and $N$ are two oriented manifolds, then $M\times N$ is oriented by the first-factor-first convention. Unless otherwise is specified, all manifolds in the present article are oriented. We will also write $-M$ for the manifold $M$ with the reverse orientation.

In various places in the paper, we need bump functions to glue differential forms, metrics, etc. on different manifolds. Thus we fix two smooth functions:
\begin{equation}\label{phi-1-2}
	\varphi_1, \varphi_2:[0,1] \to [0,1]
\end{equation}
such that $\varphi_1+\varphi_2=1$, $\varphi_1(x)=1$ on $[0,\frac{1}{3}]$ and $\varphi_2(x)=1$ on $[\frac{2}{3},1]$. These functions will be used in several places in the paper.

We will write $[n]$ for the set of integers $\{0,1,2,\dots,n-1\}$. In general, if we write $\{x_1,x_2, \dots,x_n\}$ for a finite subset of $\R$, then we assume that $x_1<x_2<\dots<x_n$.

In what follows, all chain complexes, $A_{\infty}$-categories, etc. are defined over $\Z/2\Z$. In particular, Theorem \ref{surgery-polygon-thm} is proved only for instanton Floer homology with coefficients in $\Z/2\Z$. We hope this allows the main geometrical ideas of the paper to stand out. Although we have not checked the details, we believe studying orientations of the moduli spaces involved in the proof of Theorem \ref{surgery-polygon-thm} would give rise to an extension of Theorem for arbitrary coefficient ring. 

{\it Acknowledgements.}
We thank Paul Kirk, Tomasz Mrowka and Paul Seidel for helpful conversations. We are very grateful to the  {\it Simons Center for Geometry and Physics} for providing an opportunity to work on this project. We especially would like to thank Simon Donaldson, Kenji Fukaya and John Morgan who organized the program {``\it Mathematics and Physics of Gauge Fields''}.

\section{Homological Algebra of the Surgery Polygon} \label{exact-polygon}
In this section, we define exact $n$-gons and exact $n$-cubes. It is shown in Subsection \ref{n-cube} that any exact $n$-cube induces an exact $(n+1)$-gon. We shall obtain the exact $(N+1)$-gon of Theorem \ref{surgery-polygon-thm} by constructing an exact $N$-cube and then applying this algebraic construction.
\subsection{Exact $n$-gons}
A slight variation of the following proposition, commonly called {\it triangle detection lemma}, appears in Oszv\'ath and Szab\'o's pioneering work \cite{OS:HF-br}. The original idea seems to go back to Seidel and Kontsevich.

\begin{prop}\label{3-detec}
	For each $i \in \Z/3\Z$, let $(C_i,d_i)$ be a chain complex, with homology $H_i$.
	Suppose that for all $i\in \Z/3\Z$ we are given maps
	\[
	  f_i:C_i \to C_{i+1} \hspace{.5cm} g_i: C_i \to C_{i+2} \hspace{.5cm} h_i: C_i \to C_{i+3},
	\]
	which satisfy the following properties:
	\begin{eqnarray*}
		d_i^2&=& 0 \label{identity-0}\\
		d_{i+1}  f_i + f_i  d_i  &=& 0 \label{identity-1}\\
		d_{i+2} g_i + f_{i+1}  f_i + g_i  d_i &=& 0 \label{identity-2} \\
		d_{i+3} h_i + f_{i+2}  g_i + g_{i+1}  f_i + h_i  d_i &=& 1. \label{identity-3}
	 \end{eqnarray*}
	Then the map
	\[
	  C_i \xrightarrow{(f_i,g_i)} \text{Cone}(f_{i+1}):=(C_{i+1}\oplus C_{i+2}[1],d_{i+1}+d_{i+2}+f_{i+1})  \\
	\]
	is a chain map which induces an isomorphism on homologies.
	Moreover, the following sequence of homology groups is exact:
	\[
	  \xymatrix{
	  H_i \ar[rr]^{(f_i)_*}& &
	  H_{i+1}\ar[dl]^{(f_{i+1})_*}\\
	  & H_{i+2} \ar[ul]^{(f_{i+2})_*} &
	  }
	\]
\end{prop}
Proposition \ref{3-detec} motivates the following definition:
\begin{definition}
	An \emph{exact n-gon} consists of:
	\begin{itemize}
		\item[(i)] a chain complex $(C_j,d_j)$ for each $j\in \Z$ such that $(C_{j+n},d_{j+n})=(C_j,d_j)$;
		\item[(ii)] a homomorphism $f^j_k$ for each  pair $(j,k)$ of the elements of $\Z$ with $j< k \leq n+j$
		such that $f_{j+n}^{i+n}=f_j^i$.
	\end{itemize}	
	Moreover, the above homomorphisms are required to satisfy the following conditions:
	\begin{equation}\label{exact-n-gon}
	   \sum_{j<k<l}  f_l^k f_k^j =
	   \left\{
	   \begin{array}{lr}
	   	d_l f_l^j + f_l^j d_j  &  l - k < n \\
		d_l f_l^j + f_l^j d_j + 1 & l - k = n
	   \end{array}\right.
	 \end{equation}
\end{definition}
\begin{prop}\label{Ngon}
	Let $(C_j, f_k^j)$ be an exact $n$-gon, let $C = \bigoplus_{j=0}^{n-1} C_j$,
	and define an endomorphism of $C$ by:
	\[
	  D = \sum_{0\leq j \leq n-1} d_j + \sum_{0\leq j< k \leq n-1} f_k^j .
	\]
	Then $(C,D)$ is an acyclic complex.
	Moreover, if $C_i'=C_{i+1}\oplus\cdots \oplus C_{i+n-1}$ and the endomorphism $D_i'$ of $C_i'$ is defined by:
    	\[
	 D_i'=\sum_{i< j < n+i} d_j + \sum_{i< j<k<n+i} f_k^j,
	\]
	then $(C_i',D_i')$ is homotopy equivalent to $(C_i,d_i)$.
\end{prop}

\begin{proof}
	Since the homomorphisms $d_j$ are differential and the maps $f^j_k$ satisfy \eqref{exact-n-gon}, the endomorphism $D$ is a differential.
	Consider the endomorphism of $C$ given by
	\[ K = \sum_{0 \leq j \leq k\leq n-1} f_{n+j}^k. \]
	Identities in \eqref{exact-n-gon} imply that $DK + KD-{\rm Id}$ is a nilpotent endomorphism of $C$, hence $(C,D)$ is acyclic.
	
	For the second part, without loss of generality, we can assume that $i=0$.
	We can again use \eqref{exact-n-gon} to show that $D_0'$ is a differential and the maps $F:C_0 \to C'_0$ and $G:C'_0 \to C_0$, defined as below, are chain maps:
	\[
	  F=\left(
	  \begin{array}{c}
	 	f^0_1\\
		f^0_2\\
		\vdots\\
		f^0_{n-1}
	  \end{array}
	   \right)\hspace{1cm}
	  G=\left(
	  \begin{array}{cccc}
	 	f^1_n&f^2_n&\dots&f^{n-1}_n
	\end{array}
	  \right)
	\]
	The chain maps $F$ and $G$ determine a homotopy equivalence of $C_0$ and $C'_0$. The map $GF$ is homotopic to ${\rm Id}$ using the homotopy given by $f^0_n$.
	The map $FG$ is homotopic to an automorphism of $C'_0$ using the homotopy:
	\[ K = \sum_{1 \leq j \leq k\leq n-1} f_{n+j}^k. \]
\end{proof}

\begin{remark}
	Note that an exact $n$-gon need not give rise to an exact sequence on homology!
	For instance, we could consider the $4$-gon
	\[ C_0=C_2=\Z/2\Z \hspace{1cm} C_1=C_3=0 \]
	where all the differentials are equal to zero. We also define $f^{j+2}_j$ to be $1$ and all the remaining maps
	to be zero.
\end{remark}
\begin{cor}\label{spectral-sequence}
	Let $(C_j, f_k^j)$ be an exact $n$-gon. Then:
	\begin{itemize}
	\item[(i)] there is a spectral sequence that converges to the trivial vector space and its second page is equal to:
       \[
         \bigoplus_{0\leq j \leq n-1} H_*(C_j).
       \]
       The differential on this page is given by:
       \[
         \sum_{0\leq j \leq n-2}(f_{j+1}^j)_*;
       \]
  	\item[(ii)] there is a spectral sequence that converges to $H_*(C_i)$ and its second page is equal to:
	       \[
         \bigoplus_{i< j<n+i} H_*(C_j).
       \]
       The differential on this page is given by:
       \[
         \sum_{i< j<n+i-1}(f_{j+1}^j)_*.
       \]
	\end{itemize}
\end{cor}
\begin{proof}
	Take the filtration $F_0\supset F_1\supset \cdots \supset F_{n-1}$ of $(C,D)$ where $F_i=C_i\oplus\cdots \oplus C_{n-1}$. The spectral sequence associated to this filtration has
	the desired properties in the first part by Proposition \ref{Ngon}.
	For the second part, without loss of generality, we can assume that $k=0$. Then the filtration  $F_1\supset \cdots \supset F_{n-1}$ on $C_0'$ from Proposition \ref{Ngon}
	gives the desired spectral sequence.
\end{proof}

\begin{cor}\label{euler-char-exact-polygon}
	Let $(C_j, f_k^j)$ be an exact $n$-gon.
	Suppose that each $C_j$ admits a $\Z/2\Z$-grading
	such that $f_k^j$ has degree $k - j - 1$ whenever $0 \leq j < k \leq n-1$.
	Let $\chi(C_j)$ denote the Euler characteristic of $C_j$.  Then:
\[ \sum_{j=0}^{n-1} (-1)^{j} \chi(C_j) = 0 .\]
\end{cor}
\begin{proof}
	The differential $D$ on $C = \bigoplus_{j=0}^{n-1} C_j[ j ]$ has odd degree.
	Since $(C,D)$ is acyclic,
	\[  \sum_{j=0}^{n-1} (-1)^j \chi(C_j)  = \chi(C) = 0. \]
\end{proof}

\subsection{Exact $n$-cubes}\label{n-cube}
Next, we describe a structure called an exact $n$-cube, from which an exact $(n+1)$-gon can be constructed.
We first define a directed graph $G_n$ with vertices labeled by the subsets of $[n]$ and edges:
\vspace{-5pt}
\begin{itemize}
  \item $\beta_{S,i}:S\to S\sqcup \{i\}$ whenever $i\notin S\subset [n]$.
  \item  $\delta : [n] \to \emptyset $ which is called the {\it connecting edge}.
\end{itemize}
A length $k$ (directed) \emph{path} in $G_n$ is defined as a sequence of $k$ consecutive edges. We only deal with paths with length at most $n+1$. These paths can be divided into two types depending on whether they contain the connecting edge or not. A path of the first type can be described by a pair $(S,\sigma)$ where $S\subset [n]$ and $\sigma :[k]\to [n]\backslash S$ is an injection. This path is the sequence of
edges $$S\to S\sqcup \{\sigma(0)\}\to S\sqcup \{\sigma(0), \sigma(1)\}\to \cdots \to S\sqcup \Im \sigma$$
A path of the second type can be described by a pair $(\sigma,\tau)$ where $\sigma:[k]\to [n]$ and $\tau:[l]\to [n]$ are injections with disjoint images. This path is the sequence of
edges:
\[
  S=[n]\backslash \Im \sigma \to S\sqcup \{\sigma(0)\}\to S\sqcup \{\sigma(0),\sigma(1)\}\to \cdots  \to [n] \xrightarrow{\delta} \emptyset \to \{\tau(0)\}
\to \cdots \to \Im\tau
\]

Suppose we are given a chain complex $(C_S,d_S)$, for each vertex $S$, and a homomorphism $f_q:C_S\to C_T$, for each path $q$ from a vertex $S$ to a vertex $T$. If $q$ is a path of the first type (respectively, second type) corresponding to the pair $(S,\sigma)$ (respectively, $(\sigma,\tau)$), then $f_q$ is also denoted by $f_{S,\sigma}$ (respectively, $f_{\sigma,\tau}$). For any two vertices $S$ and $T$, which are connected to each other by a path of length at most $n+1$, define:
$f^S_T :C_S\to C_T$ as follows:
\[
  f^S_T:=\sum_{q:S\to T} f_q.
\]
\begin{definition}\label{ncube}
	We call $(\{C_S\}, \{f_q\})$ an exact $n$-cube if the following equalities hold:
	\begin{align}
	   df^S_T+f^S_Td&=\sum_{ S\subsetneq R\subsetneq T} f^R_T f^S_R,
	   \hspace{1cm}~~~~~~~\text{if}~ S \subsetneq T;  	\label{equ1}\\
	   df^S_T+f^S_Td&=\sum_{\substack{S\subsetneq R \\ \text{or}~R\subsetneq T}}f^R_T f^S_R , 	
	   \hspace{1cm} ~~~~~~~~\text{if}~|T|\le |S|~\text{and}~ S\neq T; \label{equ2}\\
	  df^S_T+f^S_Td&=\text{1}+\sum_{\substack{ S\subsetneq R \\ \text{or}~R\subsetneq T}}f^R_T f^S_R , 		\hspace{1cm} ~\text{if}~S=T   \label{equ3}.
	\end{align}
\end{definition}

Any exact $n$-cube gives rise to an exact $(n+1)$-gon. Let $C_j:=\bigoplus_{|S|=j} C_S$ where $0\le j\le n$ and define $f^j_k : C_j\to C_k$ by:
\[ f_k^j := \sum_{\substack{|S| = j \\ |T| = k}} f^S_T \]
The the chain complexes $C_j$ and the maps $f_k^j:C_j \to C_k$ define an exact $(n+1)$-gon.

\subsection{$A_\infty$-categories and Exact Polygons}
We can frame the triangle detection lemma in a more conceptual way using the language of $A_{\infty}$ categories.  For $i\in \Z/3\Z$, suppose that $X_i$ are objects in a strictly unital $A_{\infty}$-category $\mathcal{A}$, and suppose we are given morphisms $\alpha_i: X_i \to X_{i+1}$, whose compositions satisfy the following conditions: \label{a3conditions}
\begin{eqnarray*} \mu_1(\alpha_i) &=& 0 \\
\mu_2(\alpha_{i+1},\alpha_i) &=& 0 \\
\mu_3(\alpha_{i+2},\alpha_{i+1},\alpha_i) &=& 1 \end{eqnarray*}
Then for any strictly unital $A_{\infty}$ functor $\mathcal{F}:\mathcal{A} \to \Ch$, we obtain complexes $C_i = \mathcal{F}(X_i)$ and morphisms between these complexes, as in the triangle detection lemma:
\begin{eqnarray*} f_i &=& \mathcal{F}^1(\alpha_i) \\
 g_i &=& \mathcal{F}^2(\alpha_{i+1},\alpha_i) \\
  h_i &=& \mathcal{F}^3(\alpha_{i+2},\alpha_{i+1},\alpha_i). \end{eqnarray*}
The identities stated in the triangle detection follow precisely from the identities required for $\mathcal{F}$ to be an $A_{\infty}$ functor (given our assumptions on compositions of the maps $\alpha_i$):
\begin{eqnarray*} \sum_{p+q - 1 = k }  \mathcal{F}^p(\alpha_1,\dots, \mu_q( \alpha_{i+1},\dots,\alpha_{i+q}),\dots,\alpha_{k})   \hspace{4 cm}\\
  \hspace{4 cm} = \sum_{i_1+\cdots + i_r = k} \mu_r(\mathcal{F}^{i_1}(\alpha_1,\dots,\alpha_{i_1}),\dots,\mathcal{F}^{i_r}(\alpha_{k-i_r+1},\dots,\alpha_k)) \end{eqnarray*}
Note that we only need these identities for $k \leq 3$.  In general, we will want to consider only the implications of $A_{\infty}$-relations involving compositions of low order.  To formalize this, we make the following definition.

\begin{definition}
	An $A_n$-category is a collection of objects and morphisms,
	with composition maps $\mu_k$ for all $k \leq n$, satisfying all of the $A_{\infty}$-relations which involve only these maps.
	We similarly define the notion of an $A_n$ functor between $A_n$ categories,
	by imposing only the relations which involve multiplications of order less than $n$.
\end{definition}

Denote by $\mathcal{P}_n$ the strictly unital $A_n$ category which is generated freely by a cycle of morphisms $\alpha_i:X_i \to X_{i+1}$, with $i \in \Z/n\Z$, modulo the relations
\[ \mu_{k}(\alpha_i,\dots,\alpha_{i+k-1} ) = \left \{  \begin{array}{l l} 0 & k < n \\ 1 & k =n \end{array}   \right. . \]
Note that this includes the relation $d \alpha_i = \mu_1(\alpha_i) = 0$.  An exact $n$-gon induces an $A_n$ functor $\mathcal{F}:\mathcal{P}_n \to \Ch$.

More generally, we can define an exact $n$-gon in an $A_\infty$ category $\mathcal A$ to be an $A_n$ functor from $\mathcal{P}_n$ to $\mathcal A$. For example, any $n$-periodic exact sequence in an abelian category determines an exact $n$-gon in the associated derived category. In general, if $\mathcal F:\mathcal P_n \to \mathcal A$ is an exact $n$-gon in $\mathcal A$ and $\mathcal F':\mathcal A \to \mathcal A'$ is an $A_\infty$-functor, then we can form the composition $\mathcal F' \circ \mathcal F$, which defines an exact $n$-gon in $\mathcal A'$. Given an object $y$ of an $A_\infty$-category $\mathcal A$, we can form the Yoneda functor ${\mathcal Hom}(Y,\cdot)$ which is an $A_\infty$-functor from $\mathcal A$ to ${\rm Ch}$. Thus we can use this Yoneda functor to construct an exact $n$-gon of chain complexes from an exact $n$-gon in an $\mathcal A$.

\section{Topology of the Surgery Polygon} \label{topology}
The main goal of this section is the definition of a family of cobordisms which play an important role in the proof of Theorem \ref{surgery-polygon-thm}.
\subsection{Dehn Surgeries and Cobordisms}\label{4-man}
Let $\mathcal Y$ be an oriented $3$-manifold with torus boundary.  Let $\lambda$ and $\mu$ be oriented simple closed curves on the boundary, satisfying $\lambda \cdot \mu = 1$.   Given any pair of coprime integers $(p,q)$, we can consider the Dehn filling of slope $p/q$,
\[ Y_{p/q} = \mathcal Y \cup_{f_{p/q}} H, \]
where $H = D^2 \times S^1$ is a standard genus $1$ handlebody and $f_{p/q}:\partial H \to \partial \mathcal Y$ is any diffeomorphism such that $f_{p/q}(\partial D^2 \times \mathrm{\{point\}})$ is homologous to $\mu_{p,q} = p \mu + q \lambda$.  The loop $\{0\}\times S^1$ inside $H$ induces a knot in $Y_{p/q}$ which will be denoted by $K_{p/q}$.

Given two Dehn fillings $Y_{p/q}$ and $Y_{r/s}$ such that $ps - qr = 1$, there is a natural cobordism
\begin{equation} \label{cob-Z}
	Z^{p/q}_{r/s}: Y_{p/q} \to Y_{r/s}.
\end{equation}	
To construct it, we first glue the manifolds $[-2,-1] \times H$ and $[1,2] \times H$ to $[-2,2] \times\mathcal Y$ using diffeomorphisms $\id \times f_{p/q}$ and $\id \times f_{r/s}$. The result of this gluing is a $4$-manifold with three boundary components.  By construction, the first two components are diffeomorphic to $-Y_{p/q}$ and $Y_{r/s}$, and the third component is a manifold $L$, obtained by Dehn fillings of slope $p/q$ and $r/s$ on the two ends of $[-1,1] \times T^2$.  Because the filling curves satisfy $\mu_{p,q} \cdot \mu_{r,s} = ps-qr = 1$, $L$ is diffeomorphic to $S^3$.  Attaching $B^4$, we obtain the desired cobordism $Z^{p/q}_{r/s}$.

The same construction can be also applied if $ps-qr \neq 1$, but in this case $L$ is no longer diffeomorphic to a sphere (hence cannot be filled with $B^4$).  Instead of a simple cobordism from $Y_{p/q}$ to $Y_{r/s}$, we have a 4-manifold with three boundary components $-Y_{p/q}$, $Y_{r/s}$, and a 3-manifold $M^{p/q}_{r/s}$. To describe the diffeomorphism type of $M^{p/q}_{r/s}$ explicitly, choose a curve $\lambda_{r,s}$ such that $\lambda_{r,s} \cdot \mu_{r,s} = 1$, and define integers $a = \mu_{r,s} \cdot \mu_{p,q}$ and $b = \lambda_{r,s} \cdot \mu_{p,q}$.  When $a \neq 0$,  the 3-manifold $M^{p/q}_{r/s}$ is diffeomorphic to the lens space $L(a,b)$, and when $a = 0$, it is diffeomorphic to $S^2 \times S^1$.

Motivated by the definition of $Z^{p/q}_{r/s}$, we introduce the notion of cobordisms with {\it middle ends}. We say $W$ is a cobordism from a 3-manifold $Y$ to $Y'$ with middle end $L$ if:
\[
  \partial W=-Y \sqcup Y' \sqcup L.
\]
We use the notation $W:Y \xrightarrow{L} Y'$ to specify the ends of $W$. Therefore, $Z^{p/q}_{r/s}$ is a cobordism from $Y_{p/q}$ to $Y_{r/s}$ with the middle end $M^{p/q}_{r/s}$. If we have the cobordisms:
\[
  W:Y \xrightarrow{L}Y'\hspace{1cm}W':Y' \xrightarrow{L'}Y''
\]
then we can compose these cobordisms to obtain $W\# W':Y \xrightarrow{L\sqcup L'}Y''$.

Our surgery exact polygon involves only the integer surgeries $Y_0,\dots,Y_N$, and the following infinite periodic sequence of cobordisms between them:

\[ \cdots \xrightarrow{\; \; M^N_0 \; \;} Y_0 \xrightarrow{\; \; M^0_1 \; \;} Y_1 \xrightarrow{\; \; M^1_2 \; \;} \cdots \xrightarrow{\; \; M^{N-1}_N \; \; } Y_N \xrightarrow{\; \; M^N_0 \; \;} Y_0 \xrightarrow{\; \; M^0_1 \; \;}\cdots  \]
where the arrow starting from $Y_j$ is given by the cobordism $Z^j_{j+1}$\footnote{Here we use the cyclic notation, and  when $j=N$, $j+1$ denotes $0$.}. Note that the middle boundary components of most of the cobordisms $Z^j_{j+1}$ are diffeomorphic to $S^3$ - the only nontrivial boundary component is $L(N,1)$, which occurs in the cobordism $Z^N_0$.  It will be useful for us to consider the composition of all these cobordisms in this sequence.  The resulting space, denoted by $\mathfrak W$, has an alternative construction which will be easier for us to use later on.

Let $D^2$ be a Euclidean disk of some fixed large radius (larger than $N+2$ will do), and let:
\begin{equation} \label{B}
	B:=\R \times D^2\setminus\{x_j= (j+ \frac{1}{2},0)\mid j\in \Z\}
\end{equation}
Inside $B$ we consider the arcs $\fd_j = (j-\frac{1}{2},j+ \frac{1}{2}) \times \{0\}$.  We orient these arcs using their increasing parametrization.  For any $j \in \Z$, we define $\overline{j} \in \{0,\dots,N\}$ to be the integer obtained by reducing $j$ mod $N+1$. Form a divisor $\fd $ as follows:
\begin{equation} \label{fd}
	\fd := - \sum_j \; \overline{j} \; \fd_j.
\end{equation}	
Recall that a co-oriented codimension $2$ submanifold $\fD$ in a Riemannian manifold $M$ can be used to construct a complex line bundle $\mathcal{L}(\fD)$ over $M$.  This can be done in a standard way, by choosing a tubular neighborhood $n(\fD)$, pulling back the normal bundle $\nu(\fD)$ to $n(\fD)$, trivializing the pull back of $\nu(\fD)$ on $n(\fD) \setminus \fD$ in the tautological way, and gluing the pull-back bundle on $n(\fD)$ to the trivial line bundle on $M\backslash \fD$ using this trivialization.  More generally, given a divisor $\fD = \sum_i m_i \fD_i$, where each $\fD_i$ is a manifold as above, we can define $\mathcal{L}(\fD) = \bigotimes_i \mathcal{L}(\fD_i)^{m_i}$.  Applying this general construction to the divisor $\fd$, we obtain a complex line bundle $\mathcal{L}$ over $B$. Let $\mathfrak X$ be the $\U(1)$-bundle associated to $\mathcal L$.

The 4-manifold $\fW$ is constructed by gluing $\fX$ to $\R \times \mathcal Y$. The line bundle $\mathcal{L}(\fd)$ admits a smooth section $s$, which vanishes along each $\fd_j$ (with multiplicity $\overline{j}$), and is nonzero everywhere else. Normalizing $s$, we obtain a trivialization $\overline{s}$ of $\fX$ on $B \setminus \fd$.  We can use this trivialization to identify $\left. \fX \right|_{\partial B}$ with $-S^1 \times \partial B =- \mu \times \lambda\times  \R$ where the ``base'' circle is denoted by $\lambda$, and the ``fiber'' circle is denoted by $\mu$. This gives us an orientation reversing identification of $\left. \fX \right|_{\partial B}$ with $\partial (\R\times Y)= \mu \times\lambda \times \R$. Gluing the boundaries of $\fX$ and $\R \times \mathcal Y$, using this identification, gives rise to the manifold $\fW$. Observe that $\fW$ comes with a canonical (smooth) map $f:\fW \to \R$, obtained by gluing the canonical projections of $\fX$ and $\R \times \mathcal Y$ onto $\R$. The 4-manifold $\fW$ also comes with a diffeomorphism $T:\fW \to \fW$ which translates by $N+1$ in the $\R$ direction. The fibers of $\fX$ over the interval $\fd_j$ determines a cylinder in $\fW$. We will write $\Sigma_j$ for these 2-dimensional submanifolds of $\fW$.

\begin{prop} \label{topology-W}
	For any $j \in \{0,\dots,N \}$, the pair $(f^{-1}(j), f^{-1}(j)\cap \Sigma_j)$ is diffeomorphic to $(Y_j, K_j)$.
\end{prop}
The 3-manifolds $f^{-1}(j)$ are called the {\it vertical cuts} of $\fW$. We can use the diffeomorphism $T$ to identify $f^{-1}(j)$ and $f^{-1}(k)$ when $\overline j=\overline k$. Therefore, this proposition determines the topology of all 3-manifolds $f^{-1}(j)$.

\begin{proof}
First observe that $f^{-1}(j) \cap \fX$ is a handlebody $H = D^2 \times S^1$, viewed as a circle bundle over $D^2$.  In this case the trivialization $\overline{s}$ is given (away from $0$) by:
\[ \overline{s}(z) = \left(\frac{z}{|z|}\right)^{-j} .\]
It follows that when we glue $\partial H$ to $\partial \mathcal Y$, we identify $\lambda$ with the curve $(e^{\bi\theta},e^{- \bi j\theta})$ and we identify $\mu$ with the curve $(1,e^{\bi\theta})$. Therefore, the curve $\lambda + j \mu$ is identified with $(e^{\bi\theta},1)$, which bounds a disk in $H$.  Hence, by definition, $f^{-1}(j)$ is diffeomorphic to $Y_{j}$.
It is also clear that $f^{-1}(j)\cap \Sigma_j$ is mapped to $K_j$ with respect to this diffeomorphism.
\end{proof}

We now define a second family of separating submanifolds in $\fW$, which are called {\it spherical cuts}. For every pair of integers $(j,k)$ satisfying $0 < k-j \leq 2N+2$, there exists a sphere $S^j_k \subset B$, which is centered at $(\frac{j+k}{2},0)$, the midpoint of $x_j$ and $x_{k-1}$, and its radius is equal to $r_{j,k} = \frac{k-j-1}{2}+\nu(k-j) $ with $\nu:\R_{\geq 0} \to (0,\frac{1}{2})$ being a fixed, strictly increasing function.
In particular, if $j_1 \leq j_2 \leq k_2 \leq k_1$ and $(j_1,k_1)\neq(j_2,k_2)$,
then $S^{j_1}_{k_1}$ does not intersect $S^{j_2}_{k_2}$. The sphere $S^j_k$ is one of the boundary components of the complement of the ball in $B$ which is centered at $(\frac{j+k}{2},0)$ and has radius $r_{j,k}$. Thus, we can use the orientation of $B$ to define an orientation on $S^j_k$. The spherical cut corresponding to $S^j_k$ is given by $M^j_k = \left.\fX\right|_{S^j_k}$.  It is a circle bundle over $S^j_k$ of degree
\[ d^j_k = \sum_{i} - \overline{i} (\fd_i \cdot S_k^j) = \overline{k} - \overline{j}. \]
In the special case $k=j+1$, this simplifies to
\[ d^j_{j+1} = \overline{j+1} - \overline{j}  = \left \{ \begin{array}{ll} 1 & \overline{j} \neq N  \\ -N & \overline{j} = N \end{array}\right. .\]
It follows that the manifolds $M^j_{j+1}$ are either spheres or diffeomorphic to $L(N,1)$.

Let $W^j_{j+1}$ denote the submanifold of $\fW$ bounded by $Y_j$,$Y_{j+1}$, and $M^j_{j+1}$.  This manifold is diffeomorphic to $Z^j_{j+1}$ in \eqref{cob-Z} before filling gluing the $4$-ball. More generally, for $j<k$, we define the compact manifold $W_k^j$ to be the composite cobordism:
\[
  W_k^j=W_{j+1}^j\# W_{j+2}^{j+1} \# \dots \# W_{k}^{k-1}.
\]
The interior of this manifold is diffeomorphic to $f^{-1}((j,k))$. The middle end of $W_k^j$ is a union of lens spaces. Note that this manifold, in general, is different from $Z^j_k$.

For each $0\leq l\leq N$, let $B_N(l)$ denote the submanifold of $B$ bounded by the spheres $S^l_{l+1}$, $S^{l+1}_{l+2}$, $\dots$, $S^{l+N}_{l+N+1}$ and $S^l_{l+N+1}$. Then the fibers of the $\U(1)$-bundle $\fX$ over $B_N(l)$ determine a 4-manifold, denoted by $X_N(l)$, with boundary components $ M^l_{l+1}$, $ M^{l+1}_{l+2}$, $\dots$, $ M^{l+N}_{l+N+1}$ and $M^l_{l+N+1}$. (See Figure \ref{XN(l)} for an example.) The non-trivial boundary components of $X_N(l)$ are $S^1 \times S^2$ and $L(N,1)$. Clearly, the manifolds $X_N(l)$, for different choices of $l$, are diffeomorphic to each other. In order to give an alternative description of $X_N(l)$, let $\sigma_1$ be the exceptional sphere in $\overline{{\bf CP}}^2$ and $\sigma_N$ be the connected sum of $N$ copies of $\sigma_1$ in the connected sum of $N$ copies of $\overline{{\bf CP}}^2$. We also fix a closed loop $\gamma$ and $N$ points in $\#_N \overline{{\bf CP}}^2$ such that these submanifolds and $\sigma_N$ are disjoint form each other. Then removing a regular neighborhood of the $N$ points, $\gamma$ and $\sigma_N$ gives rise to a 4-manifold which is diffeomorphic to $X_N(l)$.

For $1\leq i \leq N-l$, fix a properly embedded path in $B_N(l)$ which connects a point in the boundary of $S^{l+i-1}_{l+i}$ to  a point in the boundary of $S^{N}_{N+1}$. Then the fibers of $X_N(l)$ over this path gives rise to a cylinder, denoted by $e_i$. One boundary component of $e_i$ belongs to the lens space $L(N,1)\subset\partial X_N(l)$, and we orient $e_i$ such that the induced orientation on this boundary component matches with the orientations of the fibers of the $\U(1)$-bundle $\fX$. Similarly, we can define a cylinder $e_i$, for $N-l+1\leq i \leq N$, by fixing a path from the boundary of $S^{l+i}_{l+i+1}$ to $S^{N}_{N+1}$.  Since the restrictions of relative homology classes determined by $e_i$ to the boundary are torsion, the intersection number of $e_i$ and $e_j$ is a well-defined rational number. These intersection numbers are computed in Lemma \ref{int-num}. By Poincar\'e duality, the cylinders $e_i$ also determine cohomology classes on $X_N(l)$, which will be also denoted by the same notation. Similarly, pick a path from a point in the boundary of $S^l_{l+N+1}$ to a point in the boundary of $S^{N}_{N+1}$, and let $\fc$ be the cylinder given by the restriction of the $\U(1)$-bundle to this path. The cylinder $\fc$ has a boundary component in $L(N,1)$. We orient $\fc$ such that the induced orientation on this boundary component of $\fc$ disagrees with the orientation of the fibers of the $\U(1)$-bundle $\fX$. Note that restriction of the homology class of $\fc$ to the boundary of $X_N(l)$ is not torsion. The cohomology classes determined by $e_1$, $\dots$, $e_N$ and $\fc$ give a set of generators for $H^2(X_N(l),\Z)$. A generator for the relations among these cohomology classes is given in Lemma \ref{int-num}.

\begin{figure}
	\centering
	\begin{tikzpicture}
\draw[fill] (0,0) circle (3pt);
\draw[fill] (1,0) circle (3pt);
\draw[fill] (2,0) circle (3pt);\node[right] at (1.65,-0.6) {$M^6_7$};
\draw[fill] (3,0) circle (3pt);
\draw[fill] (4,0) circle (3pt);
\draw[thick] (5,0) circle (3pt); \node[right] at (4.65,-0.6) {$M^9_{10}$};
\draw[fill] (6,0) circle (3pt);
\draw[fill] (7,0) circle (3pt);
\draw[fill] (8,0) circle (3pt);
\draw[fill] (9,0) circle (3pt);\node[right] at (8.65,-0.6) {$M^{13}_{14}$};
\draw[thick] (4.5,0) ellipse (6cm and 3cm);
\draw [middlearrow={latex},thick,red]  (5,2.8pt)--  (5,3);  \node[right] at (5,1.3) {$\mathfrak{c}$};
\draw[middlearrow={latex},thick,blue]  (2,0.1) to [out=45,in=135] (4.9,0.05); \node[above] at (3.7,0.6) {$e_3$};

\draw[middlearrow={latex},thick,blue]  (8.9,0.05) to [out=135,in=45] (5.1,0.05); \node[above] at (7.5,0.75) {$e_9$};
	\end{tikzpicture}
	\caption{This schematic figure sketches $X_{9}(4)$: the small circles represent the 3-manifolds $M^j_{j+1}$
	for $4\leq j \leq 13$. In particular, the white circle represents the lens space $L(9,1)$. The outer ellipse
	sketches $S^1\times S^2$. The (co)homology classes $e_3$, $e_9$ and $\fc$ are also sketched in this figure.}
    	\label{XN(l)}
\end{figure}
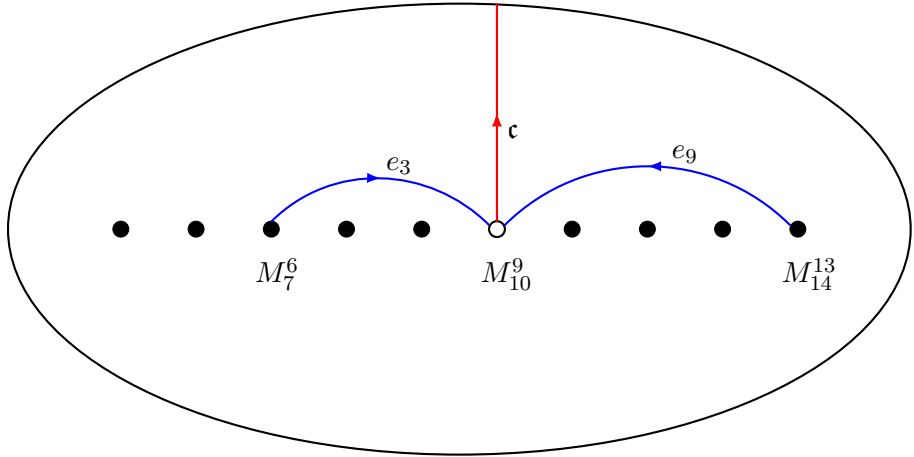

\begin{lemma} \label{int-num}
	 The manifold $X_N(l)$ is simply connected.
	 The intersection of the classes $e_i$ and $e_j$ is equal to $\frac{1}{N}-\delta_{i,j}$.
	  Moreover, the homology class $e_1+\dots+e_N$ is trivial.
\end{lemma}
\begin{proof}
	As it is explained above, $X_N(l)$ can be regarded as a submanifold of $\#_N \overline{{\bf CP}}^2$.
	The cohomology classes $e_1$, $\dots$, $e_N$ can be extended to the set of standard generators
	of $H^2(\#_N \overline{{\bf CP}}^2;\Z)$. The claims of this lemma can be easily seen from this description of
	$e_1$, $\dots$, $e_N$.
\end{proof}

\subsection{Cycles in the Surgery Polygon}

While working with $\U(N)$-instanton Floer homology, in addition to a closed 3-manifold $Y$, we need to fix a $\U(N)$-bundle on $Y$. A $\U(N)$-bundle on $Y$ is determined up to an isomorphism by the cohomology class of its first Chern class. We specify this cohomology class by a {\it $1$-cycle} which represents the Poincar\'e dual of the first Chern class. A $1$-cycle $\gamma$ in $Y$ is a linear combination of oriented submanifolds of dimension $1$ with integer coefficients. As it is explained in the previous subsection, we can associate a line bundle $L$ to any 1-cycle $\gamma$. Then $L\oplus \underline{\C}^{N-1}$ is the $\U(N)$-bundle determined by $\gamma$.

A {$2$-cycle} $c$ in a 4-manifold $W$ is a linear combination of oriented submanifolds of dimension $2$ with integer coefficients. In the case that $W$ has boundary, we assume that each connected component of $c$ is properly embedded in $W$. Similar to the 3-dimensional class, $c$ determines a $\U(N)$-bundle on $W$ for any value of $N$. Suppose $(Y,\gamma)$, $(Y',\gamma')$ and $(L,l)$ are pairs of a 3-manifold and a 1-cycle. We say $(W,c)$ is a cobordism from $(Y,\gamma)$ to $(Y',\gamma')$ with the middle end $(L,l)$ if $W$ is a cobordism from $Y$ to $Y'$ with the middle end $L$, and $c$ is a $2$-cycle in $W$ such that:
\[
  \partial c=-\gamma'\sqcup \gamma \sqcup l.
\]
Then the boundary restriction of the $\U(N)$-bundle associated to $c$ is canonically isomorphic to the $\U(N)$-bundles associated to $\gamma$, $\gamma'$ and $l$. The advantage of cycles to cohomology classes for representing $\U(N)$-bundles is that a cycle determines a $\U(N)$-bundle up to a canonical isomorphism. This allows us to compose cobordisms of pairs without any ambiguity.


We introduce a family of pairs of manifolds and cycles which are related to surgery polygons. Let $\mathcal Y$ be the 3-manifold with torus boundary from the previous subsection. Let also $w\subset \mathcal Y$ be a 1-cycle that is away from the boundary of $\mathcal Y$. Suppose $S$ is a subset of $[N]$ with $j$ elements. Recall that $S$ represents a vertex of the graph $G_N$. To $S$, we associate a pair of a 3-manifold and a 1-cycle. The 3-manifold depends only on $j$ and is given by $Y_j$. The 1-cycle $\gamma_S\subset Y_j$ is given by $w\sqcup \sigma_S\cdot K_j$, where $\sigma_S$ denotes the sum of the elements of $S$. Next, suppose that $i\in [N]$ does not belong to the set $S$. Then we have the edge morphism $\beta_{S,i}:S \to S'$ where $S':= S\sqcup \{i\}$. We associate a cobordism of pairs as below to $\beta_{S,i}$:
\[
  (W^j_{j+1},c_{S,i}):(Y_j,\gamma_S) \xrightarrow{(M^{j}_{j+1},\,l_{S,i})} (Y_{j+1},\gamma_{S'}).
\]
Half of the cylinder $\Sigma_j$ (respectively, $\Sigma_{j+1}$) belongs to $W^j_{j+1}$ and it intersect $Y_j$ (respectively, $Y_{j+1}$) in the knot $K_j$ (respectively, $K_{j+1}$) by Proposition \ref{topology-W}. We define $c_{S,i}$ as:
\begin{equation}\label{cS}
  [j,j+1]\times w\sqcup\sigma_S \cdot (\Sigma_j\cap W^j_{j+1})\sqcup \sigma_{S'} \cdot (\Sigma_{j+1}\cap W^j_{j+1})
\end{equation}
The cycle $l_{S,i}$ in $M_{j}^{j+1}$ is a union of the fibers of the Hopf fibration. Since $M_{j}^{j+1}$ is a sphere, the homology class of this cycle is trivial.   There is also a cobordism of pairs associated to the edge $\delta$ of $G_N$ as follows:
\[
  (W^N_{N+1},c_\delta):(Y_n, \gamma_{[N]})\xrightarrow{(M^{N}_{N+1},\,l_{\delta})} (Y_0, \gamma_{\emptyset})
\]
The cycle $c_\delta$ is again given by \eqref{cS} where $S=[N]$, $S'=\emptyset$ and $j=N$. The cycle $l_\delta$ in the lens space $M^{N}_{N+1}$ is also a union of the fibers of its $\U(1)$-fibration. The homology class of this cycle is equal to $\frac{N(N-1)}{2}$ times a generator of $H_1(M^{N}_{N+1};\Z)$.

\section{Geometry of the Surgery Polygon} \label{metrics}

The main goal of this section is to construct families of metrics on the 4-manifolds which are constructed in the previous section.

\subsection{Families of Cylindrical Metrics}\label{family-general}

Suppose $W$ is a compact smooth 4-dimensional manifold. Let $W^+$ be the result of gluing $[-1,0)\times \partial W$ to $W$ by identifying $\{-1\}\times \partial W$ with the boundary of $W$.
We call each connected component of $(-1,0)\times \partial W\subset W^+$ an {\it end} of $W^+$. We also fix a metric $h$ on each connected component $Y$ of $\partial W$. Then $h$ determines the metric:
\begin{equation} \label{cylinder}
	\frac{1}{r^2}dr^2+h(y)
\end{equation}
on the end $(-1,0)\times Y$, which is called the {\it cylindrical} metric associated to $h$. If we define the coordinate $t=-\ln(|r|)$ on this end, then \eqref{cylinder} has the form $dt^2+h$, which justifies the terminology of cylindrical metrics. A metric $g$ on $W^+$ is cylindrical if its restrictions to the ends of $W^+$ have the form \eqref{cylinder}.

A metric $g$ on $W^+$ is {\it asymptotically cylindrical} if for each connected component $Y$ of $\partial W$ there is a quadratic form $h'$ on $(-1,0)\times Y$ such that:
\begin{equation} \label{cylinder-2}
	g=\frac{1}{r^2}dr^2+h(y)+h'(r,y).
\end{equation}
Here $h$ is the metric associated to $Y$ and $h'$ satisfies the following decay condition. There is a positive constant $\delta$ and for each integer $k$, there is a constant $C_k$ such that:
\[
  |\nabla^k h'(r,y)|\leq C_k|r|^\delta
\]
Here the covariant derivative $\nabla^k$ and the point-wise norm $|\cdot|$ are defined with respect to the metric $g$. Under this decay assumption, we say $g$ is asymptotic to $h$ on the cylindrical end associated to $Y$.

A {\it family of smooth cylindrical metrics} on a 4-manifold $W$ parametrized by a smooth manifold $K$ consists of a cylindrical metric $g(\bx)$ on $W$ for each $\bx\in K$ such that $g(\bx)$ depends smoothly on $\bx$ and the restriction of $g(\bx)$ for each end $(-1,0)\times Y$ of $W$ is given by \eqref{cylinder}. In particular, the restriction of $g(\bx)$ on the ends is independent of $\bx$. More generally, we can consider {\it smooth families of asymptotically cylindrical metrics} by allowing the metrics in the family to have the form \eqref{cylinder-2} on the ends. For our purposes, it is important to work with even more general families of metrics that can degenerate in a controlled way. In the present article, we limit ourselves to families of such metrics which are parametrized by {\it admissible polyhedra}.

An admissible polyhedron is a polyhedron $K$ such that each face $\overline F$ of $K$ with co-dimension $l$ has the form of the product $K_0\times \dots \times K_l$ with each $K_i$ being an admissible polyhedron of lower dimension. We reserve the notation $F$ for the interior of the face $\overline F$ which is by definition the set of points in $\overline F$ with a neighborhood diffeomorphic to an open ball in a Euclidean space. Therefore, $F$ is identified with $K_0^\circ\times \dots K_l^\circ$ where $K_i^\circ$ is the interior of the polyhedron $K_i$. If $\overline F_i'$ is a face of $K_i$, then $\overline F'=\overline F_1'\times \dots \times \overline F_l'$ is also a face of $K$ and the product structure on $\overline F'$ is induced by the product structures of the faces $F_i'$.

An admissible polyhedron is also required to have a parametrized regular neighborhood for each face. For any face $\overline F=K_0\times \dots \times K_l$ of a polyhedron $K$, a regular neighborhood $U_F$ and a diffeomorphism $\Psi_F:U_F \to \overline F\times (1,\infty]^l$ is fixed such that for $p \in \overline F$:
\[
  \Psi_F(p)=(p,\infty,\dots,\infty)
\]
If $\overline F_i'$ is a face of $K_i$ with the parametrized regular neighborhood $U_{F_i'}$, then the open neighborhood $U_{F'}$ of the face $\overline F'=\overline F_1'\times \dots \times \overline F_l'$ is given by:
\[
  U_{F'}=\Psi_{F}^{-1}(U_{F_0'}\times \dots \times U_{F_l'}\times (1,\infty]^l)
\]
and the map $\Psi_{F'}$ is equal to the composition of $\Psi_F:U_{F'} \to U_{F_0'}\times \dots \times U_{F_l'}\times (1,\infty]^l$
and the map $\Psi_{F'}^{F}$ defined on $U_{F_0'}\times \dots \times U_{F_l'}\times (1,\infty]^l$ as follows:
\[
  \Psi_{F'}^{F}=(\Psi_{F_0'},\dots, \Psi_{F_l'},{\rm Id}_{(1,\infty]^l}).
\]
In the following, we use the maps $\Psi_F$ to regard $F\times (0,\infty]^l$ as open subsets of $K$. The main example of admissible polyhedra for us is the {\it associahedron} whose definition is reviewed in the next subsection.

A family of cylindrical metrics on a compact 4-manifold $W$ parametrized by an admissible polyhedron $K$ is a manifold $\mathbb W$ together with a projection map $\pi:\mathbb W \to K$ such that each fiber of $\pi$ is a smooth 4-manifold and a Riemannian metric on each fiber is fixed. The space $\mathbb W$ as a manifold has the following form:
\begin{equation}\label{familyW}
  \mathbb W=W^+\times K\backslash \left ( \bigcup_{F} Y_F \times F\right)
\end{equation}
and $\pi$ is induced by the projection map from $W^+\times K$ to $K$. Here $Y_F$, for each $F$,  is a closed 3-dimensional submanifold $Y_F$ of $W^+$. We also assume that a collar neighborhood of $Y_F$ is identified with the set $(-2,2)\times Y_F$. The 3-manifolds $Y_F$ and the metrics on the fibers of $\pi$ are required to satisfy some additional constraints.

If the face $\overline F$ has the form $K_0\times \dots \times K_l$, then we require that $Y_F$ and $W^+\backslash Y_F$ have respectively $l$ and $l+1$ connected components. We denote the connected components of $Y_F$ with $Y_1$, $\dots$, $Y_l$. The connected components of $W\backslash (-1,1)\times Y_F$ are also denoted by $W_0$, $\dots$, $W_l$. Therefore, $W^+\backslash Y_F$ can be identified with the disjoint union of the cylindrical 4-manifolds $W_0^+$, $\dots$, $W_l^+$. Each connected component of $Y_F$ is called a {\it cut} of $W$ and is equipped with a metric. Let $\overline F_i'$ be a face of $K_i$ with codimension $l_i$ and $\overline F'= \overline F_0'\times \dots \times  \overline F_l'$ be the corresponding face of $K$. For each $i$, there is a closed 3-manifold $Y_{F_i'}\subset W_i^+$ with $l_i$ connected components such that $Y_{F'}=Y_F\sqcup Y_{F_0'}\sqcup \dots \sqcup  Y_{F_l'}$. The metrics on the cuts, that $Y_F$ and $Y_{F'}$ have in common, are required to be equal to each other.

For the open face $F$, the space $\pi^{-1}(F)$ can be identified with $(W_0^+\sqcup\dots \sqcup W_l^+)\times K_0^\circ\times \dots \times K_l^\circ$. For each $1\leq i \leq l$ and each $\bx_i \in K_i^\circ$, there is a metric $g(\bx_i)$ on $W_i^+$ varying smoothly with respect to $\bx_i$ such that the metric on the fiber of $\mathbb W$ over the point $(\bx_0,\dots,\bx_l)\in F$ is given by $(g(\bx_0), \dots, g(\bx_l))$. The restriction of the metric $g(\bx_i)$ to the cylindrical ends of $W_i^+$ needs to be given as in \eqref{cylinder} where $h$ is determined by the metrics on the boundary components of $W$ and the cuts. In particular, the restriction of $g(\bx_i)$ to the cylindrical ends is independent of $\bx_i$.

The metrics on the fibers of $\mathbb W$ over $F$ determines a family of metrics on the fibers of $\mathbb W$ over the regular neighborhood $F \times (1,\infty]^l$ of $F$. To define this family, fix a 1-parameter family of smooth functions $u_s:[-1,1]\to \R^{>0}$ for $s\in [1,\infty)$ such that:
\vspace{-5pt}
\begin{enumerate}
	\item[(i)] $u_s$ depends smoothly on $s$;
	\item[(ii)] $u_s(-r)=u_s(r)$;
	\item[(iii)] $u_s(r)=\frac{1}{|r|}$ for $t\in [-1,-\frac{2}{3}]\cup[\frac{2}{3},1]$;
	\item[(iv)] $\int_{-1}^1 u_s(r)dt = s $;
	\item[(v)] The functions $u_s$ on compact subsets of $[-1,0)\cup(0, 1]$ converges uniformly in all
	derivatives to the function $u_{\infty}(r) = \frac{1}{|r|}$, as $s \to \infty$.
\end{enumerate}
\vspace{-5pt}
Now, let $p=(\bx_0,\dots,\bx_l,s_1,\dots,s_l)$ be an element of $F \times (0,\infty]^l$. If $s_i\in (0,\infty)$, then the regular neighborhood $(-1,1)\times Y_i$ is a subset of the 4-manifold $\pi^{-1}(p)$. We call this neighborhood the {\it neck} associated  to $Y_i$. On this neighborhood, we fix the metric $u_{s_i}(r)^2dr^2+h_i$ where $h_i$ is the metric on $Y_i$. We can also consider the cylindrical coordinate $t$:
\begin{equation} \label{cylindrical-coordinate}
  t(r):=\int_{0}^r u_{s_i}(\rho)d\rho
\end{equation}
Note that $t\in (-\frac{s_i}{2},\frac{s_i}{2})$ for $r \in (-1,1)$. Then the chosen metric on $(-1,1)\times Y_i$ has the form $dt^2+h_i$. If $s_i=\infty$, then $((-1,0)\cup(0, 1))\times Y_i$ is a subset of $\pi^{-1}(p)$ and we fix the metric $\frac{1}{r^2}dr^2+h_i$ on this 4-manifold. These metrics can be extended by $g(\bx_0)$, $\dots$, $g(\bx_l)$ to a metric on $\pi^{-1}(p)$. As a final requirement on the family of metrics $\mathbb W$, we demand that the metrics on the fibers of $\mathbb W$ over the regular neighborhood $F \times (1,\infty]^l$ agrees with the above family of metrics.

More generally, we can consider families of asymptotically cylindrical metrics. Similar to the cylindrical case, a family of asymptotically cylindrical metrics on $W$ parametrized by an admissible polyhedron $K$ consists of a manifold $\mathbb W$ as in \eqref{familyW}, a projection map $\pi$ and a metric on each fiber of the map $\pi$. Let $F=K_0^\circ\times \dots \times K_l^\circ$ be an open face as above such that $W^+\backslash Y_F=W_0^+\sqcup\dots \sqcup W_l^+$. For each $\bx_i\in K_i$, there is an asymptotically cylindrical metric $g(\bx_i)$ on $W_i^+$ such that $g(\bx_i)$ on each end of $W_i^+$ is asymptotic to the metric which is fixed on the corresponding boundary component of $W_i$. The metric $g(\bx_i)$ is also required to vary smoothly with respect to $\bx_i \in K_i^\circ$. The metric on the fiber of $\mathbb W$ over the point $(\bx_0,\dots,\bx_l)$ is given by $(g(\bx_0),\dots,g(\bx_l))$.

The above metrics on $F$ can be used to construct a family of metrics on the fibers of $\mathbb W$ over $F\times (0,\infty]^l$. To define the metric associated to $p=(\bx_0,\dots,\bx_l,s_1,\dots,s_l)\in F\times (0,\infty]^l$, we firstly focus on the neck associated to the connected component $Y_i$ of $Y_F$. Suppose $W_0^+$ (respectively, $W_1^+$) is the connected component of $W^+\backslash Y_F$ that has $Y_i$ (respectively, $-Y_i$) as a connected component of its boundary. If $s_i\in(1,\infty)$, then the regular neighborhood $(-1,1)\times Y_i$ is a subset of $\pi^{-1}(p)$. The restrictions of the metrics $g(\bx_0)$ and $g(\bx_1)$ to the ends associated to $Y_i$ have the following form:
\begin{equation*}
	g(\bx_0): \frac{1}{r^2}dr^2+h_i(y)+h'(r,y) \hspace{1cm}g(\bx_1): \frac{1}{r^2}dr^2+h_i(y)+h''(r,y).
\end{equation*}
for appropriate choices of quadratic forms $h'$ and $h''$. We can use these metrics to define the following metric on $(-1,1)\times Y_i$ for the parameter $s_i$:
\begin{equation} \label{glued-metric-1}
	dt^2+h_i(y)+
	\varphi_1(\frac{t+s_i}{2s_i})h'(-e^{-\frac{s_i}{2}-t},y)+
	\varphi_2(\frac{t+s_i}{2s_i})h''(-e^{-\frac{s_i}{2}+t},y)
\end{equation}
where $t\in (-\frac{s_i}{2},\frac{s_i}{2})$ is defined in \eqref{cylindrical-coordinate} and $\varphi_1$, $\varphi_2$ are given in \eqref{phi-1-2}. In the case that $s_i=\infty$, then $((-1,0)\cup(0, 1))\times Y_i$ is a subset of $\pi^{-1}(p)$ and we fix the following metric on this space:
\begin{equation}\label{glued-metric-2}
	\left\{
		\begin{array}{cc}
			\frac{1}{r^2}dr^2+h_i(y)+h'(r,y)&r\in (-1,0)\\
			\frac{1}{r^2}dr^2+h_i(y)+h''(-r,y)&r\in (0,1)\\	
		\end{array}
	\right.
\end{equation}
The metrics in \eqref{glued-metric-1} and \eqref{glued-metric-2} are equal to $g(\bx_0)$ (respectively, $g(\bx_0)$) in a neighborhood of $\{-1\}\times Y_i$ (respectively, $\{1\}\times Y_i$). Therefore, we can extend these metrics on the necks to a metric $\widetilde g$ on $\pi^{-1}(p)$ using $g(\bx_0)$, $\dots$, $g(\bx_l)$. For a family of asymptotically cylindrical metrics the metrics $g$ and $\widetilde g$ are required to have similar asymptotic behavior on $F\times (0,\infty]^l$.

To state the precise version of the property about asymptotic behavior of $g$ and $\widetilde g$, suppose $\lambda$ is the function on $\pi^{-1}(p)$ which is equal to $-\ln(|r|)$ on an end $(-1,0)\times Y$ where $Y$ is a connected component of $\partial W$, is equal to $\frac{s_i}{2}-|t|$ on the neck $(-1,1)\times Y_i$ where $t$ is defined as in \eqref{cylindrical-coordinate}, and is equal to $0$ on the remaining part of $\pi^{-1}(p)$. Then for a family of asymptotically cylindrical metrics there is a positive constant $\delta$, and for any integer $k$, there is a constant $C_k$ such that for $F$ and for $(\bx_0,\dots,\bx_l,s_1,\dots,s_l)\in F\times (1,\infty]^l$ the following inequality holds
\[
  |\nabla^k(g-\widetilde g)|\leq C_k e^{-\delta (\lambda+\tau)}.
\]
Here $\tau$ is equal to the minimum of the parameters $s_1$, $\dots$, $s_l$.

\subsection{Associahedron}

In this subsection, we review the definition of a family of polytopes which are known as {\it associahedra}. The families of metrics that we use in this paper are all parametrized by associahedra. These polytopes, discovered by Stasheff \cite{St:assoc}, also play an important role in the definition of the $A_\infty$-structure of Fukaya categories in symplectic geometry \cite{FOOO1,FOOO2}. In the first part of this subsection, we review the definition of a few combinatorial notions necessary for the definition of an associahedron. Then we define associahedra as the moduli spaces of points in the 3-dimensional Euclidean space.

\subsubsection*{Ribbon Trees}

\begin{definition}
	Let $\Gamma$ be a graph with vertex set $V(\Gamma)$ and edge set $E(\Gamma)$.
	A \emph{ribbon structure} on $\Gamma$ consists of a choice, for every vertex $v \in V(\Gamma)$,
	of a cyclic ordering of the set of all edges incident to $v$.
\end{definition}	
\begin{example} \label{ribbon-planar}
	Any planar graph has a preferred ribbon structure,
	in which edges incident to a common vertex are ordered in the counterclockwise sense.
\end{example}
\begin{definition}	
	A graph $T$ is a {\it tree}, if it is contractible.
	We say $T$ is {\it marked}, if one of the vertices of degree $1$ is marked as the {\it root} of $T$.
	If $T$ has no vertices of degree $2$, we say it is \emph{reduced}.
	In a reduced tree, any edge incident to a vertex with degree $1$ is called a \emph{leaf},
	and any other edge is called an \emph{interior edge}. An edge of $T$ is called a {\it non-root leaf}, if
	it is a leaf which is not incident to the root. An {\it $n$-ribbon tree} is a reduced
	and marked ribbon tree with $n+1$ leaves.
\end{definition}
Let $T$ be an $n$-ribbon tree. We shall need the following notation about $T$:
\begin{itemize}
	\item[(i)] The set of vertices with degree at least 3 are called the {\it interior vertices} of $T$
		      and are denoted by
		      $V^{{\rm Int}}(T)$. If $v\in V^{{\rm Int}}(T)$, then the integer number
		      $d(v)$ is defined such that the degree of $v$ is equal to $d(v)+1$.
	\item[(ii)] For $e\in E(T)$, let the endpoints of $e$ be labeled with $s(e)$ and $t(e)$ such that
		       the unique path from $s(e)$ to the root does not contain the edge $e$.
		       Then the unique path from $t(e)$ to the root starts with the edge $e$.
		       The vertices $s(e)$ and $t(e)$ are respectively called the
		       {\it source} and the {\it target} of $e$. We say $e$ is an {\it outgoing} edge of $s(e)$
		       and the {\it incoming} edge of $t(e)$. The set of all interior edges are denoted by $E^{{\rm Int}}(T)$.
		       Note that $|E^{{\rm Int}}(T)|=|V^{{\rm Int}}(T)|-1$.
	\item [(iii)]	Since each vertex $v$ has a unique incoming edge, the ribbon structure on $v$ is equivalent
			to an ordering of the outgoing edges. Therefore, we sometimes confuse the two types of ordering.
\end{itemize}

\begin{example}
	Two ($6$)-ribbon trees are given in Figures \ref{(6)-ribbon-tree} and \ref{s(6)-ribbon-tree}.
  	The roots of these two trees are labeled with $l_0$ and the remaining leaves are denoted by $l_1$, $l_2$, $\dots$, $l_6$.
	The planar description of these graphs determine their ribbon strcutures.
	The tree in Figure \ref{(6)-ribbon-tree} has three interior vertices $v_0$, $v_1$, $v_2$ and two interior edges $e_1$, $e_2$.
	The sources of $e_1$ and $e_2$ are both equal to $v_0$. The targets of $e_1$ and $e_2$ are repectively equal to $v_1$ and $v_2$.
	The ribbon tree in Figure \ref{s(6)-ribbon-tree} has only two interior vertices $\widetilde v$ and $v_2$. The only interior edge of this graph
	has $\widetilde v$ as its source and $v_2$ as its target.
	
	\begin{figure}
    \begin{minipage}{.5\textwidth}
	\begin{tikzpicture}
		\draw[thick] (-1,0) to (-1,-1) ;   \node[above] at (-1,0) {$l_0$} ; \node[right] at (-1,-1) {$v_0$};
		\draw[thick] (-1,-1)  to  (-5,-3)  ;  \node[below] at (-5,-3) {$l_1$};
		\draw[thick] (-1,-1)  to  (-2,-2)  ;   \node[left] at (-2,-2) {$v_1$};
		\draw[thick] (-1,-1)  to  (0,-2)   ;  \node[left] at (0,-2) {$v_2$};
		\draw[thick] (-2,-2)  to  (-2,-3)  ;  \node[below] at (-2,-3) {$l_3$};
		\draw[thick] (-2,-2)  to  (-2.8,-3) ;  \node[below] at (-2.8,-3) {$l_2$};
		\draw[thick] (-2,-2)  to  (-1.2,-3) ;  \node[below] at (-1.2,-3) {$l_4$};
		\draw[thick] (0,-2)   to  ( 0.8 ,-3  ) ; \node[below] at (0.8,-3) {$l_6$};
		\draw[thick] (0,-2)   to  ( -0.8 ,-3  ) ; \node[below] at (-0.8,-3) {$l_5$};
		\node[left] at (-1,-1.7) {$e_1$};  \node[left] at (.2,-1.5) {$e_2$};
	\end{tikzpicture}
	\captionof{figure}{A ($6$)-ribbon tree}\label{(6)-ribbon-tree}
	\end{minipage}%
    \begin{minipage}{.5\textwidth}
	\begin{tikzpicture}
\draw[thick] (-1,0) to (-1,-1) ;   \node[above] at (-1,0) {$l_0$} ; \node[right] at (-1,-1) {$\widetilde v$};
\draw[thick] (-1,-1)  to  (-5,-3)  ;  \node[below] at (-5,-3) {$l_1$};
\draw[thick] (-1,-1)  to  (0,-2)   ;  \node[left] at (0,-2) {$v_2$};
\draw[thick] (-1,-1)  to  (-2,-3)  ;  \node[below] at (-2,-3) {$l_3$};
\draw[thick] (-1,-1)  to  (-2.8,-3) ;  \node[below] at (-2.8,-3) {$l_2$};
\draw[thick] (-1,-1)  to  (-1.2,-3) ;  \node[below] at (-1.2,-3) {$l_4$};
\draw[thick] (0,-2)   to  ( 0.8 ,-3  ) ; \node[below] at (0.8,-3) {$l_6$};
\draw[thick] (0,-2)   to  ( -0.8 ,-3  ) ; \node[below] at (-0.8,-3) {$l_5$};
\end{tikzpicture}
\captionof{figure}{A ($6$)-ribbon tree obtained from shrinking edge $e_1$ in Figure \ref{(6)-ribbon-tree}}\label{s(6)-ribbon-tree}
	\end{minipage}
\end{figure}

\end{example}

\begin{definition}
	Let $e$ be an interior edge in an $n$-ribbon tree $T$ joining vertices $s(e)$ to $t(e)$.
	One can collapse the graph, by removing $e$ and identifying $s(e)$ with $t(e)$.
	We say that $T'$, the resulting graph, is obtained by {\it shrinking the edge $e$ of $T$}.
	The tree $T'$ inherits a ribbon structure of its own, by merging the cyclic orderings at $v$ and $w$.
	To be a bit more detargeted, the set $V(T')$ is the union of $V(T)\backslash \{s(e),t(e)\} $
	and a new vertex $\widetilde v$.
	The labeling of the outgoing edges of any vertex in
	$V(T)\backslash \{s(e),t(e)\} \subset V(T')$ is inherited from the ordering of the outgoing
	edges of the corresponding vertex of $T$. Suppose $e$ is labeled as the $k^{\rm th}$ outgoing
	edge of $s(e)$. Then the set of the outgoing edges of $\widetilde v$, as an ordered set, is given
	as below:
	\begin{equation*}
		e_1,\dots,e_{k-1},f_1,f_2,\dots,f_{d(t(e))},e_{k+1},\dots,e_{d(s(e))}
	\end{equation*}
	where:
	\begin{equation*}
		e_1,e_2,\dots,e_{d(s(e))}\hspace{2cm} f_1,f_2,\dots,f_{d(t(e))}
	\end{equation*}
	are respectively the ordered set of the outgoing edges of $s(e)$ and $t(e)$. Therefore,
	$T'$ is also an $n$-ribbon tree and:
	\begin{equation*}
		|V^{in}(T')|=|V^{in}(T)|-1.
	\end{equation*}
	In general, if a series of edge shrinking in an $n$-ribbon tree $T$ gives rise to the graph
	$S$, then we write $T<S$. We also write $T\leq S$ if
	$ T< S$ or $ T= S$.
\end{definition}

\begin{example}
	The (6)-ribbon tree in Figure \ref{s(6)-ribbon-tree} is obtained  by shrinking the edge $e_1$ in the (6)-ribbon tree of Figure \ref{(6)-ribbon-tree}.

\end{example}

Next, we introduce a different method to encode the information of a ribbon tree.
\begin{definition}
	A subset $A$ of $[n]$ is called a {\it cyclic bisection} if it has the following form:
	\[
	  \hspace{2.5cm}A=\{i,\,i+1,\, \dots,\, j\} \hspace{2cm}
	\]
	with $j-i\geq1$.
	More generally, we can define a cyclic bisection of an ordered set $S=\{l_1,\dots,l_n\}$ by identifying $S$ with $[n]$ in
	an order-preserving way. Two cyclic bisections $A_0$ and $A_1$ of $S$ are {\it crossed}
	if $A_0\cap A_1$ is non-empty and none of $A_0$ and $A_1$ contain the other set.
	Otherwise, we say the two cyclic bisections are {\it uncrossed}.
\end{definition}
\begin{figure}
	\centering
	\begin{tikzpicture}
                \draw[thick] (0,0)    circle [radius=2];
                \foreach \a in {0,1,...,6} { 
                \draw[fill] ( 360/7 *\a:2cm) circle (1pt); 
                }
                \foreach \a in {0,1,...,6}{
                \node at (360/7*\a:2.32cm) {$l_\a$} ; }
                \node[red] at (35:1.7cm) {$A_2$};
                \node at (130:0.8cm) {$A_0$};
                \node at (300:1.65cm) {$A_1$};
                \draw[thick] (77:2cm) to (230:2cm);
                \draw[thick] (240:2cm) to (336:2cm);
                \draw[thick,red] (15:2cm) to (135:2cm);
	\end{tikzpicture}
	\caption{Three cyclic bi-sections of a set of $7$ elements with a cyclic ordering}
	\label{cyclic-bisection}
\end{figure}
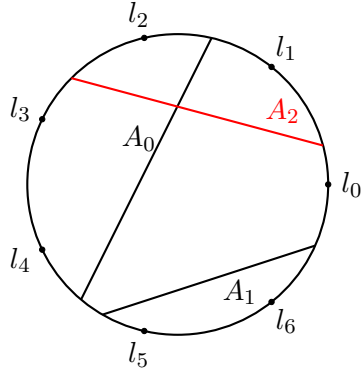
\begin{example}
	Figure \ref{cyclic-bisection} demonstrates various cyclic bi-sections of the set $S=\{l_1,\dots,l_6\}$.
	In this figure, we put the elements of $S\cup \{l_0\}$ on a circle such that the order of the set is given by removing
	$l_0$ and moving on the punctured circle in the counter-clockwise direction.
	Any line segment in Figure \ref{cyclic-bisection} divides the set $\{l_0,l_1,\dots,l_6\}$ into two sets, and the subset which
	does not contain $l_0$ is the corresponding cyclic bisection.
	The bisections $A_0=\{l_2,l_3,l_4\}$ and
	$A_2=\{l_1,l_2\}$ are crossed. However,
	each of these bisections and $A_1=\{l_5,l_6\}$  are
	uncrossed. The presentation of cyclic bisections as in Figure \ref{cyclic-bisection} justifies our terminology.
\end{example}

The set of non-root leaves of a ribbon tree $T$ admits a natural ordering. To obtain this ordering we shrink all the interior edges of $T$. The ribbon structure on the final graph determines the ordering among the non-root leaves of $T$. Therefore, we can talk about cyclic bisections of the non-root leaves of $T$. Let $e$ be an interior edge of $T$. Removing $e$ divides the set of leaves into two sets with at least two elements.
The set that does not contain the root determines a cyclic bisection of the set of non-root leaves. Given two interior edges of $T$, the corresponding cyclic bisections of the leaves are uncrossed. Therefore, an $n$-ribbon tree $T$ with $k$ interior edges gives rise to $k$ distinct cyclic bisections of the leaves of $T$, which are mutually uncrossed. For example, the cyclic bisections associated to the edges $e_0$ and $e_1$ of the $n$-ribbon tree in Figure \ref{(6)-ribbon-tree} are respectively equal to $A_0$ and $A_1$ given in Figure \ref{cyclic-bisection}. The following lemma asserts that an $n$-ribbon tree is determined uniquely by the set of its cyclic bisections:

\begin{lemma} \label{bisection-to-tree}
	Suppose $S$ is an ordered set.
	Suppose also $k$ mutually uncrossed cyclic bisections $\{A_i\}_{1\leq i \leq k}$ of $S$ are given.
	Then there is a unique $n$-ribbon tree $T$ with $k$ interior edges
	such that the set of non-root leaves of $T$ is $S$
	and the cyclic bisections associated to its interior edges are equal to $\{A_i\}_{1\leq i \leq k}$.
\end{lemma}
\begin{proof}
	This claim can be proved by induction on $k$. By changing the indices if necessary,
	we can assume that $A_1$ is minimal among the sets $\{A_i\}_{1\leq i \leq k}$.
	Thus $A_i$, for any $i \geq 2$, either contains $A_1$ as a proper subset or it is disjoint from $A_1$.
	By collapsing the elements of $A_1$ into one element $\widetilde v$, define a new set $S'$ out of $S$.
	The set $S'$ inherits an order from $S$.
	For any $i\geq 2$, if $A_i$ contains $A_1$, define $A_i'$ to be the following subset of $S'$:
	\[
	  A_i'= A_i \backslash A_1 \cup \{\widetilde v\}
	\]
	Otherwise, $A_i'$ is defined to be $A_i$. The sets $\{A_i'\}_{2\leq i \leq k}$ form
	mutually uncrossed cyclic bisections of the set $S'$. Therefore, the induction hypothesis can be used to
	construct a tree $T'$ with $k-1$ interior edges whose induced cyclic bisections are give by
	$\{A_i'\}_{2\leq i \leq k}$. Thus we can construct $T$ by adding the elements of $A_1$
	to the set of vertices of $T'$ and connecting the vertex $\widetilde v$ to the elements of $A_1$. The ribbon structure on $T'$
	and the ordering on the set $A_1$ induce a ribbon structure on $T$. It is clear that the set of non-root leaves
	of $T$ is equal to $S$ and the cyclic bisections
	associated to $T$ are given by $\{A_i\}_{1\leq i \leq k}$.
	Moreover, it is easy to check that $T$ is the unique tree with these properties.
\end{proof}

\subsubsection*{Compactifications of Arrangements of Points}
Let $\bx=\{x_1,x_2,\dots,x_n\}$ be a subset of $\R=\R \times \{(0,0)\}\subset \R^3$ that contains $n$ distinct points. We assume
\begin{equation} \label{arrangment}
  x_1<x_2<\dots<x_n
\end{equation}
where ``$<$'' is defined with respect to the natural order of the real line. We say a set of balls $\fD$ as below:
\[
   B_{r_1}(x_1)\hspace{.7cm} B_{r_2}(x_2)
  \hspace{.7cm}  \cdots \hspace{.7cm} B_{r_n}(x_n) \hspace{.7cm} B_R(\bar{x})
\]
defines a set of {\it territory balls}, if for $1\leq i \neq j\leq n$:
\begin{equation} \label{territory-cond}
	x_i\notin B_{4r_j}(x_j)\hspace{2cm} B_{r_j}(x_j)\subset B_{R/4}(\bar{x})
\end{equation}
We call $B_{r_j}(x_j)$ the territory ball of the point $x_j$. The ball $B_{R}(\bar{x})$ is called the {\it parent territory ball}. Note that \eqref{territory-cond} implies that the distance of a point $z$ in $B_{r_j}(x_j)$ from any other point in the same ball is less than the distance of $z$ from a point in the ball $B_{r_i}(x_i)$ with $i\neq j$. In particular, the balls $B_{r_1}(x_1),B_{r_2}(x_2) ,  \cdots , B_{r_n}(x_n)$ are disjoint.

We say two arrangements of points
$\bx_0$ and $\bx_1$
are {\it affinely equivalent} to each other if there is an affine linear transformation that maps $\bx_0$ to $\bx_1$.
The space of all arrangements of points as above modulo this equivalence relation forms a space $\mathcal P_n$. Two sets of territory balls $\fD_0$ and $\fD_1$ for affinely equivalent arrangements $\bx_0$ and $\bx_1$ are equivalent to each other if $\fD_0$ is mapped to $\fD_1$ by the affine linear transformation which maps $\bx_0$ to $\bx_1$. The space of all territory balls modulo this relation defines a fiber bundle $\mathcal D_n$ over $\mathcal P_n$ with contractible fibers. For each $n$, we fix a smooth section of the fiber bundle $\mathcal D_n$. Later in this section, we require that these sections satisfy some constraints.

The space $\mathcal P_n$ is diffeomorphic to an open ball of dimension $n-2$ and we shall review two different compactifications of this space. We can compactifiy the space $\mathcal P_n$ by allowing some, but not all, of the inequalities in \eqref{arrangment} to be equalities. The resulting compact space is a simplex of dimension $n-2$. We say this compactification of $\mathcal P_n$ is the {\it weak compactification} of $\mathcal P_n$. Alternatively, there is a more refined compactification of this moduli space which is called {\it associahedron} and is denoted by $\mathcal{K}_n$ \cite{St:assoc}. We also call
$\mathcal{K}_n$ the {\it strong compactification} of $\mathcal P_n$.

Given an $n$-leafed tree $T$, define $F_T$ to be the following space:
\begin{equation} \label{P-TandNP-T}
	 F_T := \prod_{v\in V^{\rm Int}(T)} \mathcal P_{d(v)}
\end{equation}
Then the associahedron $\mathcal{K}_n$, as a set, is defined to be the following disjoint union:
\begin{equation} \label{KNdecom}
	\mathcal{K}_n := \bigsqcup_{T} F_T
\end{equation}
where the union is over all $n$-ribbon trees. For each $n$-ribbon tree $T$, we define a map $\Phi_T:N(F_T) \to \mathcal{K}_n$ where $N(F_T)$ is given as below:
\begin{equation*}
	N(F_T):= F_T\times (M_0,\infty]^{|E^{\rm Int}(T)|}.
\end{equation*}
Here $M_0$ is a constant greater than $1$. For the definition of $\Phi_T$ we can assume that $M_0=1$. But as we move forward throughout the paper, we need to increase the value of this constant.

 Let $T$ have $l+1$ interior vertices and label
 interior vertices and the interior edges of $T$ respectively
 with $v_0$, $v_1$, $\cdots$, $v_l$ and $e_1$, $\cdots$, $e_l$. Consider an element $(\bx_0,\dots,\bx_l,s_1,\dots,s_{l})\in N(F_T)$. In the case that all parameters $s_i$ are equal to $\infty$, define $\Phi_T$ as follows:
\[
  \Phi_T(\bx_0,\dots,\bx_l,\infty,\dots,\infty):=(\bx_0,\dots,\bx_l)\in F_T \subset \mathcal{K}_n.
\]
Otherwise, choose a finite $s_i$ with the smallest index $i$. Without loss of generality, we assume that this parameter is $s_1$, $s(e_1)=v_0$ and $t(e_1)=v_1$. Suppose $T'$ is obtained from $T$ by shrinking the edge $e_1$. Let also $\widetilde v$ denote the new vertex of $T'$. We shall define $\widetilde \bx$ as an element of $\mathcal P_{d(\widetilde v)}$. Then the map $\Phi_T$ is defined inductively by the following property:
\begin{equation*}
	\Phi_T(\bx_0,\bx_1,\bx_2,\dots,\bx_l,s_1,\dots,s_{l})=
	\Phi_{T'}(\widetilde \bx,\bx_3,\dots,\bx_l,s_2,\dots,s_{l})
\end{equation*}

Let also $\bx_0$ (respectively, $\bx_1$) be represented by the arrangement of points $\{x_1,\dots,x_{d}\}$ (respectively, $\{x_1',\dots,x_{d'}'\}$)
where $d=d(v_0)$ and $d'=d(v_1)$.
The sections of the bundles $\mathcal D_{d}$ and $\mathcal D_{d'}$ determine territory balls:
\[
   B_{r_1}(x_1)\hspace{.5cm} B_{r_2}(x_2)
  \hspace{.5cm}  \cdots \hspace{.5cm} B_{r_{d}}(x_{d}) \hspace{.5cm}
  B_R(\bar{x})
\]
and:
\[
    B_{r_1'}(x_1')\hspace{.5cm} B_{r_2'}(x_2')
  \hspace{.5cm}  \cdots \hspace{.5cm} B_{r_{d'}'}(x_{d'}') \hspace{.5cm}
   B_{R'}(\bar{x}')
\]
We can make a correspondence between the set of outgoing edges of $v_0$ and the set of points $x_i$ using the ordering of these two sets. Suppose that the point $x_j$ is matched with the edge $e_1$. There is a unique affine transformation which maps the ball $B_{R'}(\bar{x}')$ to
$B_{e^{-s_1}r_j}(x_j)$. Use this transformation to replace the point $x_j$ with the image of the points $x_1'$, $\dots$, $x_{d'}'$. This arrangement of points defines
$\widetilde \bx \in \mathcal P_{d_1+d_2-1}$ which is independent of the choices of the representatives for $\bx_0$ and $\bx_1$.

\begin{example}
	Let $T$ and $T'$ denote the (6)-ribbon trees of Figures \ref{(6)-ribbon-tree} and \ref{s(6)-ribbon-tree}.
	Figures \ref{x123} and \ref{xx2} represent points in $F_T$ and $F_{T'}$. The arrangement $\widetilde{\bx}$ is obtained by
	merging  $\bx_1$ and $\bx_2$.
	Choices of territory balls for these arrangements are also illustrated in Figures \ref{x123} and \ref{xx2}.
	Intersections of territory balls
	with the line $\R\times \{(0,0)\}$ are demonstrated by two round brackets around each point.
	The intersections of parent territory balls with the line $\R\times \{(0,0)\}$ are also
	sketched by square brackets. In order to have a clearer figure,
	\eqref{territory-cond} is not completely satisfied in Figures \ref{x123} and \ref{xx2}.

	\begin{figure}
		\centering
\begin{tikzpicture}[scale=1]
\draw [thick,->]  (0,2)--(10,2); \node at (0.5,2) {[}; \node at (9.5,2) {]}; \node at (-0.5,2) {$\mathbf{x}_0:$};
\draw[fill] (2,2) circle (1pt);  \node at (1.5,2) {(}; \node at (2.5,2) {)};
\draw[fill] (4,2) circle (1pt);  \node at (2.8,2) {(}; \node at (5.2,2) {)};
\draw[fill] (8,2) circle (1pt);  \node at (7,2) {(}; \node at (9,2) {)};

\draw [thick,->]  (0,0)--(10,0); \node at (0.5,0) {[}; \node at (9.5,0) {]}; \node at (-0.5,0) {$\mathbf{x}_1:$};
\draw[fill] (3,0) circle (1pt);  \node at (2,0) {(}; \node at (4,0) {)};
\draw[fill] (6,0) circle (1pt);  \node at (5.5,0) {(}; \node at (6.5,0) {)};
\draw[fill] (8,0) circle (1pt);  \node at (7.6,0) {(}; \node at (8.4,0) {)};
\draw[dashed] (0.5,0) to (3,2);
\draw[dashed] (9.5,0) to (5,2);

\draw [thick,->]  (0,-2)--(10,-2); \node at (0.5,-2) {[}; \node at (9.5,-2) {]}; \node at (-0.5,-2) {$\mathbf{x}_2:$};
\draw[fill] (2,-2) circle (1pt);  \node at (1,-2) {(}; \node at (3,-2) {)};
\draw[fill] (6,-2) circle (1pt);  \node at (4.8,-2) {(}; \node at (7.2,-2) {)};

\end{tikzpicture}
	\captionof{figure}{The arrangements of points $\bx_0, \bx_1, \bx_2$ give an element of $F_T$ with $T$ being the (6)-ribbon tree of Figure
	\ref{(6)-ribbon-tree}. Territory balls for these arrangements are also sketched in this figure.}\label{x123}
\end{figure}
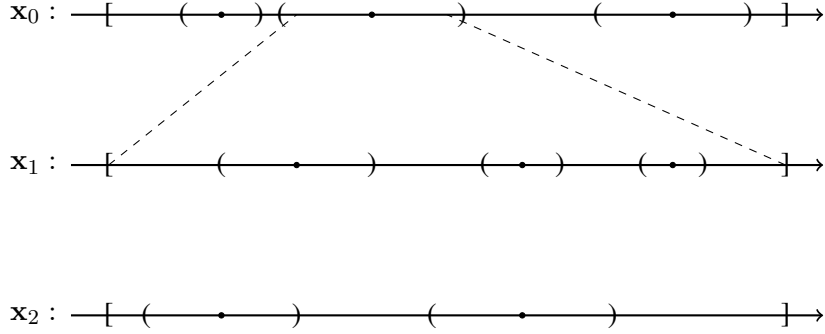

\begin{figure}
\centering
\begin{tikzpicture}[scale=1]
\draw [thick,->]  (0,2)--(10,2); \node at (0.5,2) {[}; \node at (9.5,2) {]}; \node at (-0.5,2) {$\widetilde{\mathbf{x}}:$};
\draw[fill] (2,2) circle (1pt);  \node at (1.5,2) {(}; \node at (2.5,2) {)};

\draw[fill] (3.555,2) circle (1pt);  \node at (3.333,2) {(}; \node at (3.777,2) {)};
\draw[fill] (4.4,2) circle (1pt);  \node at (4.34,2) {(}; \node at (4.46,2) {)};
\draw[fill] (4.666,2) circle (1pt);  \node at (4.62,2) {(}; \node at (4.735,2) {)};

\draw[fill] (8,2) circle (1pt);  \node at (7,2) {(}; \node at (9,2) {)};

\draw [thick,->]  (0,0)--(10,0); \node at (0.5,0) {[}; \node at (9.5,0) {]}; \node at (-0.5,0) {$\mathbf{x}_2:$};
\draw[fill] (2,0) circle (1pt);  \node at (1,0) {(}; \node at (3,0) {)};
\draw[fill] (6,0) circle (1pt);  \node at (4.8,0) {(}; \node at (7.2,0) {)};

\end{tikzpicture}
\captionof{figure}{The arrangements of points $\widetilde{\bx}, \bx_2$ give an element of $F_{T'}$ with $T'$ being the (6)-ribbon tree of Figure
	\ref{s(6)-ribbon-tree}. Territory balls for these arrangements are also sketched in this figure.}\label{xx2}
\end{figure}

\end{example}

We equip $\mathcal{K}_n$ with the weakest topology such that the maps $\Phi_T$ are all continuous.  The space $\mathcal{K}_n$ with this topology is a compact Hausdorff space
\cite{FOn, FOOO1,FOOO2, Se:Pic}. In fact, $\mathcal{K}_n$ admits the structure of a smooth manifold with corners and is a polytope of dimension $n-2$. For example, the spaces $\mathcal{K}_3$ and $\mathcal{K}_4$ are respectively an interval and a pentagon. If $M_0$ is large enough, then the maps $\Phi_T$ are diffeomorphisms. Therefore, they define charts for $\mathcal{K}_n$. We assume this requirement on $M_0$ holds for $n\leq 2N+1$.
\begin{remark}
	In \cite{FOn, FOOO1,FOOO2, Se:Pic}, $\mathcal{K}_n$ is regarded as the compactification of the moduli space of
	2-dimensional (conformal) discs with the choice of $(n+1)$ marked points on the boundary. This moduli space can
	be identified with $\mathcal P_n$ in the following way. Any conformal structure on the 2-dimensional disc
	is conformally equivalent to the standard conformal structure on the upper half-plane $\H\subset \C$.
	This conformal equivalence can be chosen such that one of the marked points is mapped to $\infty$
	in the boundary of $\H$.
	Then the remaining marked points determine an arrangement of points on the real line in $\C$, which
	is well-defined up to an affine transformation. Let $\bx=\{x_1,\dots,x_n\}$ be an element of $\mathcal P_n$.
	A territory ball $B_{r_j}(x_j)$ determines a half-disc $D_j$
	in $\H$ with the center $x_j$ and radius $r_j$.
	We can also define a half-disc associated to $B_R(\bar{x})$. The interior of the complement of
	this half-disc, denoted by $D_\infty$, can be regraded as an open neighborhood of $\infty$.
	The discs $D_1$, $\dots$, $D_n$, $D_\infty$ define disjoint open neighborhoods of the marked points.
	Such auxiliary structures are called {\it strip-like ends} in \cite{Se:Pic}.
\end{remark}

An unsatisfactory point in the definition of $\Phi_T$ is its dependence on the labeling of the interior edges of $T$. This issue can be fixed by requiring that the sections of the bundles $\mathcal D_{n}$ satisfy a compatibility condition. Suppose $T$ is an $n$-ribbon tree with exactly one interior edge. Then any element of $F_T$ consists of two arrangements of points. Moreover, we can use the above gluing construction to define a (well-defined) gluing map from $\Phi_T: F_T \times (M_0,\infty)\to\mathcal P_n$. For any element $(\bx_1,\bx_2,s)\in F_T\times (M_0,\infty)$, the territory balls for $\bx_1$ and $\bx_2$ induce a set of territory balls for the point $\Phi_T (\bx_1,\bx_2,s)$. That is to say, the sections of the bundles $\mathcal D_{n'}$, for $n'<n$, induce a section of the bundle $\mathcal D_n$ over the following open subset of $\mathcal P_n$:
\begin{equation} \label{Un}
  \mathcal U_n=\bigcup_{T} \Phi_T(F_T\times (M_0,\infty))
\end{equation}
where the union is over all $n$-ribbon trees with one interior edge. A priori, this section might not be well-defined because the spaces $\Phi_T(F_T\times (M_0,\infty))$ overlap with each other.
 \begin{definition}\label{compatible-territory}
	Suppose that sections of the fiber bundles $\mathcal D_{n}$ are fixed for $n\leq n_0$.
	We say that these sections of $\mathcal D_n$ are {\it compatible with gluing}, if for each $n$ the section of
	$\mathcal D_n$ over $\mathcal U_n$ agrees with the one induced by the sections of the bundles
	$\mathcal D_{n'}$ with $n'<n$.
 \end{definition}

The following lemma is the analogue of Lemma 9.3 in \cite{Se:Pic}. For the remaining part of this paper, we assume that the claim of this lemma holds for $n_0=2N+1$.
\begin{lemma} \label{consistency-territory-balls}
	For any $n_0$, if $M_0$ is large enough, then the sections of the bundles $\mathcal D_n$
	with $n\leq n_0$ can be chosen such that they are compatible with gluing.
\end{lemma}
\begin{proof}[Sketch of the proof]
	The sections of the bundle $\mathcal D_n$ can be constructed by induction.
	Suppose these sections are constructed for $n'<n$. The induction assumption shows that the induced section
	of $\mathcal D_n$ on the open set $\mathcal U_n$ is well-defined if $M_0$ is large enough.
	Since the fibers of $\mathcal D_n$ are contractible, this section can be extended to
	$\mathcal P_n$.
\end{proof}

For an $n$-leafed tree $T$ with $l$ interior edges, the space $F_T$ forms an open face of co-dimension $l$. From the description of the topology, it is straightforward to see that the corresponding face, denoted by $\overline F_T$, is given by:
\[
  \bigcup_{T'\leq T} F_{T'}.
\]
This also implies that $\overline F_T$ is equal to $\prod_{v\in V^{\rm Int}(T)} {\mathcal{K}_{d(v)}}$. Therefore, each face of the associahedron $\mathcal{K}_n$ is a cartesian product of smaller associahedra.

Co-dimension one faces of the associahedra $\mathcal{K}_n$ are labeled by $n$-ribbon trees with one interior edge. Therefore, Lemma \ref{bisection-to-tree} implies that the set of co-dimension one faces are in correspondence with the cyclic-bisections of the set $[n]$. Suppose $A$ is a cyclic bi-section of $[n]$. We write $T_{A}$ and $\overline F_{A}$ for the corresponding $n$-ribbon tree and the corresponding face of $\mathcal{K}_n$. We also define $U_{A}$ to be the following subspace of $\mathcal{K}_n$:
\[
  U_{A}:=\bigcup_{T\leq T_{A}} \Phi_T(F_T\times (M_0,\infty]^{|E^{\rm Int}(T)|}).
\]
The space $U_{A}$ is clearly an open neighborhood of $\overline {F}_{A}$.

Let $T$ be an $n$-ribbon tree such that $T\leq T_{A}$. Suppose $T$ has $l$ interior edges denoted by $e_1$, $\dots$, $e_{l}$  such that the cyclic bisection associated to $e_1$ is $A$. Let $\bx= (\bx_0,\dots,\bx_l,s_1,\dots,s_{l})\in F_T\times (M_0,\infty]^{l}$ and define:
\[
  \pi_{A}(\bx):=\Phi_T(\bx_0,\dots,\bx_l,\infty,s_2,\dots,s_{l})
\]
By choosing $M_0$ large enough, the point $ \pi_{A}(\bx)$ and the number $l_{A}(\bx):=s_1$ depend only on $\bx$ (and not on $T$). This allows us to identify $ U_{A}$ with $\overline F_{A}\times (M_0,\infty]$. That is to say, $ U_{A}$ is a regular neighborhood of the face $\overline F_{A}$, as in Subsection \ref{family-general}.


More generally, we define $U_T$, for an $n$-ribbon tree $T$ with $l$ interior edges, to be the following neighborhood of $\overline{F}_T$:
\[
  U_T:=\bigcap_{T\le T_{A}} U_{A}
\]
Similar to $U_{A}$, $U_T$ can be identified with $\overline{F}_T\times (M_0,\infty]^{l}$. Note that if $T_A\geq T$, then the projection $\pi_A$ maps $U_T$ to itself . If $A'$ is another cyclic bisection that $T_{A'}\geq T$, then compatibility of the sections of the bundles $\mathcal D_{n'}$ with respect to gluing implies that the maps $\pi_A$ and $\pi_{A'}$ commute with each other. Define $\pi_T:U_T \to \overline{F}_T$ to be the map given by composing all projection maps $\pi_A$ such that $T_A\geq T$. Above discussion shows that this map is independent of the order of composition of the projection maps $\pi_A$. The map $\pi_T$ defines the first factor of the diffeomorphism from $U_T$ to $\overline{F}_T\times (M_0,\infty]^{l}$. The component of this diffeomorphism in $(M_0,\infty]^{l}$ can be also defined similarly.

\begin{prop} \label{open-sets}
	For two cyclic bisections $A_0$ and $A_1$, the open sets
	$U_{A_0}$ and $U_{A_1}$ intersect if and only if $A_0$ and $A_1$ are uncrossed. More generally,
	suppose $T_1$ and $T_2$ are two $n$-ribbon trees.
	Then $U_{T_1}$ and $U_{T_2}$ intersect if and only if there is an $n$-ribbon tree $T$ such that
	$T\le T_1$ and $T\le T_2$.
\end{prop}
\begin{proof}
	The second part of the proposition is a consequence of the first part.
	To prove the first part of the proposition, let $U_{A_0}$ and $U_{A_1}$ contain a point in their intersection.
	Since these sets are open, we can assume that $\bx \in \mathcal P_n$.
	Therefore, $\bx$ is in correspondence with the arrangement of $n$ points $\{x_1,\dots,x_n\}$ on the real line.
	The assumption \eqref{territory-cond} implies that:
	 \[
	  \hspace{3cm}  |x_i-x_j|<|x_i-x_k|\hspace{2cm} i,\,j \in A_0,\,\,\, k \notin A_0
	 \]
	 A similar claim holds if we replace $A_0$ with $A_1$. Consequently, either the sets $A_0$ and $A_1$
	 are disjoint or one of them contain the other one. That is to say, $A_0$ and $A_1$ are uncrossed.
	 Conversely, if $A_0$ and $A_1$ are uncrossed, then there is an $n$-ribbon tree $T$ such that
	 $T<T_{A_0}$ and $T<T_{A_1}$. This implies that $U_{A_0}$ and $U_{A_1}$
	 both contain the space $ \Phi_T(F_T\times (M_0,\infty]^{|E^{\rm Int}(T)|})$.	
\end{proof}

There is a map from the strong compactification to the weak compactification of $\mathcal P_n$, denoted by $\mathfrak{F}: \mathcal{K}_{n}\to\Delta_{n-2}$, which is equal to the identity map on the space $\mathcal P_n$. Intuitively, the compactification $\mathcal{K}_n$ allows different points to merge. It also has the the additional information that records the rate of convergence of different points to each other. The map $\fF$, called the {\it forgetful map}, forgets this additional information. To give the definition of $\fF$, let $T$ be an $n$-ribbon tree with root $l_0$ and the non-root leaves $l_i$. Suppose the root $l_0$ is connected to the interior vertex $v$ of $T$. Suppose also $e_1$, $\dots$, $e_{d(v)}$ are the edges of $T$ whose sources are equal to $v$. We assume that the labeling of these edges is given by the ribbon structure of the vertex $v$. We also define $n_i$ to be the number of leaves $l_j$ such that the unique path from $l_j$ to $l_0$ contains the edge $e_i$. Any $p\in F_T$ associates an arrangement $\{x_1,\dots,x_{d(v)}\}$ of $d(v)$ points to $v$. Then $\mathfrak{F}(p)$ is the arrangement of $n$ points in which $x_i$ appears with multiplicity $n_i$. It is straightforward to check that the map $\fF\circ \Phi_T$ for any $n$-ribbon tree $T$ is continuous. Therefore, the map $\mathfrak{F}$ is continuous.

\subsubsection*{Annular Decomposition}

Let $T$ be an $n$-ribbon tree with interior vertices $v_0$, $\dots$, $v_l$ and interior edges $e_1$, $\dots$, $e_l$. Suppose $\bx \in \mathcal P_n$ belongs to the open subset $F_T\times (0,\infty)^l$ of $U_T$. Then $\bx=\Phi_T(\bx_0,\dots,\bx_l,s_1,\dots,s_l)$ for appropriate choices of $(\bx_0,\dots,\bx_l)\in F_T$ and $(s_1,\dots,s_l)\in (M_0,\infty)^l$. For a given $i$, suppose $\{x_1,\dots,x_{d(v_i)}\}$ is a representative for $\bx_i$, $B_{r_j}(x_j)$ is the territory ball around $x_j$, and $B_{R}(\overline x)$ is the parent territory ball. Then we can form the  3-dimensional annular regions:
\begin{equation} \label{annulus}
  B_{r_j}(x_j)\backslash \{x_j\}
\end{equation}
around $x_j$ and the annular region $\R^3\backslash B_{R}(\overline x)$ around infinity. These annular regions are in correspondence to the edges of $T$. Each interior edge $e_k$ of $T$ with source $v_{i}$ and target $v_{i'}$ is in correspondence with two annular regions; one of these regions is a 3-dimensional annulus of the form \eqref{annulus} and the other one is an annular region around infinity. On the other hand, each leaf of $T$ is matched with exactly one annular region.

In order to form the punctured Euclidean space $\R^3\backslash \bx$, we glue together annular regions corresponding to the interior edges\footnote{For two different representatives of $\bx$, there is a unique affine transformation that identifies the corresponding punctured Euclidean spaces. Therefore, $\R^3\backslash \bx$ is well-defined up to a canonical isomorphism. See also the definition of $\fB_\bx$ in Subsection \ref{GH-family}.}. Suppose $e_k$ is an interior edge, connecting its source $v_i$ to its target $v_{i'}$. Suppose also the annular region associated to $e_k$ in $\R^3\backslash \bx_i$ and $\R^3\backslash \bx_{i'}$ are respectively denoted by $B_{r}(x)\backslash \{x\}$ and $\R^3\backslash B_{R}(\overline x)$. Remove the ball of radius $e^{-\frac{s_k}{2}}r$ from $B_{r}(x)\backslash \{x\}$ to obtain $A_1$. The space $A_1$ is an annulus with inner radius $e^{-\frac{s_k}{2}}r$ and outer radius $r$. We also take the intersection of $\R^3\backslash B_{R}(\overline x)$ with the closure of the ball centered at $\overline x$ and with radius $e^{\frac{s_k}{2}}R$ to obtain $A_2$. Then $A_2$ is also an annulus with inner radius $R$ and outer radius $e^{\frac{s_k}{2}}R$. We glue $A_1$ along its inner sphere to $A_2$ along its outer sphere using an affine transformation. We write $A(e_k,\bx,T)$ for the the union of the  following subsets of $A_1$ and $A_2$ after gluing:
\[
  A_1\cap B_{e^{-\frac{M_0}{2}}r}(x)\hspace{1cm} A_2\backslash  \overline{B_{e^{\frac{M_0}{2}}R}(\overline x)}
\]
The space $A(e_k,\bx,T)$ is also an annular region where the ratio of its outer radius to the inner radius is equal to $e^{s_k-M_0}$. We call $A(e_k,\bx,T)$ the {\it neck} of $\R^3\backslash \bx$ associated to the edge $e_k$ of $T$. Applying the above gluing construction to all interior edges gives rise to $\R^3\backslash \bx$.

We can also associate subspaces of $\R^3\backslash \bx$ to each leaf and each interior vertex of $T$. Suppose $l$ is a leaf of $T$ and $A(l,\bx,T)$ is the subspace of $\R^3\backslash \bx$ induced by the annular region associated to $l$. We call $A(l,\bx,T)$ the {\it end} of $\R^3 \backslash \bx$ associated to the leaf $l$ of $T$. The complement of the necks and the ends of $\R^3\backslash \bx$ is a union of $l+1$ compact connected spaces, one for each interior vertex. The {\it fat region} of $\R^3\backslash \bx$ associated to the interior vertex $v_i$ of $T$ is the connected component of this space corresponding to $v_i$. The fat region associated to the vertex $v_i$ can be also regarded as a subset of $\R^3\backslash \bx_i$. Decomposing $\R^3\backslash \bx$ as the above disjoint union of necks, ends and fat regions is called the {\it annular decomposition} of $\R^3\backslash \bx$ associated to the tree $T$.

The definition of annular decomposition can be extended to the open neighborhood $F_T\times (0,\infty]^l$ of $F_T$. In this neighborhood the parameters $s_k$ are allowed to be $\infty$. If $s_k=\infty$, then we leave the annular regions associated to the interior edge $e_k$ without change. Furthermore, $A(e_k,\bx,T)$ is the disjoint union of two sets of the following form:
\begin{equation} \label{neck-infinite}
  B_{e^{-\frac{M_0}{2}}r}(x)\backslash \{x\}\hspace{1cm} \R^3\backslash  \overline{B_{e^{\frac{M_0}{2}}R}(\overline x)}.
\end{equation}
According to Proposition \ref{open-sets}, for any point $p\in \mathcal{K}_n$, there is a unique $n$-ribbon tree $T$ such that $T$ is minimal with respect to the relation $\leq$ and $p \in F_T\times (0,\infty]^l$. The annular decomposition of $p$ (without reference to any ribbon tree) is defined as the annular decomposition with respect to this minimal tree.

Suppose again $T$ is a ribbon tree with interior vertices $v_0$, $\dots$, $v_l$ and interior edges $e_1$,$\dots$, $e_l$. Suppose $\bx\in \mathcal P_n$ is equal to $\Phi_T(\bx_0,\dots,\bx_l,s_1,\dots,s_l)$ where $ (\bx_0,\dots,\bx_l)\in F_T$ and $(s_1,\dots,s_l)\in (0,\infty)^l$. Let $\gamma_0, \dots, \gamma_l:\R^3\backslash \bx \to \R$ be the functions determined by the properties (i), (ii) and (iii) below.
\begin{itemize}
	\item[(i)] The function $\gamma_i$ is equal to $1$ on the fat region associated to the interior vertex $v_i$ and vanishes
	on the fat region associated to any other vertex.
	\item[(ii)] If $l$ is a leaf incident to the vertex $v_i$, then $\gamma_i$ is equal to $1$
	on the end associated to the leaf $l$.
	\item[(iii)] Suppose the neck $A(e_k,\bx,T)$ associated to the interior edge $e_k$ is a 3-dimensional annulus
	with inner radius $r$ and the outer radius $e^{\tau}r$. If $v_i$ is the target of $e_k$, then $\gamma_i$ at $y\in A(e_k,\bx,T)$
	is defined as follows:
	\[
	  \gamma_i(y)=\varphi_1(\frac{\ln(\frac{|y|}{r})}{\tau})
	\]
	where the function $\varphi_1$ is given in \eqref{phi-1-2}.
	If $v_i$ is the source of $e_k$, then $\gamma_i$ at $y\in A(e_k,\bx,T)$
	is defined as follows:
	\[
	  \gamma_i(y)=\varphi_2(\frac{\ln(\frac{|y|}{r})}{\tau}).
	\]
	where the function $\varphi_2$ is given in \eqref{phi-1-2}.
	If $v_i$ is not incident to $e_k$, then $\gamma_i$ vanishes on $A(e_k,\bx,T)$.
\end{itemize}
The functions $\gamma_0$, $\dots$, $\gamma_l$ can be used to glue a list of functions $\{f_i:\R^3\backslash \bx_i \to \R\}_{0\leq i \leq l}$ and to form a function $f:\R^3\backslash \bx \to \R$. The function $f$ is defined as below:
\[
  f:=\sum_{i=0}^l \gamma_i f_i
\]
We say $f$ is given by gluing $f_0$, $\dots$, $f_l$ along the ribbon tree $T$ with parameters $(s_1,\dots,s_l)$. We can apply a similar gluing process in the case that $f_i$ are differential forms of the same degree (or any other set of tensorial objects of the same type) on the spaces $\R^3\backslash \bx_i$.

Definition of the functions $\gamma_0$, $\dots$, $\gamma_l$ can be extended to the case that some of the parameters $s_k$ are equal to $\infty$. If $s_k=\infty$, then the neck corresponding to $e_k$ consists of two annular regions as in \eqref{neck-infinite}. In this case, if $v_i$ is the source of $e_k$, then $\gamma_i$ is equal to $1$ on the first set in \eqref{neck-infinite} and vanishes on the second set. If $v_i$ is the target of $e_k$, then $\gamma_i$ is equal to $1$ on the second set in \eqref{neck-infinite} and vanishes on the first set. If $v_i$ is not incident to $e_k$, then $\gamma_i$ vanishes on the neck corresponding to $e_k$. Using this extension, we can glue functions along the ribbon tree $T$ with parameters $(s_1,\dots,s_l)$ where some of the $s_k$ could be $\infty$.

\subsection{Associahedron of Metrics} \label{abst-assoc-family}
Suppose $W$ is a 4-manifold with $n+1$ boundary components $Y_0$, $\dots$, $Y_n$. For a cyclic bisection $A$ of the set $[n]$, let $Y_{A}$ be a cut of $W$ such that $W\backslash Y_A$ has connected components $W^1_{A}$ and $W^2_{A}$ with:
\[
  \partial W^1_{A}=\{Y_{A}\}\cup \{Y_i\mid i \notin A\} \hspace{1cm}
  \partial W^2_{A}=\{-Y_{A}\}\cup \{Y_i\mid i \in A\}
\]
We say the set of cuts $\{Y_{A}\}$ is of  {\it associahedron type}, if for any two uncrossed bisections $A_0$ and $A_1$, the cuts $Y_{A_0}$ and $Y_{A_1}$ are disjoint. We then pick, for every cut $Y_{A}$, a tubular neighborhood $N(Y_{A})$ which is diffeomorphic to $[-2,2] \times Y_{A} $. We also assume that  $N(Y_{A_0})$ and $N(Y_{A_1})$ are disjoint when  $Y_{A_0}$ and $Y_{A_1}$ are disjoint. Suppose $T$ is an $n$-ribbon tree with interior vertices $v_0$, $\dots$, $v_l$ and interior edges $e_1$, $\dots$, $e_l$. Let also $A_1$, $\dots$, $A_l$ denote the cyclic bisections associated to $T$. Then we define $Y_T$ to be the union of the disjoint cuts $Y_{A_1}$, $\dots$, $Y_{A_l}$. Thus each interior edge of $T$ is in correspondence with one of the connected components of $Y_A$ and each leaf of $T$ is in correspondence with one of the boundary components of $W$. The 4-manifold $W\backslash (-1,1)\times Y_T$ has $l+1$ connected components, one for each interior vertex of $v$. We denote the connected component of $W\backslash (-1,1)\times Y_T$ associated to the interior vertex $v_i$ by $W_{i}$. The set of the boundary components of $W_{i}$ are given by the 3-manifolds associated to the edges which are incident to $v_i$.

\begin{example} \label{asso-type-cuts}
	The $4$-manifold $W_{k}^{j}$, introduced in Section \ref{topology},
	has $k-j+2$ boundary components:
	\[
	  -Y_j\hspace{1cm}M^{j}_{j+1}\hspace{1cm}M^{j+1}_{j+2}\hspace{1cm}\dots\hspace{1cm}M^{k-1}_{k}
	  \hspace{1cm}Y_k
	\]
	 The vertical and the spherical cuts provide
	a set of cuts of associahedron type for $W^j_k$.
\end{example}

For a manifold $W$ as above, we can construct a family of metrics parametrized by the associahedron $\mathcal K_{n}$ where the cut associated to the face $F_T$ of $\mathcal K_{n}$ is equal to $Y_T$. We first choose a background Riemannian metric $g_0$ on $W$ which has cylindrical ends along the boundary components $Y_0$, $\dots$, $Y_n$. We equip $Y_A$ with a Riemannian metric $h_{A}$, and on $N(Y_{A})$ we construct a smooth $1$-parameter family of metrics $g_{A}(s)$, for $s \in [0,\infty)$, with the following properties:
\hspace{-5pt}
\begin{enumerate}
	\item[(i)] For all values of $s$, $g_{A}(s)$ extends to a smooth metric on $W$
	which is equal to $g_0$ on $W \;\setminus \; N(Y_A)$.
	\item[(ii)] For $s<1$, we have $g_{A}(s) = g_0$ on $N(Y_{A})$.
	\item[(iii)] For $s\geq 2$, the restriction of $g_{A}(s)$ to the region
	$Y_{A} \times [-2,-1) \cup Y_{A} \times (1,2]$ is independent of $s$.
	\item[(iv)]For $s\geq 2$, we have $g_{A}(s) = u_s(r)^2dr^2 + h_{A}$ on $Y_{A} \times [-1,1]$, where
	 the function $u_s(t)$ is fixed in Subsection \ref{family-general}.
\end{enumerate}
\hspace{-5pt}
For each cyclic bisection $A$, we also fix a function $\gamma_A:\mathcal K_{n} \to \R\cup \{\infty\}$ such that $\gamma_{A}$ is zero on the complement of $\overline F_{A}\times (M_0,\infty]\subset \mathcal K_{n}$, and $\gamma_{A}(\bx,s)=s-M_0$ for $(\bx,s) \in \overline F_{A}\times (M_0+1,\infty] \subset U_{A}$. Then Proposition \ref{open-sets} implies that for any $p\in \mathcal K_n$, the numbers $\gamma_{A_0}(\bx)$ and $\gamma_{A_1}(\bx)$ are both non-zero only if $A_0$ and $A_1$ are uncrossed.

We can now describe the family of metrics on $W$. Fix an element $\bx$ in $\mathcal K_{n}$. We need to define a (possibly broken) metric $g(\bx)$ on $W$. We firstly define this metric on $N(Y_{A})$ for a cyclic bisection $A$. If $\gamma_{A}(\bx)$ is a finite number, then the metric $g(\bx)$ on the tubular neighborhood $N(Y_{A})$ is equal to $g_{A}(\gamma_{A}(\bx))$. Otherwise, we define $g(\bx)$ to be defined as follows:
\[
  g(\bx)=
  \left \{
  \begin{array}{ll}
 	g_{A}(2) &\text{on }Y_{A} \times [-2,-1) \cup Y_{A} \times (1,2]\\
	\frac{1}{r^2}dr^2 + h_{A}&\text{on }Y_{A} \times [-1,1]
  \end{array}
  \right.
\]
On the complement of the tubular neighborhoods $N(Y_{A})$, we define $g(\bx)$ to be equal to $g_0$. The properties of the metrics $g_A$ and the functions $\gamma_A$ show that $g(\bx)$ is a well-defined metric and these metrics together form a family of metrics on $W$ parametrized by $\mathcal K_{n}$.

\subsection{Families of Gibbons-Hawking Metrics on ALE spaces}\label{GH-family}
In the previous subsection, we introduced a general construction to define a family of metrics parametrized by an associahedron. In particular, this construction can be applied to $W^j_k$. For our purposes, we need to modify this family of metrics. In this subsection, we construct families of metrics on some 4-manifolds which appear as submanifolds of $W^j_k$. The modified family of metrics on $W^j_k$ shall be discussed in the next subsection.

\subsubsection*{Gibbons-Hawking Metric}

Suppose $x_1$, $\dots$, $x_n$ are distinct points in $\R=\R \times \{(0,0)\}\subset \R^3$. Suppose also  $m_0$, $m_1$, $\dots$, $m_n$ is an increasing sequence of integer numbers. Then define:
\begin{equation} \label{Green-function}
  \hspace{2cm} u(q)=\sum_{i=1}^{k}\frac{m_i-m_{i-1}}{|q-x_i|}\hspace{2cm}q\in \R^3
  \setminus\{x_1,\dots,x_n\}.
\end{equation}
The function $u$ is harmonic and hence $\alpha:=*du$ defines a closed 2-form  on the space $\R^3\setminus\{x_1,\dots,x_n\}$ where $*$ is the Hodge operator associated to the Euclidean metric.
Integrality of the numbers $m_i$ implies that $\alpha$ represents an integral cohomology class on $\R^3\setminus\{x_1,\dots,x_n\}$. For $0\leq i \leq n$, let $\fd_i$ be the straight path $(x_i,x_{i+1})$ on the line $\R \times \{(0,0)\}$ which is oriented in the increasing direction. Here we assume that $x_0=-\infty$ and $x_{n+1}=\infty$. Then the cohomology class of $\alpha$ is Poincar\'e dual to:
\begin{equation} \label{divisor}
  -\sum_{i=0}^n m_i\fd_i.
\end{equation}
There is also a $\U(1)$-bundle $L$ over $\R^3\setminus\{x_1,\dots,x_n\}$ whose first Chern class is represented by $\alpha$. The divisor in \eqref{divisor} determines a canonical choice of this bundle. The total space of $L$ in a punctured neighborhood of $x_i$ is diffeomorphic to $\R^{>0} \times L(m_{i-1}-m_i,1)$ where the first factor parametrizes $|q-x_i|$.

The set $A=\{x_1,\dots,x_n\}$ represents an element $\bx\in \mathcal P_{n}$. Given another representative $A'=\{x_1',\dots,x_n'\}$ for $\bx$, there is a unique affine diffeomorphism $\Phi^A_{A'}:\R^3\setminus \{x_1,\dots,x_n\}\to \R^3\setminus \{x_1',\dots,x_n'\}$. We define $\fB_\bx$ to be:
\begin{equation} \label{can-ball}
  \fB_{\bx}=(\bigsqcup_{A} \R^3\setminus A)/\sim
\end{equation}
where $A$ runs over all representatives of $\bx$ and $\sim$ is the equivalence relation defined by the maps $\Phi^A_{A'}$. A differential form $\beta$ on $\fB_\bx$ is given by differential forms $\beta_A$ on $\R^3\setminus A$, for any $A$, such that $(\Phi^A_{A'})^*(\beta_{A'})=\beta_A$. Similarly, we define any other geometrical object on $\fB_\bx$. The function $u$ in \eqref{Green-function} does not give rise to a function on $\fB_\bx$. However, $\alpha=*du$ determines a well-defined differential 2-form on $\fB_\bx$, which we denote by $\alpha_{\bm,\bx}$. Similarly, $Q_{\bm,\bx}$ defined by the quadratic form:
\[
 u^2(dx^2+dy^2+dz^2)
\]
is well-defined on $\fB_{\bx}$. The definition of annular decomposition is also compatible with respect to the equivalence relation $\sim$ in \eqref{can-ball}. Therefore, it makes sense to talk about annular decomposition of $\fB_\bx$. We also define $\fX_{\bm,\bx}$ by defining an analogous equivalence relation on the set of all $\U(1)$-bundles $L$, defined for different choices of representatives of $\bx$. If we are only interested in the diffeomorphism type of the 4-manifold $\fX_{\bm,\bx}$ and the choice of $\bx$ is not important for us, then we denote this space by $\fX_\bm$.

Let $\omega$ be a connection on $\fX_{\bm,\bx}$ whose curvature is equal to $\alpha_{\bm,\bx}$. This connection is given by a 1-form on the total space of $\fX_{\bm,\bx}$, which we also denote by $\omega$. Consider the following metric on $\fX_{\bm,\bx}$:
\begin{equation} \label{GH-metric}
	g:=Q_{\bm,\bx}+\omega^2
\end{equation}
Since $\R^3\setminus\{x_1,\dots,x_n\}$ is simply connected, any other connection $\omega'$ with curvature $\alpha_{\bm,\bx}$ is the pull-back of the connection $\omega$ by an automorphism of the bundle $\fX_{\bm,\bx}$. Therefore, the metrics induced by $\omega$ and $\omega'$ are isometric.

\begin{prop} \label{GH-prop}
	The metric $g$ on $\fX_{\bm,\bx}$ is an anti-self-dual asymptotically cylindrical metric with positive scalar curvature.
\end{prop}
\begin{proof}
        Let $\widetilde g$ be the metric $g/u$, a metric in the same conformal class as $g$. (This metric depends on the choice of the representative $\bx$.) The metric $\widetilde g$, introduced by Gibbons and Hawking, is hyper-K\"ahler \cite{GH:multi-ins}. In particular, it is an anti-self-dual metric and has vanishing scalar curvature. Thus, the metric $g$ is also anti-self-dual. The scalar curvature of $g$ is given by the following formula:
        \begin{align*}
        	R&=6u^{-\frac{3}{2}}\Delta_{\widetilde g}(u^{\frac{1}{2}})
       \end{align*}
        where $\Delta_{\widetilde g}$ is the Laplace-Beltrami operator\footnote{We use the the convention of differential geometry for the sign of Laplace-Beltrami operator, i.e., $\Delta=d^*d$.}
        with respect to the metric $\widetilde g$. In general, if $\widetilde f:L \to \R$ is the pull-back of a function $f$ on $\R^3\setminus\{x_1,\dots,x_n\}$, then $\Delta_{\widetilde g}(\widetilde f)$ is equal to the pull-back of $u^{-1}\Delta(f)$ where $\Delta(f)$ is the standard Laplace operator. Therefore, the scalar curvature of $g$ is given by the following expression which is always positive:
        \begin{equation*}
        	\frac{3}{2}u^{-4}(u_{x_1}^2+u_{x_2}^2+u_{x_3}^2)
        \end{equation*}
        	
        The metric $g$ has a nice asymptotic behavior. Firstly let $k=1$ and $x_1$ be the origin in $\R^3$. Let also $u_0$ be the harmonic function $\frac{m_1-m_0}{r}$, where $r=|q|$. The associated $\U(1)$-bundle, the 2-form, the connection and the metric are denoted by $L_0$, $\alpha_0$, $\omega_0$ and $g_0$. The associated divisor is also equal to $-m_0\fd_0-m_1\fd_1$ which is invariant with respect to the dilation maps $\phi_\lambda(q)=\lambda\cdot q$. The 2-form $\alpha_0$ is also invariant with respect to $\phi_\lambda$. Moreover, the contraction of $\alpha_0$ in the radial direction vanishes.

        The bundle $L_0$ is the pull-back of a $\U(1)$-bundle $l_{m_0,m_1}$ on $S^2$ with respect to the radial projection map $\pi$. Here $S^2$ denotes the sphere of radius one in $\R^3$ centered at the origin with the non-standard orientation. Let $\fr,\fl\in S^2$ be the points $(1,0,0)$ and $(-1,0,0)$. Then $l_{m_0,m_1}$ is the $\U(1)$-bundle on $S^2$ associated to the divisor $m_0\fl-m_1\fr$. For each choice of $m_0$ and $m_1$, we fix a connection $\mu_{m_0,m_1}$ on $l_{m_0,m_1}$ whose curvature is represented by a harmonic 2-form (with respect to the standard metric $\sigma$ on $S^2$). We can pick $\omega_0$ to be the pull-back connection $\pi^*(\mu_{m_0,m_1})$. Then the metric $g_0$ can be written as:
        \begin{equation}\label{g0}
        	g_0=(m_1-m_0)^2\frac{dr^2}{r^2}+\pi^*((m_1-m_0)^2I_0+\mu_{m_0,m_1}^2)
        \end{equation}
        where the quadratic form $I_0$ on $l_{m_0,m_1}$ is the pull-back of the standard metric $\sigma$ on $S^2$. We reparametrize $L_0$ using the (orientation preserving) diffeomorphism $\Phi: \R \times l_{m_0,m_1} \to L_0$, defined as follows:
        \begin{equation} \label{repar}
        	\Phi(t,x)=e^{-t}\cdot x.
        \end{equation}	
	Then $g_0$ is the product metric $(m_1-m_0)^2dt^2+h$ with $h$ being the metric $(m_1-m_0)^2I_0+\mu_{m_0,m_1}^2$ on $ l_{m_0,m_1}$.
        	
        Next, we consider the more general case that $k>1$. Let the following balls be the set of territory balls for $\{x_1,\dots,x_n\}$:
        \[
           B_{r_1}(x_1)\hspace{.7cm} B_{r_2}(x_2)
          \hspace{.7cm}  \cdots \hspace{.7cm} B_{r_n}(x_n) \hspace{.7cm} B_R(\overline {x})
        \]
        For $1\leq i \leq k$, let $D_i$ be the punctured ball $B_{r_i}(x_i)\setminus \{x_i\}$, and $D_\infty$ be the interior of the complement of $B_R(\overline x)$. We also define $\widetilde D_i$ to be the punctured closed ball $\overline {B_{r_i/2}(x_i)}\setminus \{x_i\}$ and $\widetilde D_\infty$ to be the complement of $B_{2R}(\overline x)$. Consider the diffeomorphisms $\Psi_i: (0,\infty)\times S^2 \to D_i$ and $\Psi_\infty: (-\infty,0)\times S^2 \to D_\infty$ defined as follows:
        \begin{equation}\label{cylinder-diffeo}
          \Psi_i(t,y)=r_i\cdot e^{-t}\cdot y\hspace{1cm}\Psi_i(t,y)=R\cdot e^{-t}\cdot y
        \end{equation}

        As in \eqref{repar}, we can fix a canonical isomorphism of $\fX_{\bm,\bx}|_{D_i}$ with the cylinder $(0,\infty) \times l_{m_{i-1},m_i}$ and a canonical isomorphism of $\fX_{\bm,\bx}|_{D_\infty}$ with $(-\infty,0) \times l_{m_0,m_k}$ by lifting the above diffeomorphisms. We also fix a connection $\omega'$ on $\fX_{\bm,\bx}$ such that it agrees with the pullback of the connection $\mu_{m_{i-1},m_i}$ on $\fX_{\bm,\bx}|_{D_i}$ and the pullback of the connection $\mu_{m_0,m_k}$ on $\fX_{\bm,\bx}|_{D_\infty}$. Then $a:=\alpha_{\bm,\bx}-F(\omega')$ induces a 2-form on $\fB_{\bx}$ which satisfies the following exponential decay condition for any point $(t,y)\in D_i$ or $D_\infty$ and any integer number $k$:
        \begin{equation} \label{decay}
          \hspace{4cm} |\nabla^ka(t,y)|\leq C_k e^{-\delta |t|} \hspace{1cm} (t,y)\in D_i \text{ or } D_\infty.
        \end{equation}
        Here $(t,y)$ is defined using the maps in \eqref{cylinder-diffeo}, $\delta$ is a positive real number independent of $k$ and $C_k$ is a positive constant.
	Analogous to $M_0$ in the previous section, we might need to increase the value of constants $\delta$ and $C_k$ as we move forward throughout the paper.
        The above norm is defined with respect to the cylindrical metric $dt^2+\sigma$ on $D_i$ rather than the Euclidean metric.

        The 2-form $a$ is closed and we wish to find a 1-form $b$ with the similar decay condition as in \eqref{decay} such that $db=a$.
        For $1\leq i \leq n$ or $i=\infty$, we can write the restriction of $a$ to $D_i$ in the following way:
	\[
          \left . a\right|_{D_i}=\alpha_i(t)+\beta_i(t) dt
        \]
        where $\alpha(t)$ and $\beta(t)$, for each $t$, are 2- and 1-forms on $S^2$. Then define:
	\begin{equation} \label{integrate-2-form-1}
          \hspace{2cm} b_i(t)=-\int_t^\infty \beta_i(s)\,ds \hspace{1cm}1\leq i \leq n
	\end{equation}
        and
	\begin{equation} \label{integrate-2-form-2}
          \hspace{2cm} b_\infty (t)=\int_{-\infty}^t \beta_\infty(s) \,ds
	\end{equation}
        Then $b_i$ has the similar decay condition as in \eqref{decay} over $D_i$ and $db_i=a$. Suppose $\phi_i$ is a function on $\fB_\bx$ which is zero outside of $D_i$ and is equal to $1$ on a neighborhood of $\widetilde D_i$. Then:
        \begin{equation} \label{reminder}
          a-d(\phi_1b_1+\dots+\phi_nb_n+\phi_\infty b_\infty)
        \end{equation}
        is a closed 2-form on $\R^3$ which is supported in the complement of the union of $\widetilde D_i$. Since $\fB_{\bx}\backslash \bigcup \widetilde D_i$ has trivial second cohomology with compact support, we can find a 1-form $b'$ supported outside of the union of $\widetilde D_i$'s such that $db'$ is equal to the 2-form in \eqref{reminder}. Therefore, the exterior derivative of:
        \[
          b:=\phi_1b_1+\dots+\phi_nb_n+\phi_\infty b_\infty+b'
        \]
        is equal to $a$ and $b$ is exponentially decaying as in \eqref{decay}. In particular, $\omega:=\omega'+b$ defines a connection on $\fX_{\bm,\bx}$ such that $F(\omega)=\alpha_{\bm,\bx}$. Given this, it is straightforward to show that
        $g$ is an asymptotically cylindrical metric.

        We can give an explicit choice for the 1-form $b'$. Although we will not use it here, this canonical choice will be useful for us later when we need to repeat the above construction in family.
        To construct the 1-form $b'$, let $U_i$, for $1\leq i\leq n-1$ be a small closed neighborhood of the path along the $x$-axis in $\R^3$ which connects $\widetilde D_i$ to $\widetilde D_{i+1}$.
        (See Figure \ref{integrate}.) By applying Poincar\'e lemmas for regular cohomology and compactly supported cohomology \cite[Chapter 3]{BoTu:diff-form},
        we can reduce our problem into finding a compactly supported 1-form $b'$ in
        $V:=\R^{3}\backslash \( \bigcup \widetilde D_i\cup \bigcup U_i\)$ such that $db'$ is equal to a given 2-form compactly supported in $V$.
        Projection\footnote{To obtain a smooth map, we can enlarge $V$ slightly and then use the red lines to identify $V$ and $S^2\times \R$}
        along the red lines in Figure \ref{integrate} identifies $V$ with $S^2\times \R$.
        By another application of Poincar\'e Lemma for cohomology with compact support,
        we can reduce our problem into finding a function $f$ on $S^2$ such that $df$ is equal to a given 1-form.
        This problem has a unique solution if we require that the integral of $f$ over $S^2$ is equal to $0$.
        \begin{figure}
        	\centering
                \begin{tikzpicture}
                        \draw[thick] ([shift=(10.5:1)]0,0)  arc (10.5:349.5:1) ;
                        \draw[thick] ([shift=(190.5:1)]10,0) arc (190.5:529.5:1);
                        \draw[thick] ([shift=(8.2:1.3)]5,0)    arc  (8.2:171.8:1.3);
                        \draw[thick] ([shift=(188.2:1.3)]5,0)    arc  (188.2:351.8:1.3);

                        \draw[thick] (0.98,0.2)--(3.72,0.2);
                        \draw[thick] (0.98,-0.2)--(3.72,-0.2);
                        \draw[thick] (6.28,-0.2)--(9.02,-0.2);
                        \draw[thick] (6.28,0.2)--(9.02,0.2);

                        \foreach \a in {0,1,...,6}{
                        \draw[thick,red] (90+180/6*\a:1)-- (90+180/6*\a:3); }
                        \begin{scope}[shift={(10,0)}]
                        \foreach \a in {0,1,...,6}{
                            \draw[thick,red] (270+180/6*\a:1)-- (270+180/6*\a:3); }
                         \end{scope}

                        \draw[thick,red] (5,1.3)--(5,3);
                        \draw[thick,red] (5,-1.3)--(5,-3);

                        \draw[thick,red] (0.8,0.6)--(0.8,3);
                        \draw[thick,red] (0.8,-0.6)--(0.8,-3);
                        \draw[thick,red] (4.2,1.02)--(4.2,3);
                        \draw[thick,red] (4.2,-1.02)--(4.2,-3);
                        \draw[thick,red] (5.8,1.02)--(5.8,3);
                        \draw[thick,red] (5.8,-1.02)--(5.8,-3);
                        \foreach \a in {1,2,3}{
                        \draw[thick,red] (1+0.75*\a,0.2)--(1+0.75*\a,3);
                        \draw[thick,red] (1+0.75*\a,-0.2)--(1+0.75*\a,-3);
                        }
                        \begin{scope}[shift={(5,0)}]

                        \draw[thick,red] (4.2,0.6)--(4.2,3);
                        \draw[thick,red] (4.2,-0.6)--(4.2,-3);
                        \foreach \a in {1,2,3}{
                        \draw[thick,red] (1+0.75*\a,0.2)--(1+0.75*\a,3);
                        \draw[thick,red] (1+0.75*\a,-0.2)--(1+0.75*\a,-3);}
                         \end{scope}
                \end{tikzpicture}
        	\captionof{figure}{In this example, there are three territory balls. The ribbon neighborhood $V_1$ (respectively, $V_2$) connects the second territory ball to the first one (respectively, third one).}\label{integrate}
        \end{figure}
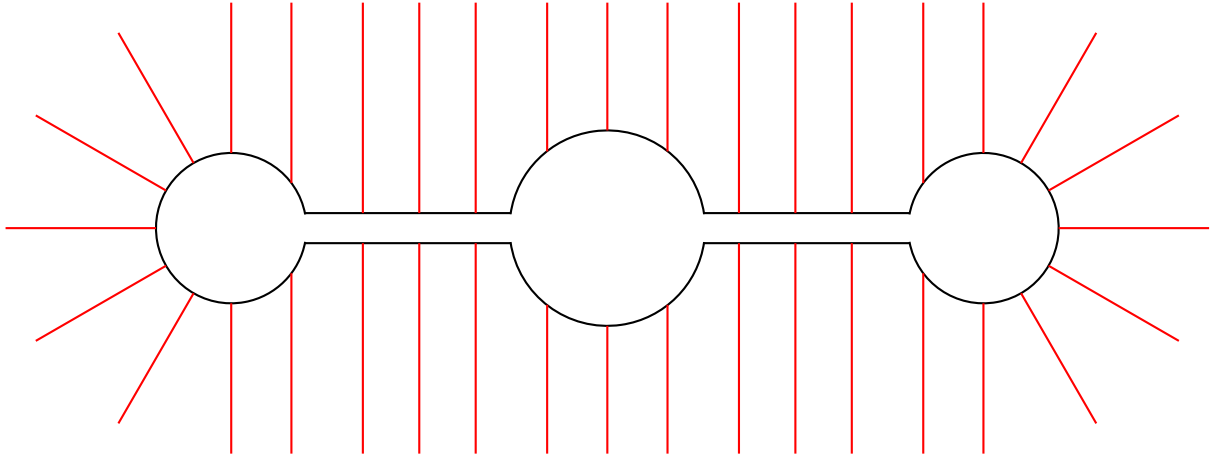
\end{proof}

\begin{definition}
	We call the 4-manifold $\fX_{\bm,\bx}$ the Gibbons-Hawking manifold 
	associated to the parameters $\bm$ and $\bx$.
	The metric $g$ in \eqref{GH-metric} is also called the Gibbons-Hawking metric 
	on $\fX_{\bm,\bx}$.
\end{definition}

The definition of GH metrics can be extended to a slightly more general choices of parameters $\bm$ and $\bx$. For any sequence of (not necessarily increasing) integers $(m_0,\dots,m_n)$ and $\{x_1,\dots,x_n\}\subset \R$, we can form a $\U(1)$-bundle $L$ over $\R^3\setminus \{x_1,\dots,x_n\}$. Let $\H$ be the subspace of $\R^3$ given by $y\geq 0$ and $z=0$. Thus, $\H$ can be identified with the standard upper half-plane. There is a M\"obius transformation acting on $\H$ which maps $x_i$ to $\infty$. This diffeomorphism of $\H$ is unique up to post-composition with an affine transformation. Therefore, it induces a diffeomorphism from $\R^3\setminus \{x_1,\dots,x_n\}$ to $\R^3\setminus \{x_1',\dots,x_n'\}$ which is well-defined up to post-composition with an affine transformation of $\R^3\setminus \{x_1',\dots,x_n'\}$. (In particular, $(x_1',\dots,x_n')$ induces a well-defined element of $\mathcal P_n$.) This diffeomorphism can be lifted to a diffeomorphism of $L$ and the $\U(1)$-bundle $L'$ associated to $(m_i,m_{i+1},\dots,m_n,m_0,\dots,m_{i-1})$ and $(x_1',\dots,x_n')$. If $L$ admits a Gibbons-Hawking metric, then we can push-forward this metric to $L'$. In summary, we can extend the definition of Gibbons-Hawking manifold $\fX_{\bm,\bx}$ to the case that $\bm=(m_0,\dots,m_n)$ is increasing possibly after shifting parameters.

\subsubsection*{Associahedron of Gibbons-Hawking Metrics}

Fix a sequence of increasing integers $\bm=(m_0,m_1,\dots,m_n)$, let $\fX_\bm$ be the Gibbons-Hawking manifold associated to $\bm$ and the points $x_i=i+\frac{1}{2}$ for $1\leq i \leq n$. Suppose $A$ is a cyclic bi-section of $[n]$ given by the set $\{j,j+1,\dots,k\}$. Analogous to Section \ref{topology}, let $S^j_k$ be the sphere that is centered at $(\frac{j+k}{2},0,0)$ and has radius $\frac{k-j-1}{2}+\nu(k-j)$ with $\nu$ being the same function as in Section \ref{topology}. Then $\R^3\backslash S^j_k$ has two connected components; one component contains the points $x_i$ with $j\leq i < k$, and the other component contains the remaining $x_i$'s. Let $M^j_k$ be the 3-manifold given by the fibers of $\fX_\bm$ over $S^j_k$. Then the set of cuts given by 3-manifolds $M^j_k$, for various choices of $A$, is of associahedron type. Thus, associated to this set of cuts, there is a family of metrics $\bbX_\bm$ on $\fX_\bm$ parametrized by $\mathcal {K}_n$. As in Subsection \ref{abst-assoc-family}, we denote the union of all cuts associated to an $n$-ribbon tree $T$ by $Y_T$. For each sphere $S^j_k$, we fix a regular neighborhood $[-2,2]\times S^j_k$ such that if $S^{j_1}_{k_1}$ and $S^{j_2}_{k_2}$ are disjoint then the regular neighborhoods $[-2,2]\times S^{j_1}_{k_1}$ and $[-2,2]\times S^{j_2}_{k_2}$ are also disjoint. These regular neighborhoods induce regular neighborhoods $[-2,2]\times M^j_k$ of the spherical cuts $M^j_k$.

The connected components of $\fX_\bm\backslash Y_T$ are also Gibbons-Hawking manifolds. Let $v$ be one of the interior vertices and $e_1$, $\dots$, $e_{d(v)}$ be the outgoing edges of $v$, labeled using the ribbon structure around the vertex $v$. Let $l_1$, $\dots$, $l_n$ denote the non-root leaves of $T$, again ordered using the ribbon structure. We write $u_i$ for the vertex of degree one incident to $l_i$. Let the integers $i_0$, $\dots$, $i_{d(v)}$ be chosen such that
\begin{equation} \label{multip}
  0\leq i_0 < i_1<\dots<i_{d(v)}\leq n
\end{equation}
and the unique path from the root to any of the vertices $u_{i_{j-1}+1}$, $\dots$, $u_{i_{j}}$ contains the edge $e_j$. Then we say $v$ has type $(i_0,\dots,i_{d(v)})$. The connected component of $\fX_\bm\backslash Y_T$ corresponding to the vertex $v$ is a Gibbons-Hawking manifold associated to the parameters $\bm'=(m_{i_0},m_{i_1},\dots,m_{i_{d(v)}})$.

\begin{example}
	Suppose $T$ is the (6)-ribbon tree given in Figure \ref{(6)-ribbon-tree}.
	The types of the vertices $v_0$, $v_1$ and $v_2$ of $T$ are listed below:
	\[
	  v_0: i_0=0, i_1=1, i_2=4, i_3=6\hspace{0.7cm}v_1: i_0=1, i_1=2, i_2=3, i_3=4\hspace{0.7cm}
	  v_2: i_0=4, i_1=5, i_2=6
	\]
\end{example}

The main goal of this subsection is to modify the family of metrics $\bbX_\bm$ such that each metric in the family is a Gibbons-Hawking metric. To achieve this goal, we need to do some preliminary work.

Fix an $n$-ribbon tree  $T$ with $l$ interior edges whose interior vertices and edges are labeled with $v_0$, $\dots$, $v_l$ and $e_1$, $\dots$, $e_l$. Let $(d(v_i)+1)$-tuple $\bm_i$ is given by the
type of the vertex $v_i$ and the $n$-tuple $\bm$. The open face $F_T$ is identified with $\mathcal P_{d(v_0)}\times \dots \times \mathcal P_{d(v_l)}$. Given $(\bx_0,\dots, \bx_l) \in F_T$ and $(s_1,\dots,s_l)\in (M_0,\infty)^l$, we can form the 2-form $\alpha_{\bm_i,\bx_i}$ and the quadratic form $Q_{\bm_i,\bx_i}$ on the space $\fB_{\bx_i}$. Suppose $\bx=\Phi_T(\bx_0,\dots,\bx_l,s_1,\dots,s_l) \in \mathcal P_n$ and $\widetilde \alpha$ (respectively, $\widetilde Q$) is the result of gluing $\alpha_{\bm_0,\bx_0}$, $\dots$, $\alpha_{\bm_l,\bx_l}$ (respectively, $Q_{\bm_0,\bx_0}$, $\dots$, $Q_{\bm_l,\bx_l}$). Recall from the previous subsection that there are functions $\gamma_0,\dots, \gamma_l:\fB_{\bx}\to \R$ such that:
\[
  \widetilde \alpha=\sum_{i=0}^l\gamma_i\alpha_{\bm_i,\bx_i} \hspace{1cm}\widetilde Q=\sum_{i=0}^l\gamma_iQ_{\bm_i,\bx_i}.
\]
We wish to compare $(\alpha_{\bm,\bx}, Q_{\bm,\bx})$ with $(\widetilde \alpha,\widetilde Q)$. We need to fix a metric $J_\bx$ on $\fB_{\bx}$ to perform this comparison.

For any $2\leq d\leq n$ and any $\bx\in \mathcal P_d$, we wish to pick a metric $J_\bx$ on $\fB_\bx$ such that $J_\bx$ satisfies the following two conditions:
\vspace{-5pt}
\begin{itemize}
	\item[(i)] We identify the interior of each territory ball in $\fB_\bx$ with $(0,\infty)\times S^2$ as in \eqref{cylinder-diffeo}.
	Similarly, we identify the interior of the complement of the parent ball with
	$(-\infty,0)\times S^2$. The metric $J_\bx$ in these open sets are given by $dt^2+\sigma$ where $\sigma$ is the standard metric on $S^2$.
	\item[(ii)] The second constraint is a consistency condition with respect to the inductive nature of the spaces
	$\mathcal P_d$, similar to the condition in Definition \ref{compatible-territory}.
	Suppose $T$ is a $d$-ribbon tree with two interior vertices $v_0$ and $v_1$.
	Suppose $(\bx_1,\bx_2)\in F_T$ and $\bx=\Phi_T(\bx_0,\bx_1,s)$ for $s\in (0,\infty)$. Then the metrics $J_{\bx_1}$ and
	$J_{\bx_2}$ on $\fB_{\bx_1}$ and $\fB_{\bx_2}$ can be glued
	to each other to define a metrics on $\fB_{\bx}$. We require that $J_\bx$ agrees with this metric when $s\in (M_0,\infty)$.
\end{itemize}
\vspace{-5pt}
A family of such metrics can be constructed by induction on $d$. In the case that $d=2$, we need to define one metric that extends a given metric on the territory balls and the complement of the parent ball. We choose these extensions in an arbitrary way. To carry out the induction step, we note that the first condition determines the desired metrics on the territory balls and the complement of parent balls. The second condition determines the metrics on the subset $\mathcal U_d$ of $\mathcal P_d$. Analogous to Proposition \ref{consistency-territory-balls}, we can use the induction step to show that the metrics determined by (ii) are consistent. We extend these metrics in an arbitrary way to complete the induction step.

We also need to define a function $\lambda_\bx:\fB_\bx \to \R$ for any $2\leq d\leq n$ and any $\bx\in \mathcal P_d$. This function is defined in terms of the annular decomposition of $\fB_\bx$. The function $\lambda_\bx$ vanishes on any fat region of $\fB_\bx$, is equal to $t$ on an end $(0,\infty) \times S^2$, and that is equal to $\tau-|t|$ on a neck $(-\tau,\tau) \times S^2$. Note that we again identify ends and necks by cylinders of the form $(a,b)\times S^2$ using diffeomorphisms similar to the maps in \eqref{cylinder-diffeo}.

\begin{prop} \label{exp-decay-Q-alpha}
	The constants $\delta$ and $C_k$ can be chosen such that:
	\begin{equation}\label{exponential-decay}
	  |\nabla^k(\alpha_{\bm,\bx}-\widetilde \alpha)|_{\bx}\leq C_k e^{-\delta(\tau+\lambda_\bx)} \hspace{1cm}
	  |\nabla^k(Q_{\bm,\bx}-\widetilde Q)|_{\bx}\leq C_k e^{-\delta(\tau+\lambda_\bx)}
	\end{equation}
	where $\nabla^k$ is defined with respect to the metric $J_\bx$, $|\cdot |_{\bx}$ is the point-wise norm with respect to $J_\bx$ and
	$\tau=\min_{i}(s_i)$.
\end{prop}

\begin{proof}
	We verify the inequalities in \eqref{exponential-decay} on a neck region for constants $\delta$ and $C_k$ which are independent of
	$s_1$, $\dots$, $s_l$. Analogous arguments can be used to prove
	similar inequalities for fat regions and cylindrical ends. Moreover, An inductive argument and compactness of associahedra can be employed
	 to show that these constants can be made independent of $\bx_0$, $\dots$, $\bx_l$.
	
	We assume that the neck region is in correspondence with
	the interior edge $e_1$. We also assume that the source and the target of $e_1$ are $v_0$ and $v_1$.
	Moreover, $e_1$ appears as the first element with respect to the ribbon structure of the vertex $v_0$. Then
	the neck region has the form $(\frac{M_0-s_1}{2},\frac{s_1-M_0}{2})\times S^2$ and the restriction of the metric $J_\bx$
	to this region is given by the cylindrical metric $dt^2+\sigma$. The 2-form $\widetilde \alpha$
	on $[\frac{s_1-M_0}{6},\frac{s_1-M_0}{2})\times S^2$ is equal to $\alpha_{\bm_0,\bx_0}$, on
	$[\frac{M_0-s_1}{6},\frac{s_1-M_0}{6}]\times S^2$ is equal to a convex linear combination of $\alpha_{\bm_0,\bx_0}$, $\alpha_{\bm_1,\bx_1}$,
	and on $(\frac{M_0-s_1}{2},\frac{M_0-s_1}{6}]\times S^2$ is equal to $\alpha_{\bm_1,\bx_1}$. Therefore, it suffices to show that
	$\alpha_{\bm,\bx}-\alpha_{\bm_0,\bx_0}$ has the desired decay properties on $(\frac{M_0-s_1}{6},\frac{s_1-M_0}{2})\times S^2$. A similar
	argument can be applied to $Q_{\bm,\bx}-\widetilde Q$.
	
	Suppose $\bx$ and $\bx_0$ are represented
	as below:
	\[
	  \bx:x_1<x_2<\dots<x_n \hspace{1cm}\bx_0:x_1'<x_2'<\dots<x_{d(v_0)}'
	\]
	Let the vertex $v_0$ have type $(i_0,\dots,i_{d(v_0)})$.
	In order to obtain $\bx$ from $\bx_0$, we replace each point $x_j'$ of $\bx_0$ with $i_j-i_{j-1}$ points which belong to
	the ball $B_{r_j/4}(x_j')$ with $r_j$ being the radius of the territory ball around $x_j'$. In fact, we can substitute
	$B_{r_j/4}(x_j)$ with $B_{e^{-\tau} r_j/4}(x_j)$ because $s_1$, $\dots$, $s_l$ are larger than $\tau$.
	We also need to include $i_0+n-i_{d(v_0)}$ points which belong to the complement of
	the ball $B_{2e^{\tau}R}(\overline x')$ where $B_{R}(\overline x')$ is the parent ball of $\bx_0$.
	Among these points, $i_0$ elements are less than $x_1'$ and $n-i_{d(v_0)}$ elements are greater than $x_{d(v_0)}'$.
	Moreover, the neck region associated to $e_1$ is the region between the balls centered at $x_1'$ and with radii
	$r_1e^{-\frac{M_0}{2}}$ and $r_1e^{-s_1+\frac{M_0}{2}}$.
	
	The 2-form $\alpha_{\bm,\bx}$ is given by:
	\[
	  \sum_{k=1}^{n}(m_{k-1}-m_k)\frac{(x-x_k)dy\wedge dz+ydz\wedge dx+zdx\wedge dy}{|q-x_k|^3}
	\]
	Suppose $q$ is a point in the neck region associated to the edge $e_1$. This region can be regarded as a subset
	of $\fB_{\bx_0}$. The 2-form $\alpha_{\bm_0,\bx_0}$ at a point $q$ in this region can be described as follows.
	In the above sum, replace $x_k$ with $x_j'$ if $x_k\in B_{r_j/4}(x_j')$.
	Moreover, remove the terms that the corresponding $x_k$ belongs to the complement of $B_{R}(\overline x')$.
	Then the resulting expressions is equal to $\alpha_{\bm_0,\bx_0}$ at the point $q$.

	The above description of $\alpha_{\bm_0,\bx_0}$ implies that
	$\nabla^k(\alpha_{\bm,\bx}-\widetilde \alpha)$ at a point $q$ in the neck region associated to the edge $e_1$ can be bounded by finding
	appropriate upper bounds for the expressions of the form:
	\begin{equation}\label{alpha-ineq-1}
	  |\nabla^k(\frac{(x-x_k)dy\wedge dz+ydz\wedge dx+zdx\wedge dy}{|q-x_k|^3}-
	  \frac{(x-x_j')dy\wedge dz+ydz\wedge dx+zdx\wedge dy}{|q-x_j'|^3})|_\bx
	\end{equation}
	where $x_k$ belongs to the territory ball of $x_j'$, and:
	\begin{equation}\label{alpha-ineq-2}
	  |\nabla^k(\frac{(x-x_k)dy\wedge dz+ydz\wedge dx+zdx\wedge dy}{|q-x_k|^3})|_\bx.
	\end{equation}	
	where $x_k$ belongs to the complement of $B_{R}(\overline x')$.
	It is straightforward to see that there is a constant $C$ such that the expression in \eqref{alpha-ineq-1} is less than
	$C\frac{|x_k-x_j'||q-x_1'|}{|q-x_j'|^2}$ and the expression in \eqref{alpha-ineq-2} is less than
	$C\frac{|q-x_1'|}{|q-x_k|}$. These inequalities imply that \eqref{alpha-ineq-1} and \eqref{alpha-ineq-2}  on the region between the
	balls centered at $x_1'$ with radii $r_1e^{-\frac{M_0}{2}}$ and $r_1e^{-\frac{M_0}{2}-2\frac{s_1-M_0}{3}}$ are less than
	$C_ke^{-\delta(\tau+\lambda_\bx)}$ for an appropriate choices of $\delta$ and $C_k$.
	A similar argument can be employed to estimate the difference between $\alpha$ and $\alpha_1$ on the
	region between the balls centered at $x_1'$ with radii
	$r_1e^{-\frac{M_0}{2}-\frac{s_1-M_0}{3}}$ and $r_1e^{-s_1+\frac{M_0}{2}}$. By combining these results, we can
	verify \eqref{exponential-decay} over the neck region associated to the edge $e_1$.
\end{proof}

As the next step, we need to fix a connection on $\omega_{\bm',\bx}$ on $\fX_{\bm',\bx}$ for any $2\leq d\leq n$, $\bm'=(m_{i_0},\dots,m_{i_d})$ and $\bx \in \mathcal K_d$ such that the curvature of $\omega_{\bm',\bx}$ is equal to $\alpha_{\bm',\bx}$ and these families of connections satisfy the analogue of Proposition \ref{exp-decay-Q-alpha}. As in the proof of Proposition \ref{GH-prop}, we firstly fix a connection $\omega_{\bm',\bx}'$ on $\fX_{\bm',\bx}$ and then modify it by adding an appropriate 1-form. The connections $\omega_{\bm',\bx}'$ are required to satisfy the following two conditions:
\begin{itemize}
	\item[(i)] On each territory ball and on the complement of the parent ball,
	the connection $\omega_{\bm',\bx}'$ is the pull-back of
	a connection of the form $\mu_{m,m'}$.
	\item[(ii)] Suppose $T$ is a $d$-ribbon tree with two interior vertices $v_0$ and $v_1$.
	Suppose $\bm'_0$ and $\bm'_1$ are defined using $\bm'$ and the types of the vertices $v_0$,  $v_1$.
	Suppose $(\bx_1,\bx_2)\in F_T$ and $\bx=\Phi_T(\bx_0,\bx_1,s)$ for $s\in (0,\infty)$. Then the $\U(1)$-connections
	$\omega_{\bm_1',\bx_1}'$ and $\omega_{\bm_2',\bx_2}'$ on $\fX_{\bm_1',\bx_1}$ and $\fX_{\bm_2',\bx_2}$ can be glued
	to each other to define a connection on $\fX_{\bm',\bx}$. We require that $\omega_{\bm',\bx}'$ agrees with this connection
	when $s\in (M_0,\infty)$.
\end{itemize}
These two conditions are similar to the properties that the metrics $J_\bx$ satisfy. The same inductive argument can be used to prove the existence of the family of connections satisfying these two properties.

\begin{prop}
	The constants $\delta$ and $C_k$ can be chosen such that if $a_{\bm',\bx}:=\alpha_{\bm',\bx}-F(\omega_{\bm',\bx})$, then:
	\begin{equation} \label{difference-approx}
	  |\nabla^k(a_{\bm',\bx})|_{\bx}\leq C_ke^{-\delta\cdot \lambda_\bx}
	\end{equation}
	where $\nabla^k$ is defined with respect to the metric $J_\bx$ and
	$|\cdot |_{\bx}$ is the point-wise norm with respect to $J_\bx$.
\end{prop}
\begin{proof}
	Given a fixed $\bm'$ and $\bx$, we firstly show that there are constants $C_k$ and $\delta$ such that \eqref{difference-approx} holds. This is essentially a consequence of \eqref{decay}.
	Next, let $T$ be a $d$-ribbon tree, $(\bx_0,\dots, \bx_l) \in F_T\subset \mathcal K_d$ and $(s_1,\dots,s_l)\in (M_0,\infty)^l$ such that $\bx=\Phi_T(\bx_0,\dots,\bx_l,s_1,\dots,s_l)$.
	Let also $\bm_i$ be the vector associated to $\bx_i$ by the $d$-tuple $\bm'$. Then the 2-forms $\alpha_{\bm_i,\bx_i}$ determines a 2-form $\widetilde \alpha$ on $\mathcal B_\bx$
	and Proposition \ref{exp-decay-Q-alpha} asserts that $\nabla^k(\widetilde \alpha-\alpha_{\bm',\bx})$ is bounded by $C_ke^{-\delta\cdot \lambda_\bx}$. This observation can be used to show that
	there are constants $C_k$ and $\delta$ which work for every point in $\Phi_T(\{(\bx_0,\dots,\bx_l)\}\times(M_0,\infty)^l)$. Next, we can show that the constants $C_k$ and $\delta$ can be made independent
	of $\bx$ using induction and compactness of associahedra.
\end{proof}

\begin{prop} \label{1-form-modify}
	There is a 1-form $b_{\bm',\bx}$ for any $2\leq d\leq n$, $\bm'=(m_{i_0},m_{i_1},\dots,m_{i_d})$
	and $\bx\in \mathcal P_d$ satisfying:
	\begin{itemize}
	\item[(i)] $d(b_{\bm',\bx})=a_{\bm',\bx}$;
	\item[(ii)] $|\nabla^kb_{\bm',\bx}|_{\bx}\leq C_k e^{-\delta\cdot \lambda_\bx}$
	\item[(iii)] Fix $(\bx_0,\dots, \bx_l) \in F_T$ and
	$(s_1,\dots,s_l)\in (M_0,\infty)^l$ and let:
	\begin{equation}\label{nbhd-face-T}
		\bx=\Phi_T(\bx_0,\dots,\bx_l,s_1,\dots,s_l) \in \mathcal P_n.
	\end{equation}	
	Suppose $\widetilde b$
	is the result of gluing $b_{\bm_0,\bx_0}$, $\dots$, $b_{\bm_l,\bx_l}$.
	The constants $C_k$ and $\delta$ can be chosen such that:
	\begin{equation}\label{exponential-decay-2}
	  |\nabla^k(b_{\bm',\bx}-\widetilde b)|_{\bx}\leq C_k e^{-\delta(\tau+\lambda_\bx)}
	\end{equation}
	where $\nabla^k$ is defined with respect to the metric $J_\bx$, $|\cdot |_{\bx}$
	is the point-wise norm with respect to $J_\bx$ and
	$\tau=\min_{i}(s_i)$.
	\end{itemize}
\end{prop}

We can use Proposition \ref{1-form-modify} to define a Gibbons-Hawking manifold $\fX_{\bm',\bx}$ for each $d$-tuple $\bm'=(m_{i_0},\dots,m_{i_d})$ and $\bx\in \mathcal K_d$. The metric on this space is given by:
\begin{equation} \label{family-GH-metric}
  g_{\bm',\bx}=Q_{\bm,.\bx}+(\omega_{\bm',\bx}'+b_{\bm',\bx})^2.
\end{equation}

\begin{proof}[Proof of Proposition \ref{1-form-modify}]
	Let $\bx\in \mathcal K_d$ and $a$ be a closed 2-form on $\fB_\bx$ which decays exponentially on the ends.
	Given an annular decomposition of $\fB_\bx$ associated to a $d$-ribbon tree $T$,
	we can define a 1-form $b$ such that $db=a$
	by modifying the proof of Proposition \ref{GH-prop}. For an interior edge $e$ of $T$,
	let the associated neck region to $e$ have the form $(-\frac{\tau}{2},\frac{\tau}{2})\times S^2$ in the cylindrical coordinate.
	Analogous to \eqref{integrate-2-form-1} and \eqref{integrate-2-form-2},
	we can define a 1-form $b_e$ on this neck region such that $db_e=a$. Let the restriction of $a$ to the neck region be
	equal to $\alpha(t)+\beta(t)dt$ with
	$\alpha(t)\in \Omega^2(S^2)$ and $\beta(t)\in \Omega^1(S^2)$.
	There is a unique 1-form $b_e(0)$ on $S^2$ such that $db_e(0)$
	is equal to $\alpha(0)$ and $d^*(b_e)$ vanishes\footnote{We use the
	standard metric on $S^2$ to define $d^*$.}. Then we define:
	\begin{equation*}
		b_e(t):=b_e(0)+\int_0^t\beta(s)\,ds
	\end{equation*}
	Let $\phi_e:\R\to [0,1]$ be a bump function supported in
	$(-\frac{\tau}{2},\frac{\tau}{2})$ such that $\phi_e$ is equal to $1$ on $(-\frac{\tau}{2}+1,\frac{\tau}{2}-1)$.
	The closed 1-form $a-d(\phi_eb_e))$ is supported outside the subset
	$(-\frac{\tau}{2}+1,\frac{\tau}{2}-1)\times S^2$ of the neck region.
	By repeating the same construction for all neck regions and
	 using the 1-forms in  \eqref{integrate-2-form-1} and \eqref{integrate-2-form-2},
	 we can essentially reduce the problem to the case that $a$ is supported in fat regions. (In fact, we have to
	 slightly enlarge the fat regions.) Now we apply the last step of the proof of
	 Proposition \ref{GH-prop} to each fat region to complete the construction of the desired 1-form $b$.
	 This construction has three properties which are useful for us:
	 \begin{itemize}
	 	\item[(i)] this construction can be performed in family for all
		elements $\bx\in \mathcal K_d$ which have an annular decomposition with respect to the ribbon tree $T$;
	 	\item[(ii)] this construction is linear in $a$;
		\item[(iii)] If $a$ and its derivatives are exponentially decaying on an end,
		then $b$ and its derivatives are also exponentially decaying with the same exponent. A similar claim also holds
		for neck regions.
	\end{itemize}
	
	The above construction allows us to construct the 1-forms $b_{\bm',\bx}$
	for any $2\leq d\leq n$, $\bm'=(m_{i_0},m_{i_1},\dots,m_{i_d})$ and $\bx\in \mathcal P_d$.
	The 1-form $b_{\bm',\bx}$ is given as a sum of 1-forms:
	\[
	  b_{\bm',\bx}=b_{\bm',\bx}^0+b_{\bm',\bx}^1+\dots+b_{\bm',\bx}^n
	\]
	such that if $\bx$ is given as in \eqref{nbhd-face-T} for a $d$-ribbon tree with $l$ interior edges,
	then:
	\[
	  \hspace{4cm} b_{\bm',\bx}^k=0\hspace{2cm}\text{if }\,\,\bx\in \Phi_T(F_T\times (2M_0,\infty)^l),\,\, k> \dim(F_T)=d-l-2
	\]
	The 1-forms $b_{\bm',\bx}^k$ are constructed by induction on $k$. Firstly let $k=0$. Let $F_T$ be a $0$-dimensional face
	of an associahedron $\mathcal K_d$ with $d\leq n$. Then for any point $\bx\in \Phi_T(F_T\times (M_0,\infty)^{d-2})$, we can
	consider the the annular decomposition of $\fB_\bx$ with respect to the tree $T$. We apply the construction of the previous
	paragraph to define a 1-from ${\widetilde b}_{\bm',\bx}^0$ for each point in this neighborhood. Let also
	$\varphi:\R\to [0,1]$ is a smooth bump function supported in $(M_0,\infty)$ which is equal to $1$ on $(2M_0,\infty)$. Then
	If $\bx=\Phi_T(\bx_0,\dots,\bx_l,s_1,\dots,s_l)$, then
	$b_{\bm',\bx}^0:=\varphi(s_1)\dots \varphi(s_l)\cdot{\widetilde b}_{\bm',\bx}^0$. This 1-form can be extended trivially to
	the points $\bx$ which are not close to the vertices of $\mathcal K_d$. By construction, $db_{\bm',\bx}^0=a_{\bm',\bx}$ if
	$\bx\in \Phi_T(F_T\times (2M_0,\infty)^{d-2})$ for a $0$-dimensional face $F_T$. Moreover, $b_{\bm',\bx}^0$ satisfies
	the decay condition in Item (ii) of the statement of the proposition. Item (iii) is also satisfied if $F_T$ is $0$-dimensional.
	Next, we consider the case that $F_T$ is $1$-dimensional. We repeat the above construction for the points in
	$\Phi_T(F_T\times (M_0,\infty)^{d-3})$ and the $2$-form $a_{\bm',\bx}-db_{\bm',\bx}^0$. In the present case,
	we again use the annular region with respect to $T$ and define a 1-form $\widetilde b_{\bm',\bx}^1$ for any point
	in $\Phi_T(F_T\times (M_0,\infty)^{d-3})$. Then we modify this 1-form by the bump function $\varphi$ and define a
	$1$-form $b_{\bm',\bx}^1$ on $\mathcal P_d$ which is equal to $\widetilde b_{\bm',\bx}^1$
	on $\Phi_T(F_T\times (2M_0,\infty)^{d-3})$. Repeating this construction for the faces of all dimensions in $\mathcal K_d$
	provides the 1-forms satisfying the required properties.
\end{proof}

Now we can come back to our main task of constructing a family of metrics on $\fX_\bm$ parametrized by $\mathcal K_n$. We observed that if $T$ is an $n$-ribbon tree with $l$ interior vertices $v_0$, $\dots$, $v_l$, then $\fX_\bm\setminus Y_T$ has $l+1$ connected components $Z_0^+$, $\dots$, $Z_l^+$, and $Z_l^+$ is diffeomorphic to the Gibbons-Hawking manifold $\fX_{\bm_i}$ wehre $\bm_i=(m_{i_0},m_{i_1},\dots,m_{i_{d(v_i)}})$. Here $(i_0,i_1,\dots,i_{d(v_i)})$ is the type of the vertex $v_i$. We already defined a family of Gibbons-Hawking manifolds $\{\fX_{\bm_i,\bx_i}\}_{\bx_i\in \mathcal K_{d(v_i)}}$. Therefore, we can push-forward the Gibbons-Hawking metrics on these manifolds to $Z_i^+$ using a family of diffeomorphisms:
\[
  f_{\bm_i,\bx_i}:\fX_{\bm_i,\bx_i} \to Z_i^+.
\]
To be more precise, $((f_{\bm_0,\bx_0})_*(g_{\bm_0,\bx_0}),\dots, (f_{\bm_l,\bx_l})_*(g_{\bm_l,\bx_l}))$ would define a metric $\fX_\bm\setminus Y_T$ for each $(\bx_0,\dots,\bx_l)\in F_T$.
\begin{prop}\label{fix-diffeo}
	For any $n$-ribbon tree $T$, and any choice of $(\bm_i,\bx_i)$, the diffeomorphism $\fX_{\bm_i,\bx_i}$
	can be chosen such that the metrics
	$((f_{\bm_0,\bx_0})_*(g_{\bm_0,\bx_0}),\dots, (f_{\bm_l,\bx_l})_*(g_{\bm_l,\bx_l}))$ form a family of asymptotically
	cylindrical metrics $\bbX_\bm$ on $\fX_\bm$ parametrized by $\mathcal K_n$.
\end{prop}
\begin{proof}
	To define the maps $f_{\bm_i,\bx_i}$, it is helpful to limit ourselves to a special type of diffeomorphisms.
	Suppose $\H$ denotes the upper half-plane:
	\[
	 \H=\{(x,y)\in \R^2 \mid y\geq 0\}
	\]
	and $\partial \H$ is the subset of $\H$ given by the pairs $(x,0)$.
	Suppose $\Omega$ is a simply connected open subset of $\H$ such that
	$\partial \Omega$ (as a subset of $\H$) is a union of smooth curves,
	$\Omega \cap \partial \H$ consists of $k$ open intervals, and
	$\partial \Omega \cap \partial \H$ consists of the end points of the intervals in $\Omega \cap \partial \H$.
	Let $\mathcal S$ be the solid obtained by rotating $\Omega$ around the axis $\partial \H$. Suppose
	also $D\subset \mathcal S$ is a divisor supported in the $x$-axis. Then we can form the $\U(1)$-bundle $L$
	associated to $D$. The manifolds $Z_i$ are obtained this way by definition.
	
	Suppose $\Omega'$ is another region in $\H$ which gives rise to the manifold $\mathcal S'$. Suppose also $L'$
	is a $\U(1)$-bundle over $L$ associated to a divisor $D'$ supported in the $x$-axis.
	We assume $f$ is a diffeomorphism from a regular neighborhood of $\partial \Omega \cup (\Omega \cap \partial \H)$ in
	$\H$ to a regular neighborhood of $\partial \Omega' \cup (\Omega' \cap \partial \H)$ in $\H'$ which maps
	$\partial \Omega$ to $\partial \Omega'$, $\Omega \cap \partial \H$ to $\Omega' \cap \partial \H$ and the divisor $D$ to $D'$.
	If $\overline f:\Omega\to \Omega'$ is an extension of $f$, then $\overline f$ induces a diffeomorphism from $L$ to $L'$
	in an obvious way.
	We call any such diffeomorphism from $L$ to $L'$ a nice diffeomorphism. The key feature for us is that the
	space of all nice diffeomorphisms from $L$ to $L'$ is a non-empty contractible space \cite{Sm:S2-diff}.

	Let $T$ be an $n$-ribbon tree and $v$ be a vertex of degree $d$ of $T$.
	Let $Z$ be the connected component of $\fX_\bm\backslash Y_T$ corresponding to the vertex $v$.
	Let also $\bm'$ be the $(d+1)$-tuple defined using the type of the vertex $v$ and $\bm$.
	We need to define diffeomorphisms $f_{\bm',\bx}:\fX_{\bm',\bx} \to Z^+$ for each $\bx\in \mathcal P_d$. We construct these diffeomorphisms by induction on $d$ such that
	they satisfy the following conditions:
	\begin{itemize}
		\item[(i)] $f_{\bm',\bx}$ is a nice diffeomorphism depending smoothly on $\bx\in \mathcal P_d$
		\item[(ii)] The diffeomorphism $f_{\bx',\bm}$ maps the fibers of $\fX_{\bm',\bx}$ over the ends (respectively, necks) of $\fB_{\bx}$
			to the corresponding ends (respectively, necks) of $Z$ determined by regular neighborhoods of the spherical cuts fixed earlier in this section.
		\item[(iii)] Suppose $S$ is a $d$-ribbon tree with vertices $v_0$, $\dots$, $v_l$ and the edges $e_1$, $\dots$, $e_l$.
		Suppose also $\bm_i$ is the  $(d(v_i)+1)$-tuple determined by the type of the vertex $v_i$ and the $d$-tuple $\bm'$. For any
		$(\bx_0,\dots,\bx_l)\in \mathcal P_{d(v_0)}\times \dots \times \mathcal P_{d(v_l)}$ and $(s_1,\dots,s_l)\in (M_0,\infty]^l$, let:
		\begin{equation} \label{bxS}
			\bx=\Phi_S(\bx_0,\dots,\bx_l,s_1,\dots,s_l)
		\end{equation}
		The fat regions of $\fB_\bx$ can be identified with the fat regions of $\fB_{\bx_0}$, $\dots$, $\fB_{\bx_l}$.
		In particular, we can compare $\nabla^kf_{\bm',\bx}$ restricted to the fibers of $\fX_{\bm',\bx}$ over fat regions of $\fB_\bx$ to
		the maps $\nabla^kf_{\bm_i,\bx_i}$ restricted to the fibers of $\fX_{\bm_i,\bx_i}$ over fat regions of $\fB_{\bx_i}$. We require that the distance between these maps are
		controlled by $C_ke^{-\tau}$ where $\tau=\min_i(s_i)$.
	\end{itemize}
	
	In order to show how the induction step can be carried out, let $T$ and $v$ be chosen as above. We firstly define $f_{\bm',\bx}$
	in the case that $\bx=\Phi_S(\bx_0,\dots,\bx_l,s_1,\dots,s_l)$ as in \eqref{bxS} for a $d$-ribbon tree $S$ representing a vertex of $\mathcal K_d$.
	Condition (ii) determines $f_{\bm',\bx}$ on the fibers of $\fX_{\bm',\bx}$ over the ends and the necks of $\fB_\bx$. We extend $f_{\bm',\bx}$ to the fibers over fat regions
	such that Conditions (i) and (iii) are satisfied in $\Phi_S(F_S\times (M_0,\infty]^{d-2})$. Next, let $S$ represent an edge of $\mathcal K_d$.
	We can use contractibility of space of nice diffeomorphisms to define $f_{\bm',\bx}$ in the case that $\bx$ is given as in
	\eqref{bxS}. Repeating this construction for all faces of $\mathcal K_d$ allows us to define $f_{\bm',\bx}$ for $\bx\in \mathcal P_d$.
	Propositions \ref{exp-decay-Q-alpha} and \ref{1-form-modify} guarantee that the push-forward metrics define a family of asymptotically cylindrical metrics on $\fX_\bm$
	parametrized by $\mathcal K_n$.
\end{proof}

\subsection{Gibbons-Hawking Families of Metrics on $W^j_k$} \label{GH-family-cob}

Gibbons-Hawking manifolds are relevant for the present paper because they appear as submanifolds of the 4-manifolds $W^j_k$. Suppose $T$ is a $(k-j+1)$-ribbon tree with interior vertices $v_0$, $\dots$, $v_l$. Suppose also $Y_T$ is the cut in $W^j_k$ given by Example \ref{asso-type-cuts}, and $W_0$, $\dots$, $W_l$ are connected components of a the complement of a regular neighborhood of $Y_T$ in $W^j_k$. Suppose the component $W_i$ has the property that any boundary component of $W_i$ is given by a spherical cut $M_{k'}^{j'}\subset Y_T$, with $1\leq k'-j'\leq N$. Then $W_i$ is diffeomorphic to a $\U(1)$-bundle defined over a subspace $\fB$ of $B$ in \eqref{B} which is bounded by spheres $S_{k'}^{j'}$, with $1\leq k'-j'\leq N$. In particular, $W_i$ is diffeomorphic to a Gibbons-Hawking manifold $\fX_\bm$ where $\bm$ is determined by the intersection of $\fB$ with the divisor $\fd$ in \eqref{fd}. In this case, we say $W_i$ is of {\it GH type}. If $W_i$ is not of GH type, then we say it is of {\it NGH type}.

In Theorem \ref{Wmetrics}, we shall construct a family of metrics $\bbW^j_k$ on $W^j_k$ such that all components of GH type appearing in the family has a Gibbons-Hawking metric. In order to construct this family of metrics, we firstly need to fix metrics on the boundary components and cuts that show up in $W^j_k$. We fix arbitrary metrics on 3-manifolds $Y_i$ which are $(N+1)$-periodic in $i$. For a spherical cut or boundary component $M_{k'}^{j'}$, we fix the metric $\sigma+\mu_{m,m'}^2$ for an appropriate choice of $m$ and $m'$. Here $\sigma$ is induced by identification of the spherical base of $M_{k'}^{j'}$ with $S^2$ and $\mu_{m,m'}$ is the $\U(1)$-connection introduced in the proof of Proposition \ref{GH-prop}. For any $Y$ in this family of 3-manifolds, we choose regular neighborhood $[-2,2]\times Y$ such that if two three manifolds $Y$ and $Y'$ in this family are disjoint, then $[-2,2]\times Y$ are $[-2,2]\times Y'$ are also disjoint.

\begin{theorem}\label{Wmetrics}
	There is a family of asymptotically cylindrical metrics $\mathbb W^j_k$ on $W^j_k$ parametrized by $\mathcal {K}_{k-j+1}$ such that the
	following two conditions hold for any $(k-j+1)$-ribbon tree $T$:
	\vspace{-5 pt}
	\begin{itemize}
		\item[(i)] the cut associated to $F_T$ is equal to $Y_T$ given by Example \ref{asso-type-cuts};
		\item[(ii)] let $W$, the connected component of $W^j_k\setminus Y_T$ associated to a vertex $v$, be of GH type diffeomorphic to $\fX_{\bm}$ where
		$\bm$ is a $(d(v)+1)$-tuple $(m_0,\dots,m_{d(v)})$.
		The restriction of the family of metrics $\bbW^j_k$ to $F_T$ induces a family of metrics on $W$ parametrized by $\mathcal K_{d(v)}$. We require that
		the metric corresponding to $\bx\in \mathcal K_{d(v)}$ on $W$ is the pull-back of the Gibbons-Hawking metric on $\fX_{\bm,\bx}$ constructed in the previous subsection.
	\end{itemize}
\end{theorem}

\begin{proof}
	A family of metrics parametrized by an associahedron, determines families of metrics on the connected components of the complement of cuts. Each of these families of metrics
	are parametrized by an associahedron. In order to prove the proposition, we inductively construct such families of metrics on any component of $W^j_k\setminus Y_T$
	where $T$ is a $(k-j+1)$-ribbon tree. To be more precise, if $v$ is a vertex of $T$ with degree $d$ and $W$ is a connected component of $W^j_k\setminus Y_T$ corresponding to $v$,
	then by induction on $d$, we construct a family of metrics on $W$ parametrized by $\mathcal K_d$.
	We also show that we can assume the following additional conditions hold:
	\begin{itemize}
		\item Each face $S$ of $\mathcal K_d$ determines a union of cuts $Y_S$ in $W$.
			Let $S$ have vertices $v_0$, $\dots$, $v_l$ and $d_i=d(v_i)$.
			Let also $W_0$, $\dots$, $W_l$ be the connected components of the complement of the regular neighborhood of $Y_S$ in $W$.
			Then $W_i$ also appears as a connected component of the complement of a union of cuts in $W^j_k$.
			We require that the restriction to $\overline F_S$ of the family of metrics on $W$ is given
			by the families of metrics on $W_i$ parametrized by $\mathcal K_{d_i}$.
		\item If $W$ is of GH type diffeomorphic to the space $\fX_{\bm}$, then we demand that the metric on $W$ corresponding to $\bx\in \mathcal K_d$
			is induced by a diffeomorphism $f_{\bm,\bx}:\fX_{\bm,\bx} \to W$.
			Moreover, we require that $f_{\bm,\bx}$ satisfies the same properties as (i), (ii) and (iii) in the proof of Proposition \ref{fix-diffeo}.
	\end{itemize}

In order to preform the induction step, we assume that $W$ is of NGH type. In the case that $W$ is of GH type, the proof is similar to the proof of Proposition \ref{fix-diffeo}.
	Our induction assumption specifies the metrics on $W$ parametrized by the points in the boundary of $\mathcal K_d$. In order to define the metrics associated to the points in
	$\mathcal P_d$, we fix a function $\gamma_T:\mathcal P_d \to \R^{\geq 0}$ for each face $F_T$ of $\mathcal K_d$ with codimension $1$. The function $\gamma_T$ is supported in
	$\Phi_T(F_T\times (M_0,\infty))$. Moreover, the value of $\gamma_T$ at $\Phi_T(\bx_0,\bx_1,t)$is equal to $e^t$ if $(\bx_0,\bx_1)\in F_T$ and $t\geq 2M_0$.
	We also fix a compactly-supported function $\gamma_0:\mathcal K_d \to \R^{\geq 0}$ and a metric $g_0$ on $W$ which has cylindrical ends corresponding to the metrics fixed on vertical and
	spherical cuts. We also assume that $\gamma_0$ is chosen such that $\gamma_0$ and the functions $\{\gamma_T\}_T$ do not have a common zero.
	The metric on $W$ parametrized by $\bx\in \mathcal P_d$ is defined as follows:
	\[
	  \frac{\gamma_0(\bx)g_0+\sum_{T}\gamma_T(\bx)\cdot g_T(\bx)}{\gamma_0(\bx)+\sum_{T}\gamma_T(\bx)}.
	\]
	Here $g_T(\bx)$ is a metric on $W$ which we only need to define for the points in $\Phi_T((M_0,\infty)\times F_T)$. If $\bx=\Phi_T(t,(\bx_0,\bx_1))$, then by induction we have a
	broken asymptotically cylindrical metic on $W$ associated to the point $(\bx_0,\bx_1)$. We glue these two metrics as it is explained in \eqref{glued-metric-1} and the following discussion there.
	The resulting metric determines $\gamma_T(\bx)$. It is straightforward to see that the above family of metrics has the required properties.
\end{proof}

\section{Moduli Spaces of Anti-Self-Dual Connections}\label{ASD-mod-spaces}

This section concerns the moduli spaces of ASD connections that appear in the construction of surgery exact polygons. In the first subsection, we review some general facts about such moduli spaces. Some special properties of moduli spaces of ASD connections over Gibbons-Hawking spaces are discussed in Subsection \ref{Mod-GH}. In the final subsection of the present section, we study the moduli spaces of {\it completely reducible connections} on $X_{N}(l)$.
\subsection{Moduli Spaces on Manifolds with Cylindrical Ends}\label{cyl-end-mod}
\subsubsection*{Chern-Simons Functional and Flat Connections}

Suppose $Y$ is an oriented connected closed 3-manifold and $\gamma$ is a 1-cycle in $Y$. The 1-cycle $\gamma$ determines a $\U(N)$-bundle $E$ on $Y$ whose isomorphism class depends only on the homology class of $\gamma$. There is a Chern-Simons functional $\CS$ on the space of connections on the bundle $E$ which is well-defined up to a constant. For two connections $A_0$ and $A_1$ on $E$, we have:
\[
  \CS(A_1)-\CS(A_0)= \frac{1}{16N\pi^2}\int_{[0,1]\times Y} \tr(\ad(F(\bA)) \wedge \ad(F(\bA))).
\]
Here $\bA$ is a connection on $[0,1]\times Y$ whose restriction to $\{i\}\times Y$ is equal to $A_i$ for $i=0,1$. Moreover, $\ad(F(\bA))$ is a 2-form with coefficients in ${\rm End}(\su(P))$ and $\tr$ is given by taking fiberwise trace. We fix an arbitrary connection on $\Lambda^{N}E$ and focus only on connections on $E$ whose central part is equal to this fixed connection on $\Lambda^{N}E$. This space is invariant with respect to automorphisms of $E$ whose fiberwise determinants are equal to $1$. The stabilizer of a connection $\eta $ with respect to the action of such automorphisms is denoted by $\Gamma_\eta$.

A connection $A$ is a critical point of the Chern-Simons functional, if $F_0(A)\in \Omega^2(Y,\su(E))$ vanishes. In general, for a connection $A$, $F_0(A)$ denotes the trace-free part of the curvature of $A$. We will write $R(Y,\gamma)$ for the solutions of $F_0(A)=0$ modulo the action of automorphisms of $E$ with determinant $1$. An element $\eta$ of $R(Y,\gamma)$ is {\it non-degenerate} if the Hessian of $\CS$ at $\eta$ is non-degenerate modulo the action of the gauge group. This is equivalent to say that the cohomology group $H^1(Y,\ad(\eta))$ is trivial. In general, we will write $h^i(\eta)$ for the dimension of the cohomology group $H^i(Y,\ad(\eta))$. If $\CS$ has degenerate critical points, we can make a small perturbation of the Chern-Simons functional using holonomy perturbations such that the critical points of the resulting functional are non-degenerate \cite{Don:YM-Floer, KM:YAFT}. When it does not make any confusion, we denote the critical points of such perturbed Chern-Simons functional with $R(Y,\gamma)$, too.

A pair $(Y,\gamma)$ is called admissible if there is $\sigma\in H_2(Y,\Z)$ such that the algebraic intersection number of $\gamma$ and $\sigma$ is coprime to $N$. For an admissible pair $(Y,\gamma)$, the elements of $R(Y,\gamma)$ are irreducible. That is to say, for any $\eta\in R(Y,\gamma)$, the stabilizer $\Gamma_\eta$ consists of only scalar transformations. In fact, even after a small perturbation of the Chern-Simons functional the same result holds \cite[Lemma 3.11]{KM:YAFT}.

There are two other families of pairs $(Y,\gamma)$ which appear in this paper. If $Y$ is the lens space $L(p,q)$, then the (unperturbed) Chern-Simons functional has non-degenerate critical points for any choice of $\gamma$. Let $\zeta$ denotes a flat $\U(1)$-connection on $L(p,q)$ whose holonomy along the generator of $H_1(Y,\Z)$ is equal to $e^{2\pi \bi/p}$. Then any flat $\U(N)$-connection on $(L(p,q),\gamma)$ is isomorphic to $\zeta^{i_1}\oplus \dots \oplus \zeta^{i_N}$ where $0\leq i_j < p$ and the sum $i_1+\dots+i_N$ is determined by the homology class of $\gamma$. In particular, all elements of $R(L(p,q),\gamma)$ are reducible. We shall be also interested in the case that $Y=S^1\times S^2$. The space $R(S^1\times S^2,\gamma)$ is empty unless $\gamma$ represents the trivial homology class.

By taking holonomy along the $S^1$ factor, we can identify $R(S^1\times S^2,\emptyset)$ with $T/W$ where $T$ is the space of diagonal matrices in $\SU(N)$ and $W=S_N$ is the Weyl group. The space $T/W$ can be described as the quotient of the Lie algebra:
\begin{equation} \label{abelian-Lie-alg}
	\mathfrak{t} = \{ (t_1,\dots,t_N) : t_1 + \cdots + t_N = 0 \}
\end{equation}	
by the action of the semi-direct product $W \ltimes L$, where $L = \Z^N \cap \mathfrak{t}$ is the root lattice of $\SU(N)$, acting by translation. The Weyl group $W$ also acts by permuting the coordinates. To identify $T/W$ with a standard simplex, we introduce functions $r_i: \mathfrak{t} \to \R$, given by
\[ r_i = \left \{ \begin{array}{ll} t_{i+1} - t_{i}  & i<N \\  t_{1} - t_{N} + 1 &i=N \end{array}\right.\]
Then the map $r = (r_1,\dots,r_N): \mathfrak{t} \to \R^N$ identifies $\mathfrak{t}$ with the plane $r_1 + \cdots + r_N = 1$.  Let $\Delta_{N-1}^\ft$ denote the locus in $\mathfrak{t}$ cut out by the inequalities $r_i \geq 0$ for all $i$.

\begin{prop}
	$\Delta_{N-1}^\ft$ is a fundamental domain for the action of $W \ltimes L$ on $\mathfrak{t}$.
	The elements of $T/W$ with repeated eigenvalues are in correspondence with the boundary of $\Delta_{N-1}^\ft$.
\end{prop}
\begin{proof} Define a {\it width} function $w: \R^N \to \R$ by
\[ w(t_1,\dots,t_N) = \sup_{i,j} | t_i - t_j | .\]
Given any vector $t = (t_1,\dots,t_N) \in \mathfrak{t}$, we can find a \emph{unique} vector $k \in L$ such that $w(t-k) \leq 1$.  Indeed, if we define a positive integer $d$ by
\[ d = \sum_{i=1}^N  \lceil t_i  \rceil, \]
and choose a permutation $\sigma$ such that:
\[ t_{\sigma(i)} -  \lceil t_{\sigma(i)}  \rceil \leq t_{\sigma(i+1)} - \lceil t_{\sigma(i+1)} \rceil \]
for every $i$, then the components of $k$ are given by
\[
  k_{\sigma(i)} = \left \{
  	\begin{array}{ll}
		 \lceil t_{\sigma(i)}  \rceil-1 & \sigma(i) \leq d \\
		\lceil t_{\sigma(i)}  \rceil  & \sigma(i) > d
	\end{array}
  \right.
\]

We can then find a permutation $\tau$ such that
\[ t_{\tau(i)}-k_{\tau(i)} \leq t_{\tau(i+1)}-k_{\tau(i+1)} \]
for every $i$.  It follows that the element $(\tau, - k) \in W \ltimes L$ takes $t$ into $\Delta_{N-1}^\ft$.
To complete the proof, we must show that no two points in $\Delta_{N-1}^\ft$ can be taken to one another by an element of $W \ltimes L$.  But this follows from the uniqueness of $k$, the fact that width is preserved by the action of $W$, and the fact that $\tau$ is unique up to right multiplication by an element of the stabilizer of $t-k$.
\end{proof}

\subsubsection*{Moduli Space of ASD Connections}

Let $W$ be a compact oriented smooth four-manifold with non-empty boundary and
$c$ be a 2-cycle. Let $E$ be the $\U(N)$ bundle determined by $c$. The bundle $E$ induces a $\U(N)$-bundle on $W^+$ which is also denoted by $E$. We assume that $(W,c)$ satisfies the following condition:
\begin{condition} \label{noS1S2}
	For any boundary component $(Y,\gamma)$ of $(W,c)$, either the 3-manifold $Y$ is a lens space, or
	the pair $(Y,\gamma)$ is admissible.
\end{condition}
\noindent
We also assume that for each admissible boundary component $(Y,\gamma)$ of $(W,c)$ the Chern-Simons functional is perturbed such that all of its critical points are non-degenerate.

Let $A_0$ and $A_1$ be two connections on $W^+$ such that the restriction of $A_i$ on any end $[-1,0)\times Y$ of $W^+$ is the pullback of a critical point of the corresponding (perturbed) Chern-Simons functional on $(Y,\gamma)$. We say $A_0$ and $A_1$ represent the same {\it path} along $(W,c)$ if there is an automorphism $u$ of $E$ with determinant $1$ such that:
\[
  A_0-u^*A_1
\]
is supported in $W\subset W^+$. The equivalence classes of this relation are called {\it paths along $(W,c)$}. We say the restriction of a path $p$, represented by $A_0$, to a boundary component $(Y,\gamma)$ of $(W,c)$ is equal to $\alpha\in R(Y,\gamma)$ if the restriction of $A_0$ to the end $[-1,0)\times Y$ is the pull-back of a connection that represents $\alpha$.

For a path  $p$ along $(W,c)$ represented by a connection $A_0$, let $\mathcal A_p(W,c)$ be the space of connections of the form $A_0+a$ where $a\in L^2_{k,\delta}(W^+,\Lambda^1\otimes \su(E))$. A differential form $a$ on $W^+$ is in $L^2_{k,\delta}$ if it belongs to $L^2_{k,loc}$. Moreover, the form $|r|^{-\delta}a(r,y)$ on $[-1,0)\times Y$ is required to be an element of $L^2_k$. Here $k\geq 3$ and $\delta$ is a small positive real number. Note that all elements of $\mathcal A_p(W,c)$ has the same central part. For a connection $A\in \mathcal A_p(W,c)$, the Chern-Weil integral:
\begin{equation} \label{energy}
  \kappa(A):=\frac{1}{16N\pi^2}\int_{W^+} \tr(\ad(F(A)) \wedge \ad(F(A))).
\end{equation}
is called the {\it topological energy} of $A$. This integral is independent of $A$ and depends only on $p$. We also define the {\it gauge group} $\mathcal G_p(W,c)$ as follows:
\[
  \mathcal G_{p}(W,c):=\left \{u\in {\rm Aut}(E) \mid \det(u)=1,\, \nabla_{A_0} u \in L^2_{k,\delta}(W, \Lambda^1\otimes\su(E))\right\}.
\]
The quotient of $\mathcal A_{p}(W,c)$ by the action of the gauge group $\mathcal G_{p}(W,c)$ is denoted by $\mathcal B_{p}(W,c)$. We say a connection $A\in \mathcal A_{p}(W,c)$ is irreducible if the stabilizer $\Gamma_{A}$ of the action of $\mathcal G_{p}(W,c)$ has positive dimension.

After fixing a metric on $W^+$ compatible with cylindrical ends, we can define the moduli space of {\it ASD} connections on $W^+$. Firstly assume that the critical points of the Chern-Simons functional for all boundary components of $(W,c)$ are non-degenerate. A connection $A\in \mathcal A_{p}(W,c)$ is ASD if $F_0^+(A)$, the self-dual part of $F_0(A)$, vanishes. This equation is invariant with respect to the action of $\mathcal G_{p}(W,c)$ and the quotient space of solutions forms a subspace $\mathcal M_{p}(W,c)$. In  general, holonomy perturbations of the Chern-Simons functional associated to a boundary component $(Y,\gamma)$ of $(W,c)$ induces a gauge invariant perturbation of the ASD equation on $[-1,0)\times Y$. Using a smooth function supported in $(-1,0)\times Y$, we can extend this perturbation to $W^+$. In the case that we need to perturb the ASD equation to achieve non-degenracy of critical points, a connection is called ASD if:
\begin{equation} \label{pert-ASD}
  F^+_0(A)+U(A)=0
\end{equation}
where $U(A)$ is the sum of the perturbation terms induced by the perturbations of the Chern-Simons functionals. We still denote the space of the all ASD connections modulo the action of the gauge group by $\mathcal M_{p}(W,c)$. Let $(Y_0,\gamma_0)$, $\dots$, $(Y_k,\gamma_k)$ be boundary components of $(W,c)$. Let also $\alpha_i$ be the restriction of $p$ to $(Y_i,\gamma_i)$. We will also write $\mathcal M_{p}(W,c;\alpha_0,\dots,\alpha_k)$ for $\mathcal M_{p}(W,c)$ if we want to emphasize on the limiting connections of the elements of the moduli space of ASD connections. For the sake of exposition, we assume that there is no need to perturb the Chern-Simons functional in the rest of this subsection.

There is a variant of the moduli space $\mathcal M_{p}(W,c)$ which is useful for our purposes. Let $(Y,\gamma)$ be a boundary component of $(W,c)$. Let also $p$ be a path along $(W,c)$ whose restriction to $(Y,\gamma)$ is given by $\alpha\in R(Y,\gamma)$. The restriction of an element of $\mathcal G_p(W,c)$ to $\{t\}\times Y\subset (-1,0)\times Y$ is asymptotic to an element of $\Gamma_\eta$ as $t$ approaches $0$. Let $\mathcal G_{p,\eta}(W,c)$ be the subgroup of $\mathcal G_p(W,c)$ consisting automorphisms which are asymptotic to the identity on the end $(-1,0)\times Y$ of $W^+$. Then the moduli space of ASD connection in $\mathcal A_p(W,c)$ modulo the action of $\mathcal G_{p,\eta}(W,c)$ gives rise to the {\it framed moduli space} $\widetilde {\mathcal M}_{p,\eta}(W,c)$. There is an action of $\Gamma_\eta$ on the based moduli space $\widetilde {\mathcal M}_{p,\eta}(W,c)$ and the quotient space is equal to $\mathcal M_p(W,c)$. This construction can be extended in an obvious way to the case that $(Y,\gamma)$ is replaced with a union of boundary components of $(W,c)$.

To any ASD connection $A$ we can associate the following Fredholm complex:
\begin{equation}\label{ASD-com}
	L^2_{l+1,\delta}(W, \mathfrak {su}(E))  \xrightarrow{\hspace{1mm}d_A\hspace{1mm}}L^2_{l,\delta}(W,  \Lambda^1\otimes \mathfrak {su}(E))
	\xrightarrow{\hspace{1mm}d^+_A\hspace{1mm}} L^2_{l-1,\delta}(W,  \Lambda^+\otimes \mathfrak {su}(E))
\end{equation}
The second map is the linearization of the ASD equation. We say $A$ is {\it regular} if the cohomology group $H_A^2$ associated to the last term of the above complex vanishes. Suppose $(Y,\gamma)$ is the union of all boundary components of $(W,c)$, $\eta$ is an element of $R(Y,\gamma)$. If $A$ is a regular ASD connection representing an element of the framed moduli space $\widetilde {\mathcal M}_{p,\eta}(W,c)$, then $\widetilde {\mathcal M}_{p,\eta}(W,c)$ is smooth in a neighborhood of $[A]$ and the cohomology group $H^1_A$ associated to the middle term in the above complex can be identified with the tangent space of $\widetilde {\mathcal M}_{p,\eta}(W,c)$ at $[A]$. In Section \ref{perturbations}, we review to what extent we can achieve regularity of the moduli spaces that appear in this paper by perturbing the ASD equation.

Next, we consider pairs $(W,c)$ that satisfy the following condition:
\begin{condition} \label{S1S2}
	For any boundary component $(Y,\gamma)$ of $(W,c)$, the 3-manifold $Y$ is a lens space,
	the pair $(Y,\gamma)$ is admissible,
	or the pair is diffeomorphic to $(S^1\times S^2,\emptyset)$. Moreover, there
	is exactly one boundary component which is diffeomorphic to $S^1\times S^2$.
\end{condition}
\noindent
The critical points of the Chern-Simons functional for the pair $(S^1\times S^2,\emptyset)$ are degenerate. However, they are non-degenerate in the Morse-Bott sense and we can construct well-behaved moduli space for the pairs $(X,c)$ satisfying \eqref{S1S2} without much additional work. (See \cite{Taubes:L2-mod-space,MMR:L2-mod-space,Don:YM-Floer}.) We can define paths along $(W,c)$ as before. Let $p$ be a path whose restriction to $S^1\times S^2$ is equal to $\alpha\in \Delta_{N-1}^\ft$. The same constructions as above give rise to the definition of the spaces $\mathcal A_{p}(W,c;\alpha)$, $\mathcal G_{p}(W,c;\alpha)$, $\mathcal B_{p}(W,c;\alpha)$, $\mathcal M_{p}(W,c;\alpha)$ and $\widetilde{\mathcal{M}}_{p,\alpha}(W,c;\alpha)$. It is also important to consider the moduli spaces where the limiting connection $\alpha$ on $S^1\times S^2$ is free to vary.

Suppose $\Delta_{N-1}^{\ft,\circ}$ denotes the interior of the simplex $\Delta_{N-1}^\ft$. Let $p$ be a path along $S^1\times S^2$ represented by a connection $A_0$ such that the restriction of $A_0$ to $S^1\times S^2$ is equal to $\alpha_0\in \Delta_{N-1}^{\ft,\circ}$. We can arrange for a family of connections $A_\alpha$, for $\alpha\in \Delta_{N-1}^{\ft,\circ}$, such that $A_\alpha$ depends smoothly on $\alpha$, the restriction of $A_\alpha$ to $[-1,0)\times (S^1\times S^2)\subset W^+$ is the pull-back of a connection representing $\alpha$, and the restriction of $A_\alpha$ to the complement of a neighborhood of $[-1,0)\times (S^1\times S^2)$ is equal to $A_0$. We can also assume that stabilizers of the restrictions of $A_\alpha$ to $S^1\times S^2$ are all equal to each other. We define $\mathcal A_{p}(W,c;\Delta_{N-1}^{\ft,\circ})$ to be the space of all connections which can be written as $A(\alpha)+a$ where $\alpha \in \Delta_{N-1}^{\ft,\circ}$ and $a \in L^2_{k,\delta}(W^+,\Lambda^1\otimes \su(E))$. The gauge groups $\mathcal G_{p}(W,c;\alpha)$ are independent of $\alpha$ and they act on $\mathcal A_{p}(W,c;\Delta_{N-1}^{\ft,\circ})$. Let $\mathcal B_{p}(W,c;\Delta_{N-1}^{\ft,\circ})$ be the quotient space. The moduli space of the elements in $\mathcal B_{p}(W,c;\Delta_{N-1}^{\ft,\circ})$ represented by the ASD connections are denoted by $\mathcal M_{p}(W,c;\Delta_{N-1}^{\ft,\circ})$. If we only mod out the solutions of the ASD equation by the based moduli space $\mathcal G_{p,\alpha}(W,c;\alpha)$, then the resulting space is denoted by $\widetilde{\mathcal{M}}_{p,\Delta_{N-1}^{\ft,\circ}}(W,c;\Delta_{N-1}^{\ft,\circ})$. The analogue of the ASD complex for a connection $A$ which represents an element of $\mathcal M_{p}(W,c;\Delta_{N-1}^{\ft,\circ})$ and is asymptotic to $\alpha\in \Delta_{N-1}^{\ft,\circ}$ is the following Fredholm complex:
\begin{equation}\label{ASD-comS}
	L^2_{l+1,\delta}(W, \mathfrak {su}(E))
    \xrightarrow{\hspace{1mm}d_A\hspace{1mm}}
    L^2_{l,\delta}(W,  \Lambda^1\otimes \mathfrak {su}(E))\oplus T_\alpha \Delta_{N-1}^{\ft,\circ}
	\xrightarrow{\hspace{1mm}d^+_A\hspace{1mm}}
   L^2_{l-1,\delta}(W,  \Lambda^+\otimes \mathfrak {su}(E))
\end{equation}
A point $[A] \in \mathcal M_{p}(W,c;\Delta_{N-1}^{\ft,\circ})$ is regular if the cohomology group at the last term of the above deformation complex vanishes. The moduli space $\widetilde {\mathcal M}_{p,\Delta_{N-1}^{\ft,\circ}}(W,c;\Delta_{N-1}^{\ft,\circ})$ is a smooth manifold in a neighborhood of a regular connection. We can repeat the same construction in the case that $\Delta_{N-1}^{\ft,\circ}$ is replaced with an open face $\Gamma$ of $\Delta_{N-1}^\ft$. In particular, we can form the moduli spaces $ \mathcal M_{p}(W,c;\Gamma)$, $\widetilde{\mathcal M}_{p,\Gamma}(W,c;\Gamma)$ and the analogue of \eqref{ASD-comS} for an element of $ \mathcal M_{p}(W,c;\Gamma)$. The regularity of an element of $ \mathcal M_{p}(W,c;\Gamma)$ is again defined using the associated ASD complex. Even more generally, if $\mathcal V$ is a subspace of $\Delta_{N-1}^{\ft}$, we define:
\[
  \mathcal M_{p}(W,c,\mathcal V)=
  \bigcup_{\beta\in \mathcal V}\mathcal M_{p(\beta)}(W,c,\beta)
\]
where $p(\beta)$ is given by composing $p$ and a path in $R(S^1\times S^2)$ from the endpoint of $p$ to $\beta$.

Suppose  $(W,c)$ is a 4-manifold satisfying either Condition \ref{noS1S2} or Condition \ref{S1S2}. Suppose $p$ is a path along $(W,c)$ and $A\in \mathcal A_p(X,c)$ is an arbitrary connection. The ASD operator associated to $A$, denoted by $\mathcal D_A$, is defined as follows:
\begin{equation}\label{ASDop}
  \mathcal D_A:= d_A^\ast + d_A^+ :L^2_{l,\delta}(X,  \Lambda^1\otimes \mathfrak {su}(E)) \to
   L^2_{l-1,\delta}(W, \mathfrak {su}(E))
	\oplus L^2_{l-1,\delta}(W,  \Lambda^+\otimes \mathfrak {su}(E))
\end{equation}
If the positive number $\delta$ is small enough, then the ASD operator is Fredholm. Moreover, $\ind(\mathcal D_A)$ depends only on the path $p$. Therefore, we define $\ind(p)$ to be the index of the operator $\mathcal D_A$.

If $(W,c)$ satisfies Condition \ref{noS1S2}
 and $A$ is an ASD connection representing a regular element of $\mathcal M_p(W,c;\alpha)$, then the based moduli space $\widetilde{\mathcal M}_{p,\alpha}(W,c;\alpha)$ is a smooth manifold in a neighborhood of $[A]$ whose dimension is equal to $\ind(\mathcal D_A)+h^0(\alpha)$. If $(W,c)$ satisfies Condition \ref{S1S2}
 and $A$ is an ASD connection representing a regular element of $\mathcal M_p(W,c;\Gamma)$ for a face $\Gamma$ of $\Delta_{N-1}^\ft$, then the based moduli space $\widetilde{\mathcal M}_{p,\Gamma}(W,c;\Gamma)$ is a smooth manifold in a neighborhood of $[A]$ whose dimension is equal to $\ind(\mathcal D_A)+h^0(\alpha)+\dim(\Gamma)$ where $\alpha$ denotes an element of $\Gamma$.

\subsubsection*{Moduli Spaces for a Family of Metrics}

We can form moduli spaces of ASD connections with respect to a family of metrics on a 4-manifold. Suppose $(W,c)$ is a pair satisfying Condition \ref{noS1S2}. Suppose we are also given a family of smooth metrics on $W$ parametrized by a smooth manifold $K$. That is to say, $\bbW:=W^+\times K$ admits a metric on each fiber $W^+\times \{g\}$ which varies smoothly with respect to $g\in K$. For a path $p$ along $(W,c)$, we define the moduli space of ASD connections with respect to this family of metrics to be the following subspace of $\mathcal B_p(W,c)\times K$:
\begin{equation} \label{par-mod-space}
  \mathcal M_p(\bbW,c):=\bigcup_{g\in K} \mathcal M_p(W,c;g)
\end{equation}
where $\mathcal M_p(W,c;g)$ denotes the moduli space $\mathcal M_p(W,c)$ defined with respect to the metric $g$. A similar construction can be used to define framed moduli spaces $\widetilde {\mathcal M}_p(\bbW,c)$. We say $([A],g)\in \mathcal M_p(\bbW,c)$ is a regular element of $ \mathcal M_p(\bbW,c)$ if the linearization of $F_0^+(A)$ as a map from $L^2_{l,\delta}(X,  \Lambda^1\otimes \mathfrak {su}(E))\oplus T_gK$ to $L^2_{l-1,\delta}(X,  \Lambda^+\otimes \mathfrak {su}(E))$ is surjective. Note that this condition is weaker than $[A]$ being a regular element of $ \mathcal M_p(W,c;g)$. If the latter property holds, we say $\mathcal M_p(\bbW,c)$ is {\it fiberwise regular} at $([A],g)$. The definition of $\mathcal M_p(\bbW,c)$ can be adapted to the pairs satisfying Condition \ref{S1S2} and one can form the moduli spaces $\mathcal M_p(\bbW,c;\Gamma)$ for any face $\Gamma$ of $\Delta_{N-1}^{\ft}$. In this case, the linearization of the ASD equation will be a map from $L^2_{l,\delta}(X,  \Lambda^1\otimes \mathfrak {su}(E))\oplus T_\alpha \Gamma \oplus T_gK$ to $L^2_{l-1,\delta}(X,  \Lambda^+\otimes \mathfrak {su}(E))$. We say an element $([A],g)$ is a regular element of the moduli space if this map is surjective.

It is important for our purposes to extend the definition of moduli space of ASD connections to families of metrics containing broken metrics. Suppose $\bbW$ is a family of metrics on a smooth 4-manifold $W$ parametrized by an admissible polyhedron $K$. Let also the projection map from $\bbW$ to $K$ be denoted by $\pi$. We assume that $c$ is a 2-cycle on $W$, which satisfies either Condition \ref{noS1S2}. We also assume that all components of $c$ are transversal to the cuts in the family of metrics $\bbW$. In particular, for any cut $Y$ of $\bbW$, the intersection $Y\cap c$ defines a 1-cycle on $Y$. We require that the pair $(Y,Y\cap c)$, for any cut $Y$, is either admissible or $Y$ is a lens space. We shall define a moduli space $\mathcal M_p(\bbW,c)$ which is equipped with a projection map ${\rm Pr}: \mathcal M_p(\bbW,c) \to K$ as in the case of family of smooth metrics. By definition $\pi^{-1}(K^{\circ})$ is a family of smooth metrics parametrized by $K^{\circ}$, the interior of $K$. The subspace $({\rm Pr})^{-1}(K^{\circ})$ of $\mathcal M_p(\bbW,c)$ is defined to be $\mathcal M_p(\pi^{-1}(K^{\circ}),c)$ where the latter moduli space is given by \eqref{par-mod-space}.

Next, we describe the fibers of $\mathcal M_p(\bbW,c)$ over the interior points of a codimension one face $\overline F$ of $K$. Let $Y_F$ be the cut associated to the face $F$ and $\gamma_F:=Y_F\cap c$. Then the complement of a regular neighborhood of $Y_F$ in $W$ has two connected components  denoted by $W_0$ and $W_1$. By our assumption on admissible polyhedra, the face $\overline F$ has the form $K_0\times K_1$. Moreover, $\pi^{-1}(\bbW)$ is determined by two families of metrics $\bbW_0$, $\bbW_1$ on $W_0$, $W_1$ parametrized by $K_0$ and $K_1$. If $\pi_i:\bbW_i \to K_i$ are the projection maps, then $\pi_i^{-1}(K_i^0)$ defines a family of smooth metrics on $W_i$ and we can form the based moduli space $\mathcal M_{p_i,\eta}(\bbW_i,c_i)$ where $c_i=W_i\cap c$, $p_i$ is a path along $(W_i,c_i)$, and $\eta\in R(Y_F,\gamma_F)$. We define:
\begin{equation} \label{fiber-codim}
	\mathcal M_p(\pi^{-1}(F),c)=\bigcup_{\substack{\eta\in R(Y_F,\gamma_F)\\ p_0\# p_1=p}} \widetilde{\mathcal M}_{p_0,\eta} (\pi^{-1}(K_0^{\circ}),c_0) \times_{\Gamma_\eta}
	\widetilde{\mathcal M}_{p_1,\eta} (\pi^{-1}(K_1^{\circ}),c_1)
\end{equation}
Here $p_0\# p_1$ denotes the path along $(W,c)$ obtained by gluing the paths $p_0$, $p_1$ along $(W_0,c_0)$, $(W_1,c_1)$. The space in \eqref{fiber-codim} is equal to ${\rm Pr}^{-1}(F)$, the fibers of $\mathcal M_p(\bbW,c)$ over $F=K_0^\circ\times K_1^\circ$.  An element $[A_0,A_1]$ in \eqref{fiber-codim} is regular if $A_0$ is a regular element of $ \widetilde{\mathcal M}_{p_0,\eta} (\pi^{-1}(K_0^{\circ}),c_0)$ and $A_1$ is a regular element of $\widetilde{\mathcal M}_{p_1,\eta} (\pi^{-1}(K_1^{\circ}),c_1)$.
The fibers of $\mathcal M_p(\bbW,c)$ over the faces of higher codimensions and the regularity of elements of these fibers can be defined similarly. The family of metrics over each face of codimension $k$ is induced by the families of metrics on $k+1$ manifolds. Then the moduli space over the interior of the face is given then by fiber products of based moduli spaces associated to the $k+1$ families of metrics. We follow the standard approach to define a topology on $\mathcal M_p(\bbW,c)$. (See, for example, \cite{KM:monopoles-3-man}.) If  $\eta$ is a flat connection on one of the boundary components of $W$, the based moduli spaces  $\widetilde {\mathcal M}_{p,\eta}(\bbW,c)$ can be also defined in a similar way.

Let $(W,c)$ be a pair satisfying Condition \ref{S1S2} and $\bbW$ is a family of metrics on $W$ parametrized by an admissible polyhedron $K$. We assume that $\bbW$ satisfies similar conditions as above. In particular, for any cut $Y$ in this family, either $(Y,Y\cap c)$ is admissible or $Y$ is a lens space. Then we can follow a similar approach to defining the moduli spaces $\mathcal M_p(\bbW,c;\Gamma)$ and $\widetilde {\mathcal M}_{p,\Gamma}(\bbW,c;\Gamma)$ for an open face $\Gamma$ of $\Delta_{N-1}^\ft$.

\subsection{ASD Connections on Gibbons-Hawking Spaces} \label{Mod-GH}
The following proposition in the case of closed 4-manifolds is proved in \cite{AHS:SDonSD}. Essentially the same proof, which uses the Weitzenb\"{o}ck formula,
can be used to verify this proposition:

\begin{prop}\label{regularGH}
	Suppose $X$ is a 4-manifold with an ASD asymptotically cylindrical metric which has positive scalar curvature.
	Suppose $c$ is a 2-cycle and $p$ is a path along $(W,c)$.
	Then any ASD connection associated to the triple $(X,c,p)$ is regular. In particular, the based moduli spaces for the triple
	$(X,c,p)$ are smooth manifolds.
\end{prop}

Let $\fX_\bm$ be a Gibbons-Hawking manifold of type $\bm=(m_0,\dots,m_n)$ and $\bbX_\bm$ be the family of metrics on $\fX_\bm$ parametrized by $\mathcal K_n$. This family consists of asymptotically cylindrical metrics. However, it can be approximated by families of cylindrical metrics. Fix a positive real number $T_0$. Using inductive arguments as in the proofs of Proposition \ref{fix-diffeo} and Theorem \ref{Wmetrics}, we can kill the exponentially decaying term on the subsets $(T_0,\infty)\times Y$ of the ends $(0,\infty)\times Y$ and the subsets $(-\tau+T_0,\tau-T_0)\times Y$ of the necks $(-\tau,\tau)\times Y$ that appear in the family. This gives rise to a family of cylindrical metrics which approaches the original family as $T_0$ goes to infinity. We still denote the resulting family of metrics with $\bbX_\bm$.

\begin{cor} \label{regular-GH}
	Let $c$ be a 2-cycle on $\fX_\bm$ and $p$ be a path along $(\fX_{\bm},c)$. Let the restriction of $p$ to one of the boundary
	components of $\fX_\bm$ be $\eta$. Assume that $\ind(p)\leq 1-h^0(\eta)$. Then any element of the moduli space
	 $\widetilde {\mathcal M}_{p,\eta}(\mathbb X_\bm,c)$ is fiberwise regular if $T_0$ is large enough. In particular, $\widetilde {\mathcal M}_{p,\eta}(\mathbb X_\bm,c)$
	 is either empty or at least $(n-2)$-dimensional for large enough values of $T_0$.	
\end{cor}
If $\index(p)$ is small enough, then the moduli space $\widetilde {\mathcal M}_{p,\eta}(\mathbb X_\bm,c)$ is automatically empty. Therefore, there are finitely many paths $p$ which are relevant for the above corollary. In particular, we can assume that the desired large constant $T_0$ given by the corollary is independent of $p$.
\begin{proof}
	For an integer number $k$, let $T_0=k$ and suppose that there is $p$ such that the moduli space
	$\widetilde {\mathcal M}_{p,\eta}(\mathbb X_\bm,c)$ for $T_0=k$ has an element $[A_k]$ which is not
	fiberwise regular. Suppose also $g_k$ is the metric on $\fX_\bm$ with respect to which $A_k$ is ASD. Since $\mathcal K_n$ is compact,
	the metrics $g_k$, after passing to a subsequence, converge to a (possibly broken) asymptotically cylindrical Gibbons-Hawking
	metric $g_\infty$ on $\fX_\bm$. Firstly, let $g_\infty$ be a non-broken metric. Then we claim that the connections $A_k$,
	up to action of the gauge group and after passing to a subsequence, are strongly
	convergent to a connection $A_\infty$ on $\fX_\bm$ which is ASD with respect to $g_\infty$.
	A priori, standard Floer-Uhlenbeck compactness theorems\footnote{These compactness theorems go back to Floer's original paper on instanton Floer homology \cite{Fl:I}.
	A good reference for these compactness results is \cite[Section 5.1]{Don:YM-Floer}. Although \cite{Fl:I,Don:YM-Floer} are mainly concerned with the Lie group $\U(2)$, the compactness theorems there
	can be adapted to the case of higher rank unitary Lie groups without any change.
	A concise review of the results that we use here is given in \cite[Subsection 6.1]{AX:suture-higher}.
	See \cite{Don:YM-Floer,AX:suture-higher} for the definition of weak chain convergence.} imply that the connections $A_k$,
	up to action of the gauge group and after passing to a subsequence, are weakly chain convergent to a connection on $(\fX_\bm,c)$.
	Here the convergence could be
	weak because of the bubbling of instantons and we need to consider chain convergence because part of the energy of connections
	 $A_k$ might slide off the ends of $\fX_\bm$.
	 However, neither of these two phenomena happen as it can be easily checked using
	 the following observations. Firstly, let $q$ be a path along $(\fX_\bm,c)$ whose restriction to one of the boundary components
	is $\lambda$ and $\ind(p)<h^0(\lambda)$. Then Proposition \ref{regularGH} implies that
	the moduli space $\mathcal M_q(\fX_\bm,c;g_\infty)$ is empty. Secondly, let $A$ be a non-flat ASD connections on
	$\R\times L(p,q)$ with the product metric associated to the round metric on the lens space $L(p,q)$.
	If $A$ is asymptotic to a flat connection $\chi$ on $L(p,q)$, then $\dim(\Gamma_A)<\dim(\Gamma_\chi)$. Therefore,
	$\ind(p)\geq 1-h^0(\chi)$. In fact, we can conclude $\ind(p)\geq 2-h^0(\chi)$ because of translational symmetry on $\R\times L(p,q)$.
	
	The connection $A_\infty$ is a regular element of $\mathcal M_p(\fX_\bm,c;g_\infty)$ because $g_\infty$ is a Gibbons-Hawking metric.
	Since the connections $A_k$ are strongly convergent to $A_\infty$, the connection $A_k$, for large enough values of $k$,
	has to be a regular element of $\mathcal M_p(\fX_\bm,c;g_k)$. This contradicts our assumption that $A_k$ is not a fiberwise
	regular element of $\mathcal M_{p}(\bbX_\bm,c)$. In the case that $g_\infty$ is a broken metric, a similar argument as in the previous
	paragraph shows that the connections $A_k$, up to action of the gauge group and after passing to a subsequence,
	converge to an element of $\mathcal M_p(\fX_\bm,c;g_\infty)$, i.e., a broken ASD connection.
	Again, $A_\infty$ is regular because $g_\infty$ is a (broken)
	Gibbons-Hawking metric. Moreover, regularity of $A_\infty$ implies regularity of $A_k$, for large enough values of $k$.
	(See, for example, \cite[Proposition 3.9]{Don:YM-Floer}.)
\end{proof}

In Subsection \ref{GH-family-cob}, we constructed a family of asymptotically metrics on cobordisms $W_k^j$ for any pair $k\geq j$. In particular, we can apply this construction to define $\bbW_k^j$ in the case that $j=0$ and $k=2N+1$. This family of metrics induces families of metrics $\bbW_k^j$ on $W_k^j$ for $0 \leq k<j\leq 2N+1$. If $k'=k+N+1$, $j'=j+N+1$ and $0\leq j,j',k,k'\leq 2N+1$, then an examination of our construction from Subsection \ref{GH-family-cob} shows that we can assume $\bbW_k^j=\bbW_{k'}^{j'}$. We also use the above trick for a large value of $T_0$ to turn these families into families of cylindrical metrics. In fact, we choose $T_0$ large enough, such that Corollary \ref{regular-GH} holds for families of Gibbons-Hawking metrics on GH components of $\bbW_{2N+1}^0$. Form now on, when we work with the family of metrics $\bbW_k^j$, we implicitly assume that $0\leq j,j',k,k'\leq 2N+1$ and the above modification is applied to $\bbW_k^j$.

Next, we look more closely at moduli spaces on the Gibbon-Hawking manifold where $n=2$ and $(m_0,m_1,m_2)$ are given as below:
\[
  (i-1,i,i+1)\hspace{1cm}(i+N,i,i+1) \hspace{1cm}(i+N,i,i+N-1)
\]
These manifolds appear as GH components in the families of metrics $W^j_k$. They have three boundary components. In the first case, the only non-trivial boundary component is ${\bf RP}^3$. In the other two cases, the 4-manifolds have two non-trivial boundary components which are $L(N,1)$ and $L(1-N,1)$.

\begin{prop}\label{dual}
	Let the 2-cycle $c$ on $\fX_\bm$ have the form $j_0\fd_0+j_1\fd_1+j_2\fd_2$:
	\begin{itemize}
		\item[(i)] Let $(m_0,m_1,m_2)=(i-1,i,i+1)$. Consider the 2-cycle $\overline c=j_0\fd_0+(j_2-j_1+j_0)\fd_1+j_2\fd_2$.
		Then $c$ and $\overline c$ induce the same cycle $\gamma$ on the boundary component ${\bf RP}^3$.
		Let $\alpha$ be a flat connection on $({\bf RP}^3,\gamma)$.
		Then there is a path $\overline p$ based at $\alpha$ on $(\fX_\bm,\overline c)$ such that the moduli spaces
		$\widetilde {\mathcal M}_{p,\alpha}(\bbX_\bm,c;\alpha)$ and
		$\widetilde {\mathcal M}_{\overline p,\alpha}(\bbX_\bm,\overline c;\alpha)$
		are diffeomorphic to each other.
		\item[(ii)] Let $(m_0,m_1,m_2)=(i+N,i,i+1)$ or $(i+N,i,i+N-1)$. Suppose $\alpha$ is a flat connection on $L(1-N,1)$
		or the trivial boundary component of $\fX_\bm$. Suppose also $\alpha'$ is the following flat connection on
		$L(N,1)$:
		\[
		  \alpha'=1\oplus \zeta \oplus \dots \oplus \zeta^{N-1}
		\]
		Suppose $p$ is a path whose restrictions to $L(1-N,1)$, $L(N,1)$ are $\alpha$, $\alpha'$ and
		$\ind(p)=-\dim(\Gamma_\alpha)$.
		Then the $0$-dimensional moduli space $\widetilde {\mathcal M}_{p,\alpha}(\bbX_\bm,c;\alpha,\alpha')$ is empty.
	\end{itemize}
\end{prop}
\begin{proof}
	Complex conjugation, as an involution on $\su(N)$, maps $p$
	to a path $\overline p$ along $(\fX_\bm,\overline c)$. This map induces a diffeomorphism at the level of moduli spaces.
	The claim in (ii) is also obvious, because the stabilizer of any element of ${\mathcal M}_{p}(\bbX_\bm,c)$ is a subset of
	$\Gamma_{\alpha'}$. On the other hand, $\dim(\Gamma_\alpha)>\dim(\Gamma_{\alpha'})$. Therefore, the moduli space
	$\widetilde {\mathcal M}_{p,\alpha}(\bbX_\bm,c;\alpha,\alpha')$ is either empty or its dimension is at least
	$\dim(\Gamma_\alpha)-\dim(\Gamma_{\alpha'})$.
\end{proof}

\subsection{Completely Reducible ASD Connections on $X_N(l)$} \label{comp-red}

In this subsection, we study a special family of ASD connections on $X_N(l)$. Before we take up this task, we gather some notations which will be used throughout the rest of the paper. Let $v=(x_1,\dots,x_N)$ be a vector in $\R^N$. Then for $1\leq p\leq \infty$, the $l^p$ norm of $v$ is defined as:
\[
  |v|_p:=\(\sum_{i=1}^N x_i^p\)^{\frac{1}{p}} \hspace{0.75cm}1\leq p< \infty\hspace{2cm}
  |v|_\infty:=\max_{1\leq i \leq N}\{|x_i|\}
\]
We also introduce the following notation:
\[
[v]_+:=\sum_{i=1}^N x_i
\]
Recall that $\ft$ denote the space of all vectors $v\in \R^N$ with $[v]_+=0$. There is a standard map from $\R^N$ to $\ft$ which maps a vector $v$ as above to the vector:
\[
  \overline v:=v-(\frac{[v]_+}{N},\frac{[v]_+}{N},\dots,\frac{[v]_+}{N})
\]
The vector $\overline v$ is called the {\it normalization} of $v$. The space $\ft$ can be also identified with the standard Cartan subalgebra of $\su(N)$ by mapping a vector $w=(y_1,\dots, y_N)\in \ft$ to a diagonal matrix whose $i^{\rm th}$ diagonal entry is equal to $2\pi \bi y_i$.

We shall also work with a special family of vectors in $\R^N$. For $0\leq i \leq N-1$:
\begin{equation}\label{lambdai}
	\lambda_i=(\underbrace{0,\dots,0}_{N-i},\underbrace{1,\dots,1}_{i})
\end{equation}
Then the normalization of $\lambda_i$ is equal to:
\begin{equation} \label{lambdai}
  \overline{\lambda_i}=(\underbrace{-\frac{i}{N},\,\dots,\,-\frac{i}{N}}_{N-i},\,
  \underbrace{\frac{N-i}{N},\,\dots,\,\frac{N-i}{N}}_{i})
\end{equation}
The vertices of the simplex $\Delta^\ft_{N-1}$ in Subsection \ref{cyl-end-mod} are given by the vectors $\overline {\lambda_0}$, $\dots$, $\overline {\lambda_{N-1}}$ in \eqref{lambdai}.
\subsubsection*{Completely Reducible Connections}
Suppose $X$ is a negative-definite 4-manifold with an asymptotically cylindrical end modeled on a 3-manifold $Y$. A Hermitian line bundle $L$ on $X$ supports an ASD connection, whose curvature has finite $L^2$ norm, if and only if the restriction of $c_1(L)$ to $Y$ is a torsion element of $H^2(Y,\Z)$. The space of all such ASD connections, up to isomorphism, forms a torus of dimension $b^1(X)$. In particular, this ASD connection, up to isomorphism, is unique in the case that $H_1(Y,\R)$ is trivial. By taking the direct sum of such abelian ASD connections, we can produce higher rank ASD connections. Any such ASD connection is called a {\it completely reducible} connection. The primary concern of this subsection is to study completely reducible $\U(N)$-connections on the negative definite manifold $X_N(l)$ equipped with the family of metrics from Section \ref{metrics}.

Recall that we introduced cohomology classes $e_1$, $\dots$, $e_N$ on $X_N(l)$ in Subsection \ref{4-man}, which generate the group of cohomology classes in $H^2(X_N(l),\Z)$ whose restriction to the boundary is torsion. For $1\leq j \leq k$, suppose $L_j$ is a Hermitian line bundle on $X_N(l)$ whose first Chern class is given by:
\begin{equation} \label{c1}
	i_j^1e_1+i_j^2e_2+\dots+i_j^{N}e_{N}.
\end{equation}
It is convenient to introduce a vector $v_j\in \Z^N$ and a $k$-tuple of vectors defined as below:
\[
  v_j=(i_j^1,\dots,i_j^N)\hspace{1cm}\bv=(v_1,v_2,\dots,v_k).
\]
We also assume that $[v_j]_+$ is a non-negative integer numbers less than $N$. Since $e_1+\dots+e_N$ is the a cohomology class of $X_N(l)$, we can always subtract a multiple of this vector from $v_j$ to ensure this assumption holds. Fix a metric $g$ on $X_N(l)$ with asymptotically cylindrical ends and let $B_j(g)$ be the unique ASD connection on $L_j$. This connection also determines flat connections $\beta_j(g)$ and $\chi_j$ on the non-trivial boundary components $S^1\times S^2$ and $L(N,1)$. The flat $\U(1)$-connection $\chi_j$ is uniquely determined by the first Chern class of the underlying bundle which is equal to $[v_j]_+$ times the standard generator. On the other hand, the flat connection $\beta_j(g)$ depends on the metric $g$.

Let $\widetilde B_{\bv}(g)$ be the completely reducible ASD $\U(k)$-connection given as the direct sum of the connections $B_j(g)$. We also define $c_\bv$ to be the 2-cycle associated to $\widetilde B_{\bv}(g)$. The connections $\widetilde B_{\bv}(g)$ for different choices of the metric $g$ have different central parts. To avoid this, we fix a $\U(1)$-connection $A_0$ on $X_N(l)$ associated to the 2-cycle $c_\bv$. For simplicity, we can assume that the restriction of $A_0$ on the end corresponding to $S^1\times S^2$ is trivial. By adding a central 1-form to  $\widetilde B_{\bv}(g)$, we can define a new completely reducible connection $B_\bv(g)$ whose central part is equal to $A_0$. We will write $\beta_\bv(g)$ and $\chi_\bv$ for the restriction of the connection $B_\bv(g)$ to $S^1 \times S^2$ and $L(N,1)$. Note that $h^0(\beta_{\bv}(g))=h^1(\beta_{\bv}(g))$ by the K\"unneth formula.

\begin{lemma} \label{ind-Bv}
	For $B_\bv(g)$ chosen as above:
   \begin{equation} \label{comp-ind}
	\ind(\mathcal D_{B_\bv(g)})=\sum_{1\leq i, j\leq k}|v_i-v_j|_2^2-\sum_{1\leq i,j \leq k}|[v_i]_+-[v_j]_+|
	-h^0(\chi_\bv)-h^0(\beta_{\bv}(g))
   \end{equation}
Moreover, $\ind(\mathcal D_{B_\bv(g)})\geq -h^0(\chi_\bv)-h^0(\beta_{\bv}(g))$
	and the equality holds if and only if the
	vectors $v_j$, possibly after a permutation of indices, satisfy the following relations:
	\begin{equation} \label{vector-cond}
		v_1 \geq v_2 \geq \dots \geq v_k \hspace{1cm}	|v_i-v_j|_{\infty}\leq 1.
	\end{equation}
	Here $v\geq u$ for two vectors $u,v\in \R^N$, if each component of $v$ is not less than the corresponding component of $u$.
\end{lemma}
\begin{proof}
	According to the index formula:
	\begingroup\makeatletter\def\f@size{10}\check@mathfonts
	\begin{align*}
		\ind(\mathcal D_{B_\bv(g)})+h^0(\beta_{\bv}(g))+h^0(\chi_\bv)&=4k\cdot \kappa(A)-(k^2-1)(\frac{\chi(X_N(l))+\sigma(X_N(l))}{2})
		+\frac{h^0(\chi_\bv)+\rho(\chi_\bv)}{2}\\	
		&=4k\cdot \kappa(A)-\frac{k^2-1}{2}
		+\frac{k-1+\sum_{ i<j}h^0(\chi_{i}-\chi_{j})+\rho(\chi_{i}-\chi_{j})}{2}
	\end{align*}
	\endgroup
	Here $\rho(\chi_\bv)$ is the $\rho$-invariant of the flat connection $\ad{\chi_\bv}$ which is defined on a vector bundle of rank $N^2-1$ \cite{APS:II}.
	It is straightforward to check that:
	\begin{equation} \label{energy-formula}
		4k\cdot \kappa(A)=
		\sum_{1\leq i,j \leq k}|v_i-v_j|^2_2-\frac{1}{N}\sum_{1\leq i,j \leq k}([v_i]_+-[v_j]_+)^2
	\end{equation}
	Moreover, the following relation for the $\U(1)$-connection $\zeta^j$, with $|j|\leq N$, is proved in \cite{APS:II}:
	\begin{equation*}
		h^0(\zeta^j)+\rho(\zeta^j)=2-4|j|+\frac{4j^2}{N}.
	\end{equation*}
	Here $\rho$ and $h^0$ are defined with respect to the 2-dimensional real representation of $\U(1)$. The formula for the index of $B_\bv(g)$ is a consequence of these relations.

	In order to verify the second part of lemma, note that we have the following inequalities:
	\begin{equation*}
		\sum_{1\leq i,j \leq k}|[v_i]_+-[v_j]_+|\leq \sum_{1\leq i,j \leq k}|v_i-v_j|_1\hspace{1cm}
		 \sum_{1\leq i,j \leq k}|v_i-v_j|_1\leq\sum_{1\leq i, j\leq k}|v_i-v_j|_2^2
	\end{equation*}	
	In the first inequality, equality holds if and only if the first condition in \eqref{vector-cond} is satisfied.
	In the second one, equality holds if and only if the second condition in \eqref{vector-cond} is satisfied.
\end{proof}

We make the following elementary observation about the vectors that satisfy \eqref{vector-cond}:
\begin{lemma} \label{sharp-bv}
	Suppose the vectors $v_1$, $\dots$, $v_k$ satisfy the conditions in \eqref{vector-cond} and:
	\[
	 v_1+v_2+\dots+v_k=(s_1,\dots,s_k)
	\]
	where $0\leq s_j \leq k-1$. Then $v_l=(i_l^1,\dots,i_l^N)$ with:
	\[
	  i_l^j=\left\{
	  \begin{array}{ll}
		1&l\leq s_j\\
		0&l> s_j\\
	  \end{array}
	  \right.
	\]
	In particular, $|v_l|_1$ is equal to the number of $s_j$'s which are greater than or equal to $l$.
\end{lemma}

From now on, we assume that $k=N$. Let $l$ be a non-negative integer number which is not greater than $N$. Let:
\begin{equation} \label{sigma-tau}
  \sigma: [N-l]\to [N] \hspace{1cm} \tau: [l] \to [N]
\end{equation}
be two injective maps. For any choice of $\sigma$ and $\tau$ as above, let $w_{\sigma,\tau}$ be a 2-cycle representing  the following cohomology class:
\begin{equation}\label{sigmatau}
   \sigma(0)e_1+\dots+\sigma(N-l-1)e_{N-l}+\tau(0)e_{N-l+1}+\dots+\tau(l-1)e_{N}
\end{equation}

\begin{prop}\label{bi-per-con}
	For a metric $g$ on $X_N(l)$, the $\U(N)$-bundle associated to $w_{\sigma,\tau}$
	supports a completely reducible connection $B_\bv(g)$ with:
	\begin{equation} \label{cond-ind-limit-flat}
	 \ind(\mathcal D_{B_{\bv}(g)})=-h^0(\chi)-h^0(\beta)\hspace{2cm} \chi=1\oplus \zeta \oplus \dots \oplus \zeta^{N-1}
	\end{equation}
	if and only if the image of the maps $\sigma$ and $\tau$ are disjoint.
\end{prop}
\begin{proof}
	Lemmas \ref{ind-Bv} and \ref{sharp-bv} imply that all the entries of the vector $v_j$ are equal to either $0$ or $1$.
	Then the second condition in \eqref{cond-ind-limit-flat} imply that
	the vectors $v_1$, $\dots$, $v_N$ is given by the rows
	of the following matrix after a permutation of its columns:

	\begin{equation}\label{matrix}
	  \left(
	  \begin{array}{ccccc}
	  	0&0&\dots&0&0\\
		0&0&\dots&0&1\\
		0&0&\dots&1&1\\
		\vdots&\vdots&\ddots&\vdots&\vdots\\
		0&1&\dots&1&1\\
	  \end{array}
	  \right)
	\end{equation}
	Therefore, the set $\{\sigma(0),\dots,\sigma(N-l-1),\tau(0),\dots,\tau(l-1)\}$ is equal to $\{0,1,\dots,N-1\}$. This implies one direction of the proposition. The other direction is straightforward.
\end{proof}
A pair $(\sigma,\tau)$ of two injective maps as in \eqref{sigma-tau} with disjoint images is called a {\it bi-permutation} of type $l$ associated to the set $[N]$.  For a bi-permutation $(\sigma,\tau)$ of type $l$ and any metric $g$ on $X_N(l)$ with cylindrical ends, Proposition \ref{bi-per-con} asserts that there is a completely reducible ASD connection $B_{\sigma,\tau}(g)$ on $X_N(l)$ such that the associated 2-cycle is $w_{\sigma,\tau}$ and the limiting flat connection of $B_{\sigma,\tau}(g)$ on the lens space end is equal to $\chi$ in \eqref{cond-ind-limit-flat}.

Fix a reference bi-permutation $(\sigma_0,\tau_0)$, and let $S$, $[N]\setminus S$ denote the image of the maps $\sigma_0$, $\tau_0$. Let $S_n$ be the symmetric group on the set $\{1,\dots,n\}$. Given any element $(f,g) \in S_{N-l}\times S_l$, we can form the bi-permutation $(\sigma_f,\tau_g)$ of type $l$ which is defined as:
\begin{equation} \label{sigmaf-taug}
  \sigma_f(i)= \sigma_0(f^{-1}(i+1)-1)\hspace{1cm}
  \tau_g(j)=\tau_0(g^{-1}(j+1)-1)\hspace{1.3cm} i\in [N-l],\, j\in [l]
\end{equation}
This gives a transitive and faithful action of $S_{N-l}\times S_l$ on the set of all bi-permutations $(\sigma,\tau)$ of type $l$ such that ${\rm image}(\sigma)=S$ and ${\rm image}(\tau)=[N]\setminus S$.

\begin{remark} \label{comp-red-barXN}
	Let $\overline{X}_N$ be the result of gluing $S^1\times D^3$ to $X_N(l)$ along the boundary component
	$S^1\times S^2$. The cohomology group $H^2(\overline{X}_N,\mathbf{Z})$ is generated by
	$e_1,\cdots, e_N$ modulo the relation $e_1+\cdots +e_N=0$.
	Given a  cylindrical metric  $g$ on  $\overline{X}_N$ and a vector:
	\[
	  \mathbf{v}=(v_1,\cdots,v_k), v_i \in \mathbf{Z}^{N},
	\]	
	 we can form a completely reducible ASD $\U(k)$-connection $\overline{B}_\mathbf{v}(g)$.
	 Let $\chi_\bv$ denote the limiting flat connection of $\overline{B}_\mathbf{v}(g)$ on $L(N,1)$.
	A similar argument as in the proof of Lemma \ref{ind-Bv} can be used to show that:
	 \begin{equation}\label{ind-Bv2}
	  \ind(\mathcal D_{\overline{B}_\bv(g)})=
	  \sum_{1\leq i, j\leq k}|v_i-v_j|_2^2-\sum_{1\leq i,j \leq k}|[v_i]_+-[v_j]_+|-h^0(\chi_\bv)
	\end{equation}
	We will need this index formula and the following variant of Lemma \ref{sharp-bv} in the final section of the
	paper.
	
	Suppose a flat $\U(N)$-connection $\chi=\zeta^{s_1}\oplus \cdots \oplus \zeta^{s_N}$ on $L(N,1)$
	is given such that $0\leq s_j\leq N-1$. Analogous to  Lemma \ref{sharp-bv}, we define $v_l=(i_l^1,\dots,i_l^N)$ as follows:
	\[
	  i_l^j=\left\{
	  \begin{array}{ll}
		1&l\leq s_j\\
		0&l> s_j\\
	  \end{array}
	  \right.
	\]
	Let $\overline{B}_\mathbf{v}(g)$ be the flat connection associated to $ \bv=(v_1,\cdots,v_N)$.
	Then $\overline{B}_\mathbf{v}(g)$ has the limiting flat connection $\chi$ and
	$\ind (\mathcal D_{\overline{B}_\mathbf{v}(g)})=-h^0(\chi)$. Let $B_1\oplus\cdots \oplus B_N$
	be the decomposition of $\overline{B}_\mathbf{v}(g)$ into $\U(1)$-connections. Then for any
	subset $\{i_1,\cdots,i_k\}\subset \{1,\cdots, N\}$, we have:
	\[	
	  \ind (\mathcal D_{B_{i_1}\oplus \cdots \oplus B_{i_k}})=-h^0(\zeta^{i_1}\oplus \cdots \oplus \zeta^{i_k}).
	\]
\end{remark}

\subsubsection*{Holonomy Maps}

In Subsection \ref{Mod-GH}, for any $0\leq l \leq N$, we fix a family of metrics $\bbW^l_{l+N+1}$ on $W^l_{l+N+1}$ which is parametrized by the associahedron $\mathcal K_{N+2}$. One of the cuts associated to this family of metrics is equal to $M^l_{l+N+1}$. The complement of a neighborhood of this cut has two connected components, one of which is identified with $X_N(l)$. In particular, the family of metrics $\bbW^l_{l+N+1}$ induces a family of metrics on $X_N(l)$ parametrized by $\mathcal K_{N+1}$ which we denote by $\bbX_{N}(l)$.

Any $g\in \mathcal K_{N+1}$ gives rise to a (possibly broken) metric on $X_N(l)$, which we also denote by $g$. Consider the completely reducible connection $B_{\sigma,\tau}(g)$ for a fixed choice of a bi-permutation $(\sigma,\tau)$ of type $l$. Let the decomposition of $B_{\sigma,\tau}(g)$ as a direct sum of $\U(1)$-connections be given as follows:
\begin{equation} \label{dec-com-red}
  B_{\sigma,\tau}(g)=B_{\sigma,\tau}^1(g)\oplus \dots \oplus B_{\sigma,\tau}^N(g)
\end{equation}
where the $\U(1)$-connection $B_{\sigma,\tau}^j(g)$ is asymptotic to $\zeta^{j-1}$ on the lens space end of $X_N(l)$. The connection $B_{\sigma,\tau}(g)$ is asymptotic to an $\SU(N)$-flat connection $\beta_{\sigma,\tau}(g)$ on $S^1\times S^2$ because we pick the central connection $A_0$ such that its restriction to $S^1\times S^2$ is trivial. The decomposition \eqref{dec-com-red} implies that the holonomy of $\beta_{\sigma,\tau}(g)$ along the $S^1$ factor comes with a preferred choice of diagonalization. Therefore, this holonomy determines a map $\hol_{\sigma,\tau}:\mathcal K_{N+1} \to T$ where $T$ is the standard maximal torus of $\SU(N)$. It is clear that the map $\hol_{\sigma,\tau}$ is smooth in the interior of $\mathcal K_{N+1}$. Standard gluing theory results about (abelian) ASD connections also show that this map is continuous.

There is a canonical way to lift the map $\hol_{\sigma,\tau}:\mathcal K_{N+1} \to T$ to a map $\widetilde \hol_{\sigma,\tau}:\mathcal K_{N+1} \to \ft$ such that $\hol_{\sigma,\tau}=\exp \circ \widetilde \hol_{\sigma,\tau}$. Recall that we introduced a cylinder $\fc$ in Subsection \ref{4-man} such that $\partial \fc=\gamma_0\sqcup -\gamma_N$ where $\gamma_0$ (respectively, $\gamma_N$) is a fiber of the $\U(1)$-fibrations of $S^1\times S^2$ (respectively, $L(N,1)$) with the standard orientation. Gauss-Bonet Theorem and our convention on the orientation of $\fc$ implies that $B_{\sigma,\tau}^j(g)$ is asymptotic to a connection on $S^1\times S^2$ with holonomy $\exp(\theta_j(g))$ where $\theta_j(g)$ is given by:
\begin{equation} \label{thetaj}
	\theta_j(g)=2\pi\bi\frac{(j-1)}{N}+\int_{\fc}F(B_{\sigma,\tau}^j(g)).
\end{equation}
Let $\Theta_{\sigma,\tau}(g)$ be the vector whose $j^{\rm th}$ entry is given by \eqref{thetaj}. Then we define:
\[
 \widetilde \hol_{\sigma,\tau}(g):=\overline{\Theta_{\sigma,\tau}(g)}
\]

\begin{remark}
	To be more precise, we pick $\fc$ such that its intersection with the
	ends corresponding to $S^1\times S^2$ and $L(N,1)$ are cylindrical.
	Moreover, if $Y$ is a cut of the family of metrics $\bbX_N(l)$,
	we also require that in the the regular neighborhood
	of $Y$ the cylinder $\fc$ has the cylindrical form corresponding to a fixed loop in $Y$.	

\end{remark}

\begin{prop} \label{hol-fac-forg}
	The map $\widetilde \hol_{\sigma,\tau}:\mathcal K_{N+1} \to \ft$ factors through the forgetful map
	$\mathfrak{F}: \mathcal K_{N+1}\to\Delta_{N-1}$. That is to say, there is  a map
	$\overline \hol_{\sigma,\tau}:\Delta_{N-1} \to \ft$ such that
	$\widetilde \hol_{\sigma,\tau}=\overline \hol_{\sigma,\tau}\circ \mathfrak{F}$.
\end{prop}

We call $\overline \hol_{\sigma,\tau}$ the {\it reduced holonomy map} associated to the bi-permutation $(\sigma,\tau)$. The proof of Proposition \ref{hol-fac-forg} is based on the following elementary lemma:

\begin{lemma}\label{pairing-curvature-surface}
	Suppose $X$ is a 4-manifold such that each connected component of its boundary is a rational homology
	sphere. Suppose $\Sigma$ is an oriented surface with boundary which is properly embedded in $X$.
	Suppose $X^+$ and $\Sigma^+$ are given by adding cylindrical ends to $X$ and $\Sigma$. Suppose $L$ is a
	Hermitian line bundle on $X^+$. Suppose $\alpha$ is a flat connection on $L|_{X^+\backslash X}$ which is pull-back of a connection on
	$L|_{\partial_X}$. Suppose $A$ is a connection on $L$ which is exponentially asymptotic to $\alpha$ on the ends of $X^+$. Then:
	\begin{equation} \label{pairingCL}
		\frac{1}{2\pi\bi}\int_\Sigma F(A)
	\end{equation}
	is independent of the choice of the connection $A$ and is equal to $c_1(L)[\Sigma]$.
\end{lemma}
A priori, $[\Sigma]\in H_2(X,\partial X)$. Since $H_1(\partial X,\Q)=H_2(\partial X,\Q)=0$, the relative homology class $\Sigma$ has a unique lift to $H_2(X,\Q)$ and the pairing $c_1(L)[\Sigma]$ is well-defined as a rational number.
\begin{proof}
	This lemma is trivial in the case that $\Sigma$ is a closed Riemann surface.
	We can reduce the general case to this special case.
	Using the assumption on $\partial X$, we can find a closed oriented surface $\widetilde \Sigma$ and a continuous family of smooth maps
	$\phi_t: \widetilde \Sigma \to X$ for $t\in \R^{+}$ such that it satisfies the following properties. For each $t$ define:
	\[
	  \Sigma_1(t):=\phi_t^{-1}(\partial X\times[t\times \infty))\hspace{2cm}\Sigma_2(t):=\phi_t^{-1}(X^+\backslash (\partial X\times[t\times \infty))).
	\]
	The map $\phi_t$ is required to have a bounded $C^1$ norm on $\Sigma_1(t)$.
	There also exists $M$, independent of $t$, such that the map $\phi_t$ is $M$ to $1$ on $\Sigma_2(t)$.
	Moreover, the homology class $(\phi_t)_*([\widetilde \Sigma])$ is a lift of the relative homology class
	 $M[\Sigma]$.
	Then we have:
	\[
	  \int_\Sigma F(A)=\frac{1}{M} \lim_{t \to \infty} \int_{\widetilde \Sigma}\phi_t^*(F(A))=2\pi \bi c_1(L)[\Sigma].
	\]
\end{proof}
\noindent	
The above lemma can be generalized in an obvious was to the case that the metric on $X^+$ is broken.

\begin{proof}[Proof of Proposition \ref{hol-fac-forg}]
	It suffices to show that if $g_1, g_2\in \mathcal K_{N+1}$ are two elements
	in the same fiber of the map $\fF$, then they are mapped to the same
	element of $\ft$ by $\widetilde \hol_{\sigma,\tau}$. Since $g_1$ and $g_2$
	belong to the same fibers of $\fF$, there is a decomposition:
	\[
	  X_N(l)=W \#_Y Z
	\]
	such that $S^1\times S^2$ is a boundary component
	of $W$, all boundary components of $Z$ (including Y) are rational homology spheres,
	and $g_1$, $g_2$ determine broken metrics of $X_N(l)$ compatible
	with the above decomposition that agree with each other
	on $W$. (See Figure \ref{XN(l)-decom}.)
	Note that the metric induced by $g_1$ and $g_2$ on $Z$ might be broken.
	In the definition of $\widetilde \hol_{\sigma,\tau}(g_1)$ and $\widetilde \hol_{\sigma,\tau}(g_2)$,
	the integrals $\int_\fc F(B^j_{\sigma,\tau}(g_1))$ and $\int_\fc F(B^j_{\sigma,\tau}(g_2))$
	have contributions from $W$ and $Z$.
	The contributions from $Z$ are equal to each other by Lemma \ref{pairing-curvature-surface}
	(or by the extension of Lemma \ref{pairing-curvature-surface} to the case of
	broken metrics). Moreover,
	the restrictions of $B^j_{\sigma,\tau}(g_1)$ and $B^j_{\sigma,\tau}(g_2)$ to $W$ agree with each other.
	Therefore, the contributions from $W$ to the integrals are also equal.	
\end{proof}

\begin{prop} \label{hol-maps-k-space-to-kspace}
	Let $k$ be a positive integer number not greater than $N$. Let $(\sigma,\tau)$ be a bi-permutation of type $l$.
	Then the holonomy map $\overline {\hol}_{\sigma,\tau}:\Delta_{N-1} \to \ft$ maps each face of the simplex
	$\Delta_{N-1}$ with dimension $k-1$ into an affine subspace of $\ft$
	with dimension $k-1$.
\end{prop}

\begin{proof}
	Fix a face $\Delta'$ of $\Delta_{N-1}$ with dimension $k-1$. Then there are $k+1$ positive integer numbers $i_0$, $\dots$, $i_k$ such that:
	\[
	  i_0+\dots+i_k=N+1
	\]
	and $\Delta'$ parametrizes the arrangements of $N+1$ points, denoted by $q_1\leq q_2 \leq \dots\leq q_{N+1}$ in $\R$ which satisfy the following property for each $0\leq j \leq k$:
	\[
	  q_{M_j+1}=q_{M_j+2}=\dots=q_{M_j+i_j}
	\]
	where $M_j=i_0+i_1+\dots+i_{j-1}$. As before, two arrangements are equivalent to each other if one of them is mapped to
	the other one by an affine map.
	Let $j_0$ be the smallest non-negative integer number that:
	\[
	  i_0+i_1+\dots+i_{j_0}>N-l.
	\]
	
	The face $\Delta'$ determines a cut of $X_N(l)$ with at most $k+1$ connected components.
	The connected components
	of $X_N(l)$ after removing this cut can be written as the union of a 4-manifold $W$ and $Z$.
	(See Figure \ref{XN(l)-decom}.)
	The 4-manifold $W$ is connected and has $S^1\times S^2$ as one of its boundary components.
	The 4-manifold $Z$
	is the union of the remaining connected components.
	In particular, all of its boundary components are rational homology sphere.
	The boundary of $W$ is equal to:
	\[
	  \partial W=\left ( \coprod_{j\neq j_0} -L(i_j,1) \right)\sqcup L(N+1-i_j)\sqcup S^1\times S^2
	\]
	Gluing $W$ to $Z$ along the boundary components of $W$, which are non-trivial lens spaces, produces $X_N(l)$.
	\begin{figure}
    		\centering
	        	
\begin{tikzpicture}
\node at (3,2) {$W$};
\draw[fill] (0,0) circle (3pt);
\draw[fill] (1,0) circle (3pt);
\draw[fill] (2,0) circle (3pt);
\draw[fill] (3,0) circle (3pt);
\draw[fill] (4,0) circle (3pt);
\draw[thick] (5,0) circle (3pt);
\draw[fill] (6,0) circle (3pt);
\draw[fill] (7,0) circle (3pt);
\draw[fill] (8,0) circle (3pt);
\draw[fill] (9,0) circle (3pt);
\draw[thick] (4.5,0) ellipse (6cm and 3cm);
\draw [middlearrow={latex},thick,red]  (5,2.8pt)--  (5,3);  \node[right] at (5,1.3) {$\mathfrak{c}$};
\draw[middlearrow={latex},thick,blue]  (2,0.1) to [out=45,in=135] (4.9,0.05); \node[above] at (3.7,0.6) {$e_3$};
\draw[middlearrow={latex},thick,blue]  (3,-0.1) to [out=-45,in=-135] (4.9,-0.05);\node[below] at (3.7,-0.45) {$e_4$};

\draw[thick] (1.5,0) ellipse (1.8cm and 1cm);
\draw[thick] (5,0) ellipse (1.35cm and 0.7cm);
\draw[thick] (8.5,0) ellipse (1cm and 0.7cm);

\draw[middlearrow={latex},thick,orange]  (2,-0.97) to [out=-45,in=-135] (4.9,-0.7); \node[below] at (4,-1.3) {$\widetilde{f_1}$};
\draw[middlearrow={latex},thick,orange]  (7,-0.1) to [out=-110,in=-45] (5.5,-0.65);  \node[below] at (6.7,-0.6) {$\widetilde{f_3}$};

\end{tikzpicture}
    		\caption{This schematic figure shows a decomposition of $X_{9}(4)$ associated to an element of a face
		of dimension $3$ in $\Delta_{8}$. In this example, $i_0=4$, $i_1=3$, $i_2=1$ and $i_3=2$ .
		The outer region sketches the 4-manifolds $W$. Two instances
		of cohomology classes $e_i$, two instances of cohomology classes $\widetilde f_i$ and the cylinder $\fc$
		are also sketched in this figure.}
	    	\label{XN(l)-decom}
	\end{figure}
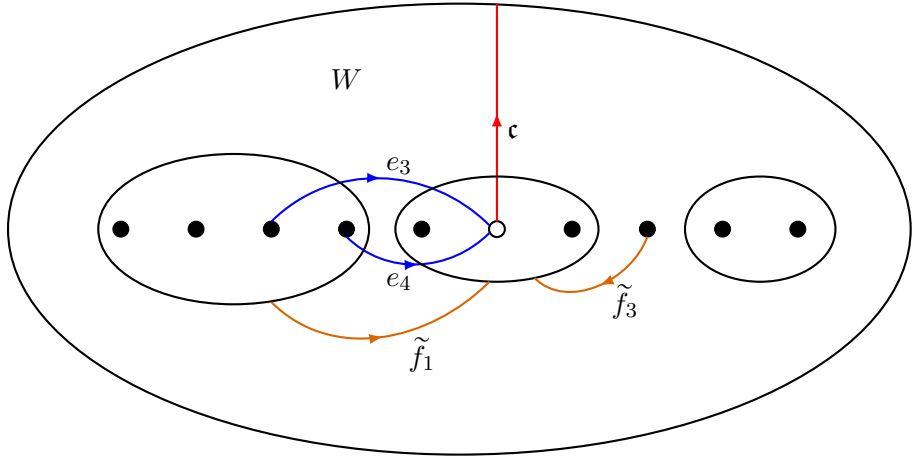

	Suppose $f_i$ is the restriction of the cohomology class
	$e_i \in H^2(X_N(l),\Z)$ to $W$. Then for each
	$0\leq j \leq k$:
	\begin{equation} \label{identity-cohomology}
	  f_{M_j+1}=f_{M+2}=\dots=f_{M_j+i_j}
	\end{equation}
	where $M_j=i_0+i_1+\dots+i_{j-1}$. We denote the common value of these cohomology classes by $\widetilde f_j$.
	These cohomology classes satisfy the following
	relations:
	\begin{equation} \label{relations-fj}
	  \widetilde f_{j_0}=0\hspace{2cm}i_0\widetilde f_0+i_1\widetilde f_1+\dots+i_k\widetilde f_k=0.
	\end{equation}
	The second identity holds because the sum of the cohomology classes $e_i$ vanishes.
	
	Each element of $\Delta'$ determines a metric with cylindrical ends on $W$.
	In particular, we can consider the ASD connection $A_j$ on the Hermitian line bundle $L_j$
	with $c_1(L_j)=\widetilde f_j$. For each $g \in \Delta'$:
	\[
	 a_j(g)=\int_{W\cap \fc} F(A_j)
	\]
	We shall show that each component of $\overline {\hol}_{\sigma,\tau}$ can be expressed affinely in terms of $a_i$'s with universal constants.
	The relationships in \eqref{relations-fj} gives rise to similar relations for $a_j(g)$.
	In light of that, we can conclude that $\overline {\hol}_{\sigma,\tau}(\Delta')$ lives in a
	$(k-1)$-dimensional affine subspace of $\ft$.
	
	In order to prove the above claim, it suffices to show that:
	\[
	  \int_\fc F(B^j_{\sigma,\tau}(g))=\int_{\fc \cap W}F(B^j_{\sigma,\tau}(g))+\int_{\fc \cap Z}F(B^j_{\sigma,\tau}(g))
	\]
	can be written affinely in terms of $a_i$'s.
	Clearly, the first term in the right hand side of the above identity is a linear expression in terms of $a_i$'s
	because the first Chern class of the
	carrying line bundle of $B^j_{\sigma,\tau}(g)|_{\fc\cap W}$ is a linear combination of the classes $\widetilde f_j$.
	By Lemma \ref{pairing-curvature-surface}, the second term in the right hand side of the above identity is also independent of the connection $B^j_{\sigma,\tau}(g)|_{\fc\cap Z}$
	and is determined by the topology of the corresponding line bundle. The isomorphism class of this line bundle is determined by $(\sigma,\tau)$ and is independent of $g$.
	(Note that the connection $B^j_{\sigma,\tau}(g)|_{\fc\cap Z}$ is not
	even well-defined because $g\in \Delta'$ does not fix any metric on $Z$.)
\end{proof}

The method of the proof of the above proposition can be used to determine explicitly the subspace containing the image of each face of $\Delta_{N-1}$. In particular, the image of each vertex can be characterized explicitly. This task is addressed in the next proposition. We denote the vertices of $\Delta_{N-1}$ with $u_k$ for $0\leq k \leq N-1$. Identifying $\Delta_{N-1}$ with the weak compactification of $\mathcal P_{N+1}$, we assume that the vertex $u_k$ corresponds to the arrangement of $(N+1)$ points $q_1$, $\dots$, $q_{N+1}$ such that:
\[
  q_1=\dots=q_{N-k}\hspace{1cm}q_{N-k+1}=\dots=q_{N+1}.
\]
As in the proof of the above proposition, each vertex also determines a decomposition of $X_N(l)$ to two 4-manifolds
$W$, $Z$ and a metric on $W$.

\begin{prop} \label{red-hol-ver}
	Suppose $(\sigma,\tau)$ is a bi-permutation of type $l$.
	Then:
	\begin{equation} \label{hol-sigma-tau-vert}
	  \overline \hol_{\sigma,\tau}(u_k)=\frac{1}{N-k}(\overline {\lambda_{\sigma(0)}}+\overline {\lambda_{\sigma(1)}}
	  +\dots+\overline {\lambda_{\sigma(N-k-1)}})
	\end{equation}
	and:
	\begin{equation} \label{hol-sigma-tau-vert-2}	
  	  \overline \hol_{\sigma,\tau}(u_{k'})=\frac{1}{k'+1}(\overline {\lambda_{\tau(l-1-k')}}+\overline {\lambda_{\tau(l-k')}}
	  +\dots+\overline {\lambda_{\tau(l-1)}})
	\end{equation}
	where $0\leq k' < l\leq k \leq N-1$.
\end{prop}

\begin{proof}
	The image of the vertex $u_k$ with respect to the map $\overline \hol_{\sigma,\tau}$ is determined by the following expression:
	\[
	  \frac{1}{2\pi \bi}\int_\fc F(B^j_{\sigma,\tau}(u_k))=
	  \frac{1}{2\pi \bi} \int_{\fc \cap W}F(B^j_{\sigma,\tau}(u_k))
	  +\frac{1}{2\pi \bi} \int_{\fc \cap Z}F(B^j_{\sigma,\tau}(u_k))
	\]
	The number $j_0$, defined in the proof of Proposition \ref{hol-maps-k-space-to-kspace} is equal to $1$ or $0$
	depending on whether $k\geq l$ or $k< l$. The first integral in the left hand side of the above expression
	vanishes because the carrying bundle of $B^j_{\sigma,\tau}(u_k)$ has vanishing $c_1$ in $W$,
	and hence $B^j_{\sigma,\tau}(g)$ is flat. In order to compute the second integral,
	we shall compute the difference
	$\int_{\fc\cap Z}F(B^j_{\sigma,\tau}(u_k))-\int_{\fc\cap Z}F(B^{j-1}_{\sigma,\tau}(u_k))$.
	By Lemma \ref{pairing-curvature-surface}, this difference is equal to the paring of the restriction of the cohomology
	classes $e_{i_j}$ and $\PD(\fc)$ to $Z$ where:
	\[
	  i_j=\left\{
	  \begin{array}{ll}
	  	\sigma^{-1}(N-j+1)+1 &N-j+1 \in {\rm image}(\sigma)\\
		\tau^{-1}(N-j+1)+N-l+1 &N-j+1 \in {\rm image}(\tau)
  	  \end{array}	
	  \right.
	\]
	Assume that $k\geq l$, $N-k\geq i_j$ as in Figure \ref{XN(l)-vertex-decom}.
	Let also $Z_1$ be the connected component
	of $Z$ which has a the lens space $L(N,1)$ as one of its connected components.
	Then we can write:
	\begin{align*}
		\int_Z e_{i_j}\cup \PD(\fc)&=\int_{Z_1} e_{i_j}\cup \PD(\fc)\\
		&=-\int_{Z_1} e_{i_j}\cup e_{i_j}\\
		&=-\int_{X_N(l)} e_{i_j}\cup e_{i_j}+\int_{X_N(l)\backslash Z_1} e_{i_j}\cup e_{i_j}\\
		&=(-\frac{1}{N}+1)-(-\frac{1}{N-k}+1)
	\end{align*}
	Here $\int_X\alpha$, for an oriented 4-manifold $X$ and $\alpha\in H^4(X,\partial X,\Q)$,
	denotes the pairing of $\alpha$ and the generator of $H_4(X,\partial X,\Q)$. The second equality above holds
	because the restriction of $\PD(\fc)$ and $-e_{i_j}$ to $Z_1$ are equal to each other.
	The last equality is the consequence of applying Lemma \ref{int-num}. Similar arguments
	can be used to show that:
	\begin{equation}\label{second-pairing}
	   \int_{\fc\cap Z}F(B^j_{\sigma,\tau}(u_k))-\int_{\fc\cap Z}F(B^{j-1}_{\sigma,\tau}(u_k))=\left\{
	   \begin{array}{ll}
		\frac{1}{N-k}-\frac{1}{N} &k\geq l,\, N-k\geq i_j\\
		-\frac{1}{N} &k\geq l,\, N-k< i_j\\
		-\frac{1}{N} &k< l,\, N-k>i_j\\
		\frac{1}{k+1}-\frac{1}{N} &k< l,\, N-k\leq i_j	
	   \end{array}
	   \right.
	\end{equation}

	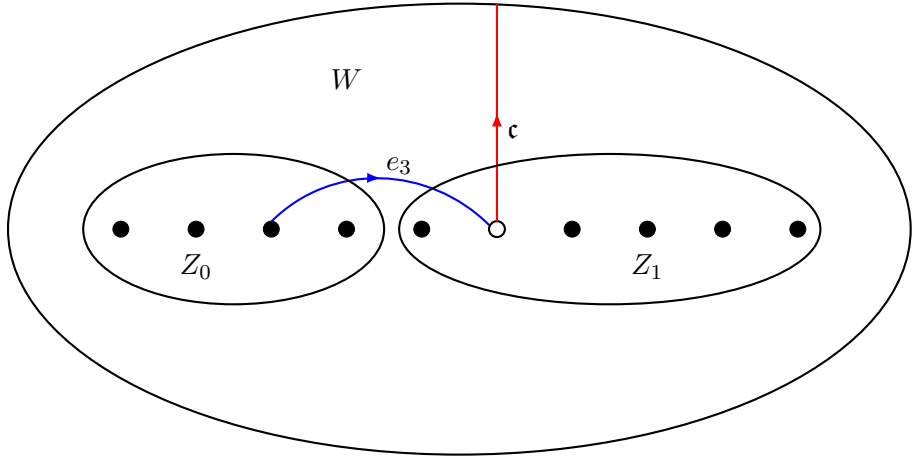
\begin{figure}
    		\centering
	    		\begin{tikzpicture}
\node at (3,2) {$W$};
\draw[fill] (0,0) circle (3pt);
\draw[fill] (1,0) circle (3pt);
\draw[fill] (2,0) circle (3pt);
\draw[fill] (3,0) circle (3pt);
\draw[fill] (4,0) circle (3pt);
\draw[thick] (5,0) circle (3pt);
\draw[fill] (6,0) circle (3pt);
\draw[fill] (7,0) circle (3pt);
\draw[fill] (8,0) circle (3pt);
\draw[fill] (9,0) circle (3pt);
\draw[thick] (4.5,0) ellipse (6cm and 3cm);
\draw [middlearrow={latex},thick,red]  (5,2.8pt)--  (5,3);  \node[right] at (5,1.3) {$\mathfrak{c}$};
\draw[middlearrow={latex},thick,blue]  (2,0.1) to [out=45,in=135] (4.9,0.05); \node[above] at (3.7,0.6) {$e_3$};

\draw[thick] (1.5,0) ellipse (2cm and 1cm);\node[below] at (1,-0.2) {$Z_0$};
\draw[thick] (6.5,0) ellipse (2.8cm and 1cm);\node[below] at (7,-0.2) {$Z_1$};

\end{tikzpicture}
    		\caption{Decomposition of $X_9(4)$ associated to the vertex $u_5$ of $\Delta_{8}$.}
	    	\label{XN(l)-vertex-decom}
	\end{figure}
	Next, for each vertex $u_k$ of $\Delta_{N-1}$, we compute the vector
	$\Theta_{\sigma,\tau}(u_k)$ associated to the bi-permutation $(\sigma,\tau)$, which is defined in \eqref{thetaj} .
	Firstly let $k\geq l$. The first entry of this vector is equal to zero.
	The first case of \eqref{second-pairing} shows that if
	$i_j\leq N-k$, then $j^{\rm th}$ entry of $\Theta_{\sigma,\tau}(u_k)$ is obtained by
	adding $\frac{1}{N-k}$ to the $(j-1)^{\rm st}$ entry.
	Otherwise, these two entries are equal to each other by the second case of \eqref{second-pairing}.
	Therefore, for $k\geq l$:
	\begin{equation} \label{Theta-sigma-tau-k<=l}
	  \Theta_{\sigma,\tau}(u_k)=\frac{1}{N-k}({\lambda_{\sigma(0)}}+ {\lambda_{\sigma(1)}}+\dots
	  +{\lambda_{\sigma(N-k-1)}})
	\end{equation}
	Similarly, for $k'< l$, we can show that:
	\begin{equation} \label{Theta-sigma-tau-k>l}
	  \Theta_{\sigma,\tau}(u_{k'})=\frac{1}{{k'}+1}({\lambda_{\tau(l-1-k')}}+ {\lambda_{\tau(l-k')}}+
	  \dots+{\lambda_{\tau(l-1)}})
	\end{equation}
	An immediate consequence of \eqref{Theta-sigma-tau-k<=l} and \eqref{Theta-sigma-tau-k>l} is the claim in \eqref{hol-sigma-tau-vert}.
\end{proof}

In the next part, we describe some standard simplicial decompositions of $\Delta_{N-1}$, called \emph{bi-barycentric subdivisions}. Such decompositions are relevant to our discussion because the reduced holonomy maps $\overline \hol_{\sigma,\tau}$, which are initially defined for each bi-permutation separately, glue together in a natural way according to bi-barycentric subdivisions.

\subsubsection*{Bi-barycentric Subdivision and Holonomy Maps}

Suppose $\Delta_{N-1}$ is an arbitrary $(N-1)$-dimensional simplex in a Euclidean spaces whose vertices are denoted by $v_1$, $\dots$, $v_N$. Then any point in this simplex has the form of $r_1v_1+r_2v_2+\dots+r_Nv_N$ with:
\begin{equation}\label{simp-rel}
 	r_1+\dots+r_N=1\hspace{1cm}r_i\geq 0
\end{equation}
This coordinate system identifies the simplex with the standard simplex which consists of the points $(r_1,\dots,r_N)\in \R^N$ that satisfy \eqref{simp-rel}. The barycentric subdivision of $\Delta_{N-1}$ is obtained by cutting $\Delta_{N-1}$ with all planes $r_i = r_j$.  This subdivision has cells $\Delta_{N-1}^{f}$, indexed by permutations $f \in S_N$, which are cut out by a series of inequalities:
\[
  r_{f(1)} \geq r_{f(2)} \geq \cdots \geq r_{f(N)}.
\]
It is clear that every such cell is a simplex. The vertices of this simplex are given by:
\[
  \hspace{2.5cm}\frac{1}{k}(v_{f(1)}+\dots+v_{f(k)})\hspace{1cm}1\leq k \leq N
\]

The barycentric subdivision of a simplex also has an interpretation in terms of the moduli spaces of points on a real line. The simplex $\Delta_{N-1}$  can be identified with the space of all $(N+1)$-tuple of points as follows:
\begin{equation} \label{DeltaN-1-points}
  q_0,q_1,\dots,q_N\in \R,\hspace{1cm} q_0\leq q_i
\end{equation}
where at least one of the above inequalities is strict and two arrangements are equivalent to each other if they are related by an affine transformation.
This identification is given by:
\[
  r_i=\frac{q_i-q_0}{\sum_{i=1}^{N} (q_i-q_0)}
\]
Each permutation $f\in S_N$ determines an ordering as follows:
\[
  q_0 \leq q_{f(1)}\leq \dots \leq q_{f(N)}
\]
The moduli space of points satisfying the above ordering forms one of the cells in the barycentric subdivision of $\Delta_{N-1}$. Therefore, any such cell can be identified with the week compactification of $\mathcal P_{N+1}$ in a natural way.

For any $0\leq l \leq N$, the {\it bi-barycentric subdivision of type} $l$ of a simplex $\Delta_{N-1}$ is defined in a similar way. Label the vertices of $\Delta_{N-1}$ as $u_1$, $\dots$, $u_{N-l}$, $u_1'$, $\dots$, $u_l'$ and let $(r_1,\dots,r_{N-l},s_1,\dots,s_l)\in \R^N$ denote the coordinate of a typical point of $\Delta_{N-1}$. Then the bi-barycentric subdivision has cells $\Delta^{(f,g)}_{N-1}$ for each $(f,g)\in S_{N-l}\times S_l$, which are cut out by two series of inequalities:
\[ r_{f(1)} \geq \cdots \geq r_{f(N-l)}\hspace{1cm}s_{g(1)} \geq \cdots \geq s_{g(l)}  \]
Again, it is clear that every such cell is a simplex.  In fact, $\Delta^{(f,g)}_{N-1}$ is the convex hull of the simplices $\Delta_{N-l-1}^{f} \times 0$ and $0 \times \Delta_{l-1}^{g}$. In particular, its vertices are given by:
\[
  \hspace{2.5cm}\frac{1}{k_1}(u_{f(1)}+\dots+u_{f(k_1)})\hspace{1cm}\frac{1}{k_2}(u_{g(1)}'+\dots+u_{g(k_2)}')\hspace{1cm}1\leq k_1 \leq N-l,\,1\leq k_2 \leq l
\]
As long as both $l$ and $N-l$ are nonzero, it follows that the bi-barycentric subdivision of type $l$ is the \emph{join} of the barycentric subdivisions of $\Delta_{N-l-1}$ and $\Delta_{l-1}$.  When $l$ or $N-l$ is equal to $0$, the bi-barycentric subdivision of type $l$ is just the ordinary barycentric subdivision of $\Delta_{N-1}$. In the case $N=3$, the two subdivisions of $\Delta_{N-1}$ is demonstrated in Figure \ref{bi-barycentric-3}.
\begin{figure}
    	\centering

\begin{tikzpicture}
\draw[thick] (-3,0)--(3,0);
\draw[thick] (-3,0)--(0,4);
\draw[thick] (3,0)--(0,4);
\draw[thick,red] (0,0)--(0,4);

\begin{scope}[shift={(8,0)}]
\draw[thick] (-3,0)--(3,0);
\draw[thick] (-3,0)--(0,4);
\draw[thick] (3,0)--(0,4);
\draw[thick,red] (0,0)--(0,4);

\draw[thick,red] (-1.5,2)--(3,0);
\draw[thick,red] (1.5,2)--(-3,0);
 \end{scope}

\end{tikzpicture}
    	\caption{Two bi-barycentric subdivisions of $\Delta_2$}
    	\label{bi-barycentric-3}
\end{figure}
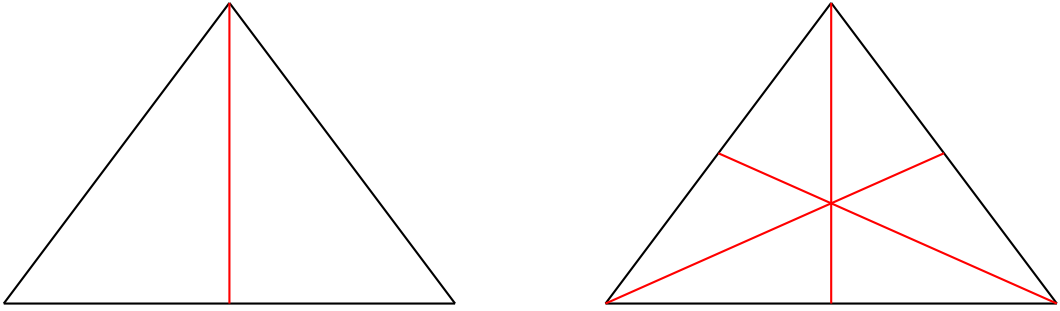

We can give a description of bi-barycentric subdivision in terms of moduli spaces of points. Suppose $\Delta_{N-1}$ is identified with the space of all arrangements of $N+1$ points  $o_1$, $\dots$, $o_{N-l}$, $p$, $q_1$, $\dots$, $q_l$ in $\R$ such that:
\begin{equation}\label{ineq-modui-points}
 o_1,\dots, o_{N-l}\leq p \leq l_{1},\dots, q_{l}
\end{equation}
and not all the above points are equal to each other. Moreover, two arrangements are equivalent to each other if they are related by an affine transformation. This identification is given by:
\[
  r_i=\frac{p-o_i}{\sum_{i=1}^{N-l} (p-o_i)+\sum_{i=1}^{l} (q_i-p)}\hspace{1cm}
  s_j=\frac{q_j-p}{\sum_{i=1}^{N-l} (p-o_i)+\sum_{i=1}^{l} (q_i-p)}
\]
Any pair of permutations $(f,g) \in S_{N-l}\times S_{l}$ determines an ordering of the points as follows:
\begin{equation}\label{ineq-modui-points-ci-bary-cent-cell}
  o_{f(1)}\leq \dots \leq o_{f(N-l)}\leq p \leq q_{g(1)}\leq \dots \leq q_{g(l)}
\end{equation}
The moduli space of points satisfying the above ordering forms one of the cells in the bi-barycentric subdivision of type $l$ of $\Delta_{N-1}$ which we denote by $\Delta^{f,g}_{N-1}$.

Suppose $(\sigma_0,\tau_0)$ is a bi-permutation of type $l$ as in \eqref{sigma-tau} such that their images are given by the subsets $S$ and $[N]\setminus S$ of $[N]$.
Identify the simplex $\Delta_{N-1}$ with the moduli points as in \eqref{ineq-modui-points}. Then any cell of the bi-barycentric subdivision of type $l$ of $\Delta_{N-1}$ is determined by a pair of permutations $(f,g)\in S_{N-l} \times S_l$. Moreover, we can associate a bi-permutation $(\sigma_f,\tau_g)$ of type $l$ to $(f,g)$ as in \eqref{sigmaf-taug} such that images of $\sigma_f$ and $\tau_g$ are also equal to $S$ and $[N]\setminus S$. Identify $\Delta_{N-1}$ with the arrangements of points as in \eqref{ineq-modui-points}. If a point $\bx=(o_1,\dots,o_l,p,q_1,\dots,q_m)$ of this simplex satisfies the inequalities in \eqref{ineq-modui-points-ci-bary-cent-cell}, then define:
\begin{equation} \label{hol-ST}
  \overline \hol_{S}(\bx):=\overline \hol_{\sigma_f,\tau_g}(\bx)
\end{equation}

\begin{prop}
	The map $\overline {\hol}_{S}$ is a well-defined map from $\Delta_{N-1}$ to $\ft$
\end{prop}

\begin{proof}
	We need to show that the maps $\overline{\hol}_{\sigma_f,\tau_g}$ agree on the overlaps of the simplices
	$\Delta_{N-1}^{f,g}$. An $(N+1)$-tuple $\bx=(o_1,\dots,o_{N-l},p,q_1,\dots,q_l)$ belongs to more
	than one cell $\Delta_{N-1}^{f,g}$, if some of its entries are equal to each other.
	For simplicity, we assume that these numbers satisfy the following order:
	\begin{equation}\label{arrangment-triv}
	  o_{1}\leq \dots \leq o_{N-l}\leq p \leq q_{1}\leq \dots \leq q_{l}
	\end{equation}
	Then $(o_1,\dots,o_l,p,q_1,\dots,q_m)$ belongs to $\Delta_{N-1}^{id,id}$.
	We will write $\bx_0$ for this element of $\Delta_{N-1}^{id,id}$.
	If some of the inequalities in \eqref{arrangment-triv} are equalities, then $(o_1,\dots,o_{N-l},p,q_1,\dots,q_l)$
	determines elements of other cells of the bi-barycentric subdivision of $\Delta_{N-1}$. To avoid cumbersome
	notations, we consider a special case which includes the subtleties of the general case.
	We assume that $p=q_1=q_2$. Then $(o_1,\dots,o_{N-l},p,q_1,\dots,q_l)$ gives rise to an element $\bx_1$ of
	$\Delta_{N-1}^{id,g}$ where $g$ is the permutation of the set $[l]$ that interchanges $1$ and $2$. The assocaited
	bi-permutations to $\bx_0$ and $\bx_1$ are equal to $(\sigma_0,\tau_0)$ and $(\sigma_0,\tau_g)$.
	
	Analogous to the proof of Proposition \ref{hol-fac-forg}, the arrangements $\bx_0$ and $\bx_1$ determine
	a decomposition of $X_N(l)$ as $W\#_Y Z$ where $Y=L(N-2,1)$,
	the non-trivial boundary components of $W$ are $S^1 \times S^2$ and $L(N-2,1)$, and
	the non-trivial boundary components of $Z$ are $-L(N-2,1)$ and $L(N,1)$. The elements $\bx_0$ and $\bx_1$
	also give rise to the same metric $g$ with cylindrical ends on $W$. The difference between $\bx_0$ and $\bx_1$
	is that they determine different 2-cycles in $X_N(l)$. Nevertheless, the restrictions of these 2-cycles to $W$
	agree with each other. Furthermore, there is a diffeomorphism of $Z$ which maps
	the cohomology classes of the 2-cycles determined by $\bx_0$ and $\bx_1$  to each other,
	maps the homology class of the cylinder $\fc \cap Z$ to itself, and is equal to the identity map on the boundary.
	Thus, we can argue as
	in the proof of Proposition \ref{hol-fac-forg} to show that
	$\overline{\hol}_{\sigma_0,\tau_0}(\bx_0)=\overline{\hol}_{\sigma_0,\tau_g}(\bx_1)$.	
\end{proof}

We can compose the map $\overline \hol_{S}$ with the exponential map to obtain a map whose target is the $(N-1)$-dimensional torus $T$ in $\SU(N)$. We compose the resulting map with the quotient map from $T$ to $T/W\cong \Delta_{N-1}^\ft$ to obtain $\hol_{S}:\Delta_{N-1}\to \Delta_{N-1}^\ft$. We shall show that $\hol_{S}$ has degree one in an appropriate sense.

\begin{prop}\label{deg-one}
	There is a map $H_{S}:[0,1]\times \Delta_{N-1} \to \Delta_{N-1}^\ft$ such that
	$H_{S}(0,\cdot):\Delta_{N-1} \to \Delta_{N-1}^\ft$ is an affine isomorphisms of simplices and
	\[
	  H_{S}(1,\cdot)=\hol_{S}.
	\]
	Moreover, there is an open dense subset $\Lambda_{S}$ of $\Delta_{N-1}^\ft$ such that
	if $H_{S}(t,x)\in \Lambda_{S}$, then $H_{S}(t,x)$ is smooth at the point $(t,x)$.
\end{prop}
\begin{proof}
	Proposition \ref{red-hol-ver} asserts that $\overline \hol_{S}$ maps the vertices of $\Delta_{N-1}$ bijectively
	to the vertices of $\Delta_{N-1}^\ft$. Let $H_0$ be the affine isomorphism from $\Delta_{N-1}$ to $\Delta_{N-1}^\ft$ whose
	restriction to the vertices agrees with the map $\overline \hol_{S}$. Define
	$H_{S}:[0,1]\times \Delta_{N-1} \to \Delta_{N-1}^\ft$  to  be the composition of the map:
	\[
	  \overline H_{S}(t,\cdot):=t\cdot \overline \hol_{S}+(1-t)H_0.
	\]
	with the projection map from $\ft$ to $ \Delta_{N-1}^\ft$.
	We can also pick $\Lambda_{S}$ to be the open cells in the bi-barycentric subdivision of $\Delta_{N-1}^\ft$
	corresponding to $(S,T)$.
	This choice of $\Lambda_{S}$ satisfies our desired property because
	$H_{S}$ maps the cells of dimension at most $N-2$ in the bi-barycentric subdivision of $\Delta_{N-1}$
	to the cells of dimension at most $N-2$ in the bi-barycentric subdivision of $\Delta_{\ft}$
	by Propositions \ref{hol-maps-k-space-to-kspace} and \ref{red-hol-ver}.
\end{proof}

\subsubsection*{Regularity of Connections on $X_N(l)$}

Let $p$ be the path along $(X_N(l),w_{\sigma,\tau})$ represented by completely reducible connections. Up to this point, we only study the subset of $\mathcal M_{p}(\mathbb{X}_N(l),w_{\sigma,\tau},\Delta_{N-1}^{\ft})$ given by completely reducible connections. It would be desirable for our purposes to guarantee that, possibly after a perturbation, these completely reducible connections are regular and there is no other connection in $\mathcal M_{p}(\mathbb{X}_N(l),w_{\sigma,\tau},\Delta_{N-1}^{\ft})$. However, we can only achieve the following weaker result under the assumption that $N\leq 4$. Essentially, this is the only place in the proof of the main theorem that we need to assume that $N\leq 4$. Proof of Proposition \ref{avoid-noncomp-red-pre} is deferred until Subsection \ref{XNl-reg}.

\begin{prop}\label{avoid-noncomp-red-pre}
	Suppose $N\leq 4$, $c$ is an arbitrary 2-cycle in $X_N(l)$. Suppose $M$ is a finite subset of $\Delta_{N-1}^{\ft,\circ}$.
	There exists an arbitrary small perturbation of the ASD equation over $\mathbb{X}_N(l)$ such that for any path $p$ along $(X,c)$,
	with limiting flat connections
	$\chi_0$ and $\beta_0$ on $L(N,1)$ and $S^1\times S^2$,
	we have:
	\begin{enumerate}
	  \item[(i)]  if $\ind(p)< -h^0(\beta_0)-h^0(\chi_0)$, then the moduli space
	  $\mathcal M_{p}(\mathbb{X}_N(l),c;\Delta_{N-1}^{\ft})$ is empty;
	  \item[(ii)]  if $\ind(p)= -h^0(\beta_0)-h^0(\chi_0)$, then
	  any element of $\mathcal M_{p}(\mathbb{X}_N(l),c;M)$ is a completely reducible connection associated to
	  an element of $\mathcal {K}_{N+1}^\circ$.
	\end{enumerate}
	\vspace{-5pt}
	Moreover, there exists a neighborhood $\mathcal V$ of $M$ such that
	$\overline{\mathcal V}\subset \Delta_{N-1}^{\ft,\circ}$ and $\mathcal M_{p}(\mathbb{X}_N(l),c;\overline{\mathcal V})$ is compact.
\end{prop}

\section{Floer Homology and Surgery $N$-gon} \label{Floer-main}
The proof of the main theorem of the paper is given in this section. In Subsection \ref{inst-floer}, we review general properties of $\U(N)$-instanton Floer homology for $N$-admissible pairs. The new ingredient of this subsection is the extension of the definition of cobordism maps to the case of cobordisms with middle ends. We formulate a more detailed statement of Theorem \ref{surgery-polygon-thm} in Subsection \ref{statement}. This theorem will be proved in Subsection \ref{proof-thm}.

\subsection{$\SU(N)$-instanton Floer Homology}\label{inst-floer}

Given an $N$-admissible pair $(Y,\gamma)$, the instanton Floer homology $\text{I}_\ast^N(Y,\gamma)$ is defined in \cite{KM:YAFT}. In the present article, we only consider the instanton Floer homology groups with coefficients in $\Z/2\Z$. The $\Z/2\Z$-vector space $\rI_*^N(Y,\gamma)$ is the homology of a chain complex $(\mathfrak{C}_\ast^N (Y,\gamma),d)$. We firstly perturb the Chern-Simons functional such that the set of critical points $R(Y,\gamma)$ are non-degenerate. Then a basis for the chain group $\fC_*^N(Y,\gamma)$ is given by the elements of $R (Y,\gamma)$ which is necessarily a finite set. The differential $d$ is defined by counting the solutions of the (perturbed) ASD equation on $\R\times Y$ with respect to a product metric. To be a bit more detailed, the solutions to the downward gradient flow equation of the (perturbed) Chern-Simons functional can be identified with the solutions of a (perturbed) ASD equation on $\R\times Y$. Let $p$ be a path along $([0,1]\times Y,[0,1] \times \gamma)$ whose restrictions to $\{0\}\times Y$, $\{1\}\times Y$ are equal to $\alpha,\beta\in R(Y,\gamma)$. Let also $\mathcal M_p(\R\times Y,\R \times \gamma,\alpha,\beta)$\footnote{The notation $\mathcal M_p([0,1]\times Y,[0,1] \times \gamma,\alpha,\beta)$ is more compatible with Section \ref{cyl-end-mod}. However, we pick this notation to make it more clear that there is an $\R$-action on the moduli space.} denote the solutions of the (perturbed) ASD equation associated to the path $p$. The perturbation of the Chern-Simons functional can be chosen such that all elements of the moduli spaces $\mathcal M_p(\R\times Y,\R \times \gamma,\alpha,\beta)$ are regular. There is an $\R$-action on $\mathcal M_p(\R\times Y,\R \times \gamma,\alpha,\beta)$ induced by translations along the $\R$ factor in $\R\times Y$. This action is free unless $p$ is the trivial path. We will write $\breve{\mathcal M}_p(\R\times Y,\R \times \gamma,\alpha,\beta)$ for the quotient of $\mathcal M_p(\R\times Y,\R \times \gamma,\alpha,\beta)$ with respect to this action. Then $d(\alpha)$ for $\alpha\in R(Y,\gamma)$ is defined as:
\begin{equation*}
	d(\alpha)=\sum_{p:\alpha \to \beta} \#\breve{\mathcal M}_p(\R\times Y,\alpha,\beta)\cdot \beta
\end{equation*}
In the above expression and in the following $\#\mathcal M$ for a manifold $\mathcal M$ is equal to the number of elements in $\mathcal M$, mod $2$, if this manifold is $0$-dimensional, and is equal to zero otherwise.

Instanton Floer homology is functorial with respect to cobordisms. We need to extend this functoriality to cobordisms which have middle ends diffeomorphic to lens spaces.  Let $(W,c)$ be a cobordism from an $N$-admissible pair $(Y,\gamma)$ to another $N$-admissible pair $(Y',\gamma')$ with middle end
$(L(p,q),\lambda)$. Let also $\eta\in R(L(p,q),\lambda)$. We also fix a Riemannian metric on $W^+$ such that the metric on the end associated to the lens space $L(p,q)$ is induced by the round metric. We can use the following proposition to define a cobordism map $\rI_*^N(W,c,\eta)$ associated to the pair $(W,c)$ and the flat connection $\eta$.
\begin{prop}\label{chainmap}
	Suppose $\fC^N_* (Y, \gamma)$ and $\fC^N_* (Y', \gamma')$ are Floer chain complexes associated to $(Y,\gamma)$ and $(Y',\gamma')$
	after fixing Riemannian metrics and appropriate perturbations of the Chern-Simons functional.
	There is a perturbation of the ASD equation on $(W,c)$ such that all moduli spaces $\mathcal{M}_p(W,c;\alpha,\eta,\beta)$ consist of regular points, where
	$\alpha,\beta$ are generators of $\mathfrak{C}_\ast^N (Y, \gamma)$, $\mathfrak{C}_\ast^N(Y', \gamma')$ and $p$ is a path along $(W,c)$
	whose restrictions to $(Y,\gamma)$, $(Y',\gamma')$ and $(L(p,q),\lambda)$ are respectively equal to $\alpha$, $\beta$ and $\eta$.
	Furthermore, the map $\fC_*^N(W,c,\eta):\mathfrak{C}_*^N (Y, \gamma)\to\mathfrak{C}_*^N(Y', \gamma')$ defined as below is a well-defined chain map:
	\begin{equation} \label{chain-map}
	  \fC_*^N(W,c,\eta)(\alpha):= \sum_{p} \#\mathcal{M}_p(W,c;\alpha,\eta,\beta) \cdot \beta.
	\end{equation}
	where the sum is over all paths $p$ whose restrictions to $(Y,\gamma)$, $(Y',\gamma')$ and $(L(p,q),\lambda)$ are respectively equal to $\alpha$, $\beta$ and $\eta$.
\end{prop}

\begin{proof}
	Since $(Y,\gamma)$ is $N$-admissible, $R(Y,\gamma)$ (defined by a small perturbation of the Chern-Simons functional) consists of irreducible connections.
	This allows us to choose the perturbation of the ASD equation on $(W,c)$
	such that all moduli spaces $\mathcal{M}_p(W,c;\alpha,\eta,\beta)$ are regular.
	(See \cite{KM:YAFT} or Subsection \ref{reg-irr}.)
	In particular, we can assume that all moduli spaces $\mathcal{M}_p(W,c;\alpha,\eta,\beta)$ are smooth manifolds.
	Note that Floer-Uhlenbeck compactness implies that $0$-dimensional moduli spaces $\mathcal{M}_p(W,c;\alpha,\eta,\beta)$ are compact.
	Therefore, \eqref{chain-map} is a well-defined map.
	
	Let $p$ be chosen such that this moduli space is $1$-dimensional.
	Then standard compactness and gluing theory results imply
	that this 1-manifold can be compactified
	and its boundary can be identified with the union of the following $0$-dimensional spaces:
	\begin{equation} \label{type-1}
          \bigcup_{\substack{\alpha'\in R(Y,\gamma)\\p=p_0\#p_1}}
          \breve{\mathcal M}_{p_0}(\R \times Y,\R \times \gamma;\alpha, \alpha')\times\mathcal{M}_{p_1}(W,c;\alpha',\eta,\beta),
	\end{equation}
	\begin{equation} \label{type-2}
          \bigcup_{\substack{\beta'\in R(Y',\gamma')\\p=p_0\#p_1}}
          \mathcal{M}_{p_0}(W,c;\alpha,\eta,\beta')\times \breve{\mathcal M}_{p_1}(\R \times Y',\R \times \gamma';\beta', \beta),
	\end{equation}
	\begin{equation} \label{type-3}
          \bigcup_{\substack{\eta'\in R(L(p,q),\lambda)\\p=p_0\#p_1}}
          \widetilde {\mathcal M}_{p_0}(W,c;\alpha,\eta',\beta)\times_{\Gamma_{\eta'}}
          \breve{\widetilde{\mathcal M}\hspace{1mm}}_{p_1,\eta'}(\R \times L(p,q),\R \times \lambda;\eta', \eta).
	\end{equation}
	Here $\breve{\widetilde{\mathcal M}\hspace{1mm}}_{p_1,\eta'}(\R \times L(p,q),\R \times \lambda;\eta', \eta)$ is the quotient of
	$\widetilde{\mathcal M}_{p_1,\eta'}(\R \times L(p,q),\R \times \lambda;\eta', \eta)$
	by the $\R$-action induced by translations. Since the product metric on $\R\times L(p,q)$
	is induced by the round metric on the lens space, the moduli space $\breve{\widetilde{\mathcal M}\hspace{1mm}}_{p_1,\eta'}(\R \times L(p,q),\R \times \lambda;\eta', \eta)$
	is regular. This space admits an action of $\Gamma_{\eta'}$ and
	the orbit of an element of $\breve{\widetilde{\mathcal M}\hspace{1mm}}_{p_1,\eta'}(\R \times L(p,q),\R \times \lambda;\eta', \eta)$
	represented by an ASD connection $A$ is $\Gamma_{\eta'}/\Gamma_A$.
	Since the connection $\eta'$ can be decomposed into a direct sum of 1-dimensional flat connections,
	the orbit $\Gamma_{\eta'}/\Gamma_A$ is at least 1-dimensional unless $p_1$ is the trivial path.
	The action of $\Gamma_{\eta'}$ on
	$ \widetilde {\mathcal M}_{p_0}(W,c;\alpha,\eta',\beta)$ is free because the elements of $\mathcal M_{p_0}(W,c;\alpha,\eta',\beta)$ are all irreducible. Therefore, one
	cannot form 0-dimensional spaces in \eqref{type-3}, which implies that \eqref{type-3} is empty. Counting the elements in \eqref{type-1} for all possible choices of $\beta$ gives rise
	to $\fC_*^N(W,c,\eta)\circ d(\alpha)$. Similarly, counting the elements in \eqref{type-2} for all possible choices of $\beta$ gives rise
	to $d\circ \fC_*^N(W,c,\eta)(\alpha)$. This implies that:
	\[
	  \fC_*^N(W,c,\eta)\circ d+d\circ \fC_*^N(W,c,\eta)=0.
	\]
\end{proof}
\begin{remark}
	Proposition \ref{chainmap} can be easily extended to the case that the middle end is a union of lens spaces. We can also adapt the proof of this proposition to show that
	the chain homotopy type of $\fC_*^N(W,c,\eta)$ is independent of the perturbations of the ASD equation and the chosen metric on $W$ as long as the metric on the lens
	space end is induced by the round metric. Similarly, we can show that the chain map $\rI_*^N$ associated to the composition of cobordisms with lens space middle ends is equal
	to the composition of the cobordism maps.
	In the special case that $L(p,q)=S^3$, we can fill the middle end of $(W,c)$ by a 4-ball to form a standard cobordism $(\overline W,\overline c)$.
	The maps $\fC_*^N(W, c,\Theta)$ and $\fC_*^N(\overline W,\overline c)$ are homotopic to each other where $\Theta$ is the trivial connection.
\end{remark}

The above proposition can be generalized to the case that a family of metrics on $W$ is fixed. Suppose $\bbW$ is a family of metrics on $W$ parametrized by an $N$-admissible polyhedron $K$. We will write $\pi$ for the projection map from $\bbW$ to $K$. Suppose a $2$-cycle $c$ is given as above such that $(W,c)$ is a cobordism of $N$-admissible pairs with the middle end $(L(p,q),\lambda)$ as above. Suppose also all connected components of $c$ are transversal to the cuts of the family of metrics $\bbW$. We assume $\bbW$ satisfies the following condition:
\begin{condition}\label{noS1S2cut}
	For any cut $Y$ of the family $\bbW$ either $(Y,Y\cap c)$ is $N$-admissible or $Y$ is a lens space.
\end{condition}
\vspace{-10pt}
\noindent
Fix $\eta\in R(L(p,q),\lambda)$. We can form the moduli spaces $\mathcal M_p(\bbW,c,\alpha,\eta,\beta)$ with the projection map ${\rm Pr}:\mathcal M_p(\bbW,c,\alpha,\eta,\beta) \to K$ for any $\alpha\in R(Y,\gamma)$, $\beta \in R(Y',\gamma')$ and path $p$ along $(W,c)$.

\begin{definition}
	We call a perturbation \emph{good} if all elements of $\mathcal M_p(\bbW,c,\alpha,\eta,\beta)$ are regular for any $\alpha\in R(Y,\gamma)$, $\beta \in R(Y',\gamma')$ and path $p$ along $(W,c)$
	such that $\ind(p)\leq 1- \dim(K)$.
\end{definition}

Suppose a good perturbation for the family of metrics $\bbW$ is fixed. Then Floer-Uhlenbeck compactness implies that the 0-dimensional moduli spaces are compact. Therefore, we can define a map $f_K$ from $\fC_\ast^N (Y, \gamma)$ to $\fC_\ast^N(Y',\gamma')$ in the following way:
\begin{equation} \label{map-family}
	f_K(\alpha)= \sum_{p} \#\mathcal M_{p}(\bbW,c;\alpha,\eta,\beta) \cdot \beta
\end{equation}
where the sum is over all paths $p$ whose restrictions to $(Y,\gamma)$, $(Y',\gamma')$ and $(L(p,q),\lambda)$ are respectively equal to $\alpha$, $\beta$ and $\eta$. The map $f_K$ is not necessarily a chain map anymore. However, its failure to be a chain map can be examined by the argument in the proof of Proposition \ref{chainmap}.

Let $p$ be a path along $(W,c)$ such that the moduli space $\mathcal M_p(\bbW,c)$ is 1-dimensional, i.e., $\ind(p)+\dim(K)=1$. Standard gluing theory and compactness results\footnote{See, for example, \cite[Chapters 4 and 5]{Don:YM-Floer}.} can be employed to show that the 1-dimensional manifold $\mathcal M_p(\bbW,c;\alpha,\eta,\beta)$ can be compactified by adding following $0$-manifolds:
\begin{equation} \label{type-1-family}
	\bigcup_{\substack{\alpha'\in R(Y,\gamma)\\p=p_0\#p_1}}
	\breve{\mathcal M}_{p_0}(\R \times Y,\R \times \gamma;\alpha, \alpha')\times\mathcal{M}_{p_1}(\bbW,c;\alpha',\eta,\beta),
	\end{equation}
\begin{equation} \label{type-2-family}
	\bigcup_{\substack{\beta'\in R(Y',\gamma')\\p=p_0\#p_1}}
	\mathcal{M}_{p_0}(\bbW,c;\alpha,\eta,\beta')\times \breve{\mathcal M}_{p_1}(\R \times Y',\R \times \gamma';\beta', \beta),
\end{equation}
\begin{equation} \label{type-3-family}
	\mathcal{M}_{p}(\partial \bbW,c;\alpha,\eta,\beta):=  \bigcup_{\overline F} \mathcal{M}_{p}(\pi^{-1}(F),c;\alpha,\eta,\beta)
\end{equation}
A priori, there is a contribution to the boundary similar to the spaces in \eqref{type-3}. However, the same argument as in the proof of of Proposition \ref{chainmap} shows that this contribution is empty. The union in \eqref{type-3-family} is over all codimension $1$ faces $\overline F$ of $K$. We can use the moduli spaces in \eqref{type-3-family} to define linear maps $f_{F}:\fC_\ast^N (Y, \gamma)\to\fC_\ast^N(Y',\gamma')$ as follows:
\begin{equation}\label{fboundaryK}
	f_{F}(\alpha):=\sum_{p}\#\mathcal M_{p}(\pi^{-1}(F),c;\alpha,\eta,\beta) \cdot \beta
\end{equation}
We also define $f_{\partial K}$ to be the sum of all maps $f_F$.

There is a special case that the map $f_F$ can be simplified further. Let $Y_F$ be the cut associated to $F$ and $\gamma_F:=c\cap Y_F$. We assume that the pair $(Y_F,\gamma_F)$ is $N$-admissible. Furthermore, removing a neighborhood of $Y_F$ from $W$ produces two 4-manifolds $W_0$, $W_1$, which give rise to the following cobordisms:
\begin{equation} \label{decom1}
  (W_0,c_0):(Y,\gamma)\xrightarrow{(L(p,q),\lambda)}(Y_F,\gamma_F)\hspace{1cm}(W_1,c_1):(Y_F,\gamma_F)\xrightarrow{\hspace{1.2cm}}(Y',\gamma')
\end{equation}
where $c_i=W_i\cap c$. The restriction of the family of metrics $\bbW$ to the face $F$ induces families of metrics $\bbW_0$, $\bbW_1$ on $W_0$, $W_1$ parametrized by $K_0$, $K_1$. By definition we have:
\begin{equation} \label{end-face}
  \mathcal M_{p}(\pi^{-1}(F),c;\alpha,\eta,\beta)=\bigcup_{\substack{\xi\in R(Y_F,\gamma_F)\\ p=p_0\# p_1}}\mathcal M_{p_0} (\bbW_0,c_0;\alpha,\eta,\xi) \times
	\mathcal M_{p_1} (\bbW_1,c_1;\xi,\beta)
\end{equation}
The families $\bbW_0$ and $\bbW_1$ give rise to maps:
\[
  f_{K_0}:\fC_*^N(Y,\gamma) \to \fC_*^N(Y_F,\gamma_F)\hspace{1cm}f_{K_1}:\fC_*^N(Y_F,\gamma_F) \to \fC_*^N(Y',\gamma')
\]
Moreover, \eqref{end-face} for all choices of $\beta$ implies that $f_F=f_{K_1}\circ f_{K_0}$. We can obtain a similar description for $f_F$ if removing a neighborhood of $Y_F$ from $W$ produces the following cobordisms:
\begin{equation} \label{decom2}
  (W_0,c_0):(Y,\gamma)\xrightarrow{\hspace{1.2cm}}(Y_F,\gamma_F)\hspace{1cm}(W_1,c_1):(Y_F,\gamma_F)\xrightarrow{(L(p,q),\lambda)}(Y',\gamma')
\end{equation}
If a cut $Y_F$ induces a decomposition as in \eqref{decom1} or \eqref{decom2}, then we say that $Y_F$ is an {\it upright} cut. This discussion can be easily extended to the case that the middle end is a union of lens spaces.
\begin{example}
	The vertical cuts of the family of metrics $\bbW_k^j$ are upright.
\end{example}

Our discussion about the map $f_K$ can be summarized as follows:
\begin{prop}\label{hom-form}
	The map $f_K$ satisfies the following homotopy relation:
	\begin{equation}\label{Gmap}
		d\circ f_K +f_K\circ d= f_{\partial K}.
	\end{equation}
	The homomorphism $f_{\partial K}$ is the sum of the maps $f_F$. Moreover, if $Y_F$ is an upright cut, then $f_F=f_{K_1}\circ f_{K_0}$ where the homomorphisms $f_{K_1}$
	and $f_{K_0}$ are defined by the restriction of $\bbW$ to the face $F$ as above.
\end{prop}

The Floer homology group $\text{I}_\ast^N(Y,\gamma)$ admits a relative $\mathbf{Z}\slash 4N$-grading denoted by $\deg$. To define this $\Z/4N\Z$-grading, let $\alpha_0$ and $\alpha_1$ be two generators of $\fC_*^N(Y,\gamma)$. We consider an arbitrary path on $([0,1]\times Y,[0,1]\times \gamma)$ such that the restriction of $p$ to $\{i\}\times Y$ is equal to $\alpha_i$ for $i=0,\,1$. If $A$ is a connection representing $p$ then $\deg(\alpha_0)-\deg(\alpha_1)=\ind(\mathcal D_A)$. The relative grading $\deg$ can be lifted to an absolute $\Z/2\Z$-grading following the method used in \cite[Section 25.4]{KM:monopoles-3-man}:

\begin{prop}
	There is a $\mathbf{Z}\slash 2$-absolute grading on $\I_\ast^N$ which lifts $\deg$ and is uniquely characterized by the following properties:
	\vspace{-5pt}
	\begin{enumerate}
	  \item[(i)] The degree of a map induced by a cobordism $(W,c):(Y,\gamma)\to (Y',\gamma')$ is determined by the parity of:
	       \begin{equation*}
        		\iota(W)=\frac{N^2-1}{2}(\chi(W)+\sigma(W)+b_0(Y')+b_1(Y')-b_0(Y)-b_1(Y))
         \end{equation*}
	  \item[(ii)] The grading is normalized such that the generator of $\I_\ast^N(\emptyset):=\Z/2\Z$ has degree $0$.
	 \end{enumerate}
\end{prop}
\begin{remark}
	In the case that $N$ is odd, the homomorphisms induced by cobordisms $(W,c):(Y,\gamma)\to (Y',\gamma')$ have always degree $0$
	with respect to the absolute $\Z/2\Z$-grading.
	When $N$ is odd, the absolute grading can be also lifted to a $\mathbf{Z}\slash 4$.
\end{remark}

\subsection{Statement of the Main Theorem} \label{statement}
As in Section \ref{topology}, let $\mathcal Y$ be a 3-manifold with torus boundary and $\gamma$ be a 1-cycle in the interior of $\mathcal Y$. For any subset $S\subset [N]$ with $j$ elements, we defined a pair $(Y_j,\gamma_S)$. We assume that $(\mathcal Y,\gamma)$ are chosen such that all pairs $(Y_j,\gamma_S)$ are $N$-admissible. For example, it suffices to have a closed oriented embedded surface $\Sigma$ in $\mathcal Y$ such that $\Sigma\cdot \gamma$ is coprime to $N$. Therefore, this construction associates an $N$-admissible pair to each vertex of the directed graph $G_N$, defined in Subsection \ref{n-cube}.  If $S\subset [N]$ and $i \notin S$, then there is an edge $\beta_{S,i}$ in $G_N$ connecting $S$ to $S\sqcup\{i\}$. In Section \ref{topology}, we assigned the following cobordism to this edge:
\begin{equation}\label{edge-cob}
  (W^j_{j+1}, c_{S,i}): (Y_j,\gamma_S) \xrightarrow{(M^{j}_{j+1},\,l_{S,i})} (Y_{j+1},\gamma_{S'}) \hspace{1cm}
\end{equation}
There is also one edge in $G_N$ from $[N]$ to $\emptyset$. The corresponding cobordism of pairs for this edge is the following:
\begin{equation}\label{conn-cob}
  (W^N_{N+1},c_\delta):(Y_N, \gamma_{[N]})\xrightarrow{(M^{N}_{N+1},\,l_{\delta})} (Y_0, \gamma_{\emptyset})
\end{equation}

We can also construct a cobordism for each path $q$ in $G_N$ by composing the cobordisms associated to the edges of $q$. For a set $S\subset [N]$ with $j$ elements and an injection map $\sigma:[k]\to [N]\backslash S$, let $q$ be the path of length $k$ assigned to the pair $(S,\sigma)$. The cobordism associated to this path is:
\begin{equation} \label{cob-path-1}
  (W^{j}_{j+k},c_{S,\sigma}):(Y_{j}, \gamma_S)\xrightarrow{\coprod_{0\le i\le k-1}(M^{j+i}_{j+i+1},\,l_{S\sqcup\sigma([i]),\sigma(i)})} (Y_{j+k}, \gamma_{S\sqcup \sigma([k])} )
\end{equation}
We also labeled paths $q$ containing the edge $\delta$ with a pair of injection maps $\sigma:[k]\to [N]$ and $\tau:[l]\to [N]$ whose images are disjoint. This path starts from $S=[N]\backslash\Im \sigma$, ends at $T=\Im \tau$ and has length $k+l+1$. The cobordism assigned to this path is equal to:
\begin{equation} \label{cob-path-2}
  (W^{N-k} _{N+l+1}, c_{\sigma,\tau}): (Y_{N-k}, \gamma_{S}) \xrightarrow{\coprod_{0\le i\le k+l}(M^{N-k+i}_{N-k+i+1},\,l_{i})} (Y_{l}, \gamma_{T})
\end{equation}
where
\[
  l_i=
  \left\{
  \begin{array}{ll}
    l_{S\sqcup \sigma([i]),\sigma(i)}& ~\text{when}~ 0\leq i\leq k-1; \\
    l_\delta & ~\text{when}~i=k; \\
    l_{\tau([i-k-1]),\tau(i-k-1)} & ~\text{when}~k+1\leq i\leq k+l.
  \end{array}
  \right.
\]

Previously we fixed Riemannian metrics on 3-manifolds $Y_j$ and families of metrics $\bbW_k^j$ on cobordisms $W_k^j$ compatible with the Riemannian metrics on the 3-manifolds $Y_j$. We fix a small perturbation of the Chern-Simons functional of the pair $(Y_j,\gamma_S)$ for each set $S$ with $j$ elements such that the Floer chain complex $\mathfrak{C}_\ast^N(Y_j,\gamma_S)$ is well-defined.
\begin{theorem} \label{main}
	For each path $q$ in $G_N$ from a vertex $S\subset [N]$ to another vertex $T\subset [N]$ with length at most $N+1$,
	there is a map $f_q:\fC_*^N(Y_{|S|}, \gamma_S) \to \fC_*^N(Y_{|T|}, \gamma_T)$ such that
	$(\{\mathfrak{C}^N_\ast(Y_{|S|}, \gamma_{S})\}, \{f_q\})$ forms an exact $N$-cube.
	Let $\nu$ denote a generator of the kernel of the map $H_1(\partial \mathcal Y) \to H_1(\mathcal Y)$.
	Let $S\subset [N]$ with $j$ elements, $\sigma:[k]\to [N]\backslash S$ and $q$ be the path associated to the pair $(S,\sigma)$.
	If $N$ is odd, then the degree of the map $f_q$ with respect to the $\Z/2\Z$-grading is equal to $k-1$.
	If $N$ is odd, then the degree of the map $f_q$ with respect to the $\Z/2\Z$-grading is equal to $k-1+\sum_{m=j+1}^{k+j}\epsilon_m'$ where
	$\epsilon_m'$ is defined in \eqref{epsilon-i}.
\end{theorem}
Note that Corollary \ref{Euler-char} is a consequence of the second part of the above theorem and Corollary \ref{euler-char-exact-polygon}.

\subsection{Proof of the Main Theorem}\label{proof-thm}
In this subsection, we use the families of metrics $(\bbW^{j}_{j+k},c_{S,\sigma})$ and $(\bbW^{N-k} _{N+l+1}, c_{\sigma,\tau})$ on the pairs \eqref{cob-path-1} and \eqref{cob-path-2} to define maps $f_q$ as in Section \ref{exact-polygon}.
 For any two subsets $S$, $T$ of $[N]$, we can then define $f^S_T$ to be the sum of all maps $f_q$ where $q$ is a path from $S$ to $T$ in $G_N$. We shall also check the following identities, which verify Theorem \ref{main}:
\begin{align}
   df^S_T+f^S_Td&=\sum_{ S\subsetneq R\subsetneq T} f^R_T f^S_R,  ~~~~~~~\text{if}~ S \subsetneq T;  \label{eq1}\\
   df^S_T+f^S_Td&=\sum_{\substack{S\subsetneq R \\ \text{or}~R\subsetneq T}}f^R_T f^S_R , ~~~~~~~~\text{if}~|T|\le |S|~\text{and}~ S\neq T; \label{eq2}\\
  df^S_T+f^S_Td&=\text{1}+\sum_{\substack{ S\subsetneq R \\ \text{or}~R\subsetneq T}}f^R_T f^S_R , ~\text{if}~S=T   \label{eq3}
\end{align}

\subsubsection*{Maps $f_q$ for paths of length at most $N$}\label{A8vanishing}

Let $S$ and $T$ be subsets of $[N]$ such that $S\subsetneq T$. Let $|S|=j$, $|T|=k$. Let also $\sigma:[k-j]\to T\setminus S$ be a bijection. Then the cobordism of pairs $(W^{j}_{k},c_{S,\sigma})$, defined in \eqref{cob-path-1}, satisfies Condition \ref{noS1S2}. Moreover, the pair $(\bbW^{j}_{k},c_{S,\sigma})$ satisfies Condition \ref{noS1S2cut}. All middle ends of the cobordism $W^{j}_{k}$ are 3-dimensional spheres. To any such end, we associate the trivial $\SU(N)$-connection and we will write $\eta$ for this choice of flat connections on the middle ends. In the next section, we shall show that $(\bbW^{j}_{k},c_{S,\sigma},\eta)$ admits a good perturbation. Therefore, we can construct a map using the triple $(\bbW^{j}_{k},c_{S,\sigma},\eta)$ as in \eqref{map-family}. We will write $f_{S,\sigma}$ for this map.

Let $\overline W^{j}_{k}$ denote the result of gluing $D^4$ to $W^{j}_{k}$ along the middle boundary components. Since the parametrizing polyhedron of $\bbW^{j}_{k}$ is the associahedron $\mathcal K_{k-j+1}$, the degree of the map $f_{S,\sigma}$ is equal to $\iota(\overline W^{j}_{k})+\dim(\mathcal K_{k-j+1})$. Thus, the degree of this map is equal to:
\begin{equation} \label{degree-Z/2-ele}
  k-j-1+(N^2-1)\sum_{l=j+1}^{k}\frac{1}{2}\(\chi(\overline W^{l-1}_{l})+\sigma(\overline W^{l-1}_{l})+b_1(Y_{l-1})-b_1(Y_l)\)
\end{equation}
In particular, if $N$ is odd, then \eqref{degree-Z/2-ele} is equal to $k-j-1$. In general, $\overline W^{l-1}_{l}:Y_{l-1}\to Y_{l}$ is an elementary cobordism and the term $\iota(\overline W^{l-1}_{l})$ in the above sum is equal to $\epsilon_l'$ \cite[Subsection 42.3]{KM:monopoles-3-man}.

For each open face $F$ of $\mathcal K_{k-j+1}$ with codimension one, we can follow \eqref{fboundaryK} to define a map:
\[
  f_{F}:\fC_*^N(Y_{|S|}, \gamma_S) \to \fC_*^N(Y_{|T|}, \gamma_T).
\]
Firstly let $Y_F$, the cut associated to the face $F$, be vertical. There is $l$ such that $j<l<k$ and the restriction of the family of metrics $\bbW^j_{k}$ to the face $F$ is given by the families of metrics $\bbW^j_{l}$, $\bbW^{l}_{k}$ on the cobordisms $W^j_{l}$, $W^{l}_{k}$. Thus Proposition \ref{fboundaryK} asserts that $f_F$ is equal to the composition $f_{S',\sigma_2} \circ f_{S,\sigma_1}$ where $\sigma_1:[l-j] \to [N]$, $\sigma_2:[k-l] \to [N]$ and $S'$ are defined as below:
\[
  \sigma_1(i):=\sigma(i)\hspace{1cm} \sigma_2(i):=\sigma(i+l-j) \hspace{1cm} S'=S\sqcup \sigma_1.
\]

Next, let $Y_F$ be the spherical cut $M^{j'}_{k'}$ where $j\le j'<k'\le k$. Removing a regular neighborhood of $Y_F$ has two connected components $W_0$ and $W_1$. We assume that the indices are chosen such that $Y_k$ is one of the boundary components of $W_0$. Thus $W_1$ is a Gibbons-Hawking manifold diffeomorphic to $\fX_{\bm}$ where $\bm=(j',j'+1,\cdots,k')$. Suppose the restriction of the family of metrics $\bbW^{j}_{k}$ to the face $F$ is given by families of metrics $\bbW_0$, $\bbW_1$ on $W_0$, $W_1$ respectively parametrized by the associahedra of dimensions $(k-k')-(j-j')$, $k'-j'-2$. Let also $c_0$ and $c_1$ denote the restrictions of $c_{S,\sigma}$ to $W_0$ and $W_1$.
By the definition in \eqref{fboundaryK}:
\[
  f_{F}(\alpha)=\sum_{\beta, p_0,p_1,\eta'}
    \# \left(\widetilde{\mathcal{M}}_{p_0}(\bbW_{0},c_0; \alpha,\eta',\beta) \times_{\Gamma_{\eta'}}
   \widetilde{\mathcal{M}}_{p_1} (\bbW_1,c_1;\eta')\right) \cdot \beta
\]
where $p_0$ is a path along $(W_0,c_0)$ whose restriction to the boundary components are $\alpha$, $\eta'$ and $\beta$, and $p_1$ is a path along $(W_1,c_1)$ whose restriction to the boundary components are given by $\eta'$. Note that in our notation we did not specify the flat connections on the boundary components of $W_0$ and $W_1$ diffeomorphic to $S^3$. Of course, these flat connections are necessarily trivial connections. The moduli space:
\begin{equation}
  \widetilde{\mathcal{M}}_{p_0}(\bbW_{0},c_0; \alpha,\eta',\beta) \times_{\Gamma_{\eta'}}\widetilde{\mathcal{M}}_{p_1} (\bbW_1,c_1;\eta')
\end{equation}
has contribution in the definition of the map $f_{F}$ if it is 0-dimensional. Moreover, $\widetilde{\mathcal{M}}_{p_0}(\bbW_{0},c_0; \alpha,\eta',\beta)$ contains only irreducible connections because $(W_0,c_0) $ has two $N$-admissible ends. Therefore, the regularity assumption implies that $\widetilde{\mathcal{M}}_{p_1} (\bbW_1,c_1;\eta',\eta)$ is $0$-dimensional. On the other hand, Corollary \ref{regular-GH} asserts that the dimension of this moduli space is at least $k'-j'-2$. Therefore, $f_{F}$ vanishes unless $k'=j'+2$.

If $k'=j'+2$, then we define $\overline \sigma$:
\[
  \overline \sigma(i)=\left\{
  \begin{array}{ll}
  	\sigma(j'-j+1)&i=j'-j \\
	\sigma(j'-j)  &i=j'-j+1 \\
	\sigma(i)&\text{otherwise}
  \end{array}
  \right.
\]
The pairs $(W^{j}_{k},c_{S,\sigma})$ and $(W^{j}_{k},c_{S,\overline \sigma})$ have the same $N$-admissible ends. If we define $\overline c_0=c_{S,\overline \sigma}\cap W_0$ and $\overline c_1=c_{S,\overline \sigma}\cap W_1$, then $\overline c_0=c_0$. Moreover, Proposition \ref{dual} implies that there is a path $\overline p_1$ on $(W_1,\overline c_1)$ such that $\widetilde{\mathcal{M}}_{p_1} (\bbW_1,c_1;\eta')=\widetilde{\mathcal{M}}_{\overline p_1} (\bbW_1,\overline c_1;\eta')$. Therefore, we have:
\begin{equation}\label{mod-space-boundary}
  \widetilde{\mathcal{M}}_{p_0}(\bbW_{0},c_0; \alpha,\eta',\beta) \times_{\Gamma_{\eta'}}\widetilde{\mathcal{M}}_{p_1} (\bbW_1,c_1;\eta')=
  \widetilde{\mathcal{M}}_{p_0}(\bbW_{0},c_0; \alpha,\eta',\beta) \times_{\Gamma_{\eta'}}\widetilde{\mathcal{M}}_{\overline p_1} (\bbW_1,\overline c_1;\eta')
\end{equation}
Associated to the cut $Y_F$ and the bijections $\sigma$ and $\overline \sigma$, we have two maps $f_{F}$, which are equal to each other by \eqref{mod-space-boundary}. This observation and Proposition \ref{hom-form} shows that the map $f_T^S$, the sum of the maps $f_{S,\sigma}$ for all bijections $\sigma:[k-j] \to T\setminus S$, satisfies \eqref{eq1}.

Next, let $S$ and $T$ be subsets of $[N]$ such that $|S|>|T|$. Let $|S|=j$, $|T|=k$. Let $\sigma:[N-j]\to [N]\backslash S$ and $\tau:[k] \to T$ be two bijections. Then the pair $(\sigma,\tau)$ defines a path of length $N+1-j+k$ from $\sigma$ to $\tau$. The pair $(W^j_{k+N+1},c_{\sigma,\tau})$ from \eqref{cob-path-2} is the cobordism of pairs associated to the pair $(\sigma,\tau)$. This cobordism is a cobordism from $(Y_j,\gamma_S)$ to $(Y_k,\gamma_T)$ with a middle end. The middle end of this cobordism is the union of $L(N,1)$ and $N-j+k$ copies of $S^3$. Let $\eta$ on the middle end of this cobordism be given by the connection $\text{1}\oplus \zeta \oplus \cdots \oplus \zeta^{N-1}$ on $L(N,1)$ and the trivial connections on the remaining boundary components. Following \eqref{map-family}, we can use the triple $(\bbW^j_{k+N+1},c_{\sigma,\tau},\eta)$ to define a map:
\[
  f_{\sigma,\tau}:\fC_*^N(Y_{|S|}, \gamma_S) \to \fC_*^N(Y_{|T|}, \gamma_T).
\]
We can follow the same strategy as above to show that $f^S_T$, the sum of the maps $f_{\sigma,\tau}$ over all bijections $\sigma:[N-j]\to [N]\backslash S$ and $\tau:[k] \to T$, satisfies \eqref{eq2}. The only difference is that we need to replace part (i) of Proposition \ref{dual} with part (ii) of the same proposition.

\subsubsection*{Maps $f_q$ for paths of length $N+1$}\label{A8identity}

Let $S$ and $T$ be subsets of $[N]$ such that $|S|=|T|=l$. Let $\sigma:[N-l]\to [N]\backslash S$ and $\tau:[l]\to T$ be two bijections. Then $(\sigma,\tau)$ determines a path of length $N+1$ in $G_N$. The pair associated to this path is the cobordism $(W^l_{l+N+1},c_{\sigma,\tau})$ from $(Y_l,\gamma_S)$ to $(Y_l,\gamma_T)$. The middle end of this cobordism is the union of $L(N,1)$ and $N$ spheres. Let $\eta$ be the connection on the middle end of $(W^l_{l+N+1},c_{\sigma,\tau})$ induced by the connection $\text{1}\oplus \zeta \oplus \cdots \oplus \zeta^{N-1}$ on $L(N,1)$ and trivial connections on the remaining boundary components. We want to define a map $f_{\sigma,\tau}$ using the 0-dimensional moduli spaces associated to the triple $(\bbW^l_{l+N+1},c_{\sigma,\tau},\eta)$. We also wish to study the map $df_{\sigma,\tau}+f_{\sigma,\tau}d$ using the 1-dimensional moduli spaces associated to $(\bbW^l_{l+N+1},c_{\sigma,\tau},\eta)$. However, we cannot use the general results of Subsection \ref{inst-floer} because the family $\bbW^l_{l+N+1}$ has a cut which is diffeomorphic to $S^1\times S^2$. For the rest of this section, we also need to assume that $N\leq 4$.

Let $p$ be a path along  $(W^l_{l+N+1},c_{\sigma,\tau})$ such that $\ind(p)$ is at most $-N+1$. Let also $\pi$ denote the projection map from $\bbW^l_{l+N+1}$ to $\mathcal K_{N+2}$. As in Subsection \ref{cyl-end-mod}, we define a moduli space $\mathcal M_p(\bbW^l_{l+N+1},c_{\sigma,\tau})$ of dimension $\ind(p)+N$ together with a projection map ${\rm Pr}:\mathcal M_p(\bbW^l_{l+N+1},c_{\sigma,\tau}) \to \mathcal{K}_{N+2}$.
Firstly let $\overline F_0$ be the face of $\mathcal K_{N+2}$ whose associated cut is $S^1\times S^2$. Removing a neighborhood of $S^1\times S^2$ produces a 4-manifold with two connected components. One of these connected components is diffeomorphic to $X_N(l)$. We denote the other connected component by $W_0$. The restriction of $c_{\sigma,\tau}$ to $X_N(l)$ is equal to $w_{\sigma,\tau}$, which is defined in Section \ref{comp-red}. Let also $c_0$ denote $c_{\sigma,\tau}\cap W_0$. The topology of the pair $(W_0,c_0)$ can be described as follows. Remove a regular neighborhood of the knot $\{\frac{1}{2}\}\times K_l$ in the 4-manifold $[0,1]\times Y_l$. (Recall that the knot $K_j$ is the core of the Dehn filling torus.) The resulting manifold is a cobordism form $Y_l$ to $Y_l$ with middle end $S^1\times S^2$ and can be identified with $W_0$. We also have:
\[
  c_0=(\gamma_S\times [0,\frac{1}{2}]\cup\gamma_T\times [\frac{1}{2},1])\cap W_0
\]
The family of metrics parametrized by the face $\overline F_0$ is given by a fixed metric on $W_0$ and the family of metric $\bbX_{N}(l)$ on $X_{N}(l)$. In particular, the face $\overline F_0$ is equal to $\mathcal K_{N+1}$.

\begin{remark} \label{decomp-product}
	In the case that $S=T$, the above description implies that there is a cut in $[0,1]\times Y_l$
	 diffeomorphic to $S^1\times S^2$ which decomposes the pair
	 $([0,1]\times Y_l,[0,1]\times \gamma_S)$ into the following pairs:
	 \[
	   (W_0,c_0)\hspace{1cm} (S^1\times D^3,c_1').
	 \]
	 Here $c_1'$ is a multiple of a the cylinder $S^1\times \gamma$ where $\gamma$ is a path connecting
	 two points in the boundary of the 3-ball $D^3$.
\end{remark}

We define:
\begin{equation} \label{mod-over-Fbar}
  \mathcal M_p(\pi^{-1}(\overline F_0),c_{\sigma,\tau}):=
  \bigcup_{\substack{p_0\#p_1=p\\ \alpha\in \Delta_{N-1}^{\ft}}}
  \widetilde {\mathcal M}_{p_0,\alpha}(W_0,c_0;\alpha)\times_{\Gamma_{\alpha}}
  \widetilde {\mathcal M}_{p_1,\alpha}(\bbX_N(l),w_{\sigma,\tau};\alpha)
\end{equation}
An element $[A_0,A_1]$ of this moduli space is regular if it satisfies the following conditions:
\begin{itemize}
  \item[(i)] The common limiting flat connection of $A_0$ and $A_1$ on $S^1\times S^2$ lies in
  		$\Delta_{N-1}^{\ft,\circ}$.
  \item[(ii)] $A_0$ represents a regular element of $\mathcal M_{p_0}(W_0,c_0;\Delta_{N-1}^{\ft,\circ})$ and
  		$A_1$ represents a regular element of
		$\mathcal M_{p_1}(\bbX_N(l),w_{\sigma,\tau};\Delta_{N-1}^{\ft,\circ})$.
  \item[(iii)]	 The maps $r_0:\widetilde{\mathcal M}_{p_0}(W_0,c_0;\Delta_{N-1}^{\ft,\circ})\to \Delta_{N-1}^{\ft,\circ}$
  		in a neighborhood of $[A_0]$ is transversal to
		$r_1: \widetilde{\mathcal M}_{p_1}(\bbX_N(l),w_{\sigma,\tau};\Delta_{N-1}^{\ft,\circ}) \to
		\Delta_{N-1}^{\ft,\circ}$ in a neighborhood of $[A_1]$.
\end{itemize}

Let $\overline F'$ be any other face of $\mathcal K_{N+2}$ such that $F'$ and $\overline F_0$ are disjoint. We can proceed as in Subsection \ref{cyl-end-mod} to define $\mathcal M_p(\pi^{-1}(F'),c_{\sigma,\tau})$ and regularity of connections in this subspace of $\mathcal M_p(\bbW^l_{l+N+1},c_{\sigma,\tau})$. The disjoint union of \eqref{mod-over-Fbar} and the spaces of the form $\mathcal M_p(\pi^{-1}(F'),c_{\sigma,\tau})$ defines $\mathcal M_p(\bbW^l_{l+N+1},c_{\sigma,\tau})$ as a set. A topology on this set is also defined in a standard way.

In Subsection \ref{good-perb}, it will be shown that we can use a perturbation in the definition of the ASD equation such that all elements of the moduli spaces $\mathcal M_p(\bbW^l_{l+N+1},c_{\sigma,\tau})$ are regular in the case that $\ind(p)\leq -N+1$. More specifically, we shall show that this perturbation can be chosen such that the moduli space $\mathcal M_p(\bbW^l_{l+N+1},c_{\sigma,\tau})$ is empty if $\ind(p)\leq -N-1$. In the case that $\ind(p)=-N$, the moduli space is a compact 0-manifold which is mapped to the interior of $\mathcal K_{N+2}$ by $\Pr$. If $\ind(p)=-N+1$, then the moduli space $\mathcal M_p(\bbW^l_{l+N+1},c_{\sigma,\tau})$ is a smooth 1-manifold whose boundary is equal to:
\begin{equation}\label{moduli-space-face}
	\bigsqcup_{ F} \mathcal M_p(\pi^{-1}(F),c_{\sigma,\tau})
\end{equation}
where the sum is over all codimension one faces of $\mathcal K_{N+2}$. For the face $F_0$ corresponding to the cut $S^1\times S^2$, the subspace $\mathcal M_p(\pi^{-1}(F_0),c_{\sigma,\tau})$ is equal to the fiber product:
\begin{equation}\label{mod-M0}
 \mathcal M_{p_0}(W_0,c_0;\Delta_{N-1}^{\ft,\circ})\hspace{1mm}   _{r_0}\hspace{-1mm}\times_{r_1} \mathcal M_{cr}(\bbX_N(l),w_{\sigma,\tau})
\end{equation}
where $\mathcal M_{cr}(\bbX_N(l),w_{\sigma,\tau})$ is the moduli space of (perturbed) ASD completely reducible connections with respect to the path $p_1$ with $\ind(p_1)=-2(N-1)$. The perturbation in the definition of $\mathcal M_{cr}(\bbX_N(l),w_{\sigma,\tau})$ can be made arbitrarily small. The path $p_0$ is also chosen such that it has the same limiting flat connection as $p$ on $(Y_l,\gamma_S)$, $(Y_l,\gamma_T)$ and its index is equal to $-(N-1)$. In particular, Proposition \ref{bi-per-con} implies that the space in \eqref{mod-M0} is empty unless the images of $S$ and $T$ are disjoint. The moduli space $\mathcal M_{p_0}(W_0,c_0;\Delta_{N-1}^{\ft,\circ})$ is a compact $0$-dimensional manifold consisting of regular connections. The maps:
\[
  r_0:\mathcal M_{p_0}(W_0,c_0;\Delta_{N-1}^{\ft,\circ}) \to \Delta_{N-1}^{\ft,\circ} \hspace{1cm}r_1:\mathcal M_{cr}(\bbX_N(l),w_{\sigma,\tau}) \to \Delta_{N-1}^{\ft}
\]
are given by restrictions of connections to $S^1\times S^2$. We cannot guarantee that $\mathcal M_{cr}(\bbX_N(l),w_{\sigma,\tau})$ consists of only regular solutions. Nevertheless, any point in $r_1^{-1}({\rm image}(r_0))$ is regular.

The moduli space $\mathcal M_p(\bbW^l_{l+N+1},c_{\sigma,\tau})$ can be compactified into a compact 1-manifold by adding the points in the following space:
\begin{equation}\label{end-type-1}
	\bigsqcup_{p_0\#p_1=p}\breve {\mathcal M}_{p_0}(Y_l\times \R,\gamma_S\times \R)\times
	\mathcal M_{p_1}(\bbW^l_{l+N+1},c_{\sigma,\tau})
\end{equation}
and:
\begin{equation}\label{end-type-2}
	\bigsqcup_{p_0\#p_1=p} \mathcal M_{p_0}(\bbW^l_{l+N+1},c_{\sigma,\tau})
	\times \breve {\mathcal M}_{p_1}(Y_l\times \R,\gamma_S\times \R).
\end{equation}
In \eqref{end-type-1}, $\ind(p_0)=-1$ and $\ind(p_1)=-N$. A similar assumption holds for the paths $p_0$ and $p_1$ in \eqref{end-type-2}.

Let $\tilde f_{\sigma,\tau}:\fC_*^N(Y_l,\gamma_S) \to \fC_*^N(Y_l,\gamma_T)$ be a map defined as:
\[
  \tilde f_{\sigma,\tau}(\alpha)=
  \sum_{p}\#\mathcal M_p(\bbW^l_{l+N+1},c_{\sigma,\tau},\alpha,\eta,\beta)\cdot \beta
\]
where the sum is over all paths $p$ such that $\ind(\mathcal D_p)=-N$ and the restriction of $p$ to the incoming, the outgoing and the middle ends of $W^l_{l+N+1}$ are $\alpha$, $\beta$ and $\eta=1\oplus \zeta \oplus \cdots \oplus \zeta^{N-1}$.
Moreover, the description of the boundary components in the compactification of 1-dimensional moduli spaces gives rise to the following analogue of Proposition \ref{hom-form}:
\begin{equation}\label{hom-form-XNl}
  d\tilde f_{\sigma,\tau}+\tilde f_{\sigma,\tau}d=\sum_{F}f_F
\end{equation}
where the sum on the left hand side of the above identity is over all codimension one faces of $\mathcal K_{N+2}$. The map $f_F$ is defined as in \eqref{fboundaryK} by counting the points in the spaces that appear in \eqref{moduli-space-face}. In particular, one of the terms in \eqref{hom-form-XNl} is $f_{F_0} $ in correspondence with the face $F_0$. To emphasize that this map depends on the choice of $\sigma$ and $\tau$, we will use the alternative notation $f^\partial_{\sigma,\tau}$ to denote this map.

Let $\tilde f^S_T$ be the sum of all maps $\tilde f_{\sigma,\tau}$. By arguing as in the previous part, we can show that:
\[
  d\tilde f^S_T+\tilde f^S_Td=\sum_{\substack{ S\subsetneq R \\ \text{or}~R\subsetneq T}}f^R_T f^S_R+
  \sum_{\sigma,\tau}f^\partial_{\sigma,\tau}
\]
where the second sum in the left hand side of the above identity is over all bijections $\sigma:[N-l]\to [N]\setminus S$ and $\tau:[l]\to T$.

\begin{prop}\label{homotopy-to-identity}
	Let $S$ and $T$ be as above. If $S\neq T$, then the map $f^\partial_{\sigma,\tau}$ vanishes.
	If $S=T$, then there is a map $g_S:\fC^N_*(Y_l,\gamma_S)\to \fC^N_*(Y_l,\gamma_S)$
	such that:
	\[
	  dg_S+g_Sd=1-\sum_{\sigma,\tau}f^\partial_{\sigma,\tau}
	\]
	where the sum is over all bijections $\sigma:[N-l]\to [N]\setminus S$ and $\tau:[l]\to S$.
\end{prop}
This lemma completes the proof of the main theorem. If $S\neq T$, then we define $f_{\sigma,\tau}$ to be $\tilde f_{\sigma,\tau}$. In the case that $S=T$, we fix bijections $\sigma_0:[N-l]\to [N]\setminus S$ and $\tau_0:[l]\to S$ and define $f_{\sigma_0,\tau_0}$ to be $g_S+\tilde f_{\sigma_0,\tau_0}$. For the remaining choices of bijections $\sigma:[N-l]\to [N]\setminus S$ and $\tau:[l]\to S$ we define $f_{\sigma,\tau}$ to be $\tilde f_{\sigma,\tau}$.

\begin{proof}[Proof of Proposition \ref{homotopy-to-identity}]
	Let $\alpha$ be a generator of $\fC_*^N(Y_l,\gamma_S)$ and $\beta$ be a generator of $\fC_*^N(Y_l,\gamma_T)$.
	The coefficient of $\beta$ in $f^\partial_{\sigma,\tau}(\alpha)$ is given by counting the elements of the finite set in \eqref{mod-M0}.
	As it is explained above, this space is empty in the case that $S\neq T$.
	Assume $S=T$, and let $M(\alpha,\beta)$ denote the factor $\mathcal M_{p_0}(W_0,c_0;\Delta_{N-1}^{\ft,\circ})$ in \eqref{mod-M0}. The limiting value of $p_0$
	on the two admissible ends are $\alpha$, $\beta$. It is shown in Subsection \ref{good-perb} that
	we can assume that the map $r_0:M(\alpha,\beta)\to \Delta_{N-1}^{\ft,\circ}$ is transversal to the maps $\hol_S$ and $H_S$ from Subsection \ref{comp-red}.
	This assumption shows that if the perturbation to form the moduli space $\mathcal M_{cr}(\bbX_N(l),w_{\sigma,\tau})$ is small enough, then the number of
	points in \eqref{mod-M0} does not change if we replace the perturbation in the definition of $\mathcal M_{cr}(\bbX_N(l),w_{\sigma,\tau})$ with the zero perturbation.
	Therefore, we have:
\begin{align}
		\sum_{\sigma,\tau} f^\partial_{\sigma,\tau}(\alpha)
		=& \#\left (  M(\alpha,\beta)  \hspace{1mm}
		_{r_0}\hspace{-1mm}\times_{\hol_{S}}\Delta_{N-1}\right)\cdot \beta\nonumber\\
		=& \#\left (  M(\alpha,\beta) \hspace{1mm}
		_{r_0}\hspace{-1mm}\times_{H_{S}(0,\cdot)}\Delta_{N-1}\right)\cdot \beta \nonumber \\
		=& \#  M(\alpha,\beta) \cdot \beta
		\label{equation-2}
\end{align}
	The second equality is the consequence of counting the boundary points of the
	1-dimensional manifold obtained by counting the fiber product of the map
	$r_0:M(\alpha,\beta)\to  \Delta_{N-1}^{\ft,\circ}$ and the map $H_S: [0,1]\times \Delta_{N-1} \to  \Delta_{N-1}^{\ft}$.
	Since $H_{S}(0,\cdot)$ is an isomorphism of simplices, \eqref{equation-2} shows that
	$ \sum_{\sigma,\tau} f^\partial_{\sigma,\tau}$ is equal to the map associated to the pair
	$([0,1]\times Y_l,[0,1]\times \gamma_S)$ corresponding to a broken metric which induces the
	decomposition in Remark \ref{decomp-product}. The pair $([0,1]\times Y_l,[0,1]\times \gamma_S)$
	with the standard metric induces the identity map. Therefore, there is a map $g_S$ such that:
	 \[
	  dg_S+g_Sd=1-\sum_{\sigma,\tau}f^\partial_{\sigma,\tau}
	\]
\end{proof}

\section{Perturbations, Regularity and Gluing}\label{perturbations}

The goal of this section of the paper is to put together transversality results used throughout the paper. To achieve this goal, we use a slightly modified version of the perturbation of the ASD equation in \cite{K:higher}. These perturbations are reviewed in Subsection \ref{hol-per-def}. In Subsection \ref{reg-irr}, we discuss regularity of irreducible connections of moduli spaces for a pair $(X,c)$ satisfying either Condition \ref{noS1S2} or Condition \ref{S1S2}. The case of reducible connections is significantly more subtle and in Subsections \ref{comp-red-reg} and \ref{partial-reg-sub}, we prove some partial results in this direction. Perturbations of the ASD equation for a family of cylindrical metrics is discussed in Subsection \ref{per-family}. The final two subsections are devoted to applications of our general transversality results to pairs $(X,c)$ which are of interest to us. In this section, to abbreviate our notation, we write $\Omega^1_{k,\delta}(X,\su(E))$ or $\Omega^1_{\delta}(X,\su(E))$ for the Banach space $L^2_{k,\delta}(X,\Lambda^1\otimes \su(E))$ and so on.

\subsection{Holonomy Perturbations}\label{hol-per-def}
Suppose $X$ is a 4-manifold with boundary. In \cite{K:higher}, a Banach space $\mathcal W$ is constructed where each element $\omega\in \mathcal W$ introduces a perturbation of the ASD equation over $X^+$. The definition of $\mathcal W$ depends on the choice of a family $\{B_i, q_i, C_i\}_{i\in \mathcal{I}}$ where $\mathcal{I}$ is a countable infinite set, $B_i$ is an embedded ball in $X^+$, $q_i$ is a map from $B_i\times S^1$ to $X^+$ and $C_i$ is a positive constant number. For each $x\in B_i$, we use $q_{i,x}$ to denote the loop $q_i|_{\{x\}\times S^1}$.
 Moreover, this collection is required to satisfy the following conditions:
\begin{itemize}
	\item[(i)] The map $q_i$ is a submersion.
	\item[(ii)] $q_i(1,x)=x$ for all $x\in B_i$.
	\item[(iii)] The constants $C_i$ are growing ``sufficiently fast''. (See \cite[Sections 3 and 5]{K:higher} for a more accurate condition.)
	\item[(iv)] For every $x\in X$ the following set is $C^1$-dense in the space of loops based at $x$:
		\[
		 \{q_{i,x}\mid i \in \mathcal{I}, x \in B_i \}.
		\]
\end{itemize}
Define
\[
  \mathcal{J}:=\{ (i,j)\in \mathcal{I}\times \mathcal{I} |i\neq j,B_{i,j}= B_i\cap B_j\neq \emptyset \}.
\]
For $(i,j)\in \mathcal J$, let $q_{i,j}: B_i\cap B_j \times S^1 \to X, q_{i,j}|_{\{x\}\times S^1} $ be given as:
\[
  q_{i,j}:=q_{i,x}\ast q_{j,x}\ast q_{i,x}^{-1}\ast q_{j,x}^{-1}
\]
where $\ast$ denotes the composition of loops. We also assume that a family of constants $\{C_{i,j}\}_{(i,j)\in \mathcal J}$ is fixed such that they satisfy the analogue of Condition (iii) above. The Banach space $\mathcal W$ is the space of all sequences of self-dual 2-forms $\omega=\{\omega_i\}_{i\in \mathcal{I}}\cup \{\omega_{i,j}\}_{j\in \mathcal{J}}$ such that $\omega_i \in \Omega^2(B_i,\mathbf{C})$, $\omega_{i,j} \in \Omega^2(B_i\cap B_j,\mathbf{C})$ and the following expression is finite:
\begin{equation} \label{W-norm}
 \sum_{i\in \mathcal I} C_i |\!|\omega_i|\!|_{C^m}+\sum_{(i,j)\in \mathcal J} C_{i,j} |\!|\omega_{i,j}|\!|_{C^m}
\end{equation}
Here $m\geq 3$ is a fixed integer number. Let $c$ be a 2-cycle on $X$ and $p$ be a path along $(X,c)$. Suppose also $E$ is the $\U(N)$-bundle on $X$ associated to $c$. Given $\omega \in \mathcal W$, we can define a map $V_\omega: \mathcal A_p(X,c)\to \Omega^+(X,\su(E))$.
\begin{equation}\label{defholper}
  V_\omega(A):=\sum_{i\in \mathcal{I}} \pi(\omega_i^+\otimes \Hol_{q_i}(A))
  +\sum_{(i,j)\in\mathcal{J}} \pi(\omega_{i,j}^+\otimes \Hol_{q_{i,j}}(A))
\end{equation}
where $\omega_i^+$ and $\omega_{i,j}^+$ are the self-dual parts of $\omega_i$ and $\omega_{i,j}$, $\Hol_{q_i}(A)(x)\in \SU(E|_x)$ is the holonomy of $A$ along the loop $q_{i}(\cdot,x)$ and $\pi$ is the orthogonal projection from $\mathfrak{gl}(E)$ to $\mathfrak{su}(E)$\footnote{Our Banach space of perturbations is different from that of \cite{K:higher} in two ways. Here elements of $\mathcal W$ are sequences of 2-forms rather than ASD 2-forms. This variation is more suitable when we consider the moduli spaces associated to families of metrics. Secondly, the terms $\omega_{i,j}$ do not exist in the definition of $\mathcal W$ in \cite{K:higher}.}.

Suppose $(X,c)$ is a pair satisfying \eqref{noS1S2} and $p$ is a path along $(X,c)$. Recall that when we form the moduli space $\mathcal M_p(X,c)$, we initially might need to perturb the ASD equation as in \eqref{pert-ASD}. This primary perturbation is induced by perturbations of the Chern-Simons functional for each admissible end of $(X,c)$ such that the perturbed Chern-Simons functional has only non-degenerate critical points. Throughout this section, by an abuse of notation, we will write $F^+_0(A)=0$ for the perturbed equation \eqref{pert-ASD}. Given $\omega\in \mathcal W$, we can define a secondary perturbation and consider the equation $F^+_0(A)+V_\omega(A)=0$. In order to be more specific about the secondary perturbation $\omega$, we will write $\mathcal M^\omega_p(X,c)$ for the solutions of the equation $F^+_0(A)+V_\omega(A)=0$. We use a similar convention for the other moduli spaces which were constructed in Subsection \ref{cyl-end-mod}.

\subsection{Regularity of Irreducible Connections}\label{reg-irr}

\begin{prop} \label{noS1S2-reg}
	Suppose $(X,c)$ is a pair satisfying Condition \ref{noS1S2}.
	Suppose $\bbX$ is a family of smooth metrics on $X$ with cylindrical ends parametrized by
	a smooth manifold $K$. Then there is a residual subset $\mathcal W_{\rm reg}$ of $\mathcal W$ such that for any
	perturbation $\omega \in \mathcal W_{\rm reg}$
	the irreducible elements of moduli space $\mathcal M^{\omega}_p(\bbX,c)$, denoted by $\mathcal M^{*,\omega}_p(X,c)$,
	consists of regular connections for any choice of $p$.
\end{prop}
\begin{proof}
	Let $\mathcal E$ be the Banach bundle over $\mathcal A_p(X,c) \times K$ whose fiber over the metric
	$(A,g) \in K$ is the space $\Omega^{+_g}_\delta(\su(E))$.
	The action of the gauge group lifts to $\mathcal E$. By taking the quotient, we obtain a Banach bundle over $\mathcal B_p(X,c)\times K$ which we also denote by $\mathcal E$.
	We use the same notation for the pull back of $\mathcal E$ to
	$\mathcal B_p(X,c)\times K\times \mathcal W$ via the projection.
	Define $\Phi: \mathcal B_p(X,c)\times K\times \mathcal W \to \mathcal E$ as follows:
	\[
	  \Phi([A],g,\omega)=F^{+_g}(A)+V_{\omega}(A,g)
	\]
	Suppose $\bx=([A],g,\omega)$ is a zero of $\Phi$ and $A$ is an irreducible connection.
	Then It is shown in \cite[Lemma 13]{K:higher} that the derivative of the above map at $\bx$, as a map from the tangent space
	of $\mathcal B_p(X,c)\times K\times \mathcal W$ to the fiber of $\mathcal E$ at $\bx$, is surjective.
	In fact, it is shown there that the derivate of $\Phi$ maps
	$T_{[A]}\mathcal B_p(X,c)$ to a closed subspace of $\mathcal E|_{\bx}$ with finite codimension and maps
	$\mathcal W=T_{\omega} \mathcal W$ to a dense subspace of $\mathcal E|_{\bx}$.
	Since we will use a similar argument later, we discuss this in more detail here.

	Given any point $x\in X$, the irreducibility of $A$ and Condition (iv) in previous section imply that we can find
	$\alpha_1,\cdots,\alpha_{N^2}\in \mathcal{I}$ such that
	$\Hol_{q_{\alpha_i}}(A)(x)$ span $\mathfrak{gl}(E)_x$. Thus, they form a local frame of $\mathfrak{gl}(E)$
	 in a neighborhood $U$ of $x$. Given any $g$-self-dual 2-form $\eta\in \Omega^{+_g}_{\delta}(X, \mathfrak{su}(E))$,
	supported in $U$, we can find 2-forms $\omega_{\alpha_i}\in \Omega_\delta^{2}(X)$, for $i=1,\cdots, N^2$, which are
	supported in $U$ and:
	\[
	  \eta=\sum_{i=1}^{N^2} \pi (\omega^{+_g}_{\alpha_i}\otimes \Hol_{q_{\alpha_i}}(A))
	\]
	Thus, partition of unity shows that the image of the restriction of the derivative of $\Phi$ to
	$T_{\omega}\mathcal W$ contains all compactly supported $g$-self-dual 2-forms,
	which is a dense subspace of $\mathcal E|_{\bx}$.

	According to implicit function theorem, the space of irreducible elements in $\Phi^{-1}(0)\subset \mathcal B_p(X,c)\times K\times \mathcal W$,
	denoted by $\mathbb M^*_p(\bbX,c)$,
	is a smooth Banach manifold. Moreover, the projection map from $\mathbb M^*_p(\bbX,c)$ to $\mathcal W$
	is Fredholm. The set of regular values of this projection map, denoted by $\mathcal W_{\rm reg}$,
	is residual by Sard-Smale theorem. The space $\mathcal W_{\rm reg}$ has the required property claimed in the proposition.
\end{proof}

\begin{remark}
	We can consider an alternative space of perturbations.
	In Proposition \ref{noS1S2-reg}, we use perturbations which are constant within the family of metrics.
	Suppose a family of cylindrical metrics on $X$ parametrized by a manifold $K$ is given. Let
	$\widetilde {\mathcal W}$ be the space of maps from $K$ to $\mathcal W$. Then each element
	$\widetilde \omega \in \widetilde {\mathcal W}$ determines a perturbation of the ASD equation which is given as follows:
	\[
	  F_0^{+_g}(A)+V_{\widetilde \omega(g)}(A,g)=0
	\]
	Such perturbations will be essential for us when we consider the ASD equation
	for a family of possibly broken metrics.
\end{remark}

Later, we will need a relative version of Proposition \ref{noS1S2-reg}. Suppose $(X,c)$ and $\bbX$ are given as in Proposition \ref{noS1S2-reg}. Suppose also $J$ is an open subset of $K$ and $H$ is a compact subspace of $J$. We also fix functions $\gamma_0$ and $\gamma_1$ such that $\gamma_0$ is supported in $J$, $\gamma_1$ is supported in the complement of $H$, and $\gamma_0=1$ when $\gamma_1=0$. Assume that $\widetilde \omega \in \widetilde {\mathcal W}$ is also given such that the moduli space $\mathcal M^{*,\widetilde \omega}_p(\bbX,c)$ is regular over $J$, i.e., the moduli space $\mathcal M^{*,\widetilde \omega|_{J}}_p(\bbX|_{J},c)$ is a regular moduli space, We wish to replace $\widetilde \omega$ with another element $\widetilde \omega'$ such that the restrictions of $\widetilde \omega'$ and $\widetilde \omega$ to $H$ are equal to each other and the total moduli space $\mathcal M^{*,\widetilde \omega'}_p(\bbX,c)$ is regular:
\begin{prop} \label{noS1S2-reg-rel}
	Given $\widetilde \omega$ as above, there is a residual subset $\mathcal W_{\rm reg}$ of $\mathcal W$ such that
	for any perturbation $\omega \in \mathcal W_{\rm reg}$, the following element of $\widetilde {\mathcal W}$ defines a regular moduli space over $K$:
	\[
	  \widetilde \omega':=\gamma_0 \widetilde \omega+\gamma_1 \omega.
	\]
\end{prop}
\begin{proof}
	The proof is similar to the proof of Proposition \ref{noS1S2-reg}. Define a map $\Phi: \mathcal B_p(X,c)\times K\times \mathcal W \to \mathcal E$ in the following way:
	\[
	   \Phi([A],g,\omega)=F_0^{+_g}(A)+\gamma_0 V_{\widetilde \omega}(A,g)+\gamma_1 V_{\omega}(A,g)
	\]
	Let $\bx=([A],g,\omega)$ be a zero of $\Phi$. If $\gamma_1(g)\neq 0$,
	then the same argument as before shows that the derivative of $\Phi$ is surjective at $\bx$. In the case that
	$\gamma_1(g)= 0$, the point $g$ belongs to $J$, and the assumption about $\widetilde \omega$ implies that $\bx$ is a regular point of the map $\Phi$. Now we can
	proceed as in the proof of Proposition \ref{noS1S2-reg} to construct $\mathcal W_{\rm reg}$.
\end{proof}

\begin{prop}\label{S1S2-reg}
	Suppose $(X,c)$ is a pair satisfying Condition \ref{S1S2}.
	Suppose $\bbX$ is a family of smooth metrics on $X$ parametrized by
	a smooth manifold $K$. Then there is a residual set
	$\mathcal W_{\rm reg} \subset\mathcal W$
	such that $\mathcal M^{*,\omega}_p(\bbX,c;\Gamma)$, for $\omega \in \mathcal W_{\rm reg}$,
	consists of regular connections
	for any choice of an open face $\Gamma$ of $\Delta_{N-1}^\ft$ and a path $p$ along $(X,c)$.
	Moreover, for any element $i$ of a finite set $I$, suppose a smooth map
	$\phi_i$ from a smooth manifold $M_i$ to an open face $\Gamma_i$ of $\Delta_{N-1}^{\ft}$
	is given. We can also assume that for any $i\in I$ and $\omega \in \mathcal W_{\rm reg}$, the map
	$r:\mathcal M^{*,\omega}_p(\bbX,c;\Gamma_i) \to \Gamma_i$ is transversal to the map $\phi_i$.
\end{prop}
\begin{proof}
	The proof is similar to the proof of Propositions \ref{noS1S2-reg}. We can define a map
	$\Phi:\mathcal B_p(X,c;\Gamma)\times K \times \mathcal W \to \mathcal E \times \Gamma$. Suppose
	$\bx=([A],g,\omega)$ is such that $A$ is an irreducible connection and $\Phi(\bx)=(0,\beta)$.
	Then we can consider:
	\[
	  D_{\bx} \Phi:T_{[A]} \mathcal B_p(X,c,\Gamma) \times T_{g}K \times \mathcal W \to
	  \mathcal E|_{\bx} \times T_\beta \Gamma
	\]
	Then the space $D_{\bx} \Phi(T_{[A]} \mathcal B_p(X,c,\Gamma))$ is a closed subspace of
	$\mathcal E|_{\bx} \times T_\beta \Gamma$ with finite codimension. Furthermore, the projection of this space to
	$T_\beta \Gamma$ is also surjective. The space $D_{\bx} \Phi(\mathcal W)$ is mapped to $\mathcal E|_{\bx} \times \{0\}$
	and is dense in this space. Therefore, $\{0\}\times \{\beta\}$ is a regular value of $\Phi$ for any choice of
	$\beta$. Therefore, we can proceed as before to verify this proposition.
\end{proof}

We can also state a relative version of Proposition \ref{S1S2-reg}, similar to Proposition \ref{noS1S2-reg-rel}:

\begin{prop} \label{S1S2-reg-rel}
	Suppose $(X,c)$, the family of metrics $\bbX$ and
	$\phi_i$ are given as in Proposition \ref{S1S2-reg}. Suppose also $H\subset J\subset K$, $\gamma_0$ and $\gamma_1$
	are given as in Proposition \ref{noS1S2-reg-rel}. Suppose $\widetilde \omega \in \widetilde {\mathcal W}$ is given such that
	the moduli space $\mathcal M^{*,\widetilde \omega|_{J}}_p(\bbX|_{J},c;\Gamma)$ consists of regular solutions for each open face $\Gamma$
	of $\Delta_{N-1}^{\ft}$, and for any $i\in I$, the map $r:\mathcal M^{*,\widetilde \omega|_{J}}_p(\bbX|_{J},c;\Gamma_i) \to \Gamma_i$
	is transversal to $\phi_i$.
	Then there is a residual subset $\mathcal W_{\rm reg} \subset \mathcal W$ such that for any
	$\omega \in \mathcal W_{\rm reg}$, the parameterized
	moduli space over $K$ and associated to the following perturbation has similar properties:
 	\[
	  \widetilde \omega':=\gamma_0 \widetilde \omega+\gamma_1 \omega.
	\]
	That is to say, the moduli space $\mathcal M^{*,\widetilde \omega'}_p(\bbX,c;\Gamma)$ consists of regular solutions
	for each open face $\Gamma$ of $\Delta_{N-1}^{\ft}$, and for any $i\in I$, the map $r:\mathcal M^{*,\widetilde \omega'}_p(\bbX,c;\Gamma_i) \to \Gamma_i$
	is transversal to $\phi_i$.
\end{prop}

\begin{remark} \label{admissible-comp-supp-per}
	If $(X,c)$ has an admissible end, then all elements of the Banach manifolds $\mathcal B_p(X,c)$ are irreducible.
	Therefore, we can use the general results of this section to ensure that the moduli spaces of interest to us consists of only regular solutions.
	Moreover, we can assume that the perturbation term is supported only in the compact part $[-1,-\frac{1}{2}]\times \partial X$ of the cylindrical ends \cite{KM:YAFT}. (See also \cite[Section 24]{KM:monopoles-3-man}.)
	In this case, Propositions \ref{noS1S2-reg}, \ref{noS1S2-reg-rel}, \ref{S1S2-reg} and \ref{S1S2-reg-rel} can be proved even
	if we restrict our attention to holonomy perturbations which are supported in $[-1,-\frac{1}{2}]\times \partial X$.
	
	In the case that $(X,c)$ does not have any admissible end, then the moduli spaces of ASD connections might have reducible elements. Therefore, the results of this subsection cannot be used to achieve regularity.
	Nevertheless, we shall prove some partial results about the regularity of reducible elements of moduli spaces in the next two subsections.
\end{remark}

\subsection{Regularity of Completely Reducible Connections} \label{comp-red-reg}
Suppose $X$ is a 4-manifold with boundary such that $b_1(X)=b^+(X)=0$. We fix a cylindrical metric on $X$. Suppose $L_1$, $\dots$, $L_N$ are $\U(1)$-bundles on $X$, and $E$ denotes the $\U(N)$-bundle $L_1\oplus \dots \oplus L_n$. For simplicity, we assume that the $\U(1)$-bundles are mutually non-isomorphic to each other. Let $c$ be a 2-cycle representing $c_1(E)$. We assume that $(X,c)$ is a pair satisfying Condition \ref{S1S2} without any admissible end.  Each of the $\U(1)$-bundles $L_i$ admits an ASD connection, that is unique up to isomorphism. For simplicity, we also assume that the metric on $X$ is chosen such that the limiting flat connection of these ASD $\U(1)$-connections on $S^1\times S^2$ are different.

If we fix a connection $B_i$ on $L_i$ which is the pull-back of a flat $\U(1)$-connection on the ends of $X$, then $B_1\oplus \dots \oplus B_N$ determines a path $p$ along $(X,c)$. As in Subsection \ref{cyl-end-mod}, we can also use this connection to define the space $\mathcal B_p(X,c)$. Let $\mathcal B_{cr}$ be the subspace of $\mathcal B_p(X,c)$ represented by connections $A=A_1\oplus \dots \oplus A_N$ where $A_i$ is a connection on $L_i$. For any perturbation term $\omega$, let $\mathcal M_{cr}^{\omega}(X,L_1\oplus\cdots\oplus L_N)$ be the intersection of the moduli space $\mathcal M_p^\omega(X,c)$ and $\mathcal B_{cr}$.

\begin{prop}\label{ccr}
	For a small perturbation $\omega$,
	the moduli space $\mathcal{M}_{cr}^{\omega}(X,L_1\oplus \cdots, L_N)$ consists of a single point.
\end{prop}
\begin{proof}
	We can form a bundle $\mathcal E$ over $\mathcal B$ whose fiber over the class of a completely
	reducible connection $A$ has the form $\Omega^+_\delta(X,\R^{N-1})$
	where  for each $x\in X$, $\R^{N-1}$ is identified with the subspace of
	$\su(L_1\oplus \dots \oplus L_N)|_{\{x\}}$ consisting of the diagonal transformations.
	We can use the ASD equation on the space of completely reducible connection to define a section of this bundle.
	Since $b^+(X)=0$, this section has a unique zero which is cut down regularly.
	By abuse of notation, we will also write $\mathcal E$ for the pull-back of the bundle $\mathcal E$
	to $\mathcal B_{cr}\times \mathcal W$. The perturbed ASD equations $F_0^+(A)+V_\omega(A)$
	induces a section of the above bundle. The inverse function theorem implies that
	there is a neighborhood $\mathcal U$ of the unique solution of the solution of $F_0^+(A)=0$
	in $\mathcal B$ such that if $\omega$ is small enough, then
	$\mathcal{M}_{cr}^{\omega}(X,L_1\oplus \cdots, L_N) \cap \mathcal U$ has a unique solution.

	Suppose there is a perturbation $\omega_n$ such that $|\omega_n|<\frac{1}{n}$ and
	$\mathcal{M}_{cr}^{\omega_n}(X,L_1\oplus \cdots, L_N)$ has a solution $[B_n]$
	which does not belong to the open set $\mathcal U$.
	After passing to a subsequence,
	these connections are convergent to a connection $[B_\infty]$.
	Note that Uhlenbeck compactness a priori implies that we could have bubbles or
	part of the energy might slid off the ends. However, because $B_n$ are direct sums of
	abelian connections, neither of these phenomena happen
	 and the limiting connection is also completely reducible.
	 Since the convergence is strong,
	 the limiting connection $[B_\infty]$ belongs to the complement of $\mathcal U$ and satisfies
	 the ASD equation which is a contradiction.
\end{proof}

We wish to study the regularity of the moduli space $\mathcal M_{cr}^{\omega}(X,L_1\oplus \dots \oplus L_N)$. We assume that the perturbation $\omega$ is small enough such that $\mathcal M_{cr}^{\omega}(X,L_1\oplus \dots \oplus L_N)$ has a unique element represented by $A=A_1\oplus \dots \oplus A_N$. Let $\alpha$ denote the limiting value of this connection on  $S^1\times S^2$. The decomposition of $E$ as a direct sum of $\U(1)$-bundles, induces the following decomposition:
\begin{equation*}
 \mathfrak{su}(E)= \underline{ \R}^{N-1}\oplus \bigoplus_{i<j}L_i\otimes L_j^\ast
\end{equation*}
where $\underline {\R}^{N-1}$ denotes the bundle of trace free diagonal skew-hermitian endomorphisms of $E$. Since $V_\omega$ is an equivariant map under the gauge group action, the ASD complex of the connection $A=A_1\oplus \dots \oplus A_N$ decomposes into a direct sum of the diagonal part
\begin{equation}\label{diagonal}
	\Omega^0_{k+1,\delta} (X,\underline {\R}^{N-1}) \xrightarrow{\hspace{1mm} 	d\hspace{1mm}}
	\Omega^1_{k,\delta}(X,\underline {\R}^{N-1})\oplus
	T_{\alpha}\Delta_{N-1}^{\ft,\circ}\xrightarrow{\hspace{1mm} d^+ + DV_\omega |_{A,\underline {\R}^{N-1}}
	\hspace{1mm}}\Omega^+_{k-1,\delta}(X,\underline {\R}^{N-1})
\end{equation}
and the off-diagonal part:
\begin{equation}\label{offdiag}
	\Omega^0_{k+1,\delta}(X,L_i\otimes L_j^\ast)
	\xrightarrow{\hspace{1mm}d_{A_i\otimes A_j^\ast}\hspace{1mm}}
	\Omega^1_{k,\delta}(X,L_i\otimes L_j^\ast)\xrightarrow{\hspace{1mm}d_{A_i\otimes A_j^\ast}^+ +
	DV_\omega |_{A,L_i\otimes L_j^\ast} 	\hspace{1mm}}
	\Omega^+_{k-1,\delta}(X,L_i\otimes L_j^\ast)
\end{equation}
where $i,j$ run through all the pairs that $ 1\leq i<j\leq N$. Note that \eqref{diagonal} gives a Kuranishi structure for the moduli space of completely reducible connections $\mathcal M_{cr}^{\omega}(X,L_1\oplus \dots \oplus L_N)$.

In the remaining part of this subsection, we regard $L_i\otimes L_j^\ast$ as a sub-bundle of $\su(E)$. This means that for a given point $x\in X$, we should consider any element of $L_i\otimes L_j^\ast|_x$ as a matrix $(h_{kl})$ with $h_{ij}=-\bar{h}_{ji}$ and all the other entries zero.
The complex structure of $L_i\otimes L_j^\ast$ maps $(h_{kl})$ to a matrix $h'$ such that $h'_{ij}=\bi h_{ij}$, $h'_{ji}=-\bi h_{ji}$ and all the remaining entries of $h'$ are zero.

\begin{definition}
	Let $A$ be a connection on $E$ (not necessarily completely reducible), which represents a path $p$ along $(X,c)$.
	Let $\gamma$ be a piecewise smooth loop in $X$ with base point $x\in X$.
	Let $a\in T_A \mathcal A_p(X,c) \cong \Omega^1_{k,\delta}(X,\mathfrak{su}(E))$.
	By fixing a trivialization of $E|_x$, we can identify holonomy of any connection along $\gamma$ with an element of $\U(N)$.
	(In fact, this holonomy belongs to a the coset $g\cdot \SU(N)$ which is independent of the connection $A$
	because elements of $\mathcal A_p(X,c)$ have the same central parts.)
	We define:
	\begin{equation}\label{thol}
	  \mathfrak {hol}_{A,\gamma}(a):=D\Hol_\gamma|_A(a)=\frac{d\Hol_\gamma(A+ta)}{dt}|_{t=0}
	\end{equation}
	Note that $  \mathfrak {hol}_{A,\gamma}$ takes values in the left translation of $\su(N)$ by the element $\Hol_{\gamma}(A)\in \U(N)$.
\end{definition}
We may think of $\gamma$ as a map from $[0,1]$ to $X$ with $\gamma(0)=\gamma(1)$. By parallel transport, we obtain a trivialization of $\gamma^\ast (E)$. Now, $\gamma^\ast  a$ can be regarded as an $\su(N)$-valued 1-form under this trivialization, and we have:
\begin{equation}\label{holtan}
	\mathfrak {hol}_{A,\gamma}(a)=-{\Hol}_{\gamma}(A)\int_{[0,1]}\gamma^\ast a \in T_{\Hol_\gamma(A)}\U(N)
\end{equation}
In the case that $A$ is a completely reducible connection, we can write an explicit semi-global formula for $\mathfrak {hol}_{A,\gamma}(a)$. Fix a trivialization of $L_k$ ($1\le k \le N)$ in a neighborhood of $\gamma$ such that $A_k$ can be written as $d+a_k$ with $a_k$ being a 1-form with values in $i\mathbf{R}$. Then we have:
\begin{equation}\label{holtana}
  \mathfrak {hol}_{A,\gamma}(a)=-{\Hol}_{\gamma}(A)\int_{[0,1]} e^{\int_0^\tau a_i(s)-a_j(s) ds} \gamma^\ast a(\tau)  d\tau
\end{equation}

\begin{prop}\label{exacthol}
	Fix a base point $x\in X$. Let $a$ be an elements of $\Omega^1_{k,\delta}(X,L_i\otimes L_j^\ast)$. Then the following are equivalent:
	\begin{itemize}
		\item[(i)] there is a $0$-form $b$ with values in $L_i\otimes L_j^\ast$ such that:
			\[
			   a =d_{A_i\otimes A_j^\ast} b \hspace{1cm}b(x)=0.
			\]
		\item[(ii)] For any loop $\gamma$ with $\gamma(0)=x$, we have $\mathfrak {hol}_{A,\gamma}(a)= 0$.

   \end{itemize}
\end{prop}
\begin{proof}
	Suppose $\mathfrak {hol}_{A,\gamma}(a)$ vanishes for any loop $\gamma$. We define $b$ as follows:
\begin{equation} \label{b(y)}
b(y):= \int_{\sigma:x\to y}  a=\int_{[0,1]}\sigma^\ast a \in  L_i\otimes L_j^\ast |_y
\end{equation}
where $\sigma$ is an arbitrary path from $x$ to $y$. This integral is defined in a similar way as \eqref{holtan} using parallel transport along the path $\sigma$. The expression in \eqref{b(y)} does not depend on the
choice of the path $\sigma$ because $\mathfrak {hol}_{A,\gamma}a$ vanishes for any loop $\gamma$. Now it is easy to see that $d_A b=a$ and $b(x)=0$. This proves
$(ii) \Rightarrow (i)$.

Next, let $a=d_{A_i\otimes A_j^\ast} b$ and $\gamma$ be a loop with $\gamma(0)=x$. We have
\begin{eqnarray}
  \mathfrak {hol}_{A,\gamma}(a) &=& -{\Hol}_{\gamma}(A)\int_{[0,1]}\gamma^\ast a \nonumber\\
    &=& -{\Hol}_{\gamma}(A)\int_{[0,1]} \frac{d \gamma^\ast b}{dt} \nonumber\\
    &=&  -{\Hol}_{\gamma}(A) (\gamma^\ast b (1) -\gamma^\ast b (0))  \nonumber \\
    &=&   -{\Hol}_{\gamma}(A)  (\Ad_{{\Hol}_{\gamma}(A)^{-1}}(b(x))  -b(x) ) \label{exact1f}
\end{eqnarray}
Since $b(x)=0$, we have $\mathfrak {hol}_{A,\gamma}(a)=0$. This proves $(i) \Rightarrow (ii)$.
\end{proof}
\begin{prop}\label{nonzerohol}
	Fix a base point $x\in X$ and trivialize $L_k$ at $x$ for all $k$.
	For any non-zero form $a\in \Omega^1_{k,\delta}(X,L_i\otimes L_j^\ast)$, there exist loops $\gamma_1$, $\gamma_2$ based at $x$ such that
	\[
	 \mathfrak {hol}_{A,\gamma_1\ast \gamma_2 \ast \gamma_1^{-1}\ast \gamma_2^{-1}}(a)\neq 0
	\]
	unless $a=d_{A_i\otimes A_j^\ast}b$ for a $0$-form $b$ with values in $L_i\otimes L_j^\ast$.
\end{prop}
\begin{proof}
	Fix a pair of loops $\gamma_1,\gamma_2$, and let
	$\mathfrak {hol}_{A,\gamma_1\ast \gamma_2 \ast \gamma_1^{-1}\ast \gamma_2^{-1}}(a)=0$.
	Let $H_1$, $H_2$ denote respectively $\Hol_{\gamma_1} (A)$, $\Hol_{\gamma_2}( A)$.
	With respect to our trivializations, these are two diagonal matrices.
	We also write $C_k$ and $D_k$ for the $k^{\rm th}$ diagonal entries of $H_1$ and $H_2$.
	We have:
	\begin{eqnarray*}
		0 &=& -\mathfrak {hol}_{A,\gamma_1\ast \gamma_2 \ast \gamma_1^{-1}\ast \gamma_2^{-1}}(a) \\
		&=& \int_{[0,1]}\gamma_1^\ast a + \int_{[0,1]}H_1^{-1}(\gamma_2^\ast a) H_1 +
		\int_{[0,1]}(H_1H_2)^{-1}((\gamma_1^{-1})^\ast a_1)H_1H_2 +
		 \int_{[0,1]}H_2^{-1}((\gamma_2^{-1})^\ast a)H_2\\
		&=&  \int_{[0,1]}\gamma_1^\ast a + \int_{[0,1]}H_1^{-1}(\gamma_2^\ast a) H_1 -
		\int_{[0,1]}H_2^{-1}(\gamma_1^\ast a)H_2 - 	\int_{[0,1]}\gamma_2^\ast a
	 \end{eqnarray*}
	Therefore, we have:
	\begin{equation}\label{CD}
		(D_i^{-1}D_j-1) \int_{[0,1]}\gamma_1^\ast a = (C_i^{-1}C_j-1) \int_{[0,1]}\gamma_2^\ast a.
	\end{equation}
	Since $\gamma_1, \gamma_2$ are arbitrary, \eqref{CD} implies that
	there is $C\in L_i\otimes L_j^*|_{x}$, independent of $\gamma$, such that
	\[
	 \int_{[0,1]}\gamma^\ast a=C(\Hol^{-1}_{\gamma}(A_i)\Hol_{\gamma}(A_j)-1)
	\]
	for any loop $\gamma$ with $\gamma(0)=x$.
	Take an arbitrary section  $s\in \Omega^0(X,L_i\otimes L_j^\ast)$ with $s(x)=C$. Then by \eqref{exact1f} we know:
	\[
	 \int_{[0,1]}\gamma^\ast d_{A_i\otimes A_j^\ast}s= C (\Hol^{-1}_{\gamma}(A_i)\Hol_{\gamma}(A_j)-1)
	\]
	Therefore:
	\[
	  \mathfrak {hol}_{A,\gamma}(a-d_{A_i\otimes A_j^\ast}s)=0
	\]
	for any loop $\gamma$. Proposition \ref{exacthol} implies that there is $b$ such that:
	\[
	  a=d_{A_i\otimes A_j^\ast}b.
	\]
\end{proof}

\begin{lemma}\label{thol}
	Suppose $a\in \Omega^1_{k,\delta}(X, L_i\otimes L_j^\ast)$ is a non-zero 1-form such that $a$ is not equal to $d_{A_i\otimes A_j^\ast}b$
	for any $0$-form $b$ with values in $L_i\otimes L_j^\ast$.
	Then for any given compactly supported $c\in \Omega^+_{k-1,\delta}(X, L_i\otimes L_j^\ast)$, there are two compactly supported
holonomy perturbations $\omega_1$ and $\omega_2$ such that:
\begin{itemize}
  \item $V_{\omega_1} (A)=V_{\omega_2} (A)=0$
  \item  $DV_{\omega_1}|_A (a)+DV_{\omega_2} |_A (\bi a)=c$
\end{itemize}
\end{lemma}
\begin{proof}
	By Proposition \ref{nonzerohol} and our assumptions from Subsection \ref{hol-per-def},
	for an arbitrary point $x\in X$, there exist $(l,m)\in \mathcal{J}$ such that $x\in B_l\cap B_m$
	and $ \mathfrak {hol}_{A,q_{l,x}\ast q_{m,x}\ast q^{-1}_{l,x}\ast q^{-1}_{m,x}}(a)\neq 0 $.
	There is a neighborhood $U\subset B_l\cap B_m$ of $x$  such that
	$$ \mathfrak {hol}_{A,q_{l,y}\ast q_{m,y}\ast q^{-1}_{l,y}\ast q^{-1}_{m,y}}(a)\neq 0 $$
	for all $y\in U$.

	Since $c$ has compact support, we may assume we can find a finite set of triples
	$(U_k,l_k,m_k)$ so that the above  property holds and
	$\{U_k\}$ covers the support of $c$. Let $\{\rho_k\}$ be a partition of unity of the support of $c$ with respect to the open covering
	$\{U_k\}$.
	Since $\rho_k c\in \Omega^+(U,L_i\otimes L_j^\ast)$ and
	$\mathfrak {hol}_{A,q_{l_k}\ast q_{m_k}\ast q_{l_k}^{-1}\ast q_{m_k}^{-1}}(a)$ is a section of $L_i\otimes L_j^\ast$
	which is non-zero on $U_k$,
	we can obtain a complex-valued ASD 2-form supported in $U_k$ as follows:
	\[
	  \eta_k:= \frac{\rho_k c}{\mathfrak {hol}_{A,q_{l_k}\ast q_{m_k}\ast q_{l_k}^{-1}\ast q_{m_k}^{-1}}(a)}
	\]
	We take the following perturbation:
	$$
	V_{\omega_1}(A')=\sum_k \pi (\Re(\eta_k)\otimes \Hol_{q_{l_k}\ast q_{m_k}\ast q_{l_k}^{-1}\ast q_{m_k}^{-1}}(A') )
	$$
	Recall $\pi$ is the projection from $\mathfrak{gl}(n)$ to $\mathfrak{su}(N)$.
	It is easy to check that:
	$$DV_{\omega_1}|_A(a)=\sum_k(\Re \eta_k)\otimes \mathfrak {hol}_{A,q_{l_k}\ast q_{m_k}\ast q_{l_k}^{-1}\ast q_{m_k}^{-1}}(a) $$
	Similarly, we define:
	$$
	V_{\omega_2}(A')=\sum_k \pi (\Im(\eta_k) \otimes\Hol_{q_{l_k}\ast q_{m_k}\ast q_{l_k}^{-1}\ast q_{m_k}^{-1}}(A') )
	$$
	It satisfies:
	$$DV_{\omega_2}|_A(\bi a)=\sum_k(\bi\Im \eta_k)\otimes \mathfrak {hol}_{A,q_{l_k}\ast q_{m_k}\ast q_{l_k}^{-1}\ast q_{m_k}^{-1}}(a)$$
	Therefore, we have:
	$$DV_{\omega_1}|_A(a)+DV_{\omega_2}|_A(\bi a)=c$$
	This proves the second requirement of the lemma. The first requirement is obvious because the holonomy of $A$ along a commutator of loops is the identity map.
\end{proof}

\begin{theorem}\label{comp-red-regularity}
	Let $\omega$ be a sufficiently small perturbation such that $\mathcal{M}_{cr}^{\omega}(X,L_1\oplus \cdots, L_N)$
	consists of a single point $A=A_1\oplus \cdots \oplus A_N$
	with limiting flat connection $\alpha=\alpha_1\oplus \cdots \oplus \alpha_N$ on $S^1\times S^2$.
	Suppose:
	\[ \ind (d_{A_i\otimes A_j^\ast}^\ast+d_{A_i\otimes A_j^\ast}^+ )\ge 0
	  \hspace{1cm} H^0(S^1\times S^2, \alpha_i\otimes \alpha_j^\ast)=0
	  \hspace{2cm}\text{for all } i< j.
	 \]
	Then there is a holonomy perturbation $\omega'$, arbitrarily close to $\omega$ such that
	$\mathcal{M}_{cr}^{\omega'}(X,L_1\oplus \cdots, L_N)=\mathcal{M}_{cr}^{\omega}(X,L_1\oplus \cdots, L_N)$ and
	the unique point in $\mathcal{M}_{cr}^{\omega'}(X,L_1\oplus \cdots, L_N)$ is a regular element of
	$\mathcal M_p^{\omega'}(X,c)$.
	Moreover, $\omega'-\omega$ can be chosen to be compactly supported.
\end{theorem}

\begin{proof}
	If $\omega$ is small enough, then Proposition \ref{ccr} allows us to focus only on the off-diagonal part of the ASD complex, i.e.,
	complexes of the following form:
	\begin{equation*}
		\Omega^0_{k+1,\delta}(X,L_i\otimes L_j^\ast)
	\xrightarrow{\hspace{1mm}d_{A_i\otimes A_j^\ast}\hspace{1mm}}
	\Omega^1_{k,\delta}(X,L_i\otimes L_j^\ast)\xrightarrow{\hspace{1mm}d_{A_i\otimes A_j^\ast}^+ +
	DV_\omega	\hspace{1mm}}
	\Omega^+_{k-1,\delta}(X,L_i\otimes L_j^\ast)
	\end{equation*}
	If  the second map in the complex is not surjective,
	we want to show that there is a compactly supported  small holonomy perturbation $\omega_0$ such that
	\begin{itemize}
	  \item $V_{\omega_0}(A)=0$
	  \item $d_{A_i\otimes A_j^\ast}^+ +DV_\omega + DV_{\omega_0}:\coker (d_{A_i\otimes A_j^\ast}) \to \Omega^+_{k-1,\delta} (X,L_i\otimes L_j^\ast)$ has a smaller
	   kernel than the operator  $d_{A_i\otimes A_j^\ast}^+ +DV_\omega:\coker (d_{A_i\otimes A_j^\ast}) \to \Omega^+_{k-1,\delta} (X,L_i\otimes L_j^\ast)$.
	\end{itemize}
	If this is true, then by induction we can find a perturbation $\omega'$ close to $\omega$ so that $d_{A_i\otimes A_j^\ast}^+ + DV_{\omega'}$
	 is surjective, which verifies our claim. 

	Suppose $d_{A_i\otimes A_j^\ast}^+ +DV_\omega$  in the above complex is not surjection.
	This assumption and non-negativity of $ \ind (d_{A_i\otimes A_j^\ast}^\ast+d_{A_i\otimes A_j^\ast}^+ )$ imply that there is an element
	$a\in\ker(d_{A_i\otimes A_j^\ast}^+ +DV_\omega)$ such that $a$ is not equal to $d_{A_i\otimes A_j^\ast}b$
	for any $0$-form $b$ with values in $L_i\otimes L_j^\ast$.
	Let also $c\in \Omega^+_{k-1,\delta}(X,L_i\otimes L_j^\ast)$ be a compactly supported element in $\coker (d_{A_i\otimes A_j^\ast}^+ +DV_\omega)$.
	By Lemma \ref{thol}, we can find a compactly supported perturbation $\omega_1$  such that
	\begin{itemize}
	  \item $V_{\omega_1} (A)=V_{\omega_2} (A)=0$
	  \item  $DV_{\omega_1}|_A (a)+DV_{\omega_2} |_A (\bi a)=c$
	\end{itemize}
	Either $DV_{\omega_1}|_A (a)$ or $DV_{\omega_2} |_A (\bi a)$ represents a non-zero element in $\coker (d_{A_i\otimes A_j^\ast}^+ +DV_\omega)$.
	The 1-form $\bi a$ also belongs to  $ \ker(d_{A_i\otimes A_j^\ast}^+ +DV_\omega)$ and it cannot be written in the form of $d_{A_i\otimes A_j^\ast}b$.
	If $DV_{\omega_1}|_A (a)$ represents $0$ in the cokernel, then we replace $a$ with $\bi a$ and $\omega_1$ with $\omega_2$.
	Therefore, without loss of generality, we may assume $c'=DV_{\omega_1}|_A (a)$ represents a non-zero element in the cokernel.

	Now we form a new perturbation $\omega_0=\epsilon\omega_1$ where $\epsilon$ is a small positive number.
	 If $\epsilon$ is small enough,
	 $d_{A_i\otimes A_j^\ast}^+ +DV_\omega + \epsilon DV_{\omega_1} |_{(\ker (d_{A_i\otimes A_j^\ast}^+ +DV_\omega) )^\perp   }$
	is still an isomorphism into its image and $c'$ does not lie in the image of this operator, because
	$c'\notin \Im (d_{A_i\otimes A_j^\ast}^+ +DV_\omega)$. Therefore:
	\[
	  d_{A_i\otimes A_j^\ast}^+ +DV_\omega + \epsilon DV_{\omega_1}|_{((\ker (d_{A_i\otimes A_j^\ast}^+ +DV_\omega) )^\perp \oplus \langle a \rangle}
	\]
	is also injective, which proves that
	\[
	  d_{A_i\otimes A_j^\ast}^+ +DV_\omega + DV_{\omega_0} : \coker (d_{A_i\otimes A_j^\ast}) \to \Omega^+ (X,L_i\otimes L_j^\ast)
	\]
	has a smaller kernel than $(d_{A_i\otimes A_j^\ast}^+ +DV_\omega)|_{\coker (d_{A_i\otimes A_j^\ast})}$.
\end{proof}

In our application, we will only deal with the case $X=X_N(l)$. Let
$A$ be the completely reducible connection $B_{\mathbf{v}}(g)$ in Proposition \ref{bi-per-con}.
Apply Lemma \ref{ind-Bv} to the connection $A_i\oplus A_j$, it is easy to see that
   $\ind (d_{A_i\otimes A_j^\ast}^\ast+d_{A_i\otimes A_j^\ast}^+ )=0$.
  If we also suppose the limiting flat connection $\beta$ on $S^1\times S^2$ lies in $\Delta_{N-1}^{\ft,\circ}$, then the requirement in
  Theorem \ref{comp-red-regularity} is fulfilled and we can apply the theorem to make $A$ regular.

\subsection{Partial Regularity}\label{partial-reg-sub}
Let $p$ be a path along $(X,c)$, and suppose $A$ is a connection representing this path. We fix a base point $x\in X$ and let $h_A\subseteq Aut(\left.E\right |_{x})$ be the group of the holonomies of $A$ based at $x$. Identifying the fiber of $E$ at $x$ with $\C^N$ gives an $N$-dimensional representation of $h_A$. Irreducible components of this representation determine a decomposition of $E$ into a direct sum of Hermitian vector bundles as follows:
\[
  E=\underbrace{E_1\oplus \dots\oplus E_1}_{m_1}\oplus \underbrace{E_2\oplus \dots\oplus E_2}_{m_2}\oplus \dots \oplus \underbrace{E_s\oplus \dots\oplus E_s}_{m_s}
\]
For each bundle $E_i$ there is a connection $A_i$ such that the connection $A$ is equal to the direct sum of the connections $A_i$ (where $A_i$ appears with multiplicity $m_i$). Let $N_i$ denote the rank of $E_i$. Let $c_i$ denote a 2-cycle representing $c_1(E_i)$. The path along $(X,c_i)$, determined by the connection $A_i$, is also denoted by $p_i$. We say that $A$ is a connection of type $\{(N_i,m_i,c_i,p_i)\}_{1\leq i \leq s}$. The stabilizer of the connection $A$ of this type is given by:
\[
  \Gamma_A\cong \{(V_1,\dots, V_s)\mid V_i \in \U(m_i), \det(V_1)^{N_1}\dots \det(V_s)^{N_s}=1\}
\]

We can consider the moduli theory of the space of connections which have the same type as $A$. Firstly assume that $(X,c)$ satisfies
Condition \ref{noS1S2} and $b_1(X)=b^+(X)=0$. Suppose $\mathcal A$ denotes the space of all $s$-tuples of connections $\mathbb A=(A_1,\dots, A_s)$ such that:
\begin{itemize}
  \item The connection $A_i$ is irreducible of rank $N_i$ representing the path $p_i$.
  \item The isomorphism classes of the connections $A_i$ are distinct.
  \item The central part of the connection:
	  \begin{equation}\label{conn-A}
	    A=\underbrace{A_1\oplus \dots\oplus A_1}_{m_1}\oplus \underbrace{A_2\oplus \dots\oplus A_2}_{m_2}\oplus \dots
	    \oplus \underbrace{A_s\oplus \dots\oplus A_s}_{m_s}
	  \end{equation}
	  is given by a fixed $\U(1)$-connection on the determinant of $E$.
\end{itemize}
Note that we do not fix the central part of the connections $A_i$. Suppose a family of smooth metrics $\bbX$ on $X$ parameterized by a manifold $K$ is given. Then we can form a bundle $\mathcal E$ over the space $\mathcal A \times K$ whose fiber over $(\mathbb A,g)$ is the space:
\begin{equation} \label{E-fiber}
  \Omega_{\delta}^{+_g}(X,\underline \R^{s-1})\oplus \bigoplus_{1\leq i \leq s}\Omega_{\delta}^{+_g}(X,\su(E_i)).
\end{equation}
Analogous to previous subsection, $\underline \R^{s-1}$ should be interpreted as the sub-bundle of $\su(E)$ given by scaler endomorphisms of $E_i$. Then $F_0^{+_g}(A)$ for the connection $A$ in \eqref{conn-A} can be regarded as an element of the fiber of $\mathcal E$ over $(\mathbb A,g)$.

There is also a gauge group acting on $\mathcal A$ and $\mathcal E$. An element of this gauge group is an $s$-tuple $(u_1,\dots, u_s)$ where $u_i$ is a hermitian automorphism of the bundle $E_i$ and:
\begin{equation}\label{centerofg}
  \det(u_1)^{m_1} \cdot \det(u_2)^{m_2} \dots  \det(u_s)^{m_s}=1.
\end{equation}
The quotient of $\mathcal A$ with respect to this gauge group is denoted by $\mathcal B$. The bundle $\mathcal E$ also induces a bundle over $\mathcal B\times K$, which is still denoted by $\mathcal E$. The trace free part of ASD connections induce a section $\Phi$ of the bundle $\mathcal E$.
 \begin{lemma}\label{indofphi}
	 The differential of the section $\Phi$ at a zero $z\in \mathcal B\times K$  defines a Fredholm operator $\Phi_z$ from
	 $T_z (\mathcal B\times K)$ to the corresponding fiber $\left.\mathcal E\right |_{z}$ whose index is equal to:
	\[
	  \ind(\mathcal D_{A_1})+\dots+\ind(\mathcal D_{A_s})+\dim(K)
	\]
\end{lemma}
\begin{proof}
	Without loss of generality, we may assume $K$ consists of a single point.
	Suppose $z=[A_1,\dots,A_s]$ is a zero of $\Phi$. If we deform $z$ while preserving the central parts of
	$A_1$, $\dots$, $A_s$, then the differential of $\Phi$ takes values in the second summand of \eqref{E-fiber}.
	Fredholmness of the ASD operator implies that this component of the differential of $\Phi$ is Fredholm and its index is equal to
	$\ind(\mathcal D_{A_1})+\dots+\ind(\mathcal D_{A_s})$. If we only deform the central parts of the connections
	$A_1$, $\dots$, $A_s$, while we preserve the central part of $A$ in \eqref{conn-A}, then the differential of $\Phi$ takes
	values in the first summand of \eqref{E-fiber}. This component of the differential of $\Phi$ is an isomorphism
	because $b^1(X)=b^+(X)=0$. (This claim is standard and its proof is implicit in \cite{APS:I}. See also \cite{AD:iPFH}.)
\end{proof}

If the derivative of $\Phi$ at a zero $z \in \mathcal B\times K$ is surjective, then we say $z$ is partially regular. The zeros of $\Phi$ form the moduli space of reducible ASD connections of type $\{(N_i,m_i,c_i,p_i)\}_{1\leq i \leq s}$ with respect to the family of metrics $\bbX$. We will write $\mathcal M_p(\bbX,\{(N_i,m_i,c_i,p_i)\}_{1\leq i \leq s})$ for this space.

Consider the pull-back of the Banach bundle $\mathcal E$ to $\mathcal B\times K\times \mathcal W$ via the projection map and denote this bundle with $\mathcal E$, too. Then $\Phi$ extends to a section of $\mathcal E$ defined over $\mathcal B\times K\times \mathcal W$ which we also denote by $\Phi$. Zero is a regular value of $\Phi$. The main ingredient to verify this claim is the following standard result from representation theory of compact Lie groups. Each element of the holonomy group $h_A$ determines an element of $\U(N_1)\times \dots \U(N_s)\subset \gl(E_1)\oplus \dots \oplus \gl(E_s)$ given by the holonomy of connections $A_i$. Since the connections $A_i$ are irreducible and have different isomorphism classes, the linear combination of all such matrices, generate the space $\gl(E_1)\oplus \dots \oplus \gl(E_s)$. Given this result, then we can apply the argument in the proof of Proposition \ref{noS1S2-reg} to verify the claim about regularity of zero as a value of $\Phi$.

Now, an application of Sard-Smale theorem as in the proof of Proposition \ref{noS1S2-reg} shows that there is $\mathcal W_{\rm reg}\subset \mathcal W$ such that for any $\omega\in \mathcal W_{\rm reg}$, the section:
\[
  \mathcal M^\omega_p(\bbX,\{(N_i,m_i,c_i,p_i)\}_{1\leq i \leq s}):=\Phi^{-1}(0)\cap \(\mathcal B\times K\times \{\omega\}\)
\]
consists of regular elements. This argument and Lemma \ref{indofphi} allow us to conclude:
\begin{prop}
	Suppose $(X,c)$ is a pair satisfying Condition \ref{noS1S2} with $b^1(X)=b^+(X)=0$. Suppose a family of smooth metrics $\bbX$ on $(X,c)$ parametrized by a manifold $K$
	is fixed.
	There is a residual subset $\mathcal W_{\rm reg}\subset \mathcal W$ such that for any $\omega \in \mathcal W_{reg}$ and any path $p$ along $(X,c)$, we have:
   \begin{itemize}
     \item[(i)] Any connection $[A]\in \mathcal M^\omega_p(\bbX,c)$ is partially regular.
     \item[(ii)] Suppose the decomposition of $A$ as a direct sum of
	irreducible connections is given as below:
	\begin{equation}\label{connection-A}
	  A=\underbrace{A_1\oplus \dots\oplus A_1}_{m_1}\oplus \underbrace{A_2\oplus \dots\oplus A_2}_{m_2}\oplus \dots \oplus
	   \underbrace{A_s\oplus \dots\oplus A_s}_{m_s}
	\end{equation}
	Then we have the following index inequality:
	\begin{equation}\label{index-ineq}
	  \ind(\mathcal D_{A_1})+\dots+\ind(\mathcal D_{A_s})+\dim(K)\geq 0.
	\end{equation}
   \end{itemize}
\end{prop}

The above proposition can be adapted to the case that $(X,c)$ satisfies Condition \ref{S1S2} and $b_1(X)=b^+(X)=0$. As the first step, we need to define the space $\mathcal A$.  This space consists of all $s$-tuples of connections $\mathbb A=(A_1,\dots, A_s)$ such that:
\begin{itemize}
  \item The connection $A_i$ is irreducible of rank $N_i$ representing the path $p_i$.
  \item The isomorphism classes of the connections $A_i$ are distinct.
  \item The central part of the connection:
	  \begin{equation}\label{conn-A}
	    A=\underbrace{A_1\oplus \dots\oplus A_1}_{m_1}\oplus \underbrace{A_2\oplus \dots\oplus A_2}_{m_2}\oplus \dots
	    \oplus \underbrace{A_s\oplus \dots\oplus A_s}_{m_s}
	  \end{equation}
	  is given by a fixed $\U(1)$-connection on the determinant of $E$.
  \item There is a fixed lift of an open face $\Gamma$ of $\Delta_{N-1}^\ft$ to the set of smooth connections on $S^1\times S^2$ such that
  	  the limiting flat connection of $A$ on $S^1\times S^2$ belong to this space. 
\end{itemize}
We can proceed as before to define a gauge group acting on $\mathcal A$, the quotient space $\mathcal B$, and a Banach bundle $\mathcal E$ over $\mathcal B \times K$  whose fiber over $(\bbA,g)$ is given as in \eqref{E-fiber}.  If $\bbA=(A_1,\dots,A_s)$ then $\{F_0^{+_g}(A_i)\}_{1\leq i \leq s}$ determines a section of this bundle. The moduli space $\mathcal M_p(\bbX,\{(N_i,m_i,c_i,p_i)\}_{1\leq i \leq s};\Gamma)$ is defined to be $\Phi^{-1}(0)$. For any element $z=(\bbA,g)$ of this moduli space, the derivative of $\Phi$ at $z$ defines a Fredholm operator whose index is equal to:
\[
  \(\sum_{i=1}^s \ind(\mathcal D_{A_i})\)+\dim(\Gamma)+\dim(K)-(s-1)
\]
We say $z$ is {\it partially regular}, if the derivative of $\Phi$ is Fredholm at $z$. More generally, we can define the moduli space $\mathcal M^\omega_p(\bbX,\{(N_i,m_i,c_i,p_i)\}_{1\leq i \leq s};\Gamma)$ for any $\omega \in \mathcal W$ and extend the notion of partial regularity for the elements of these spaces.  The proof of the above formula is similar to the proof of Lemma \ref{indofphi}. The first term is given by deforming the connections $A_i$ while preserving their central parts and boundary flat connections on $S^1\times S^2$. The second term is given by deforming the boundary flat connections. The third term is the contribution of deforming the metric. The final term is also obtained by varying the central parts of each $A_i$ while preserving their boundary values.

\begin{prop}\label{partial-reg}
	Suppose $(X,c)$ is a pair satisfying Condition \ref{S1S2} with $b^1(X)=b^+(X)=0$.
	Suppose a family of smooth metrics $\bbX$ on $(X,c)$ parametrized by a manifold $K$
	is fixed. Suppose also a finite subset $M$ of $\Delta_{N-1}^{\ft,\circ}$ is fixed.
	There is a residual subset $\mathcal W_{\rm reg}\subset \mathcal W$ such that for any $\omega \in \mathcal W_{reg}$ and any path $p$ along $(X,c)$, we have:
   \begin{itemize}
     \item[(i)] Any connection $[A]\in \mathcal M^\omega_p(\bbX,c;\Delta_{N-1}^\ft)$ is partially regular.
     In particular, if the decomposition of $A$ into irreducible connections has the following form:
	\begin{equation}\label{connection-A'}
	  A=\underbrace{A_1\oplus \dots\oplus A_1}_{m_1}\oplus \underbrace{A_2\oplus \dots\oplus A_2}_{m_2}\oplus \dots \oplus
	   \underbrace{A_s\oplus \dots\oplus A_s}_{m_s},
	\end{equation}
	then:
	\begin{equation} \label{inequality-sum-indices}
	    \(\sum_{i=1}^s \ind (\mathcal D_{A_{i}})\) +N-s+\dim(K)\geq 0.
	\end{equation}
	\item[(ii)] Suppose $[A]\in \mathcal M^\omega_p(\bbX,c;\Delta_{N-1}^\ft)$ is an element of
	$\mathcal M^\omega_p(\bbX,\{(N_i,m_i,c_i,p_i)\}_{1\leq i \leq s};\Delta_{N-1}^{\ft,\circ})$ and $r([A]) \in M$.
	Then the map $r:\mathcal M^\omega_p(\bbX,\{(N_i,m_i,c_i,p_i)\}_{1\leq i \leq s}; \Delta_{N-1}^{\ft,\circ})\to \Delta_{N-1}^{\ft,\circ}$
	at the point $[A]$ is transversal to the inclusion map of $M$ in $\Delta_{N-1}^{\ft,\circ}$.
	In this case, the inequality in \eqref{inequality-sum-indices} can be improved as follows:
	\begin{equation}
	    \(\sum_{i=1}^s \ind (\mathcal D_{A_{i}})\) +1-s+\dim(K)\geq 0.
	\end{equation}
  \end{itemize}
\end{prop}
\begin{proof}
	The bundle $\mathcal E$ and the section $\Phi$ can be extended to the space $\mathcal B\times K\times \mathcal W$.
	The zero is a regular value of $\Phi$. In the case $\Gamma=\Delta_{N-1}^{\ft,\circ}$,
	we also define $r:\mathcal B\times K\times \mathcal W \to \Delta_{N-1}^{\ft,\circ}$ using the limiting value of connections on $S^1\times S^2$.
	It is also true that for any $\beta\in\Delta_{N-1}^{\ft,\circ}$, the pair $(0,\beta)$ is a regular value of
	$(\Phi,r):\mathcal B\times K\times \mathcal W \to \mathcal E\times \Delta_{N-1}^{\ft,\circ}$.
	We follow a similar argument as before to verify the claims.
\end{proof}

Proposition \ref{S1S2-partial-reg-rel} is the relative version of Proposition \ref{partial-reg}:
\begin{prop} \label{S1S2-partial-reg-rel}
	Suppose $(X,c)$ is a pair satisfying Condition \ref{S1S2} with $b^1(X)=b^+(X)=0$.
	Suppose a family of smooth metrics $\bbX$ on $(X,c)$ parametrized by a manifold $K$
	is fixed. Suppose $H\subset J\subset K$, $\gamma_0$ and $\gamma_1$
	are given as in Proposition \ref{noS1S2-reg-rel}.
	Suppose a finite subset $M$ of $\Delta_{N-1}^{\ft,\circ}$ is fixed.
	Suppose $\widetilde \omega \in \widetilde {\mathcal W}$ is given such that the following two properties hold:
	\begin{itemize}
		\item[(i)] Any connection $[A]\in \mathcal M^{\widetilde \omega|_J}_p(\bbX|_{J},c;\Delta_{N-1}^\ft)$ is partially regular.
		\item[(ii)] Suppose $[A]\in \mathcal M^{\widetilde \omega|_J}_p(\bbX|_{J},c;\Delta_{N-1}^\ft)$ is an element of
		$\mathcal M^{\widetilde \omega|_J}_p(\bbX|_{J},\{(N_i,m_i,c_i,p_i)\}_{1\leq i \leq s};\Delta_{N-1}^{\ft,\circ})$
		and $r([A]) \in M$. Then the map
		$r:\mathcal M^{\widetilde \omega|_J}_p(\bbX|_{J},\{(N_i,m_i,c_i,p_i)\}_{1\leq i \leq s};
		\Delta_{N-1}^{\ft,\circ})\to \Delta_{N-1}^{\ft,\circ}$
		at the point $[A]$ is transversal to the inclusion map of $M$ in $\Delta_{N-1}^{\ft,\circ}$.
	\end{itemize}
	There is a residual subset $\mathcal W_{\rm reg} \subset \mathcal W$ such that for any $\omega \in \mathcal W_{\rm reg}$,
	if:
 	\[
	  \widetilde \omega':=\gamma_0 \widetilde \omega+\gamma_1 \omega.
	\]
	then the following properties hold:
	\begin{itemize}
		\item[(i)] Any connection $[A]\in \mathcal M^{\widetilde \omega'}_p(\bbX,c;\Delta_{N-1}^\ft)$ is partially regular.
		In particular, if the decomposition of $A$ into irreducible connections has the following form:
		\begin{equation}\label{connection-A'-2}
			A=\underbrace{A_1\oplus \dots\oplus A_1}_{m_1}\oplus
			 \underbrace{A_2\oplus \dots\oplus A_2}_{m_2}\oplus \dots \oplus
			\underbrace{A_s\oplus \dots\oplus A_s}_{m_s},
		\end{equation}
		then:
		\begin{equation} \label{inequality-sum-indices-2}
		    \(\sum_{i=1}^s \ind (\mathcal D_{A_{i}})\) +N-s+\dim(K)\geq 0.
		\end{equation}
		\item[(ii)] Suppose $[A]\in \mathcal M^{\widetilde \omega'}_p(\bbX,c;\Delta_{N-1}^\ft)$ is an element of
			$\mathcal M^{\widetilde \omega'}_p(\bbX,\{(N_i,m_i,c_i,p_i)\}_{1\leq i \leq s};\Delta_{N-1}^{\ft,\circ})$ and $r([A]) \in M$.
			Then the map
			$r:\mathcal M^{\widetilde \omega'}(\bbX,\{(N_i,m_i,c_i,p_i)\}_{1\leq i \leq s};\Delta_{N-1}^{\ft,\circ})\to
			\Delta_{N-1}^{\ft,\circ}$
			at the point $[A]$ is transversal to the inclusion map of $M$ in $\Delta_{N-1}^{\ft,\circ}$.
			In this case, the inequality in \eqref{inequality-sum-indices} can be improved as follows:
			\begin{equation}
			    \(\sum_{i=1}^s \ind (\mathcal D_{A_{i}})\) +1-s+\dim(K)\geq 0.
			\end{equation}
	\end{itemize}
\end{prop}

\subsection{Perturbations over a Family of Cylindrical Metrics and Gluing}\label{per-family}
Suppose $\bbW$ is a family of cylindrical metrics on a smooth 4-manifold $W$ parametrized by an admissible polyhedron $K$. We write $\pi$ for the projection map from $\bbW$ to $K$. Suppose $c$ is a 2-cycle on  $W$ such that $(W,c)$ satisfies Condition \ref{noS1S2} or \ref{S1S2}. If $Y$ is a cut that appears in the family $\bbW$, then we require that $Y$ intersects $c$ transversely and $Y$ is either a lens space, $S^1\times S^2$ or $(Y,Y\cap c)$ is an $N$-admissible pair. For each $N$-admissible pair $(Y,\gamma)$, appearing as a boundary component of $(W,c)$ or a cut in the family $\bbW$, we fix a perturbation of the Chern-Simons functional such that $\fC_*^N(Y,\gamma)$ is well-defined. This perturbation of the Chern-Simons functional determines a perturbation of the ASD equation on $((a,b)\times Y,(a,b)\times \gamma)$ for any interval $(a,b)$. We will write $U_{Y,\gamma}$ for this perturbation. We wish to explain how we perturb the ASD equation for the family of metrics $\bbW$.

Suppose $\overline F$ is a codimension $l$ face of $K$ and $Y_F$ is the union of $l$ cuts associated to $\overline F$. By definition, $F$ has the product form $K_0\times K_1\times \dots\times K_l$ for admissible polyhedra $K_i$. Let $Y_1$, $\dots$, $Y_l$ denote the connected components of $Y_F$. Removing a regular neighborhood of $Y_F$ from $W$ produces a 4-manifold with $l+1$ connected components which we denote by $W_0$, $\cdots$, $W_l$. We also write $c_i$ for $W_i\cap c$. As we explained in Subsection \ref{family-general}, the family of metrics parametrized by the subset $F\times (1,\infty]^l$ of $K$ is determined by families of smooth metrics $\bbW_i$ on the 4-manifolds $W_i$ parametrized by $K_i^\circ$. Suppose $\widetilde \omega_i$ is an element of $\widetilde{\mathcal W}$ associated to the family of smooth metrics $\bbW_i$. Recall that this gives a perturbation $\widetilde \omega_i(g)$ of the ASD equation for each element $g\in K_i^\circ$. We assume that there is a positive  constant number $N_0$ such that $\widetilde \omega_i(g)$ is supported in the following subset of $W_i^+$\footnote{In this subsection, we use the cylindrical coordinate $t$ on the ends instead of the coordinate $r$.}:
\[
  W_i\cup [0, \frac{N_0}{2})\times \partial W_i.
\]

We can use $\widetilde \omega_i$ and the chosen perturbations of the Chern-Simons functionals to define a perturbation $\widetilde \omega$ of the ASD equation over the subspace $F\times (N_0,\infty]^l$. If $g=(g_0,\dots,g_l,s_1,\dots,s_l)\in F\times (N_0,\infty]^l$, the support of $\widetilde \omega_i(g_i)$ can be identified naturally with a subset of the 4-manifold $\pi^{-1}(g)$. Moreover, the support of these perturbation terms are disjoint. Therefore, $\widetilde \omega_0(g_0)$, $\dots$, $\widetilde \omega_l(g_l)$ induce a perturbation of the ASD equation on $\pi^{-1}(g)$ which we denote by $\breve\omega(g)$. If $(Y_i,Y_i\cap c)$ is $N$-admissible, then we can consider the perturbation $U_{Y_i,Y_i\cap c}$ on $(-\frac{s_i}{2},\frac{s_i}{2})\times Y_i$ induced by the perturbation of the Chern-Simons functional. Let $\phi:\R\to[0,1]$ be a smooth function which is supported in $(0,\infty)$ and is equal to $1$ on $[1,\infty)$. We can extend the following perturbation of the ASD equation on $(-\frac{s_i}{2},\frac{s_i}{2})\times Y_i$ to $\pi^{-1}(g)$:
\[
  \phi(t+\frac{s_i}{2})\cdot \phi(\frac{s_i}{2}-t)\cdot U_{Y_i,Y_i\cap c}
\]
Let $\overline \omega(g)$ denote the summation of all these perturbation for all $N$-admissible $(Y_i,Y_i\cap c)$. Then the perturbation of the ASD equation $\widetilde \omega(g)$ at $g\in F\times (N_0,\infty]^l$ is defined to be $\breve \omega(g)+\overline \omega(g)$. We say a perturbation of the ASD equation for the family of metrics $\bbW$ is admissible, if there is $N_0$ such that the perturbations over $F\times (N_0,\infty]^{{\rm codim}(F)}$ for each face $\overline F$ of $K$ has the above form. The main feature of admissible perturbations for us is the following result from gluing theory:

\begin{prop}\label{S1S2glue}
	Suppose $\omega$ is an admissible perturbation for the family of metrics $\bbW$.
	Suppose $z$ is an element of $\mathcal M_p^\omega(\bbW,c)$ such that ${\rm Pr}(z)$ belongs to $F$
	where $\overline F$ is a face of $K$ with codimension $l$.
	Then there is a neighborhood of $z$ in $\mathcal M_p^\omega(\bbW,c)$ which contains only regular elements of
	the moduli space. Moreover, there is a neighborhood $U$ of $z$ in $\mathcal M_p^\omega(\pi^{-1}(F),c)$
	, a positive real number $N_1$, and a map
	$\Psi:U\times (N_1,\infty]^l \to \mathcal M_p^\omega(\bbW,c)$ which is a diffeomorphism into an open neighborhood of $z$,
	and for any $(w,s_1,\dots,s_l)\in U\times (N_1,\infty]^l$, there is $f \in F$ such that:
	\[
	  {\rm Pr}(\Psi(w,s_1,\dots,s_l))=(f,s_1,\dots,s_l).
	\]
	
\end{prop}

\begin{proof}
	This proposition is a standard consequence of gluing theory results in the context of Yang-Mills gauge theory.
	In the case $N=2$, \cite{Don:YM-Floer} is a good reference for these results. The proofs there can be
	easily adapted to higher values of $N$.
\end{proof}

We show the usefulness of Proposition \ref{S1S2glue} in an example. Suppose $\bbW$ is a family of metrics on $W$ parametrized by an admissible polyhedron $K$. Suppose $c$ is a 2-cycle on $W$ such that for any 4-manifold $W'$ that appears in the family $\bbW$, the pair $(W', W'\cap c)$ has at least one admissible end. We show that there is an admissible perturbation $\omega$ for the family of metrics $\bbW$ such that for any oath $p$ along $(W,c)$ with $\ind(p)\leq 1-\dim(K)$, the moduli space $\mathcal M_p^\omega(\bbW,c)$ is regular. 

Firstly, we show that $\omega$ can be chosen such that the moduli spaces $\mathcal M_p^\omega(\bbW,c)$ are empty if $\ind(p)<-\dim(K)$. We firstly start with the definition of perturbations for the vertices of $K$. Suppose $\overline F$ is a vertex of $K$, $W_0$, $\dots$, $W_l$ are connected components of $W\setminus (-1,1)\times Y_F$ and $c_i=W_i\cap c$. The face $F$ also determines a metric with cylindrical end on each component $W_i$. For each $W_i$, we use Proposition \ref{noS1S2-reg} (more precisely, Remark \ref{admissible-comp-supp-per}) to choose a perturbation $\omega_i$ such that the moduli space $\mathcal M_{p_i}^{\omega_i}(W_i,c_i)$ is regular for any  path $p_i$. Now, suppose $p_i$ is a path along $(W_i,c_i)$ such that gluing the paths $p_0$, $\dots$, $p_l$ gives rise to the path $p$.  Since $\ind(p_0)+\dots+\ind(p_l)<-\dim(K)$, there is a path $p_i$ such that $\ind(p_i)$ is negative. Therefore, the moduli space over $F$ is empty. We extend the chosen perturbations to a neighborhood of each face such that the resulting perturbation is admissible. Proposition \ref{S1S2glue} asserts that if these neighborhoods are small enough, then the moduli space over these neighborhoods are also regular. Next, we need to extend the family of metrics over the edges of $K$. Because we already define the perturbations of the ASD equation in a neighborhood of vertices, we need to use Proposition \ref{noS1S2-reg-rel} to extend the perturbations to the edges. Note that we might need to shrink the neighborhoods around the vertices of $K$ as a result of applying this proposition. Another application of Proposition \ref{S1S2glue} shows that we can extend the perturbations to a neighborhood of the edges. Repeating this argument inductively gives the desired perturbation for the family of metrics $K$.

Next, let $\ind(p)=-\dim(K)$. We can use the above argument to construct an admissible perturbation in a neighborhood of the boundary of $K$ such that the moduli spaces over this neighborhood associated to paths, with index not greater than $-\dim(K)$, are empty. As the last step, we can use Proposition \ref{noS1S2-reg-rel} to define an admissible perturbation over $K$ such that the moduli space is non-empty only over the complement of a neighborhood of the boundary of $K$. Then Floer-Uhlenbeck compactness theorems show that the moduli space $\mathcal M_p^\omega(\bbW,c)$ is a compact 0-dimensional manifold and ${\rm Pr}$ maps this space to the interior of $K$. Finally, let $\ind(p)=-\dim(K)+1$. As before, we can choose perturbations in a neighborhood of faces with codimension at most one such that the moduli space is empty in a neighborhood of faces of codimension at most $2$ and the moduli space over faces of co-dimension one consists of a finite set of points. Proposition \ref{S1S2glue} implies that in a neighborhood of faces of co-dimension one the moduli space is a $1$-manifold whose boundary is the moduli space over the interior of faces of codimension one. We employ Proposition \ref{noS1S2-reg-rel} to extend the perturbation over the remaining part of $K$ such that the moduli space is a 1-manifold.

If there is a 4-manifold $W'$ appearing in the family of metrics $\bbW$ without an admissible end, then the situation is more complicated. The moduli space might contain reducible connections and we cannot apply the general results of Subsection \ref{reg-irr} to achieve regularity.  Even if we can find such perturbations, we need to guarantee that these perturbations can be chosen to be compactly supported. We do not attempt to obtain general regularity results for 4-manifolds without admissible ends and we only focus on families of metrics that appear in the proof of the main theorem.

We will be ultimately interested to construct perturbations for the families of metrics $\bbW^j_k$ where $0\leq j <k\leq j+N+1\leq 2N+1$. Firstly let $ j <k\leq j+N$. Then there are some components of the family of metrics $\bbW^j_k$ without admissible ends. However, any such component is a GH component and the results of Subsection \ref{Mod-GH} allow us to use the trivial perturbation to achieve regularity. Therefore, we can run the above inductive argument again to construct a good perturbation for $\bbW^j_k$. Next, we have to treat the case that $k=j+N+1$. In this case, the component $X_N(l)$ does not have any admissible end and is not a Gibbons-Hawking manifold. The next two subsections will be concerned with regularity for the families of metrics $\bbX_N(l)$ and $\bbW^j_{j+N+1}$.

\subsection{Regularity on $X_N(l)$}\label{XNl-reg}

\begin{lemma}\label{crvspr}
	For $1\le N\le 3$, $2\le k\le 3$ or $(N,k)=(4,2)$, let $A$ be a $\U(k)$-connection on $X_N(l)$
	which is ASD with respect to a (possibly broken) metric.
	Suppose also $\chi$ denotes the restriction of $A$
	to $L(N,1)$.
	Then there is a completely reducible $\U(k)$-connection $A'$
	such that $c_1(A')=c_1(A)$, $\kappa(A')\le \kappa(A)$ and
	the restriction of $A'$ to $L(N,1)$ is $\chi$.
	The same conclusion holds if we only assume that $|\!|F^+(A)|\!|_{L^2}$ is small enough.
\end{lemma}
\begin{proof}
	Throughout the proof, we use the notations introduced in Section \ref{comp-red}.
	Let $c_1(A)=i_1e_1+i_2e_2+\cdots +i_N e_N$ and $v=(i_1,\cdots,i_N)$.
	Let also $\chi=\zeta^{s_1}\oplus \dots \oplus \zeta^{s_k}$ where $0\le s_i<N$.
	Note that $[v]_+\equiv s_1+\dots+s_k$ mod $N$.
	If $w_1,\cdots,w_k\in \mathbf{Z}^N$ satisfy:
	\begin{equation}\label{vwdecom}
	  v=w_1+\cdots +w_k \hspace{2cm} [w_i]_+\equiv s_i ~\mod N,
	\end{equation}
	then the completely reducible connection $A'=B_{w_1}(g)\oplus \cdots\oplus B_{w_k}(g)$ has the same limiting
	flat connection as $A$ on $L(N,1)$ and $c_1(A)=c_1(A')$.	 We say $w_1,\cdots, w_k$ is a {\it nice}
	$(s_1,\dots,s_k)$-decomposition of $v$, if \eqref{vwdecom} holds and $\kappa(A')$ is not greater than $\kappa(A)$ for any ASD connection
	$A$ with $c_1(A)=v$ and the limiting flat connection $\chi$ on $L(N,1)$.
	    \begin{table}[H]
            	\begin{centering}
            	\begin{tabular}{|c|c|c|c|c|c|c|c|c|c|c|c|c|c|}
            		\hline
            		$v$&\multicolumn{2}{c|}{$(0,0)$}&$(0,1)$\\
            		\hline
            		$(s_1,s_2)$&$(0,0)$&$(1,1)$&$(0,1)$\\
            		\hline
            		$w_1$&$(0,0)$&$(0,1)$&$(0,0)$\\
            		$w_2$&$(0,0)$&$(1,0)$&$(0,1)$\\
            		\hline
                  	$\kappa(A')$&$0$&$\frac{1}{2}$&$\frac{1}{8}$                          \\
                 	\hline
            	\end{tabular}
            	\caption{$N=2$, $k=2$}\label{N2r2}
            	\end{centering}
            \end{table}
	We have the following operations to simplify the discussion:
	\begin{itemize}
	\item Since the vector $(1,\cdots,1)\in \Z^N$ represents the trivial cohomology class in $X_N(l)$,
	   	if $v$ admits a nice $(s_1,\dots,s_k)$-decomposition, then $v+(m,\cdots,m)$ admits such a decomposition, too.
	\item If $v$ admits a nice $(s_1,\dots,s_k)$-decomposition, then applying a permutation to the entries of $v$
	        produces another vector with a nice $(s_1,\dots,s_k)$-decomposition.
	        Furthermore, for $\sigma\in S_k$, if $v$ admits a nice $(s_1,\dots,s_k)$-decomposition, then
	        it also admits a nice $(s_{\sigma(1)},\dots,s_{\sigma(1)})$-decomposition.
	\item Given an arbitrary vector $u\in \Z^N$, if $w_1$, $\cdots$, $w_k$
	give a nice $(s_1,\dots,s_k)$-decomposition of $v$,
		then the vectors $w_1+u$, $\cdots$, $w_k+u$ give a nice
		$(s_1+[u]_+,\dots,s_k+[u]_+)$-decomposition of $v+ku$.
	\item If $w_1$, $\cdots$, $w_k$ determines a nice $(s_1,\dots,s_k)$-decomposition of $v$, then
		$-w_1$, $\cdots$, $-w_k$ gives a nice $(-s_1,\dots,-s_k)$-decomposition of $-v$,
	\end{itemize}
            \begin{table}[H]
            	\begin{centering}
            	\begin{tabular}{|c|c|c|c|c|c|c|c|c|c|c|c|c|c|}
            		\hline
            		$v$&\multicolumn{2}{c|}{$(0,0)$}&\multicolumn{2}{c|}{$(0,1)$}\\
            		\hline
            		$(s_1,s_2,s_3)$&$(0,0,0)$&$(0,1,1)$&$(0,0,1)$&$(1,1,1)$\\
            		\hline
            		$w_1$&$(0,0)$&$(0,0)$&$(0,0)$&$(0,1)$    \\
            		$w_2$&$(0,0)$&$(0,1)$&$(0,0)$&$(0,1)$   \\
	                 $w_3$&$(0,0)$&$(1,0)$&$(0,1)$&$(1,0)$     \\
            		\hline
                   	$\kappa(A')$&$0$&$\frac{1}{2}$&$\frac{1}{6}$&$\frac{2}{3}$                               \\
                    \hline
            	\end{tabular}
            	\caption{$N=2$, $k=3$}\label{N2r3}
            	\end{centering}
            \end{table}
	We wish to construct a nice $(s_1,\dots,s_k)$-decomposition for any vector $v$ that $[v]_+\equiv s_1+\dots+s_k$ mod $N$.
	If $N=1$, we can use the first operation to reduce the problem to the case that $v=(0)$.
	In this case, we can pick $w_1=\dots=w_k=(0)$.
	That is to say, $A'$ is the trivial connection. Since $\kappa(A)$ is a non-negative number for any ASD connection
	$A$,
	we obtain a nice decomposition of $v$. For $2\leq N \leq 3$ and $2\leq k \leq 3$ and $(N,k)=(4,2)$, we can still
	use the above four operations to reduce our problem into checking the existence of nice decompositions for the vectors in
	Tables \ref{N2r2}, \ref{N2r3}, \ref{N3r2}, \ref{N3r3} and \ref{N4k2}. Except two special cases in Tables \ref{N3r3} and \ref{N4k2}, the reducible connections $A'$ given in
	these tables have energy smaller than $1$.
	Let $A$ be an ASD connection with $c_1(A)=c_1(A')$ and the same limiting flat connection as $A'$ on $L(N,1)$. Then
	$\kappa(A)$ is non-negative and $\kappa(A)-\kappa(A')$ is an integer. Therefore, except the cases that $\kappa(A')=1$,
	we can immediately conclude that $\kappa(A')\leq \kappa(A)$.
	Note that the same argument would apply if $|\!|F^+(A)|\!|_{L^2}$ is small enough.
	In the exceptional cases, if $\kappa(A)<\kappa(A')$, then $A$
	has to be a flat connection. Since $X_N(l)$ is simply connected, $A$ must be the trivial flat connection.
	However, the limiting flat connection of $A$ are required to be non-trivial which is a contradiction.
	Therefore, the moduli space of ASD connections with vanishing topological energy and the same $c_1$ and
	the same limiting flat connection on $L(N,1)$ as $A'$ is empty.
	This also implies that this moduli space is empty if we consider the perturbed ASD equation for a small perturbation.
	This completes the proof.
            \begin{table}
            	\begin{centering}
            	\begin{tabular}{|c|c|c|c|c|c|c|c|c|c|c|c|c|c|}
            		\hline
            		$v$&\multicolumn{2}{c|}{$(0,0,0)$}&\multicolumn{2}{c|}{$(0,0,1)$}\\
            		\hline
            		$(s_1,s_2)$&$(0,0)$&$(1,2)$&$(0,1)$&$(2,2)$\\
            		\hline
            		$w_1$&$(0,0,0)$&$(0,0,1)$&$(0,0,0)$&$(0,1,1)$      \\
            		$w_2$&$(0,0,0)$&$(1,1,0)$&$(0,0,1)$&$(1,0,1)$    \\
            		\hline
                    	$\kappa(A')$&$0$&$\frac{2}{3}$&$\frac{1}{6}$&$\frac{1}{2}$                                       \\
                    \hline
            	\end{tabular}
            	\caption{$N=3$, $k=2$}\label{N3r2}
            	\end{centering}
            \end{table}
            \begin{table}[H]
            	\begin{centering}
            \scalebox{0.9}{
            	\begin{tabular}{|c|c|c|c|c|c|c|c|c|c|c|c|c|c|}
            		\hline
            		$v$&\multicolumn{3}{c|}{$(0,0,0)$}&\multicolumn{3}{c|}{$(0,0,1)$}&\multicolumn{3}{c|}{$(0,1,2)$}\\
            		\hline
            		$(s_1,s_2,s_3)$&$(0,0,0)$&$(0,1,2)$&$(1,1,1)$&$(0,0,1)$&$(1,1,2)$&$(0,2,2)$&$(0,0,0)$&$(0,1,2)$&$(1,1,1)$\\
            		\hline
            		$w_1$&$(0,0,0)$&$(0,0,0)$&$(0,0,1)$&$(0,0,0)$&$(0,0,1)$ &$(0,0,0)$   &$(0,0,0)$&$(0,0,0)$&$(0,0,1)$     \\
            		$w_2$&$(0,0,0)$&$(0,0,1)$&$(0,1,0)$&$(0,0,0)$&$(0,1,0)$ &$(1,0,1)$  &$(0,0,0)$&$(0,0,1)$&$(0,0,1)$     \\
                    $w_3$&$(0,0,0)$&$(1,1,0)$&$(1,0,0)$&$(0,0,1)$&$(1,0,1)$ &$(0,1,1)$ &$(-1,0,1)$&$(0,1,1)$&$(0,1,0)$   \\
            		\hline
                    $\kappa(A')$&$0$&$\frac{2}{3}$&$1$&$\frac{2}{9}$&$\frac{8}{9}$&$\frac{5}{9}$&$\frac{2}{3}$&$\frac{1}{3}$&$\frac{2}{3}$     \\
                    \hline
            	\end{tabular} }
            	\caption{$N=3$, $k=3$}\label{N3r3}
            	\end{centering}
            \end{table}
            	\begin{table}[H]
	\begin{centering}
	            \scalebox{0.9}{
            	\begin{tabular}{|c|c|c|c|c|c|c|c|c|c|c|c|c|c|c|c|c|}
            		\hline
            		$v$&\multicolumn{3}{c|}{$(0,0,0,0)$}&\multicolumn{2}{c|}{$(0,0,0,1)$}&\multicolumn{2}{c|}{$(0,0,1,1)$}\\
            		\hline
            		$(s_1,s_2)$&$(1,3)$&$(0,0)$&$(2,2)$&$(1,0)$&$(2,3)$&$(2,0)$&$(1,1)$\\
            		\hline
            		$w_1$&$(0,0,0,1)$&$(0,0,0,0)$&$(0,0,1,1)$&$(0,0,0,1)$&$(0,0,1,1)$&$(0,0,1,1)$&$(0,0,0,1)$\\
            		$w_2$&$(0,0,0,-1)$&$(0,0,0,0)$&$(0,0,-1,-1)$&$(0,0,0,0)$&$(0,0,-1,0)$&$(0,0,0,0)$&$(0,0,1,0)$\\
            		\hline
                    	$\kappa(A')$&$\frac{3}{4}$&$0$&1&$\frac{3}{16}$&$\frac{11}{16}$&$\frac{1}{4}$&$\frac{1}{2}$\\
                    \hline
            	\end{tabular}}
	\caption{$N=4$, $k=2$}\label{N4k2}
	\end{centering}
	\end{table}
\end{proof}

The above lemma can be partially generalized to the case that $(N,k)=(4,3)$:
\begin{lemma}\label{crvspr2}
	Suppose $X_4(l)$ is equipped with a (possibly broken) cylindrical metric.
	Let $A$ be a $\U(3)$-connection over $X_4(l)$ such that $\ind (\mathcal D_A)\ge-5$ with $\beta$ being the restriction of $A$ to $S^1\times S^2$.
	Then there is a completely reducible ASD $\U(k)$ connection $A'$ such that
	$c_1(A')=c_1(A)$, the limiting flat connections on $L(N,1)$ are the same and
	$\kappa(A')\le \kappa(A)$.
\end{lemma}
 \begin{proof}
	We can use the operations introduced in Lemma \ref{crvspr}
	to reduce the problem to the cases
	shown in Tables \ref{N4k3-1} and \ref{N4k3-2}. In this table, for each choice of $v$, $s_1$, $s_2$ and $s_3$,
	we give a triple of vectors $w_1$, $w_2$ and $w_3$
	in $\Z^4$. As it is clear from the table, the corresponding completely reducible ASD connection $A'$ might have energy greater than $1$.
	By the index formula, we have:
	\begin{equation*}
	   \ind(\mathcal D_{A})=4k \kappa(A)-(k^2-1)(\frac{\chi(X_N(l))+\sigma(X_N(l))}{2})
		+\frac{-h^0(\chi)+\rho(\chi)}{2}-h^0(\beta)\\	
	\end{equation*}
	Let also $\beta'$ denote the limiting flat connection of $A'$ on $S^1\times S^2$. Therefore, the following difference:
	\begin{equation*}
	   (\ind(\mathcal D_{A})+h^0(\beta))-( \ind(\mathcal D_{A'})+h^0(\beta'))
	\end{equation*}
	is divisible by $12$ because it is equal to $12(\kappa(A)-\kappa(A'))$. From Tables  Tables \ref{N4k3-1} and \ref{N4k3-2}, it is clear that
	$\ind(\mathcal D_{A'})+h^0(\beta')\le 8$. Since $h^0(\beta)\ge 2$, the assumption implies that $\ind (\mathcal D_A)+h^0(\beta)\ge -3$.
	Therefore, we must have $\kappa(A')\le \kappa(A)$.

	\begin{table}[H]
            	\begin{centering}
            \scalebox{0.63}{
            	\begin{tabular}{|c|c|c|c|c|c|c|c|c|c|c|c|c|c|c|c|}
            		\hline
            		$v$&\multicolumn{4}{c|}{$(0,0,0,0)$}&\multicolumn{5}{c|}{$(0,0,0,1)$}\\
            		\hline
            		$(s_1,s_2,s_3)$&$(0,0,0)$&$(1,1,2)$&$(2,2,0)$&$(0,1,3)$&$(0,0,1)$&$(1,1,3)$&$(2,2,1)$&$(3,3,3)$&$(0,2,3)$\\
            		\hline
            		$w_1$&$(0,0,0,0)$&$(0,0,0,1)$&$(0,0,1,1)$&$(0,0,0,0)$&$(0,0,0,0)$&$(0,0,0,1)$ &$(0,0,-1,-1)$   &$(0,1,1,1)$&$(0,0,0,0)$\\
            		$w_2$&$(0,0,0,0)$&$(0,0,1,0)$&(0,0,-1,-1)&$(0,0,0,1)$&$(0,0,0,0)$&$(0,0,0,1)$ &$(0,0,1,1)$  &$(0,-1,0,0)$&$(0,0,1,1)$\\
			$w_3$&$(0,0,0,0)$&$(0,0,-1,-1)$&(0,0,0,0)&$(0,0,0,-1)$&$(0,0,0,1)$&$(0,0,0,-1)$ &$(0,0,0,1)$ &$(0,0,-1,0)$&$(0,0,-1,0)$\\
            		\hline
            		$\ind(\mathcal D_{A'})+h^0(\beta')$&$-8$&$8$&$4$&2&$-4$&$4$&$8$&$4$&$2$\\
            		\hline
                    $\kappa(A')$&$0$&$\frac{5}{4}$&1&$\frac{3}{4}$&$\frac{1}{4}$&$1$&$\frac{5}{4}$&$1$&$\frac{3}{4}$ \\
                    \hline
            	\end{tabular} }
            	\caption{$N=4$, $k=3$}\label{N4k3-1}
            	\end{centering}
            \end{table}
	\begin{table}[H]
            	\begin{centering}
            \scalebox{0.68}{
            	\begin{tabular}{|c|c|c|c|c|c|c|c|c|c|c|c|c|c|c|c|}
            		\hline
            		$v$&\multicolumn{4}{c|}{$(0,0,1,1)$}&\multicolumn{4}{c|}{$(0,0,1,2)$}\\
            		\hline
            		$(s_1,s_2,s_3)$&$(0,0,2)$&$(1,1,0)$&$(2,2,2)$&$(1,2,3)$&$(0,0,3)$&$(1,1,1)$&$(3,3,1)$&$(0,1,2)$\\
            		\hline
            		$w_1$&$(0,0,0,0)$&$(0,0,0,1)$&$(0,0,1,1)$&$(0,0,0,1)$&$(0,0,0,0)$&$(0,0,0,1)$ &$(0,1,1,1)$&$(0,0,0,0)$\\
            		$w_2$&$(0,0,0,0)$&$(0,0,1,0)$&(1,1,0,0)&$(0,0,1,1)$&$(0,0,0,0)$&$(0,0,1,0)$   &$(0,-1,0,0)$&$(0,0,0,1)$\\
			$w_3$&$(0,0,1,1)$&$(0,0,0,0)$&(-1,-1,0,0)&$(0,0,0,-1)$&$(0,0,1,2)$&$(0,0,0,1)$ &$(0,0,0,1)$&$(0,0,1,1)$\\
            		\hline
            		$\ind(\mathcal D_{A'})+h^0(\beta')$&$-4$&$0$&$8$&6&$4$&$0$&$0$&$-2$\\
            		\hline
                    $\kappa(A')$&$\frac{1}{3}$&$\frac{7}{12}$&$\frac{4}{3}$&$\frac{13}{12}$&$\frac{11}{12}$&$\frac{2}{3}$&$\frac{2}{3}$&$\frac{5}{12}$\\
                    \hline
            	\end{tabular} }
            	\caption{$N=4$, $k=3$ (continued)}\label{N4k3-2}
            	\end{centering}
            \end{table}
\end{proof}

Now we are ready to prove the Proposition \ref{avoid-noncomp-red-pre} from Subsection \ref{comp-red}. At various points in the proof, we use an elementary observation about the indices of direct sums of two connections. Suppose $A_1$ and $A_2$ are two connections of rank $k$ on a pair $(X,c)$ which satisfy either Condition \ref{noS1S2} or \ref{S1S2}. Suppose also the restriction of these two connections on the admissible and the lens space ends are equal to each other. Suppose also $B$ is another connection of rank $k'$ on $(X,c')$. Then we have:
\[
  \frac{\ind(\mathcal D_{A_1\oplus B})-\ind(\mathcal D_{A_2\oplus B})}{k+k'}=
  \frac{\ind(\mathcal D_{A_1})-\ind(\mathcal D_{A_2})}{k}.
\]

\begin{prop}\label{avoid-noncomp-red}
	Suppose $N\leq 4$, $c$ is an arbitrary 2-cycle in $X_N(l)$. Suppose $M$ is a finite subset of $\Delta_{N-1}^{\ft,\circ}$.
	There exists an arbitrary small perturbation of the ASD equation over $\mathbb{X}_N(l)$ such that for any path $p$ along $(X,c)$,
	with limiting flat connections
	$\chi_0$ and $\beta_0$ on $L(N,1)$ and $S^1\times S^2$,
	we have:
	\begin{enumerate}
	  \item[(i)]  if $\ind(p)< -h^0(\beta_0)-h^0(\chi_0)$, then the moduli space
	  $\mathcal M^\omega_{p}(\mathbb{X}_N(l),c;\Delta_{N-1}^{\ft})$ is empty;
	  \item[(ii)]  if $\ind(p)= -h^0(\beta_0)-h^0(\chi_0)$, then
	  any element of $\mathcal M^\omega_{p}(\mathbb{X}_N(l),c;M)$ is a completely reducible connection associated to
	  an element of $\mathcal {K}_{N+1}^\circ$.
	\end{enumerate}
	\vspace{-5pt}
	Moreover, there exists a neighborhood $\mathcal V$ of $M$ such that
	$\overline{\mathcal V}\subset \Delta_{N-1}^{\ft,\circ}$ and $\mathcal M_{p}(\mathbb{X}_N(l),c;\overline{\mathcal V})$ is compact.
\end{prop}
\begin{proof}
	Recall that the family of metrics $\bbX_N(l)$ is parametrized by $\mathcal K_{N+1}$.
	We use a similar inductive strategy as in Subsection \ref{per-family} to define the perturbation $\omega$.
	That is to say, we construct $\omega$ on open faces $F$ of $\mathcal K_{N+1}$ by induction on $\dim(F)$.
	We also introduce an open neighborhood $\mathcal V$ of
	$M$ in $\Delta_{N-1}^{\ft,\circ}$ which we might shrink in each step of our inductive construction.
	Firstly, let $T$ be an $(N+1)$-ribbon tree such that
	$F_T$ is a vertex of $\mathcal K_{N+1}$. This vertex determines
	a decomposition of $X_{N}(l)$ into
	$N$ connected components which are denoted by $W_0$, $\dots$, $W_{N-1}$.
	We also denote $c_i$ for $W_i\cap c$.
	Let $W_0$ be the connected component
	which has $S^1\times S^2$ as one of its boundary components. Therefore, the other components are
	Gibbons-Hawking manifolds. The vertex $F_T$ gives rise to a metric on each component $W_i$.
	We apply Proposition \ref{partial-reg} to $W_0$ and the set $M$ to obtain a residual subset
	$\mathcal W_{\rm reg}$ of $\mathcal W$. Let $\omega_0$ be a small element of $\mathcal W_{\rm reg}$.
	We also fix the trivial perturbation on the GH components.
	These perturbations determine a perturbation of the ASD equation for the vertex $F_T$
	of $\mathcal K_{N+1}$.
	
	Suppose $z=[A,B_1,B_2\dots,B_{N-1}]$ is an element of
	$\mathcal M_p^{\omega}(\pi^{-1}(F_T),c;\Delta_{N-1}^{\ft})$ .
	Here $B_i$ is a connection on $(W_i,c_i)$ which satisfies the ASD equation.
	The connection $A$ defined on $(W_0,c_0)$ satisfies a perturbed ASD equation.
	Let $\beta$ be the limiting flat connection of $A$ on $S^1\times S^2$.
	We claim that if $\ind(p)< -h^0(\beta_0)-h^0(\chi_0)$, then
	the moduli space $\mathcal M_p^{\omega}(\pi^{-1}(F_T),c;\Delta_{N-1}^{\ft})$ is empty.
	We shall also show that there is an open space $\mathcal V_0\subset \Delta_{N-1}^{\ft,\circ}$
	containing $M$ such that the moduli space $\mathcal M_p^{\omega}(\pi^{-1}(F_T),c;\mathcal V)$,
	for any $p$ with $\ind(p)=-h^0(\beta_0)-h^0(\chi_0)$,
	is empty. In any of these two cases, if the connection $A$ is irreducible,
	then $z$ is a regular connection. Thus, $\ind(p)\geq -(N-1)$
	which is a contradiction. Therefore, we can assume that $A=A_1\oplus \dots\oplus A_m$ where $A_i$
	is irreducible and the rank of $A_i$ is at most $3$.
	For the sake of exposition, we prove our claims in two special cases for $N$, $l$ and $T$.
	The more general case can be proved by similar tricks.
	
	Firstly, let $N=4$ and $F_T$ determine the decomposition of $X_4(3)$ given in Figure \ref{X4(1)-decom}.
	In this figure, $W_2$ is diffeomorphic to the Gibbons-Hawking manifold
	$\overline X_2$, introduced in Remark \ref{comp-red-barXN}.
	Moreover, $W_1\#W_2$ is diffeomorphic to $\overline X_3$ and gluing the components
	$W_0$, $W_1$ and $W_2$ gives rise to the manifold $X_3(3)$.
	Let $p$ be a path along $(X_4(1),c)$ such that $\ind(p)\leq -h^0(\beta_0)-h^0(\chi_0)$.
	Let also $z=[A,B_1,B_2,B_3]\in \mathcal M_p^{\omega}(\pi^{-1}(F),c;\Delta_{3}^{\ft})$.
	We use Remark \ref{comp-red-barXN} to replace the (broken) ASD connection $[B_1,B_2]$
	with a completely reducible ASD connection $[B_1',B_2']$ with the same limiting flat connection as $[B_1,B_2]$
	on the boundary component $L(3,1)$ of $W_1\#W_2$. The choice of metrics on $W_1$ and $W_2$
	implies that the index of the path associated to $[B_1',B_2']$
	is not greater than that of the path associated to $[B_1,B_2]$.

	\begin{figure}
	\centering
           \begin{tikzpicture}

                \draw[fill] (0,0) circle (3pt);
                \draw[thick] (1,0) circle (3pt);
                \draw[fill] (2,0) circle (3pt);
                \draw[fill] (3,0) circle (3pt);
                \draw[fill] (4,0) circle (3pt);

                \draw[thick] (2,0) ellipse (3cm and 2cm);
                \draw[thick] (0.5,0) ellipse (0.8cm and 0.6cm);;
                \draw[thick] (3,0) ellipse (1.4cm and 1.1cm);
                \draw[thick] (2.5,0) ellipse (0.8cm and 0.6cm);

		\node[below] at (0.5,0) {$W_3$};
		\node[below] at (2.5,0) {$W_2$};
		\node[below] at (3.5,-0.3) {$W_1$};
		\node[below] at (1.5,-1) {$W_0$};
	\end{tikzpicture}
            	\caption{}
                	\label{X4(1)-decom}
        \end{figure}
	
	Since $A=A_1\oplus\dots\oplus A_m$, the broken perturbed ASD connection $[A,B_1',B_2']$
	can be also written as the direct sum of $m$ broken connections which we denote $z_1$, $\dots$, $z_m$.
	Each $z_i$ has rank at most $3$. Thus, Lemma \ref{crvspr} asserts that for each $z_i$,
	there is a completely reducible connection $y_i$
	with the same rank, the first Chern Class, the same restriction to the boundary component $L(3,1)$ of
	$X_3(3)=W_0\#W_1\#W_2$. Moreover, $\kappa(y_i)\leq \kappa(z_i)$.
	The broken completely reducible connection $y_1\oplus\dots\oplus y_m$ of rank $4$ on
	$W_0\#W_1\#W_2$ has the same $c_1$ and the restriction to $L(3,1)$ as $[A,B_1,B_2]$.
	But the energy of $y_1\oplus \dots\oplus y_m$ is not greater than the energy of
	$[A,B_1',B_2']$. Let $p_r$ and $\beta_r$,
	respectively, denote the path determined by $y_1\oplus\dots\oplus y_m$ and the restriction of this connection to
	$S^1\times S^2$. Let also $p'$ denote the path associated to $[A,B_1',B_2']$.
	Dimension formula and Proposition \ref{ind-Bv} imply that:
	\begin{equation}\label{index-ineq}
	  \ind(p')+h^0(\beta')\geq \ind(p_r)+h^0(\beta_r)\geq -h^0(\chi')
	\end{equation}
	where $\beta'$ and $\chi'$ are the restrictions of $[A,B_1',B_2']$
	to the boundary components $S^1\times S^2$ and $L(3,1)$ of
	$W_0\#W_1\#W_2$. Inequalities in \eqref{index-ineq} immediately imply that if $\ind(p)< -h^0(\beta_0)-h^0(\chi_0)$,
	then $\mathcal M_p^{\omega}(\pi^{-1}(F),c;\Delta_{3}^{\ft})$ is empty.
	If $\ind(p)=-h^0(\beta_0)-h^0(\chi_0)$, then in all of the above inequalities, we should have equality.
	Consequently:
	\begin{equation}\label{index-ineq-2}
	  \ind(\mathcal D_{z_i})+h^0(\beta_i)=\ind(\mathcal D_{y_i})+h^0(\alpha_i)=-h^0(\chi_i)
	\end{equation}
	Here $\alpha_i$ and $\beta_i$ are the restrictions of $y_i$ and $z_i$ to $S^1\times S^2$.
	The flat connection $\chi_i$ is also given by the restriction of these two connections to $L(3,1)$.
	In this case, if the restriction of $A$ to $S^1\times S^2$ belong to $M$, then
	Proposition \ref{partial-reg} implies that:
	\[
	  \ind(\mathcal D_{z_1})+\dots+\ind(\mathcal D_{z_m})\geq m-1.
	\]
	which contradicts \eqref{index-ineq-2}.
	
	Emptiness of the moduli spaces of the paths with index smaller than $-h^0(\beta_0)-h^0(\chi_0)$ implies
	that the moduli space $\mathcal M_p^{\omega}(\pi^{-1}(F),c;\Delta_{3}^{\ft})$ is compact for any path
	$p$ with $\ind(p)=-h^0(\beta_0)-h^0(\chi_0)$. In particular, there is a neighborhood $\mathcal V_0$ of
	$M$ such that $\overline{\mathcal V_0}\subset \Delta_{N-1}^{\ft,\circ}$ and
	$\mathcal M_p^{\omega}(\pi^{-1}(F),c;\overline{\mathcal V_0})$ is empty. The only undesirable point
	is that $\omega_0$, the restriction of $\omega$ to $W_0$, is not necessarily compactly supported.
	However, $\omega_0$ can be approximated by compactly supported perturbations. Thus,
	if a compactly supported perturbation is close enough to $\omega_0$, then
	$\mathcal M_p^{\omega}(\pi^{-1}(F),c;\Delta_{3}^{\ft})$ is empty if $\ind(p)<-h^0(\beta_0)-h^0(\chi_0)$, and
	$\mathcal M_p^{\omega}(\pi^{-1}(F),c;\mathcal V_0)$ is empty if $\ind(p)=-h^0(\beta_0)-h^0(\chi_0)$.

	 Next, let $N=4$ and $F_T$ determine the decomposition of $X_4(4)$ given in Figure \ref{X4(0)-decom}.
	 In this figure, $W_2$ and $W_3$ are diffeomorphic to the Gibbons-Hawking manifold
	$\overline X_2$. Moreover, $W_2\#W_1\#W_3$ is diffeomorphic $\overline X_4$.
	Essentially the same argument as in the previous case can be used to find a compactly supported
	perturbation $\omega$ on $W_0$ and the open set $\mathcal V_0$. The only difference is that we might need
	to use Lemma \ref{crvspr2} instead of Lemma \ref{crvspr}.
	The index requirement for Lemma \ref{crvspr2}
	can be also guaranteed by Proposition \ref{partial-reg}.
	\begin{figure}
	\centering
            \begin{tikzpicture}
                \draw[thick] (0,0) circle (3pt);
                \draw[fill] (1,0) circle (3pt);
                \draw[fill] (2,0) circle (3pt);
                \draw[fill] (3,0) circle (3pt);
                \draw[fill] (4,0) circle (3pt);

                \draw[thick] (2,0) ellipse (3cm and 2cm);
                \draw[thick] (2.5,0) ellipse (2cm and 1.4cm);
                \draw[thick] (1.5,0) ellipse (0.8cm and 0.6cm);
                \draw[thick] (3.5,0) ellipse (0.8cm and 0.6cm);

                \node[below] at (0,-0.6) {$W_0$};
                \node[below] at (2.5,-0.5) {$W_1$};
                \node[below] at (1.5,-0) {$W_2$};
                \node[below] at (3.5,-0) {$W_3$};

           \end{tikzpicture}
            	\caption{}
                	\label{X4(0)-decom}
     \end{figure}

	Following the above construction,
	we can inductively extend the perturbation $\omega$ on the vertices of $\mathcal K_{N+1}$ into
	an admissible perturbation $\omega$ on any open face $F$ of $\mathcal K_{N+1}$ with codimension at least one,
	and find a neighborhood $\mathcal V$ of $M$ with $\overline{\mathcal V}\subset \Delta_{N-1}^{\ft,\circ}$
	such that the following properties hold for any path $p$:
	\begin{itemize}
	  \item If $\ind (p) < -h^0(\beta_0)-h^0(\chi_0)$, then the moduli space
		$\mathcal M_{p}^\omega(\pi^{-1}(F),c)$ is empty.
	  \item If $\ind (p)=-h^0(\beta_0)-h^0(\chi_0)$, then the moduli space
	  	$\mathcal M_{p}^\omega(\pi^{-1}(F),c;\overline{\mathcal V})$ is empty.
	\end{itemize}
	The perturbation $\omega$ determines a perturbation
	in a neighborhood of the boundary of $\mathcal K_{N+1}$, which is admissible and the analogues of the above
	two conditions hold. As the last step we extend $\omega$ to the interior of $\mathcal K_{N+1}$.
	The argument in this case is also very similar to the case of the lower dimensional faces.
	
	We use Proposition \ref{S1S2-partial-reg-rel} to extend $\omega$ to $\mathcal K_{N+1}$ such that any element of
	$\mathcal M_{p}^\omega(\bbX_{N}(l),c;\Delta_{N-1}^{\ft})$ is partially regular for any path $p$
	that $\ind (p) < -h^0(\beta_0)-h^0(\chi_0)$ and its restriction to $L(N,1)$ is equal to $\chi_0$.
	Moreover, if $p$ is a path whose index is $-h^0(\beta_0)-h^0(\chi_0)$ and whose restriction to $L(N,1)$
	is equal to $\chi_0$, then we require that $\mathcal M_{p}^\omega(\bbX_{N}(l),c;\overline{\mathcal V})$
	is partially regular and the map $r:\mathcal M_{p}^\omega(\bbX_{N}(l),c;\mathcal V) \to \mathcal V$
	is transversal to the inclusion of $M$ into $\mathcal V$. If a connection $A$ on $X_N(l)$
	represents an element of $\mathcal M_{p}^\omega(\bbX_{N}(l),c;\Delta_{N-1}^{\ft})$ for a path $p$
	with $\ind (p) \leq -h^0(\beta_0)-h^0(\chi_0)$, then $A$ has to be reducible. Let $A_1\oplus \dots\oplus A_m$
	be the decomposition of $A$ into irreducible summands. Proposition \ref{S1S2-partial-reg-rel} asserts that:
	\begin{equation} \label{inex-ineq-3}
	  \ind(\mathcal D_{A_1})+\dots+\ind(\mathcal D_{A_m})\geq m+1-2N
	\end{equation}
	In particular, if $A_i$ has rank $3$, then $N=4$, $m=2$, and we have $\ind(\mathcal D_{A_i})\geq -5$.
	By Lemma \ref{crvspr} or Lemma \ref{crvspr2}, we can find a completely reducible connection $A_i'$ such that $c_1(A_i)=c_1(A_i')$,
	$\kappa(A_i)\leq\kappa(A_i')$, and the restriction of $A_i$ and $A_i'$ to $L(N,1)$ agree with each other.
	Let $A'=A_1'\oplus\dots\oplus A_m'$, $\alpha$ be the limiting value of $A'$ on $S^1\times S^2$, $\alpha_i$ be the limiting value of $A_i'$ on $S^1\times S^2$ and
	$\beta_i$ be the limiting value of $A_i$ on $S^1\times S^2$ . Since
	\[
	  \ind(p)+h^0(\beta_0)\geq \ind(\mathcal D_{A'})+h^0(\alpha)\geq -h^0(\chi_0),
	\]
	we should have $\ind(p)=-h^0(\chi_0)-h^0(\beta_0)$
	and $\ind(\mathcal D_{A_i})+h^0(\beta_i)=\ind(\mathcal D_{A_i'})+h^0(\alpha_i)=-h^0(\chi_i)$.
	Our inductive assumption and the emptiness of the moduli spaces associated to $(\bbX_N(l),c)$ and a path with index smaller than
	$-h^0(\chi_0)-h^0(\beta_0)$ imply that
	the moduli space $\mathcal M_{p}^\omega(\bbX_{N}(l),c;\overline{\mathcal V})$ is compact
	for a path $p$ with $\ind(p)=-h^0(\chi_0)-h^0(\beta_0)$.
	If $\beta$, the limiting flat connection of $A$ on $S^1\times S^2$, belongs to $M$, then we can use part (ii) of Proposition \ref{S1S2-partial-reg-rel} to improve the
	inequality in \eqref{inex-ineq-3} as follows:
	\begin{equation}
	  \ind(\mathcal D_{A_1})+\dots+\ind(\mathcal D_{A_m})\geq m-N.
	\end{equation}
	Therefore, we have:
	\begin{equation}\label{improved-ineq}
	  -\sum_{i=1}^{m} h^0(\chi_i)+h^0(\beta_i)\geq m-N.
	\end{equation}
	Since $h^0(\chi_i), h^0(\beta_i)\geq \rk(A_i)-1$, the above inequality implies that $m=N$, i.e., $A$
	is a completely reducible connection. 
\end{proof}

Suppose $\U(1)$-bundles $L_1$, $\dots$, $L_N$ on $X_N(l)$ are chosen such that the path $p$ determined by completely reducible connections on $L_1\oplus \dots \oplus L_N$ has index at most $-2(N-1)$. Propositions \ref{ind-Bv} and \ref{bi-per-con} imply that the cycle associated to $L_1\oplus \dots \oplus L_N$ is equal to $w_{\sigma,\tau}$ for injective maps $\sigma:[N-l]\to [N]$ and $\tau:[l]\to [N]$ with disjoint images. Up to permutation, the choices of the $\U(1)$-bundles $L_i$ are also unique. We can form the moduli space of completely reducible connections $\mathcal M_{cr}(\bbX_N(l),w_{\sigma,\tau})$ associated to the family of metrics $\bbX_N(l)$. Restriction to $S^1\times S^2$ gives a map $r_{\sigma,\tau}:\mathcal K_{N+1} \to \Delta_{N-1}^{\ft}$, which is the composition of $\widetilde {\hol}_{\sigma,\tau}$ and the projection map form $\ft$ to $\Delta_{N-1}^{\ft}$. We assume that the inclusion map of $M$ in $\Delta_{N-1}^{\ft,\circ}$ is transversal to $r_{\sigma,\tau}$. Then there are finitely many completely reducible connections on $\bbX_{N}(l)$ in $\mathcal M_{cr}(\bbX_N(l),w_{\sigma,\tau})$ such that their limiting values belong to $M$. Let $g\in \mathcal K_{N+1}$ be chosen such that that the limiting value of the associated completely reducible ASD connection is an element of $M$. Using the results of Subsection \ref{comp-red-reg}, we can make a small perturbation $\omega(g)$ of the ASD equation such that the completely reducible ASD connection becomes a regular element of the moduli space associated to $\bbX_{N}(l)$. Note that the type of perturbations that we use in Subsection \ref{comp-red-reg} do not change the set of completely reducible connections. We can repeat the proof of Proposition \ref{avoid-noncomp-red} using perturbations whose values at such $g\in \mathcal K_{N+1}$ are close to $\omega(g)$. In particular, we can show:

\begin{prop}\label{avoid-noncomp-red-version-2}
	Suppose $N\leq 4$ and $M$ is a finite subset of $\Delta_{N-1}^{\ft,\circ}$ such that
	the inclusion map of $M$ in $\Delta_{N-1}^{\ft,\circ}$ is transversal to the map $r_{\sigma,\tau}$.
	Then there exist an arbitrary small perturbation of the ASD equation over $\mathbb{X}_N(l)$
	and a neighborhood $\mathcal V$ of $M$ such that for any path $p$ along $(X,w_{\sigma,\tau})$,
	with limiting connections $\chi_0=1\oplus \dots\oplus \zeta^{N-1}$ and $\beta_0$ on $L(N,1)$ and $S^1\times S^2$,
	we have:
	\begin{enumerate}
	  \item[(i)]  if $\ind(p)< -2(N-1)$, then the moduli space
	  $\mathcal M_{p}^\omega(\mathbb{X}_N(l),w_{\sigma,\tau};\Delta_{N-1}^{\ft})$ is empty;
	  \item[(ii)]  if $\ind(p)= -2(N-1)$, then
	  any element of $\mathcal M_{p}^\omega(\mathbb{X}_N(l),w_{\sigma,\tau};\mathcal V)$ is a regular completely reducible connection associated to
	  an element of $\mathcal {K}_{N+1}^\circ$.
	\end{enumerate}
\end{prop}

\subsection{Regularity on $W^j_k$}\label{good-perb}
In Subsection \ref{per-family}, we explained how one can inductively define good perturbations for the families of metrics $\bbW_k^j$ where $0\leq j< k\leq j+N\leq 2N$. In this section, we extend this construction to the families of metrics $\bbW^l_{l+N+1}$, for $N\leq 4$ and $0\leq l \leq N$, such that it satisfies the properties which were used in the proof of Theorem \ref{main}. We follow the same notation as in Subsection \ref{proof-thm}. The face of $\mathcal K_{N+2}$ corresponding to the cut $S^1\times S^2$ is denoted by $\overline F_0$. We are interested in the moduli spaces associated to the 2-cycle $c_{\sigma,\tau}\subset W^l_{l+N+1}$ where $\sigma:[N-l]\to [N]$ and $\tau:[l] \to [N]$ are two injective maps. We will write $S$ and $T$ for $[N]\backslash {\rm image}(\sigma)$ and ${\rm image}(\tau)$. Removing the cut $S^1\times S^2$ from $(W^l_{l+N+1},c_{\sigma,\tau})$ produces pairs $(W_0,c_0)$ and $(X_N(l),w_{\sigma,\tau})$. The family of metrics on the face $\overline F_0$ is given by a fixed metric on $W_0$ and the family $\bbX_N(l)$.

The pair $(W_0,c_0)$ has admissible ends. Therefore, we can use the results discussed in Subsection \ref{reg-irr} to fix a compactly supported perturbation $\omega_0$ on $W_0$ such that the moduli space $\mathcal M_p^{\omega_0}(W_0,c_0;\Gamma)$ for any choice of a path $p$ along $(W_0,c_0)$ and an open face $\Gamma$ of $\Delta_{N-1}^\ft$ is regular. In particular, if $\ind(p)<-(N-1)$, then the moduli space $\mathcal M_p^{\omega_0}(W_0,c_0;\Delta_{N-1}^{\ft})$ is empty. In the case that $\ind(p)=-(N-1)$, $\mathcal M_p^{\omega_0}(W_0,c_0;\Gamma)$ is also empty unless $\Gamma=\Delta_{N-1}^{\ft,\circ}$ in which case the moduli space is a compact $0$-dimensional manifold. If we fix $\alpha\in \fC_*^N(Y_l,\gamma_S)$ and $\beta\in \fC_*^N(Y_l,\gamma_T)$, there is at most one path $p$ such that $\mathcal M_p^{\omega_0}(W_0,c_0;\Delta_{N-1}^\ft)$ is $0$-dimensional and the restriction of $p$ to the two admissible ends are $\alpha$, $\beta$. We denote this (possibly empty) moduli space by $M(\alpha,\beta)$. Using Proposition \ref{S1S2-reg-rel} and Remark \ref{admissible-comp-supp-per}, we can assume that for any choice of $\alpha$ and $\beta$, the restriction map $r_0:M(\alpha,\beta) \to \Delta_{N-1}^{\ft,\circ}$ is transversal to the maps $\hol_S:\Delta_{N-1}\to \Delta_{N-1}^{\ft}$ and $H_S:[0,1]\times\Delta_{N-1}\to \Delta_{N-1}^{\ft}$. In particular, $r_0$ is transversal to the maps $r_{\sigma,\tau}$, defined in the previous subsection.

Let $\omega_1$ be a perturbation of the ASD equation on $(\bbX_N(l),w_{\sigma,\tau})$ provided by Proposition \ref{avoid-noncomp-red-version-2} where $M$ is given by:
\[
  M:=\bigcup_{\alpha,\beta} r_0(M(\alpha,\beta)).
\]
Suppose that the map $r_{\sigma,\tau}^{\omega_1}:\mathcal K_{N+1} \to \Delta_{N-1}^{\ft}$ is induced by the restriction map of the completely reducible connections in the moduli space $\mathcal M_{p}^\omega(\mathbb{X}_N(l),w_{\sigma,\tau};\Delta_{N-1}^\ft)$ where $\ind(p)= -2(N-1)$. By choosing $\omega_1$ small enough, we can assume that $r_{\sigma,\tau}^{-1}(q)$ and $(r_{\sigma,\tau}^{\omega_1})^{-1}(q)$, for any $q\in M$, are arbitrarily close to each other and have the same number of elements. The perturbations $\omega_0$ and $\omega_1$ define a perturbation of the ASD equation over the face $\overline F_0$ of $\mathcal K_{N+2}$. We can use the familiar inductive construction of Subsection \ref{per-family} to extend this perturbation to an admissible perturbation $\omega$ over $\mathcal K_{N+2}$. As a consequence of Proposition \ref{avoid-noncomp-red-version-2}, if we use this perturbation $\omega$ in the definition of the moduli space for the pair $(\bbW^l_{l+N+1},c_{\sigma,\tau})$, then it satisfies all the required properties.

\bibliography{references}
\bibliographystyle{hplain}
\end{document}